\documentclass[reqno,hypertex,dvipdfmx]{amsart}
\RequirePackage{amsthm,amssymb,bm}

\RequirePackage[mathscr]{eucal} 
    \let\go=\mathfrak
  \let\mc=\mathscr 
  \let\bb=\mathbb \let\mb=\mathbf

\usepackage{multicol}
 \usepackage[makeindex,nonewpage]{imakeidx}
 \makeindex[name=notation,title={Notation},program=makeindex,%
        options=-c,columns=4,columnsep=20pt]
 \def\indn#1{\index[notation]{#1}}
\makeindex[name=terminology,title={Terminology},columnsep=20pt,options=-c]
  \def\indt#1{\index[terminology]{#1}}

\makeatletter

\def\defcd#1{\expandafter\def\csname cd#1\endcsname}

 \def\activeat#1{\csname @#1\endcsname}
 \def\def@#1{\expandafter\def\csname @#1\endcsname}
 {\catcode`\@=\active \gdef@{\activeat}}

\let\ssize\scriptstyle
\newdimen\excd	\excd.2326ex

 \def\requalfill{\cleaders\hbox{$\mkern-2mu\mathord=\mkern-2mu$}\hfill
  \mkern-6mu\mathord=$}
 \def\eqfill{$\m@th\mathord=\mkern-6mu\requalfill}
 \def\deffill{\hbox{$:=$}$\m@th\mkern-6mu\requalfill}
 \def\fiberbox{\hbox{$\vcenter{\hrule\hbox{\vrule\kern1ex
     \vbox{\kern1.2ex}\vrule}\hrule}$}}

 \newdimen\arrwd 
  \newdimen\minCDarrwd \minCDarrwd=2.5pc
   \setbox\z@\hbox{$\rightarrow\,$} \minCDarrwd=\wd\z@
 	
 \def\findarrwd#1#2#3{\arrwd=#3%
  \setbox\z@\hbox{$\ssize\;{#1}\;\;$}%
  \setbox\@ne\hbox{$\ssize\;{#2}\;\;$}%
  \ifdim\wd\z@>\arrwd \arrwd=\wd\z@\fi
  \ifdim\wd\@ne>\arrwd \arrwd=\wd\@ne\fi}
 \newdimen\arrowsp\arrowsp=0.375em  	
 \def\findCDarrwd#1#2{\findarrwd{#1}{#2}{\minCDarrwd}
    \advance\arrwd by 2\arrowsp}
 \newdimen\minarrwd 
 \setbox\z@\hbox{$\longrightarrow$} \minarrwd=\wd\z@

 \def\harrow#1#2#3#4{{\minarrwd=#1\minarrwd
   \findarrwd{#2}{#3}{\minarrwd}\kern\arrowsp
    \mathrel{\mathop{\hbox to\arrwd{#4}}\limits^{#2}_{#3}}\kern\arrowsp}}

\def@]#1>#2>#3>{\harrow{#1}{#2}{#3}\rightarrowfill}
 \def@>#1>#2>{\harrow1{#1}{#2}\rightarrowfill}
 \def@<#1<#2<{\harrow1{#1}{#2}\leftarrowfill}
 \def@={\harrow1{}{}\eqfill}
 \def@:#1={\harrow1{}{}\deffill}
 \def@ N#1N#2N{\vCDarrow{#1}{#2}\UpDownarrow}
 \def\UpDownarrow{\uparrow\,\Big\downarrow}

\def\hookrightarrowfill{\hbox{$\lhook\joinrel$}\rightarrowfill}
\def@{hk>}#1>#2>{\harrow1{#1}{#2}\hookrightarrowfill}
\def@ c>#1>#2>{\harrow1{#1}{#2}\hookrightarrowfill}
\def\hookleftarrowfill{\leftarrowfill\hbox{$\joinrel\rhook$}}
\def@{hk<}#1<#2<{\harrow1{#1}{#2}\hookleftarrowfill}

 \def@.{\ifodd\row\relax\harrow1{}{}\hfill
   \else\vCDarrow{}{}.\fi}
 \def@|{\vCDarrow{}{}\Vert}
 \def@ V#1V#2V{\vCDarrow{#1}{#2}\downarrow}
 \def@ A#1A#2A{\vCDarrow{#1}{#2}\uparrow}
 \def@(#1){\arrwd=\csname col\the\col\endcsname\relax
   \hbox to 0pt{\hbox to \arrwd{\hss$\vcenter{\hbox{$#1$}}$\hss}\hss}}

 \defcd]#1>#2>#3>{\harrow{#1}{#2}{#3}\rightarrowfill}
 \defcd>#1>#2>{\harrow1{#1}{#2}\rightarrowfill}
 \defcd<#1<#2<{\harrow1{#1}{#2}\leftarrowfill}
 \defcd={\harrow1{}{}\eqfill}
 \defcd:#1={\harrow1{}{}\deffill}
 \defcd N#1N#2N{\vCDarrow{#1}{#2}\UpDownarrow}
 \def\UpDownarrow{\uparrow\,\Big\downarrow}

\def\hookrightarrowfill{\hbox{$\lhook\joinrel$}\rightarrowfill}
\defcd{hk>}#1>#2>{\harrow1{#1}{#2}\hookrightarrowfill}
\defcd c>#1>#2>{\harrow1{#1}{#2}\hookrightarrowfill}
\def\hookleftarrowfill{\leftarrowfill\hbox{$\joinrel\rhook$}}
\defcd{hk<}#1<#2<{\harrow1{#1}{#2}\hookleftarrowfill}

 \defcd.{\ifodd\row\relax\harrow1{}{}\hfill
   \else\vCDarrow{}{}.\fi}
 \defcd|{\vCDarrow{}{}\Vert}
 \defcd V#1V#2V{\vCDarrow{#1}{#2}\downarrow}
 \defcd A#1A#2A{\vCDarrow{#1}{#2}\uparrow}
 \defcd(#1){\arrwd=\csname col\the\col\endcsname\relax
   \hbox to 0pt{\hbox to \arrwd{\hss$\vcenter{\hbox{$#1$}}$\hss}\hss}}

 \def\squash#1{\setbox\z@=\hbox{$#1$}\finsm@@sh}
\def\finsm@@sh{\ifnum\row>1\ht\z@\z@\fi \dp\z@\z@ \box\z@}

 \newcount\row \newcount\col \newcount\numcol \newcount\arrspan
 \newdimen\vrtxhalfwd  \newbox\tempbox

 \def\innernewdimen{\alloc@1\dimen\dimendef\insc@unt}
 \def\measureinit{\col=1\vrtxhalfwd=0pt\arrspan=1\arrwd=0pt 
   \setbox\tempbox=\hbox\bgroup$}
 \def\setinit{\col=1\hbox\bgroup$\ifodd\row
   \kern\csname col1\endcsname
   \kern-\csname row\the\row col1\endcsname\fi}
 \def\findvrtxhalfsum{$\egroup
  \expandafter\innernewdimen\csname row\the\row col\the\col\endcsname
  \global\csname row\the\row col\the\col\endcsname=\vrtxhalfwd
  \vrtxhalfwd=0.5\wd\tempbox
  \global\advance\csname row\the\row col\the\col\endcsname by \vrtxhalfwd 
  \advance\arrwd by \csname row\the\row col\the\col\endcsname
  \divide\arrwd by \arrspan
  \loop\ifnum\col>\numcol \numcol=\col%
     \expandafter\innernewdimen \csname col\the\col\endcsname
     \global\csname col\the\col\endcsname=\arrwd
   \else \ifdim\arrwd >\csname col\the\col\endcsname
      \global\csname col\the\col\endcsname=\arrwd\fi\fi
   \advance\arrspan by -1 %
   \ifnum\arrspan>0 \repeat}
 \def\setCDarrow#1#2#3#4{\advance\col by 1 \arrspan=#1 
    \arrwd= -\csname row\the\row col\the\col\endcsname\relax
    \loop\advance\arrwd by \csname col\the\col\endcsname
     \ifnum\arrspan>1 \advance\col by 1 \advance\arrspan by -1%
     \repeat
    \squash{\mathop{
     \hbox to\arrwd{\kern\arrowsp#4\kern\arrowsp}}\limits^{#2}_{#3}}}
 \def\measureCDarrow#1#2#3#4{\findvrtxhalfsum\advance\col by 1%
   \arrspan=#1\findCDarrwd{#2}{#3}%
    \setbox\tempbox=\hbox\bgroup$}
 \def\vCDarrow#1#2#3{\kern\csname col\the\col\endcsname
    \hbox to 0pt{\hss$\vcenter{\llap{$\ssize#1$}}%
     \Big#3\vcenter{\rlap{$\ssize#2$}}$\hss}\advance\col by 1}

 \def\setCD{\makeatother %
  \def\harrow{\setCDarrow}%
  \def\\{$\egroup\advance\row by 1\setinit}
  \m@th\lineskip3\ex@\lineskiplimit3\excd \row=1\setinit}
 \def\endsetCD{$\egroup}
 \def\dr@p#1\\{\findvrtxhalfsum\advance\row by 2 \measureinit}
 \def\measure{\bgroup
  \def\harrow{\measureCDarrow}%
  \def\\##1{\ifx##1\endmeasure\endmeasure\else\expandafter\dr@p\fi}%
  \row=1\numcol=0\measureinit}
 \def\endmeasure{\findvrtxhalfsum\egroup}

\newskip\aboveCDskip \aboveCDskip=0pt
\newskip\belowCDskip \belowCDskip=0pt
 \newcount\savedcount	
 \def\LCD#1\end{\savedcount=\count11
   \measure#1\endmeasure
   \vcenter{\kern\aboveCDskip\setCD#1\endsetCD\kern\belowCDskip}%
   \global\count11=\savedcount\end}

 \newenvironment{CD}{\let\at=@\catcode`\@=\active\LCD}{\catcode`\@=12\relax}
 
\makeatother

\let\:=\colon
\let\?=\overline

\let\bu=\bullet 

\let\dg=\dagger
\let\di=\diamond
\let\ft=\flat
\let\into=\hookrightarrow
\let\la=\langle \let\ra=\rangle

\let\onto=\twoheadrightarrow
\let\ox=\otimes
\let\sh=\sharp
\let\sp=\spadesuit
\let\To=\longrightarrow
\let\ts=\textstyle
\let\ve=\varepsilon \let\vf=\varphi
\let\wh=\widehat
\let\wt=\widetilde 
\let\x=\times
\let\xto=\xrightarrow

\def\0{^{\,0}}
\def\varstrut{\vrule height10.5pt depth3.5pt width0pt}
\def\Z{\bb Z}
\def\tq#1{\text{\quad#1\quad}}  \def\and{\tq{and}}
\def\Der{\mathop{\rm Der}\nolimits}
\def\tqn#1{\text{\quad#1\enspace}}

\def\(#1){{\rm(#1)}}

\def\risom{\buildrel\sim\over{\smashedlongrightarrow}}
  \def\smashedlongrightarrow{\setbox0=\hbox{$\longrightarrow$}\ht0=1pt\box0}
 \def\emdash{\unskip\thinspace\penalty10000---\penalty-500\ignorespaces
  \thinspace}
 \def\car{\mathop{\rm char}}

\def\Grass{\mathop{\bb G\rm rass}\nolimits}
\def\Hilb{\mathop{\rm  Hilb}}
\def\Hom{\mathop{\rm Hom}\nolimits}
 
 \def\sHom{\mathop{\mc{H}\kern-1pt\textit{om\/}}}
\def\Im{\mathop{\rm Im}} \def\Coim{\mathop{\rm Coim}}
\def\im{\mathop{\rm im}} \def\coim{\mathop{\rm coim}}

 \def\implies{$\Rightarrow$}
\def\IP{\mathop{\bb P}}
\def\Jac{\mathop{\rm Jac}}
\def\Ker{\mathop{\rm Ker}} 
 \def\Cok{\mathop{\rm Cok}}  
\def\ker{\mathop{\rm ker}} \def\cok{\mathop{\rm cok}}
\def\length{\mathop{\rm length}}

\def\Proj{\mathop{\rm Proj}}
\def\Quot{\mathop{\rm {\bb Q}uot}\nolimits}

\def\rank{\mathop{\rm rank}}

\def\Rees{\mathop{\rm Rees}}
\def\sm{_{\rm sm}}
\def\Soc{\mathop{\rm Soc}\nolimits}
\def\Spec{\mathop{\rm Spec}}
\def\Supp{\mathop{\rm Supp}}

 \RequirePackage[hypertex,dvipdfmx]{hyperref} 
 \hypersetup{colorlinks=true, urlcolor=green, citecolor=red, linkcolor=blue}

 \usepackage[abbrev]{amsrefs}

\def\sbsno{(\arabic{section}.\arabic{subsection})\enspace}
\def\defspec#1{\def\headspec{#1}%
  \ifx\headspec\empty 
  \else{\unkern\enspace(#1)}\fi
}       

\newtheoremstyle{italics}
  {6pt}
  {6pt}
  {\itshape}
  {}
  {\bfseries}
  {.}
  {.5em}
  {\sbsno \thmname{#1}{\rm \defspec{#3}}}

\newtheoremstyle{roman}
  {6pt}
  {6pt}
  {\rmfamily}
  {}
  {\bfseries}
  {.}
  {.5em}
  {\sbsno \thmname{#1}\defspec{#3}}

\theoremstyle{italics}
 \newtheorem{lemma}[subsection]{Lemma}
 
 \newtheorem{corollary}[subsection]{Corollary}
 \newtheorem{proposition}[subsection]{Proposition}
 \newtheorem{theorem}[subsection]{Theorem}

\theoremstyle{roman}
 \newtheorem{definition}[subsection]{Definition}
 \newtheorem{example}[subsection]{Example}
 
 \newtheorem{remark}[subsection]{Remark}
 \newtheorem{setup}[subsection]{Setup}
 \newtheorem{sbs}[subsection]{} 

\numberwithin{equation}{subsection}

    \title{Macaulay duality and its geometry}

\author[S. L. Kleiman]{Steven L. Kleiman}
 \address
 {Dept.\ of Math., 2-172 MIT\\
 77 Mass.\ Ave.\\
 Cambridge, MA 02139, USA}
 \email{Kleiman@math.MIT.edu}

\author[J. O. Kleppe]{Jan O. Kleppe}
 \address
 {Faculty of Technology, Art and Design\\
  Oslo Metropolitan University\\
  St. Olavs Plass\\
  NO-0130 Oslo, Norway}
 \email{JanOddvar.Kleppe@oslomet.no}

\date{\today}

\thanks{Both authors are deeply grateful to Tony Iarrobino for many
discussions in person and in email about the content of this article and
for providing many pertinent references, in particular about smoothable
quotients and elementary components, which are discussed in
\eqref{exElemCpts}.
  The second author also thanks
Arvid Siqveland for discussions about deformations of modules.
  Finally, both authors are very grateful for the referee's many
thoughtful suggestions for improving the paper.}


 \indexsetup{noclearpage,level=\section*,toclevel=section,firstpagestyle=headings}

\begin{document}
\begin{abstract}
 Macaulay Duality, between quotients of a polynomial ring over a field,
annihilated by powers of the variables, and finitely generated
submodules of the ring's graded dual, is generalized over any Noetherian
ring, and used to provide isomorphisms between the subschemes of the
Hilbert scheme parameterizing various sorts of these quotients, and the
corresponding subschemes of the Quot scheme of the dual.  Thus notably
the locus of recursively compressed algebras of permissible socle type
is proved to be covered by open subschemes, each one isomorphic to an
open subscheme of a certain affine space.  Moreover, the polynomial
variables are weighted, the polynomial ring is replaced by a graded
module, and attention is paid to induced filtrations and gradings.
Furthermore, a similar theory is developed for (relatively) maximal
quotients of a graded Gorenstein Artinian algebra.
 \end{abstract}

\maketitle

\vskip 0pt plus 0pt minus 3pt

 \tableofcontents

\section{Introduction}\label{intro}

This introduction surveys previous work by many, explains our approach,
and summarizes our advances.

 In 1972, Iarrobino \cite{IM15} initiated the study of reducible
punctual Hilbert schemes.  Since then, many researchers have carried on,
studying various loci of thick points, often with common support.  Thus
certain subtle geometric issues arose, for which the present article
aims to give the first proper treatment in natural generality.

A key tool has been Macaulay duality.  It provides a means to transform
loci in the local punctual Hilbert scheme into more tractable loci in a
related Quot scheme.  However, previous treatments of Macaulay duality
are inadequate to the job, since (with a few limited exceptions) they're
developed over a base field, and cannot properly handle flat families.
Thus in the present article, our first goal is to develop a suitable
version \eqref{thGMD3} of Macaulay duality over any Noetherian base
ring.

Our main goal stems from Iarrobino and Emsalem's 1978 landmark paper
\cite{IE78}.  They fixed a polynomial ring $P$ in $r$ variables over an
algebraically closed field $k$.  Notably, they considered local
Gorenstein quotients $C$ of $P$ with Hilbert vector $(1,r,r,1)$.  Using
Macaulay duality, they found that the homogeneous $C$ fill an open
subset of the projective space on the cubic forms, and that over it, the
adically filtered $C$ fill an affine-space bundle of fiber dimension
$\binom r2$, whose structure map takes $C$ to its associated graded
algebra.

Our main goal is to treat Iarrobino's remarkable 1984 generalization in
\cite{Iar84} of the preceding results.  The original $C$ are generalized
by the {\it compressed\/} $C$; these are the $C$ whose length is maximal
 \indt{compressed@compressed $C$}
 among quotients that are annihilated by a power of each variable and
are of fixed ``permissible'' socle type.  Iarrobino discovered and
Fr\"{o}berg and Laksov reproved, see \eqref{sbPerm} below, that,
strikingly, the length of a compressed $C$ can be computed a priori from
its socle type.  This basic result grew out of \cite{IE78}*{Thm.\,3.31,
p.\,173}; related work is discussed in the note at the bottom of p.\,350
in \cite{Iar84}.  This result is assumed in the present paper.  Rather,
we give the geometry a proper treatment.

We fix two flaws in \cite{Iar84}, beyond the lack of use of a suitable
Macaulay duality.  The first concerns the locus (in the Hilbert scheme)
of compressed homogeneous $C$.  It is not true, as asserted in
\cite{Iar84}*{Prop.\,3.6, p.\,356}, that every $C$ remains compressed
under recursion on the diameter of the socle type; see \eqref{reErr}.
So to reach our main goal, we work with {\it recursively compressed\/} 
 \indt{recursively compressed}
 $C$; these are the compressed $C$ that do remain compressed under
recursion.  It is their locus that is covered by open subschemes of
affine space; see \eqref{prSmOp}(2).  However, we show in
\eqref{regenstype} on the basis of the work of Iarrobino and of
Fr\"{o}berg and Laksov mentioned above that a general $C$ of socle type
bounded by $\mb t$ is, in fact, recursively compressed of socle type
$\mb t$.

The second flaw is the poor justification of the description in
\cite{Iar84}*{Prop.\,3.11, p.\,359} of the retraction from the locus of
filtered $C$ to that of homogeneous $C$.  That justification proceeds by
induction on the minimal number of generators of the dual $C^*$.
However, adding new generators, one after the other, leads not to a
single quotient $C$, but to a flag of quotients.  Therefore, we provide
a noninductive proof in \eqref{thFHdim}.

The theory of compressed Artinian local rings $C$ has been extended by
many authors in various directions.  In his 1985 thesis
\cite{MirThesis}, Miri studied such $C$ that are quotients of direct
sums of shifts $B := \bigoplus_q P(q)^{\oplus \mb t(q)}$.  However, in
his 1993 paper \cite{MiriCA21}, which is based on \cite{MirThesis}, he
took $B := P^{\oplus t}$ with $t\ge 1$.  He focused on the $C$ with
cyclic duals of maximal length of socle type $\mb t$.  A stronger
version of his existence theorem \cite{MiriCA21}*{Prp.\,2.8, p.\,2846}
with $B := \bigoplus_q P(q)^{\oplus \mb t(q)}$ is given in
\eqref{prMB}(1) below with a simpler proof.  He also described the
varieties of graded $C$ and of filtered $C$ in
\cite{MiriCA21}*{Thm.\,3.1, p.\,2853}; his descriptions are special
cases of \eqref{prSmOp}(2) and \eqref{coFdim}.

In 2000, Boij studied graded level quotients $C$ of $B := P(s)^{\oplus
t}$ for $s,\,t \ge 1$.  A version of of his existence theorem is given
for comparison with Miri's in \eqref{prMB}(2).

In 2001, Cho and Iarrobino \cite{C-IJA241} initiated the study of
(relatively) compressed quotients $C$ of a proper quotient $A$ of $P$;
these are the $C$ whose length is maximal among quotients of fixed
``permissible'' socle type.  They gave examples of $A$ where this
maximal length is not the a priori bound computed from the socle type and
the Hilbert function of $A$.  A few years later, they and others gave
additional examples.  There's more discussion of this matter near the
end of the introduction.

In 2014, Rossi and \c{S}ega \cite{RS-AM259} studied the Poincar\'e
series of finitely generated modules over compressed local Gorenstein
$C$; they treated mixed characteristic $C$, a notable advance.  In 2018,
Kustin, \c{S}ega, and Vraciu \cite{KSV-JA505} extended that work to
compressed local $C$ of arbitrary permissible socle type.  Both papers
recall that such $C$ exist over an infinite field, but otherwise, they
study such a $C$, assuming $C$ exists.  However, since a mixed
characteristic $C$ isn't flat over any base ring $k$, such a $C$ isn't
covered by our theory.

In our work, $P$ is replaced by a finitely generated, graded algebra $A
:= \bigoplus_{n\ge0}A_n$ over an arbitrary Noetherian base ring $k$;
from Section 4 on, $A_0 = k$.  Further, $C$ is now a quotient of a
finitely generated graded $A$-module $B = \bigoplus_{n\ge 0}B_n$ with
$B_n$ a locally free $k$-module of finite rank $\mb b(n)$.  For example,
$k$ could be a field, $A$ a {\it weighted\/} polynomial ring, and $B =
A$; see \eqref{exSymNonArt} and \eqref{exArt}, which illustrate an
advantage gained from a nonstandard weighting.  However, unless
explicitly stated otherwise, a polynomi ring has the standard grading,
in which all the variables have weight 1.  Furthermore, it may well be
useful to generalize our work to the case where $B$ is filtered,
especially in the linkage results \eqref{coLink} and \eqref{prHflg} when
$C$ is filtered, but not graded.

In general, the $C$ of interest are locally free $k$-modules of finite
rank.  Below, such a $C$ is termed $k$-{\it Artinian}.  Of course, if
 \indt{kArt@$k$-Artinian}
$k$ isn't Artinian, then neither is $C$.  However, $C$ can be viewed as
providing a flat family of Artinian quotients.

To treat the various loci $L$ of Artinian quotients and the maps between
them, we use the rudiments of Grothendieck's theory of representable
functors.  The first job then is to take a $k$-scheme $T$ of finite type
and to identify the $T$-points of $L$, namely, the $k$-maps $\tau\:T\to
L$.  Below, each $L$ of interest turns out to parameterize a universal
family $\mc L$ of quotients.  Then $\tau^*\mc L$ is a family
parameterized by $T$; conversely, every family parameterized by $T$ is
obtained via pullback along a unique $\tau$.  Thus the $T$-points $\tau$
are given by certain quotients $\mc C$ on $T$. The first job is to
identify, a priori, these $\mc C$.  Doing so means generalizing the
definition of the locus $L$ in terms of $T$ and $\mc C$.  Here
$T$-flatness plays a central role.  Further, it suffices to use affine
$T$.  The second job is to give $L$ a scheme structure and to construct
the universal family $\mc L$; in other words, the job is to represent
the functor of these $\mc C$.

To do the second job, we use Grothendieck's theory of {\it flattening
stratification\/}
 \indt{flattening stratification}
 (as it's named in \cite{CAS}*{Lect.\,8}).  The key
tool is Grothendieck's Lem.\,3.6 on p.\,221-15 in \cite{SB221}: given a
coherent sheaf $\mc F$ on a locally Noetherian scheme $Q$ and an integer
$r$, there's a (unique) subscheme $Z$ of $Q$ such that a map $\tau\:
T\to Q$ factors through $Z$ iff $\tau^*\mc F$ is locally free of rank
$r$.  (To prove it, let $F_r$ be the closed subscheme of $Q$ defined by
the $(r-1)$st Fitting Ideal of $\mc F$, and set $Z = F_r - F_{r-1}$.)
Historically speaking, flattening stratification was introduced into the
study of loci of Artinian algebras by the second author in \cite{PGor}
and \cite{MaxGor}.

For example, filter $B$ by $F^nB := \bigoplus_{p\ge n}B_p$, and consider
the locus of filtered, $k$-Artinian quotients of $B$ with given Hilbert
function $\mb h$.  What's the value of the associated functor at $T :=
\Spec(K)$ with $K$ a $k$-algebra?  It's the set of $A\ox_kK$-quotients
$C$ of $B\ox_kK$ with the induced filtration such that $F^nC/F^{n+1}C$
is locally free over $K$ of rank $\mb h(n)$ for all $n$.  In
\eqref{thRep}, we prove these sets form a functor.

Let's represent this functor.  Set $\ell := \sum_p\mb h(p)$ and $X :=
\Spec(A)$.  Denote the $\mc O_X$-module associated to $B$ by $\wt{B}$.
Set $\smash{Q := \Quot_{\wt{B}/X/k}^\ell}$, and denote by $\mc Q$ the
quotient of $B_Q$ whose associated $\mc O_{X_Q}$-module is the universal
quotient of $\wt B_Q$.  Given $q$, Grothendieck's lemma yields a
subscheme $Q_q$ of $Q$ such that a map $\tau\:T\to Q$ factors through
$Q_q$ iff $\tau^*(\mc Q/F^{q}\mc Q) := (\mc Q/F^{q}\mc Q)\ox K$ is
locally free of rank $\sum_{p < q}\mb h(p)$.  Then the functor is
represented by $\bb F\Psi_B^{\mb h} := \bigcap_{q} Q_q$ and the pullback
 \indn{FbbPsi@$\bb F\Psi_B^{\mb h}$}
 of $\mc Q$; see \eqref{thRep}.  Note it's not enough to render all
$\tau^*(F^q\mc Q/F^{q+1}\mc Q)$ flat, as possibly $\tau^*F^{q}\mc Q \neq
F^{q}\tau^*\mc Q$, whereas $\tau^*(\mc Q/F^{q}\mc Q) = \tau^*\mc
Q\big/F^{q}\tau^*\mc Q$; a simple example was given by Granger
\cite{Granger}*{p.\,3, top}.

Next, consider the sublocus of homogeneous quotients.  Its subfunctor
has, as value at $T$, the subset of $C = \bigoplus_p C_p$ with $C_p$
locally free of rank $\mb h(p)$ for all $p$.  To represent it, set $Z :=
\bb F\Psi_{B}^{\mb h}$.  Define $\mc S$ by $\mc Q =: \wt B_Q/\mc S$.
Consider the pullback $\mc Q_Z$ of $\mc Q$ and its associated graded
quotient $G_\bu (\mc Q_Z)$.  Then $\mc Q_Z = \wt B_Z/\mc S_Z$, and
$G_\bu (\mc Q_Z) = \wt B_Z/G_\bu(\mc S_Z)$.  Set $\mc R := \wt B_Z/(\mc
S_Z + G_\bu(\mc S_Z))$.  Grothendieck's lemma yields a subscheme $\bb
H\Psi_{ B}^{\mb h}$ of $Z$ such that a map $T\to Z$ factors through $\bb
H\Psi_{ B}^{\mb h}$ iff $\mc R_T$ is locally free of rank $\ell$; so iff
both $\mc Q_T \onto \mc R_T$ and $G_\bu(\mc Q_T) \onto \mc R_T$ are
isomorphisms, see \eqref{thRep}.  Thus the subfunctor is represented by
$\bb H\Psi_{ B}^{\mb h}$ and the pullback of $\mc Q$.
 \indn{HbbPsi@$\bb H\Psi_B^{\mb h}$}
 
Alternatively, form the maps $\beta_n\: (\wt B_Q)_n \to \mc Q$.
Grothendieck's lemma yields a subscheme $Y_n$ of $Q$ such that a map
$\tau\:T\to Q$ factors through $Y_n$ iff $\tau^*\Cok(\beta_n)$ is
locally free of rank $\ell - \mb h(n)$.  Then the subfunctor is
represented by $\bb H\Psi_{ B}^{\mb h} := \bigcap_{n} Y_n$ and the
pullback of $\mc Q$; see \eqref{thRep} again.  Note it's not enough to
render $\Ker(\beta_n)$ locally free, as forming a kernel does not always
commute with pullback, whereas forming a cokernel does.

The previous four paragraphs fix and generalize the second author's 1998
proof of \cite{PGor}*{Thm.\,1.1, p.\,609} and his 2006 proof of
\cite{MaxGor}*{Prp.\,9, p.\,3145}.  (Note that $\bb H\Psi_B^h$ is
denoted by ${\rm GradAlg}(h)$ in \cite{PGor} and \cite{MaxGor}.)  The
particular case of homogeneous quotients also results from the 2004
general theory of Haiman and Sturmfels, see \cite{HS04}*{6.2, p.\,766},
but both of our proofs above are simpler and more direct.

There are two natural maps: a closed embedding $\bb H\Psi_{ B}^{\mb h}
\into \bb F\Psi_{ B}^{\mb h}$ and a retraction $\bb F\Psi_{ B}^{\mb
h}\onto \bb H\Psi_{ B}^{\mb h}$.  They represent maps of functors: the
inclusion of the homogeneous $C$ into the filtered $C$, and the
retraction taking a filtered $C$ to its associated graded quotient
$G_\bu C$.  A priori, $\bb H\Psi_{ B}^{\mb h} \into \bb F\Psi_{ B}^{\mb
h}$ is a monomorphism, but a posteriori, it's a closed embedding, as the
schemes are separated and a retraction exists.  See \eqref{coIncRetr}.

Jelisiejew kindly pointed out in email to the authors the following
interesting way to construct $\bb F\Psi_{ B}^{\mb h}$ and $\bb H\Psi_{
B}^{\mb h}$ and to study the retraction map $\bb F\Psi_{ B}^{\mb h}\onto
\bb H\Psi_{ B}^{\mb h}$.  This way involves the Bia{\l}ynicki-Birula
decomposition of the above Quot scheme $Q$ under the natural action of
the multiplicative group.  Namely, $\bb F\Psi_{ B}^{\mb h}$ appears as a
disjoint union of connected components of the decomposition, and $\bb
H\Psi_{ B}^{\mb h}$ is its subscheme of fixed points.  Consequently, the
retraction map is affine; in fact, near a point of $\bb H\Psi_{ B}^{\mb
h}$ where the map is smooth, it's an affine-space fibration.
 Please see Prp.\,3.1 and Rmk.\,3.2 on p.\,255 of \cite{JJ-JLMS2019},
and Thm.\,1.1 and Prp.\,5.3(3) and Sec.\,7 of \cite{JJ-LS}, and Sec.\,5
of \cite{JJ-KS}; note that in these papers, $\bb F\Psi_{ B}^{\mb h}$
would be denoted by $\bigl(\bb H\Psi_{ B}^{\mb h}\bigr)^+$.  These
papers place various restrictive hypotheses on $A$, $B$, and $S$;
however, Jelisiejew said he believes that, without change, the proofs
cover our setup.  We do not pursue these ideas here, except indirectly
in \eqref{exElemCpts} through references to \cite{JJ-JLMS2019}.

To analyze the structure of $\bb H\Psi_{ B}^{\mb h}$ and of the
retraction $\bb F\Psi_{ B}^{\mb h}\onto \bb H\Psi_{ B}^{\mb h}$,
following \cite{IE78} and \cite{Iar84}, we generalize Macaulay duality.
Macaulay introduced it in 1916 in Chap.\,IV of \cite{MacF}.  In essence
(cf.\ Robert's introduction in \cite{MacF}), Macaulay fixed a power
series ring $\wh P$ over a field $k$, and formed the span, $P^\dg$ say,
of the Laurent monomials $1/M$ for all monomials $M \in \wh P$.  He
viewed $P^\dg$ as an $\wh P$-module: the product of a monomial $L\in \wh
P$ with $1/M$ is $L/M$ if $L$ divides $M$ and is 0 if not.  He called a
$\psi\in \wh P$ and an $f\in P^\dg$ ``inverse to one another'' if
$\psi\cdot f = 0$, and he \cite{MacF}*{\ p.\,68} defined the ``inverse
system'' of an ideal $I$ containing powers of the variables to be the
set of all $f\in P^\dg$ inverse to all $\psi \in I$.  Thus $P^\dg$
is the graded dual of the polynomial ring $P$, and the inverse system is
 \indt{graded dual} \indt{inverse system}
just $\Hom_k(P/(I\cap P),\, k)$.  For instance, take $I := (X_1X_2,\,
X_1^2-X_2^3)$ in $\wh P := k[[X_1,X_2]]$; its inverse system in $P^\dg :=
k[X_1^{-1},X_2^{-1}]$ is generated by a single element $f :=
X_1^{-2}+X_2^{-3}$; see \eqref{exArt}.  The bijection between $I$, or
equivalently $\wh P/I$, and its inverse system is commonly known as
``Macaulay Duality'' (or ``apolarity''). Our generalization is
\eqref{thGMD3}.

Macaulay Duality is a special case of Matlis Duality, as introduced in
1958 in \cite{MatDu}*{Sec.\,4, pp.\,525--28} for an arbitrary complete
Noetherian local ring.  Indeed, $\Hom_k(\wh P/I,\, k) = \Hom_{\wh P}(\wh
P/I,\, P^\dg)$, see the end of \eqref{exPolyRg}; hence, $P^\dg$ is the
injective hull of $k$ by the uniqueness of dualizing modules.  In 2018,
in \cite{KES}*{Sec\,3, pp.\,15--18; Sec.\,5, pp.\,20--22}, Smith
developed a relative version of Matlis duality, for an arbitrary
Noetherian domain $k$ and a Noetherian $k$-algebra $P$ equipped with an
ideal $J$ such that $P/J$ is a finitely generated $k$-module.  For any
$P$-module $M$, set $M^\vee := \varinjlim \Hom_k(M/J^nM,\ k)$.  Smith
proved $M^\vee = \Hom_P(M,\,P^\vee)$, and when $P = k[[X_1,\dotsc,X_n]]$
and $J = \la X_1,\dotsc,X_n \ra$, then $P^\vee = H_J^n(P)$.  Notice
$P^\vee = P^\dg$.

Another version of relative Matlis Duality is given in \eqref{prRMD}.
In it, $P$ is replaced by a $k$-Artinian $k$-algebra $A$ where $k$ is a
Noetherian ring.  Set $A^* := \Hom_k(A,k)$.  This version asserts that
$A^*$ is a dualizing module for the $k$-Artinian $A$-modules.

In 1978, the first relative version of Macaulay Duality was stated,
without proof, by Emsalem \cite{PEpais}*{Prp.\,18, p.\,415} for a power
series ring over a local algebra of finite dimension over an
algebraically closed field $K$ of characteristic 0 in order to study
deformations of local Artinian $K$-algebras.

In 1998, in Rmk.\,1.9 on p.\,613 of \cite{PGor}, the second author
conjectured there should be a relative Macaulay duality yielding an
isomorphism of schemes $L\risom \Gamma$ where $\Gamma$ is the open
subscheme of $\bb H\Psi_P^{\mb h}$ of Gorenstein quotients of a
polynomial ring $P$ over an algebraically closed field $k$ of
characteristic not $2$ and where $L$ is the appropriate catalecticant
determinantal locus of $s$-forms, with $s := \sup\{\,n\mid\mb
h(n)\neq0\,\}$; see \eqref{sbLev} and \eqref{rmLev1}.  As evidence, he
constructed a canonical map of reductions $L_{\rm red}\to \Gamma_{\rm
red}$; it's bijective on $k$-points by Macaulay duality, and
corresponding $k$-points of $L$ and $\Gamma$ have the same tangent
space.  In 2006, he elaborated on this construction in the proof of
Thm.\,11 on p.\,3146 in \cite{MaxGor}, and in the following subsection,
Prb.\,12, he noted that the arguments there prove that the map is a
topological isomorphism and so the dimensions at corresponding points
are equal.  In 2007, on p.\,694 in \cite{LevAlg}, he extended this proof
to the case of level quotients.  In particular, the equality of the
tangent spaces and of the dimensions yield a smoothness criterion for
$L$.

In 2018 in \cite{JJ-VSP}*{Subsec.\,2.2, pp.\,271--78}, Jelisiejew
developed a relative Macaulay Duality for a power series ring with
coefficients in a Noetherian algebra over an algebraically closed field
$K$ in order to study loci of Gorenstein Artinian $K$-algebras.  The
Gorenstein hypothesis is weakened a bit in Sec.\,4.4 of his 2017 thesis
\cite{JJthesis}.  In the recent paper \cite{JJ-KS}*{Sec.\,4.1}, he and
\v{S}ivic developed a Macaulay Duality for a free module over a
polynomial ring with coefficients in an algebraically closed field.

Macaulay Duality was also generalized to higher-dimensional Gorenstein
quotients of the polynomial ring and the power series ring over a field
in 2017 by Elias and Rossi \cite{ER-AM}.  However, here we do not pursue
that direction.

Generalizing Macaulay Duality is our first goal.  Our treatment is
independent of earlier work, and goes further in the way explained next;
yet it's rather simple.  For it, recall we fix a Noetherian ring $k$, a
finitely generated graded $k$-algebra $A := \bigoplus_{n\ge0}A_n$, and a
$k$-flat finitely generated graded $A$-module $B = \bigoplus_{n\ge
0}B_n$.

Let $C$ be a $k$-module.  Set $C^* := \Hom_k(C,k)$.  If $C$ is an
 \indn{Cstar@$C^*$}
$A$-module, then so is $C^*$ via the adjoint action: $(af)\psi :=
f(a\psi)$ for all $a\in A$ and $f\in C^*$ and $\psi\in C$.

Assume $C$ is locally free of finite rank.  Then $C^*$ is too, and the
natural map $C\to (C^*)^*$ is an isomorphism.  In other words, $C\mapsto
C^*$ is a dualizing functor on these $C$.  Further, see
\eqref{*duality}, it induces a dualizing functor on the filtered $C$,
which restricts to a dualizing functor on the graded $C$, and which
respects the retraction that takes a filtered $C$ to its associated
graded module.

Moreover, see \eqref{sbApl} and \eqref{thGMD3}, form the graded dual
$B^\dg := \bigoplus_{n\ge0} (B_n)^*$.  Then the functor $C\mapsto C^*$
yields a bijection from the $k$-Artinian quotients $C$ of $B$ onto the
$k$-Artinian submodules $D$ of $B^\dg$ with $B^\dg/D$ flat.  Further,
this bijection respects the induced filtrations on $C$ and $D$, and it
respects the retractions taking $C$ and $D$ to their associated graded
modules.  To recast this bijection into the more traditional form of
apolarity, note that, if $C = B/I$, then $C^* = (0:_{B^\dg}I)$, where
$(0:_{B^\dg}I)$ is the apolar annihilator, which consists of all $f\in
B^\dg$ with $\psi\cdot f = 0$ for all $\psi \in I$.  Note $B^{\dg\dg} =
B$; so $D^* = B/(0:_BD)$.  Thus the bijection becomes $I\mapsto
(0:_{B^\dg}I)$ with inverse $D\mapsto (0:_BD)$.  That then is our
generalization of Macaulay Duality.

Our Macaulay Duality facilitates the study of the retraction $\bb
F\Psi_{ B}^{\mb h}\onto \bb H\Psi_{ B}^{\mb h}$ as follows.
Recall that, if $T := \Spec(K)$ with $K$ a $k$-algebra, then a $T$-point
of $\bb F\Psi_{ B}^{\mb h}$ is given by a filtered $A\ox_kK$-quotient
$C$ of $B\ox_kK$ with $F^nC/F^{n+1}C$ locally free over $K$ of rank $\mb
h(n)$.  Now, every map $K\to L$ yields a natural isomorphism
$\Hom_K(C,K)\ox_KL \risom \Hom_L(C\ox_KL,L)$.  Thus Macaulay Duality
provides an isomorphism from the functor of points of $\bb F\Psi_{
B}^{\mb h}$ onto the functor whose value at $T$ is the set of filtered
submodules $D$ of $B\ox_kK$ with $F^nD/F^{n+1}D$ locally free over $K$
of rank $\mb h^*(n)$ where $\mb h^*(n) := \mb h(-n)$ and with
$(B^\dg\ox_kK)/D$ flat over $K$.
 \indn{hmbzst@$\mb h^*(n)$}

The latter functor is automatically representable by a copy of $\bb
F\Psi_{ B}^{\mb h}$, say $\bb F\Delta_{ B^\dg}^{\mb h^*}$, and by
a subsheaf of the pullback of $ B^\dg$ to $\bb F\Delta_{
B^\dg}^{\mb h^*}$; this subsheaf is obtained from the universal quotient
on $\bb F\Psi_{ B}^{\mb h}$ by ``Hom'ing" it into the structure
sheaf.  Similarly, the subfunctor of homogeneous $D$ is representable by
a copy of $\bb H\Psi_{ B}^{\mb h}$, say $\bb H\Delta_{ B^\dg}^{\mb
h^*}$.  Also, $\bb F\Psi_{ B}^{\mb h}\onto \bb H\Psi_{ B}^{\mb h}$
becomes the retraction $\bb F\Delta_{ B^\dg}^{\mb h^*} \onto \bb
H\Delta_{ B^\dg}^{\mb h^*}$ given by $D \mapsto G_\bu D$.

Thus a locus of quotients with given Hilbert function $\mb h$ can be
transformed into an isomorphic dual locus of submodules.  Now, our main
goal is to treat, via this duality, the geometry of recursively
compressed quotients.  Since they're defined in terms of the socle type
$\mb t$, the next step is to relativize and dualize this notion.

Assume $A/F^1A =k$.  Fix a $k$-Artinian $A$-module $C$, set $\Soc_k(C)
:= \Hom_A(k,\,C)$, and call this $A$-module the {\it $k$-socle\/} of
 \indt{kso@$k$-socle}
$C$.  If $k$ is a field, we recover the traditional notion.  In general,
$\Soc_k(C) = (C^*\ox_A k)^*$.  A form of this equation was proved by
Iarrobino \cite{Iar84}*{Lem.\,2.1, p.\,346}. It's important, because
$C^*\ox_A k$ is more useful than $\Soc_k(C)$, as by Nakayama's Lemma,
$C^*\ox_A k$ locally governs the generators of $C^*$.
  The equation holds, as both sides are equal to $\Hom_A(C^*,\,k^*)$;
the left, as $C\mapsto C^*$ is a dualizing functor, and the right, by
Adjoint Associativity; see \eqref{leSoc}.

For instance, take $P := k[X_1,X_2]$ and $C := P/I$ with $I :=
(X_1X_2,\, X_1^2- X_2^3)$.  Then, as is easy to see, $C^*$ is generated
by $X_2^{-3}+ X_1^{-2}$, and so $X_2^{-3}$ and $X_1^{-1}$ form a minimal
generating set of $G_\bu(C^*)$.  Thus we obtain $\Soc(G_\bu(C)) =
(X_1,\,X_2^3)$ without even finding $G_\bu(C)$; cf.\ \eqref{exArt}.

Assume $C^*\ox_A k$ is $k$-Artinian.  Then so is $\Soc_k(C)$, and their
Hilbert functions, say $\mb t^*$ and $\mb t$, satisfy $\mb t^*(p) = \mb
t(-p)$, as $\Soc_k(C) = (C^*\ox_A k)^*$.  Extending tradition, we call
$\mb t$ the {\it $k$-socle type\/} of $C$, and $\mb t^*$ the {\it local
 \indn{tmb@$\mb t$} \indn{tmbst@$\mb t^*$} 
 \indt{ksot@$k$-socle type}   \indt{local generator type}
generator type\/} of $C^*$.  Notice that, given any maximal ideal of
$k$, its complement contains an element $f$ such that the localization
$C^*_f$ is generated by a set with $\mb t^*(p)$ elements of degree $p$
for all $p$.

For every $k$-algebra $K$, the $A\ox_kK$-module $(C\ox_kK)^*$ is also of
local generator type $\mb t^*$ as $(C\ox_kK)^* = C^*\ox_kK$; so
$C\ox_kK$ is of $K$-socle type $\mb t$.  Set $T := \Spec(K)$.  It
follows that the subsets of the $T$-points of $\bb F\Psi_{ B}^{\mb h}$
and $\bb H\Psi_{ B}^{\mb h}$ given by those quotients of socle type $\mb
t$ form functors in $T$, which are compatible with retraction.
Similarly, the subsets of the $T$-points of $\bb F\Delta_{ B^\dg}^{\mb
h^*}$ and $\bb H\Delta_{ B^\dg}^{\mb h^*}$ given by those submodules of
local generator type $\mb t^*$ form functors in $T$, which are
compatible with retraction.

The former functors are represented by subschemes of $\bb F\Psi_{
B}^{\mb h}$ and $\bb H\Psi_{ B}^{\mb h}$, say $\bb F\Psi_{ B}^{\mb h,
\mb t}$ and $\bb H\Psi_{ B}^{\mb h, \mb t}$.  Indeed, flattening
 \indn{FbbPsit@$\bb F\Psi_{ B}^{\mb h,\mb t}$}
 \indn{HbbPsit@$\bb H\Psi_{ B}^{\mb h,\mb t}$}
 stratification yields $\bb F\Psi_{ B}^{\mb h, \mb t}$ much as above for
$\bb F\Psi_{ B}^{\mb h}$, and we can take $\bb H\Psi_{ B}^{\mb h, \mb t}
:= \bb F\Psi_{ B}^{\mb h, \mb t} \cap \bb H\Psi_{ B}^{\mb h}$; see
\eqref{prLevS}.  Now, Macaulay Duality gives isomorphisms from the
former functors to the latter.  So they're represented by copies of $\bb
F\Psi_{ B}^{\mb h, \mb t}$ and $\bb H\Psi_{ B}^{\mb h, \mb t}$, say $\bb
F\Delta_{ B^\dg}^{\mb h^*, \mb t}$ and $\bb H\Delta_{ B^\dg}^{\mb h^*,
\mb t}$,
 \indn{FbbDeltst@$\bb F\Delta_{ B^\dg}^{\mb h^*, \mb t}$}
 \indn{HbbDeltst@$\bb H\Delta_{ B^\dg}^{\mb h^*, \mb t}$}
 which are subschemes of $\bb F\Delta_{ B^\dg}^{\mb h^*}$ and
$\bb H\Delta_{ B^\dg}^{\mb h^*}$.  Also, $\bb H\Delta_{ B^\dg}^{\mb h^*,
\mb t} = \bb H\Delta_{ B^\dg}^{\mb h^*} \cap \bb F\Delta_{ B^\dg}^{\mb
h^*, \mb t}$.  Now, for $\mb h$ fixed, as $\mb t$ varies, the $\bb
F\Psi_{ B}^{\mb h, \mb t}$ and $\bb H\Psi_{ B}^{\mb h, \mb t}$ {\it
stratify\/} (that is, disjointly cover) $\bb F\Psi_{ B}^{\mb h}$ and
$\bb H\Psi_{ B}^{\mb h}$; see \eqref{prLevS}. So by Macaulay Duality,
the $\bb F\Delta_{ B^\dg}^{\mb h^*, \mb t}$ and $\bb H\Delta_{
B^\dg}^{\mb h^*, \mb t}$ stratify $\bb F\Delta_{ B^\dg}^{\mb h^*}$ and
$\bb H\Delta_{ B^\dg}^{\mb h^*}$.

From Section 7 on, many arguments proceed via recursion on the {\it
diameter\/} $(s-\bar s)$ of $\mb t$, where $s := \sup\{\,p\mid \mb t(p)
> 0\,\}$ and $\bar s := \inf\,\{\, p \mid \mb t(p) \neq 0\,\}$.  Note
that, if $C$ has Hilbert function $\mb h$, plainly $s := \sup\{\,p\mid
\mb h(p) > 0\,\}$; also, $\mb t(s) := \mb h(s)$ by \eqref{leSocLev}(2).
 \indt{diameter}  \indn{s@$s$}  \indn{sbar@$\bar s$} 

The recursion exits when $\bar s = s$, or equivalently, $\mb t(p) := 0$
for $p \neq s$.  Then we call $\mb t$ and $C$ {\it level}, and set 
 \indt{level}
 $\bb F\Lambda_{ B}^{\mb h} := \bb F\Psi_{ B}^{\mb h, \mb t}$ and
 $\bb H\Lambda_{ B}^{\mb h} := \bb H\Psi_{ B}^{\mb h, \mb t}$.
 \indn{FbbLam@$\bb F\Lambda_{ B}^{\mb h}$}
 \indn{HbbLam@$\bb H\Lambda_{ B}^{\mb h}$}
  In \eqref{prLevS}, there's an alternative construction of $\bb
F\Lambda_{ B}^{\mb h}$; it shows $\bb F\Lambda_{ B}^{\mb h}$ is an open
subscheme of $\bb F\Psi_{ B}^{\mb h}$.  So $\bb H\Lambda_{ B}^{\mb h}$
is an open subscheme of $\bb H\Psi_{ B}^{\mb h}$.  Sometimes, $ \bb
H\Psi_{ B}^{\mb h} - \bb H\Lambda_{ B}^{\mb h}$ contains $k$-points; an
example is the quotient $D^*$ in \eqref{exNotAll2}.  Also, the
retraction $\bb F\Psi_{ B}^{\mb h} \to \bb H\Psi_{ B}^{\mb h}$ needn't
carry $\bb F\Lambda_{ B}^{\mb h}$ into $\bb H\Lambda_{ B}^{\mb h}$;
indeed, \eqref{exArt} discusses a (standard) example of a
($k$-Gorenstein) level filtered $C$ with $G_\bu C$ not level.  However,
by \eqref{prGAA}, if $G_\bu C$ is level, then so is $C$; thus the
preimage of $\bb H\Lambda_{ B}^{\mb h}$ lies in $\bb F\Lambda_{ B}^{\mb
h}$.  Jelisiejew develops some of the theory of $k$-Gorenstein modules
for arbitrary $k$ in \cite{JJ-VSP}, especially in Subsection 2.2.

There's a third construction of $\bb H\Lambda_{ B}^{\mb h}$ studied in
\eqref{sbLev}--\eqref{coLevS}.  This time, $\bb H\Lambda_{ B}^{\mb h}$
arises, via flattening stratification, as a subscheme of the
Grassmann\-ian $\bb G$ of rank-$\mb t(s)$ locally free quotients of $
B_s$. Moreover, for $\mb t$ fixed, as $\mb h$ varies, the $\bb
H\Lambda_{ B}^{\mb h}$ stratify $\bb G$.  Hence, if $k$ is a domain,
then $\bb G$ is reduced and irreducible, and its generic point lies in
some $\bb H\Lambda_{ B}^{\mb h}$.  This $\bb H\Lambda_{ B}^{\mb h}$ is
reduced and irreducible, and it's covered by open subschemes, each one
isomorphic to an open subscheme of the affine space over $k$ of fiber
dimension $t(\mb b(s) - t)$ where $t := \mb h(s) = \mb t(s)$ and $\mb b$
is the Hilbert function of
 \indn{bbmb@$\mb b$} \indn{tmi@$t$}
$B$.  Also, by lower semicontinuity of rank, if $\bb H\Lambda_{ B}^{\mb
h^\di} \neq \emptyset$, then $\mb h^\di(p) \le \mb h(p)$ for all $p$; so
$\sum_p\mb h^\di(p) \le \sum_p\mb h(p)$, with equality iff $\mb h^\di =
\mb h$.  This third construction yields the isomorphism $L\risom \Gamma$
conjectured in Rmk.\,1.9 on p.\,613 of \cite{PGor} and discussed above.

For a moment, suppose $B = A$.  Then $C$ is a $k$-algebra, and $C^*$ is,
by \eqref{prRMD}, its dualizing module.  Moreover, $C^*$ is invertible
iff $C$ is level and $\mb t(s) = 1$.  If so, we say $C$ is {\it
$k$-Gorenstein}.  (Likely, Grothendieck coined the term ``Gorenstein'',
 \indt{kGor@$k$-Gorenstein}
and Serre in 1960 was the first to publish it, but historically it's
more correct to honor any one of several others, starting with Macaulay;
see \cite{Hun}*{pp.\,56, 61--62} and \cite{Eis}*{Sec.\,21.2,
pp.\,529--30}.)  The condition for $C$ to be $k$-Gorenstein doesn't
depend on its filtration, but that for $G_\bu C$ does; see \eqref{exArt}
for an example.  If $B$ is $k$-Artinian and $k$-Gorenstein, but $C$ just
$k$-Artinian, then the linked ideal $\Hom_A (C, A) \subset A$ and the
Macaulay dual $C^*$ of $C$ coincide up to tensor product with a certain
invertible $k$-module; see \eqref{coLink}.

In general, when $\bar s < s$, the recursion proceeds by replacing a
given homogeneous $k$-Artinian $C$ by the unique graded quotient $C'$
with {\it socle diameter\/} $s - \bar s -1$, provided $C'$ is $k$-flat,
and
 \indt{diameter}
when $\bar s +1 < s$, similarly replaceable.  Thus we are led to form
the sequence of quotients $\Lambda^mC$ of $C$ where $(\Lambda^mC)^*$ is
the submodule of $C^*$ generated by all the homogeneous elements of
degree at most $-m$.  So $C' = \Lambda^{\bar s+1}C$.  Furthermore, every
$\Lambda^mC$ must be $k$-flat; when so, we call $C$ \textit{multilevel},
 \indn{LamC@$\Lambda^mC$} \indt{multilevel}
and denote the Hilbert function of $\Lambda^mC$ by $\mb h_m$.  Note that
$\mb h_{\bar s}$ is the Hilbert function of $C$. 
 \indn{hmbzm@$\mb h_m$}
 Of course, when $k$ is a field, flatness is automatic, and so every $C$ is
multilevel.

Then all the $(\Lambda^mC)^*$ are flat too by \eqref{prMLev}(3).  Note
$\Lambda^m\Lambda^nC = \Lambda^nC$ for $m \le n$ and
$\Lambda^m\Lambda^nC = \Lambda^mC$ for $m \ge n$ for any $n$.  Take $n
:= \bar s+1$.  Thus $C'$ is multilevel of socle type $\mb t'$.  Also,
set $\mb h'_m := \mb h_{\bar s+1}$ for $m\le \bar s+1$ and $\mb h'_m :=
\mb h_m$ for $m\ge \bar s+1$.  Then $\mb h'_m$ is the Hilbert function
of $\Lambda^m C'$. We call $\{\mb h'_m\}$ the {\it attendant\/} to
 \indt{attendant}
$\{\mb h_m\}$.  As to a filtered module $C$, we call $C$
\textit{multilevel} if $G_\bu C$ is multilevel.

The algebraic theory of multilevel modules is treated in
\eqref{sbMLev}--\eqref{prMLevCon}, and the geometric theory, in
Secs.\,7--10.  Most of the effort concerns graded modules.  For example,
in \eqref{prMLevS}, we prove that, as $\{\mb h_m\}$ varies with $\mb
h_{0}$ fixed, the locus of all graded and all filtered multilevel
quotients of $B$ are parameterized by (possibly empty) subschemes $\bb
H\Lambda_{B}^{\{\mb h_m\}}$ and $\bb F\Lambda_{B}^{\{\mb h_m\}}$ of\/
$\bb H\Psi_{B}^{\mb h_0}$ and $\bb F\Psi_{B}^{\mb h_0}$, stratifying
 \indn{HbbLamhm@$\bb H\Lambda_{B}^{\{\mb h_m\}}$}
 \indn{FbbLamhm@$\bb F\Lambda_{B}^{\{\mb h_m\}}$}
 them.  We use flattening stratification to construct $\bb
H\Lambda_{B}^{\{\mb h_m\}}$, and we take $\bb F\Lambda_{B}^{\{\mb
h_m\}}$ to be the preimage of $\bb H\Lambda_{B}^{\{\mb h_m\}}$ under the
retraction $\bb F\Psi_{\mc B}^{\mb h_0} \onto \bb H\Psi_{B}^{\mb h_0}$.

A recursive construction of the $\bb H\Lambda_{B}^{\{\mb h_m\}}$ is
given in \eqref{sbMLSch}--\eqref{coMLevS}.  It's harder, but leads to
three new notions, treated in \eqref{sbRecMax}--\eqref{thRC}.  First, we
call $\{\mb h_m\}$ {\it recursively maximal\/} for $\mb t$ and $T/k$ if
 \indt{recursively maximal}
(1) when $\bar s < s$, the attendant $\{\mb h'_m\}$ is recursively
maximal for the attendant $\mb t'$; and (2) the $\mb h_m$ termwise bound
the Hilbert functions of any multilevel quotient $C^\di$ on any
$T^\di/T$ of $T^\di$-socle type $\mb t$, provided $\Lambda^{\bar s+1}\mc
C^\di$ has the $\mb h'_m$ as its Hilbert functions when $\bar s < s$;
and (3) there's some multilevel quotient $C^\sh$ on some $T^\sh/T$ with
socle type $\mb t$ and Hilbert functions the $\mb h_m$.

Second, we call $\mb t$ {\it quasi-permissible\/} if $\bar s = s$ and
 \indt{quasi-permissible}
$\mb b(s) \ge \mb t(s)$, or if $\bar s < s$ and (1) the attendant $\mb
t'$ to $\mb t$ is quasi-permissible, and (2) given any $T^\ft/S$ and
$\{\mb h^\ft_m\}$ recursively maximal for $\mb t'$ and $T^\ft$,
necessarily $\mb b(\bar s) - \mb h^\ft_{\bar s+1}(\bar s) \ge \mb t(\bar
s)$.  Third, given a multilevel quotient $C$ on some $T/k$, we call $C$
{\it recursively compressed\/}  for $\mb t$ if $C$ is of $T$-socle type
 \indt{recursively compressed}
$\mb t$; if $\Lambda^{\bar s+1}C$ is recursively compressed for $\mb t'$
when $\bar s < s$; and if, for any multilevel quotient $C^\di$ on any
$T^\di/S$ of $T^\di$-socle type $\mb t$, with $\Lambda^{\bar s+1} C^\di$
recursively compressed for $\mb t'$ if $\bar s < s$, necessarily $\rank
(C^\di) \le \rank (C)$.

Those three notions have these properties.  If a recursively maximal
sequence for $\mb t$ and $T/k$ exists, then it's unique and $\mb t$ is
quasi-permissible.  If $\mb t$ is quasi-permissible and $k$ is a domain,
then a recursively maximal sequence $\{\mb h_m\}$ for $\mb t$ and $T/k$
exists; moreover, $\bb H\Lambda_{B}^{\{\mb h_m\}}$ is nonempty, reduced,
irreducible, and covered by open subschemes, with each one isomorphic to
an open subscheme of the affine space over $k$ of fiber dimension $N$
where $\mb N :=\sum_p\mb t(p)\bigl(\mb b(p) - \mb h_{\bar s}(p)\bigr)$.
Every recursively compressed $C$ for $\mb t$ is of the same rank.  If
$\mb t$ is quasi-permissible, then (1) there's a recursively compressed
$C$ for $\mb t$ on some $T/k$, and (2) given a recursively maximal
sequence $\{\mb h_m\}$ for $\mb t$ and some $T/k$, and given on some
$T^\sh$ a $C^\sh$ of socle type $\mb t$, each $\mb h_m$ is the Hilbert
function of $\Lambda^mC^\sh$ iff $\mc C^\sh$ is recursively compressed
for $\mb t$.

Those three notions are particularly important when $k$ is an infinite
field, $A$ is a polynomial ring in two or more variables, and $B = A$.
This fact was discovered by Iarrobino, and treated in \cite{Iar84}.  We
abstract, clarify and extend his work from \eqref{se8} on.  Our main
results\emdash \eqref{coImax}, \eqref{prSmOp}, \eqref{thFHdim},
\eqref{coFdim}\emdash extend his main results \emdash Thms.\ IIA, IIB,
IIC on p.\,351 in \cite{Iar84}.  His results describe the geometry of
$\bb H\Lambda_A^{\{\mb h_m\}}$, of $\bb F\Lambda_A^{\{\mb h_m\}}$, and
of the retraction $\bb F\Lambda_A^{\{\mb h_m\}}\onto \bb H\Lambda_A^{\{\mb
h_m\}}$ when $\{\mb h_m\}$ is recursively maximal for $\mb t$ and $\mb
t$ is permissible; these issues are discussed next.  In particular, in
our extensions, $k$ can be any Noetherian base ring, and $A =
\bigoplus_p A_p$, any finitely generated, graded $k$-algebra with $A_p$
a locally free $k$-module of rank, say, $\mb a(p)$.

In general, set $g_m(p) := \ts \sum_{q = m}^s\mb t(q)\,\mb a(q-p)$ and
$\mb h^{\rm I}_m(p) := \min\{\,\mb g_m(p),\, \mb b(p)\,\}$ and $\beta^{\rm I}_m := \ts \sum_p \mb h^{\rm I}_m(p)$ for all $m$, $p$;
 \indn{hmbzIm@$\mb h^{\rm I}_m(p)$}  \indn{gmp@$g_m(p)$}
  \indn{babetaIm@$\beta^{\rm I}_m$}
 see \eqref{se8}; the I's honor Iarrobino's work.  Then given any $m$
and any homogeneous multilevel quotient $C$ of $B$ of socle type $\mb
t$, the Hilbert function $\mb h_m$ of $\Lambda^mC$ satisfies $\mb h_m(p)
\le \mb h^{\rm I}_m(p)$ for all $p$, and also $\rank(\Lambda^mC) \le
\beta^{\rm I}_m$, with equality iff $\mb h_m = \mb h^{\rm I}_m$, as
$C^*$ is of local generator type $\mb t^*$; see \eqref{prhImax}.

We call the $C$ above I-{\it compressed\/} if $\mb h_m = \mb h^{\rm
I}_m$ for all $m$, that is, if $C$ is represented by a point of $\bb
H\Lambda_B^{\{\mb h^{\rm I}_m\}}$; in addition, $\mb t$ must be
permissible, where {\it permissible\/} means that $\mb t(p) = 0$ if $\mb
b(p) < \mb g_{\bar s}(p)$ and if \eqref{sbPerm}(d) holds; the latter
holds, for instance, when $\mb a(p)$ and $\mb b(p)$ are nondecreasing in
$p$.  In \eqref{sbPerm}, the meaning of ``permissible''
  \indt{permissible}
 is extended to rather arbitrary $A$ and $B$. 

For instance, let $k$ be a field, and $B := A := k[X_1,X_2,X_3]$ the
polynomial ring (with the standard grading).  Let the socle vector be
$(0,0,0,1,2)$; that is, the nonzero values of $\mb t$ are $\mb t(q) = 1$
if $q = 3$ and $\mb t(q) = 2$ if $q = 4$.  Then $s = 4$ and ${\bar s} =
3$; further, for $m \ge 3$ the nonzero $g_m$ are $g_4(p)=2\mb a(4-p)$
and $g_3(p)=g_4(p)+\mb a(3-p)$.  So $\mb h^{\rm I}_4(p) = g_4(p)$ and
$\mb h^{\rm I}_3(p) = g_3(p)$ for $p \ge 3$; further, the coordinates of
the vectors $(1,3,6,6,2)$ and $(1,3,6,7,2)$ given by $\mb h^{\rm I}_4$
and $\mb h^{\rm I}_3$ are at least those of the Hilbert vectors of
$\Delta^4 C$ and $C$.  The latter are of dimensions at most 18 and 19.
Thus any $C$ of socle type $\mb t$ and dimension 19 has Hilbert vector
$(1,3,6,7,2)$, and if $\Delta^4 C$ is 18-dimensional too, then $C$ is
I-compressed and recursively compressed.

Iarrobino proved (and Fr\"{o}berg and Laksov reproved) in essence that,
remarkably, there exists an I-compressed $C$, or a $k$-point of $\bb
H\Lambda_A^{\{\mb h^{\rm I}_m\}}$, if $\mb t$ is permissible, $k$ is an
infinite field, $A$ is a polynomial ring, and $B = A$ (Iarrobino also
assumed $\mb h^{\rm I}_{\bar s}(1) = \mb a(1)$, but this assumption is
easily removed); see \eqref{sbPerm}.  Consequently, there's such a $C$
if $k$ is replaced by any (Noetherian) local ring with an infinite
residue field; see \eqref{prlift}.  Furthermore, see \eqref{coImax},
then $\{\mb h^{\rm I}_m\}$ is recursively maximal for $\mb t$.
Moreover, a permissible $\mb t$ is, by \eqref{leperm}(4),
quasi-permissible.

 Preserve Iarrobino's setup.  So his existence theorem holds.  It
follows, see \eqref{coImax}, that a homogeneous quotient $C$ of $A$ of
socle type $\mb t$ is I-compressed iff it's recursively compressed;.
Also, $\bb H\Psi_{A}^{\mb h^{\rm I}_{\bar s}, \mb t}$ is
irreducible, and $\bb H\Lambda_A^{\{\mb h^{\rm I}_m\}}$ is open in it;
see \eqref{coIrrmax}.

 Further, in essence, Iarrobino, just after Prop.\,3.4 of
\cite{Iar84}*{p.\,339}, constructed an irreducible scheme $P$ whose
$k$-points represent, with repetition, all the homogeneous $C$ of socle
type bounded (termwise) by $\mb t$, and he asserted $P$ has a nonempty
open subset whose $k$-points represent most compressed $C$ of socle type
$\mb t$.  There's a more refined assertion in \eqref{regenstype}: all
the compressed $C$ of socle type $\mb t$ correspond to an open subset of
$P$, and those recursively compressed, to a nonempty, smaller open
subset.  To put it more informally, a general $C$ of socle type $\mb t$
is recusively compressed.


Conversely, when $k$ is any Noetherian ring, $A$ is a polynomial ring,
and $B = A$, if there exists a $C$ with $\rank(\Lambda^mC) = \beta^{\rm
I}_m$ for all $m$, then $\mb t$ is permissible by \eqref{prthenperm}.
On the other hand, the literature contains many interesting examples
with $k$ a field, $A$ a proper quotient of a polynomial ring, $B = A$,
and $\mb t$ permissible; sometimes there are no $C$ with
$\rank(\Lambda^mC) = \beta^{\rm I}_m$ for all $m$, and sometimes, there
are some.

Published examples of the latter case are reviewed in \eqref{exCT}.
There all the $C$ are Gorenstein; so the rank condition reduces to $\dim
C = \beta^{\rm I}_s$.  Moreover, suppose $A$ too is Artinian and
Gorenstein.  Then $\dim C = \beta^{\rm I}_s$ when both $A$ and $C$ are
suitably general for their socle degrees; see \eqref{cogenAG}.  On the
other hand, \eqref{coPowSum} constructs examples where the Hilbert
function $\mb h$ of $C$ has a nontrivial ``plateau"; the construction is
largely distilled and suitably adapted from \cite{LNM1721}*{pp.\,9--14}.

When, as above, $\{\mb h^{\rm I}_m\}$ is recursively maximal for a
permissible $\mb t$, more can be proved via careful work with the spaces
of local generators of $C^*$.  First, by \eqref{prSmOp}, whether or not
$k$ is a domain, $\bb H\Lambda_{B}^{\{\mb h^{\rm I}_m\}}$ is an open
subscheme of\/ $\bb H\Psi_{B}^{\mb h^{\rm I}_{\bar s}}$, and it's
covered by open subschemes, with each one isomorphic to an open
subscheme of the affine space over $k$ of fiber dimension $\mb H
:=\sum_p\mb t(p)\bigl(\mb b(p) - \mb h^{\rm I}_{\bar s}(p)\bigr)$.
Second, if $A$ and $B$ satisfy certain mild conditions, which hold when
$A$ is a polynomial ring and $B = A$, then by \eqref{thFHdim} the
retraction map makes $\bb F\Lambda_{\mc B}^{\mb h^{\rm I}_{\bar s}, \mb t}$
an affine-space bundle over $\bb H\Lambda_{\mc B}^{\mb h^{\rm I}_{\bar s},
\mb t}$ of fiber dimension $\mb R := \sum_p \mb t(p)\big(\sum_{q<p}(\mb
b(q) - \mb h^{\rm I}_{\bar s}(q))\bigr)$.  Finally, by \eqref{coFdim},
which is a corollary of \eqref{prSmOp} and \eqref{thFHdim}, then $\bb
F\Lambda_{\mc B}^{\{\mb h^{\rm I}_m\}}$ is covered by open subschemes,
each one isomorphic to an open subscheme of the affine space over $k$ of
fiber dimension $\mb F := \mb H + \mb R = \sum_p \mb t(p)\big(\sum_{q\le
p}(\mb b(q) - \mb h^{\rm I}_{\bar s}(q))\bigr)$; moreover, if $S$ is
irreducible, then so are $\bb H\Lambda_{\mc B}^{\{\mb h^{\rm I}_m\}}$
and $\bb F\Lambda_{\mc B}^{\{\mb h^{\rm I}_m\}}$.

An interesting issue, raised by the referee, is to determine when the
retraction map is an algebraic vector bundle.  The referee noted that
Iarrobino \cite{IaTop}*{Thm.\,1, p.\,229; Thm.\,2, p.\,232} proved in
two variables that it's not for the (nonmaximal) Hilbert vector
$(1,1,1,\dotsc,1)$ of length $n\ge 4$ and that his result was
generalized by Haboush and Hyeon \cite{HaHy}*{Thm.\,9.2, p.\,4309} to
$r$ variables.  On the other hand, Jelisiejew proved in
\cite{JJ-VSP}*{Claim 2, p.\,280} that it is a vector bundle in six
variables for the Hilbert vector $(1,6,6,1)$.  

The formula for $\mb F$ and the smoothness and irreducibility of
$\smash{\bb F\Lambda_{\mc B}^{\{\mb h^{\rm I}_m\}}}$ are useful in the
study of $d$-dimensional quotients $C$ of the polynomial ring $A$ in $r$
variables over a field $k$, particularly in showing that certain $C$
aren't smoothable within $\Hilb^d_{\bb A^r/k}$ and in finding {\it
elementary components} of $\Hilb^d_{\bb A_k^r/k}$ (ones parameterizing
 \indt{elementary component} $C$ with absolutely irreducible support\/).
Examples are discussed in \eqref{exElemCpts}.  Notably, there's a new
example of a {\it small elementary component\/} (one of dimension less
than $dr$) with $k =\bb Q$ and with Hilbert vector
 \indt{small elementary component}
$(1,5,6,1)$.  It's produced via a novel method of reduction to the case
$k = \bb Z/p$ for just one $ p > 0$, see \eqref{prLift}, accompanied by
computations over $\bb Z/p$ with Macaulay 2.

The study of elementary components is generally based on the method of
small tangent spaces or trivial negative tangents, introduced in 1978 by
Iarrobino and Emsalem \cite{IE78} and put in definitive form in 2019 by
Jelisiejew \cite{JJ-JLMS2019}.  Versions of Jelisiejew's results,
adequate for the examples in \eqref{exElemCpts}, are, in \eqref{prLift},
given self-contained alternative proofs, largely inspired by
Jelisiejew's.

Examples abound of proper quotients $A$ of a polynomial ring
over a field $k$ and of permissible-like $\mb t$ such that every
homogeneous quotient $C$ of $A$ of $k$-socle type $\mb t$ has $\dim_kC
<\beta^{\rm I}_{\bar s}$; in particular, no $C$ has $\rank(\Lambda^mC) =
\beta^{\rm I}_m$ for all $m$, and $\bb H\Psi_A^{\mb h^{\rm I}_{\bar
s},\mb t} = \emptyset$.  Some examples were given by Cho and Iarrobino
in \cite{C-IJA241}*{Prp.\,2.7, p.\,754}, and were analyzed by Geramita
et al.\ in \cite{GerMAM186}*{Ex.\,7.3, pp.\,70--72}.  The present second
author made an infinitesimal analysis of many classes of examples of
this sort; it suggests the geometry of $\bb H\Psi_{A}^{\mb h, \mb t}$
and $\bb F\Psi_{A}^{\mb h, \mb t}$ too is essentially as above.

  More examples were given by Iarrobino \cite{Pen2005}*{Ex.\,2.9,
p.\,283}, by Zanello \cite{ZanCA35}*{Rmk.\,6, p.\,1090}, and by
Migliore, Mir\'{o}-Roig, and Nagel in \cite{MMN}*{Exs.\,2.14--2.16,
pp.\,343--344}.  The last ones are discussed in some detail and refined
in \eqref{exMetal}, as they nicely illustrate \eqref{thAffBdl}, which is
the main result of Section~\ref{GCQ}.

Section 11 develops a theory like that in Sections 9 and 10, but now for
a $k$-Gorenstein, $k$-Artinian $A$ over any Noetherian base $k$.  The
maximal Hilbert function of a quotient $C$ is often not $\mb h^{\rm
I}_{\bar s}$, but is still a predictable function $\mb h$.  The main
results \eqref{thAffBdl}(1a)--(1c) are geometric: if $\mb t$ vanishes
when it should and if $\bb H\Psi_{A}^{\mb h, \mb t} \neq \emptyset$,
then $\bb H\Psi_{A}^{\mb h, \mb t}$ is covered by nonempty open
subschemes, each one isomorphic to an open subscheme of the affine space
over $S$ of fiber dimension $\mb H:= \sum_p \mb t(p)\bigl(\mb a(p) - \mb
h(p)\bigr)$; further, $\bb F\Psi_{A}^{\mb h, \mb t}$ is an
affine-space bundle over\/ $\bb H\Psi_{A}^{\mb h, \mb t}$ of fiber
dimension $\mb R := \sum_p \mb t(p)\bigl(\sum_{p'<p} (\mb a(p') - \mb
h(p')) \bigr)$; finally, $\bb F\Psi_{A}^{\mb h, \mb t}$ is covered by
nonempty open subschemes, each one isomorphic to an open subscheme of
the affine space over $S$ of fiber dimension $\mb F := \mb H + \mb R =
\sum_p \mb t(p)\big(\sum_{p'\le p}(\mb a(p') - \mb h(p')) \bigr)$.

The key to Section 11 is further work with local generators of $C^*$ for
a given $C \in \bb H\Psi_{A}^{\mb h^\di, \mb t}$ for some $\mb
h^\di$.  Here's the idea.  Set $E := \bigoplus_{q\in\Z} A(q)^{\oplus \mb
t(q)}$.  Now, $C^*$ is of local generator type $\mb t^*$.  So given any
maximal ideal $\go m$ of k, we may replace $k$ by the localization $k_h$
for an $h \notin \go m$ so that $C^*$ has a minimal set of homogeneous
generators $f_1,\dotsc,f_m$, which define a surjection $E \onto C^*$ of
degree $0$.  Then $\mb h^\di = \mb h^{\rm I}_{\bar s}$ iff there's $n <
0$ with $E_p \risom (C^*)_p$ for $p < n$ and $(C^*)_p = A^\dg_p$ for
$p\ge n$.  To proceed, for $p < n$, we let the kernel of $E_p \onto
(C^*)_p$ be nonzero in a controlled, but general, fashion.

Say $A^\dg = Af$ with $f\in A^\dg_{-a}$.  So $f_i = g_if$ for some
$g_i\in A_{d_i}$ with $d_i>0$.  Then $\mb h^\di = \mb h$ iff the Koszul
homology group $H_1(g_1,\dotsc,g_m;A^\dg)$ is $0$ in degrees below $n$;
see \eqref{leInj}.  If so and if $k_{red}$ is a domain, then
$g_1,\dotsc,g_m$ are general; see \eqref{prOpens}.  When $k$ is an
infinite field and $A$ is a general complete intersection, Fr\"oberg's
conjecture asserts that, if $g_1,\dotsc,g_m$ are general, then indeed
$\mb h^\di = \mb h$; see \eqref{reFrbg}.
 \indt{Fr\"oberg's conjecture}

In brief, Section 2 treats preliminaries about the category $\bb F_A$ of
 \indn{FbbA@$\bb F_A$}
filtered modules over a filtered ring $A$.  However, \eqref{sbRm} makes
a novel observation: the Rees functor embeds $\bb F_A$ into the Abelian
category of graded modules over the Rees algebra of $A$ as an exact
subcategory in Quillen's sense \cite{Q}*{p.\,99}.  Section~3 develops
Macaulay Duality over any Noetherian base ring $k$, providing the tool
used to transform loci of quotients into loci of submodules.  Section 4
discusses a key invariant, the $k$-socle type $\mb t$ of an $A$-module
$C$.  Its importance stems from \eqref{leSoc}: if $C$ is {\it
$k$-Artinian}, that is, locally free of finite rank over $k$, then the
``reverse'' $\mb t^*$ of $\mb t$ is the local generator type of the dual
$C^*$.  Section 5 discusses {\it $k$-Gorenstein} $k$-Artinian $A$; those
 \indt{kGor@$k$-Gorenstein}
are the $A$ with $A^*$ an invertible $A$-module.  Most of the results in
Section 5 are familiar when $k$ is a field, but these $A$ are used in
some interesting new examples.

From Section 6 on, the focus is geometry. Section 6 constructs various
schemes of quotients of a fixed graded $A$-module $B$ when $A$ is a
graded $k$-algebra with $A_0 = k$. Notably, there are several different
constructions of the schemes $\bb F\Lambda_B^{\mb h}$ and $\bb
H\Lambda_B^{\mb h}$ of filtered and of homogeneous, level quotients with
Hilbert function $\mb h$.  Section~7 constructs the schemes $\bb
F\Lambda_{B}^{\{\mb h_m\}}$ and $\bb H\Lambda_{B}^{\{\mb h_m\}}$ of
multilevel quotients, and treats three new notions: recursively maximal
$\{\mb h_m\}$, quasi-permissible $\mb t$, and recursively compressed
multilevel $C$.  Section 8 treats, in greater generality, Iarrobino's
notions of a permissible $\mb t$ and of the I-set (our term) $\{\mb
h^{\rm I}_m\}$ of any $\mb t$.  Sections 9 and 10 extend Iarrobino's
main results in \cite{Iar84}, which describe the geometry of $\bb
H\Lambda_A^{\{\mb h^{\rm I}_m\}}$, of $\bb F\Lambda_A^{\{\mb h^{\rm
I}_m\}}$, and of the retraction $\bb F\Lambda_A^{\{\mb h^{\rm
I}_m\}}\onto \bb H\Lambda_A^{\{\mb h^{\rm I}_m\}}$ when $\mb t$ is
permissible and $A$ is a polynomial ring, thus achieving our main goal.
Finally, Section 11 develops a similar theory for maximal quotients of a
graded Gorenstein Artinian algebra.

\section{Preliminaries}\label{sePr}

\begin{setup}\label{se2}
 Fix a Noetherian base ring $k$ for use throughout this paper.  (In this
paper, all rings are commutative with 1.)

Recall for frequent, but implicit, use that (even when $k$ is not
Noetherian) the following three conditions on a $k$-module $C$ are
equivalent:
 \begin{enumerate}
 \item $C$ is locally free of finite (but not necessarily constant) rank
over $k$;
 \item $C$ is projective and finitely generated over $k$;
 \item $C$ is flat and finitely presented over $k$.
 \end{enumerate}

Given a $k$-algebra $A$, denote the (Abelian) category of $A$-modules
by $\bb M_A$.
 \indn{MA@$\bb M_A$}

 Given any $k$-module $C$, define its {\it $k$-dual}  $C^*$ by
 \indt{kdual@$k$-dual}
 \begin{equation*}\label{eq1.1}
 C^* : = \Hom_k(C,\,k).
 \end{equation*}
 \indn{Cstar@$C^*$}

If $C$ is a module over a given $k$-algebra $A$, then $C^*$ is
canonically an $A$-module too, via the {\it adjoint action\/}: $(af)\psi
:= f(a\psi)$ for all $a\in A$ and $f\in C^*$ and $\psi\in C$.
 \indt{adjoint action}

 A map of $A$-modules $\gamma\: C\to D$ induces its {\it dual map\/} of
 \indt{dual map}
$A$-modules $\gamma^*\: D^*\to C^*$ by $\gamma^*g := g\circ\gamma$.  And
$C\mapsto C^*$ defines a left exact $A$-linear functor $*\:\bb M_A\to
\bb M_A$.
 \indt{dual map}
 \end{setup}

\begin{sbs}[Filtered modules]\label{sbFilt}
 Fix a $k$-algebra $A$ with a {\it (descending) filtration} $F^\bu A$;
 \indt{filtration}
namely, $F^\bu A$ is a descending chain of $A$-submodules $F^m A$ for
$m\in \Z$ with $F^m A = A$ for $m\le 0$ and with $F^m A\cdot F^nA
\subset F^{m+n}A$ for all $m$ and $n$.

Given an $A$-module $C$, by a {\it filtration} $F^\bu C$, let us mean a
descending chain of $A$-submodules $F^n C$ for $n\in \Z$ with $F^m
A\cdot F^nC \subset F^{m+n}C$ for all $m$ and $n$.  Set 
\begin{equation*}\label{eqFm1}\ts
 G_nC := F^nC/F^{n+1}C\enspace \text{for all }n \and
         G_\bu C := \bigoplus_{n\in\Z} G_nC.
 \end{equation*}
 Then $G_m A\cdot G_nA \subset G_{m+n}A$ and $G_m A\cdot G_nC \subset
G_{m+n}C$ for all $m$ and $n$.  Also, $G_nA = 0$ for $n<0$.  Thus $G_\bu
A$ is a graded $k$-algebra, $G_\bu C$ a graded $G_\bu A$-module.  As
usual, call them the {\it associated graded\/} algebra and module.
 \indn{Gbu@$G_\bu C$}
 \indt{associated graded}

Notice that $A$ always has a {\it trivial filtration}, where $F^m A := A$
for $m\le 0$ and $F^m A := 0$ for $m > 0$.  Then $G_0A = A$ and $G_nA =
0$ for $n\neq 0$.  Similarly, $C$ has a trivial filtration.  Thus it's
really no restriction to assume $A$ and $C$ are filtered.

Given $p\in \Z$, let $C[p]$ denote $C$ with its {\it shifted
filtration}, $F^n(C[p]) := F^{n+p}C$.
 \indt{shifted filtration} \indn{FnCp@$F^n(C[p])$}

Call $F^\bu C$ {\it exhaustive\/} if $\,\bigcup_n F^nC =C$, {\it
separated\/} if $\,\bigcap_n F^nC =0$, {\it discrete\/} if $F^nC =0$ for
$n\gg0$, and {\it finite\/} if $F^nC =C$ for $n\ll0$ and $F^nC =0$ for
$n\gg0$.

The filtered $A$-modules form a category $\bb F_A$.  Its maps $\beta\:
 \indn{FbbA@$\bb F_A$}
(B,\,F^\bu B)\to (C,\,F^\bu C)$ are the $A$-maps $\beta\: B\to C$ with
$\beta(F^nB) \subset F^nC$ for all $n$.

Denote by $\bb{EF}_A$ the full subcategory of $\bb F_A$ of all
$(B,\,F^\bu B)$ with $F^\bu B$ exhaustive, and by $\bb{DEF}_A$ its full
subcategory of all $(B,\,F^\bu B)$ with $F^\bu B$ discrete too. 

Given $(B,\,F^\bu B)$ and $(C,\,F^\bu C)$ in $\bb F_A$, filter the
$A$-module $\Hom_A(B,\,C)$ by
 \begin{align*}
 F^p\Hom_A(B,\,C) &:= \Hom_A\bigl(B,\,C[p]\,\bigr)\\
                  &:= \{\, \beta\: B\to C \mid\beta(F^nB)
                         \subset F^{n+p}C \enspace\text{for all }n\,\}.
 \end{align*}

Similarly, give the $k$-module $k$ the trivial filtration, and filter
$B^*$ by
 \begin{equation*}  \label{eqFm3}
 F^p(B^*) := F^p\Hom_k(B,\,k) := \{\, \beta\: B\to k \mid\beta(F^nB)
                 \subset F^{n+p}k\enspace\text{for all }n\,\}.\\
 \end{equation*}
 Then it's easy to see that
 \begin{equation} \label{eqFm4}
  F^p(B^*) = \{\, \beta\: B\to k \mid\beta(F^nB) = 0
                 \enspace\text{for }n + p =1\,\} = (B/F^{1-p}B)^*.
 \end{equation} Notice, however, that even if $F^\bu B$ is exhaustive,
$F^\bu (B^*)$ need not be exhaustive.
 
Given a map of filtered modules $\beta\: B\to C$, its dual $\beta^*\:
C^* \to B^*$ is, plainly, a map of filtered modules.  Thus  $C \mapsto
C^*$ is a contravariant, additive functor from $\bb F_A$ to itself.

Filter the tensor product $B\ox_AC$ in $\bb M_A$ by setting
\begin{equation*}\label{eqFm5}\ts
 F^p(B\ox_AC) := \Im(\tau) \text{\enspace where }
        \tau\: \bigoplus_{m+n=p}F^m(B)\ox_AF^n(C) \to B\ox_AC.
 \end{equation*}
 Then it's easy to see that this {\it adjoint associativity formula\/}
 \indt{adjoint associativity}
preserves filtrations: $\Hom_k(B\ox_AC,\,k) = \Hom_A\bigl(B,\,
\Hom_k(C,k)\bigr)$.  In other words, \begin{equation}\label{eqeqFm6}
  (B\ox_AC)^* = \Hom_A(B,\,C^*)\text{\quad in } \bb F_A.
 \end{equation}
 In particular, taking $C:=A$ yields the following (duality) formula:
 \begin{equation}\label{eqeqFm7}
 B^* = \Hom_A(B,\,A^*)\text{\quad in } \bb F_A.
 \end{equation}

 Let now $A$ be a graded $k$-algebra, and $C$ a graded $A$-module.  Say
$A = \bigoplus_{p\in\Z} A_p$ with $A_p = 0$ for $p<0$, and $C =
\bigoplus_{p\in\Z} C_p$.  For all $n$, set $F^nA := \bigoplus_{p\ge
n}A_p$ and $F^nC := \bigoplus_{p\ge n}C_p$.  Plainly, the $F^nA$ turn
 \indn{FnC@$F^nC$}
$A$ into a filtered $k$-algebra, and the $F^nC$ turn $C$ into a filtered
$A$-module; moreover, $G_pA = A_p$ and $G_pC = C_p$ for all $p$.

Also, an $A$-module map $\gamma\: C\to D$ is {\it homogeneous\/} (that
 \indt{homogeneous}
is, $D$ is graded, say $D = \bigoplus D_p$, and $\gamma C_p\subset D_p$
for all $p$) iff $\gamma\in \Hom_{\,\bb F_A}(B,\,C)$.  Thus the category
$\bb G_A$ of graded modules and homogeneous maps is a full subcategory
 \indn{GbbA@$\bb G_A$}
of $\bb F_A$.  And the functor $G_\bu\: \bb F_A\to \bb G_A$, given by
$C\mapsto G_\bu C$, is a {\it retraction\/}; that is, $G_\bu\big| \bb
G_A=1$.
 \indt{retraction}

Let $B$ be the kernel of $\gamma$ in $\bb M_A$.  Then $B$ is naturally
graded: $B = \bigoplus_{p\in\Z} B_p$ where $B_P := B\cap C_p$.  So $B$
is the kernel of $\gamma$ in $\bb G_A$.  Moreover, $F^nB = B\cap F^n C$
for all $n$; so $B$ is the kernel of $\gamma$ in $\bb F_A$ by
\eqref{sbSm}.  Similarly, the cokernel of $\gamma$ in $\bb M_A$ is
also its cokernel in $\bb G_A$ and in $\bb F_A$.  In particular, $\bb
G_A$ is, therefore, Abelian.
 \end{sbs}

\begin{sbs}[Strict maps]\label{sbSm}
 Fix a filtered $k$-algebra $(A,\,F^\bu A)$.  Following
\cite{SerreAlLoc}*{p.\,II-2} (and \cite{D71}*{(1.1.5), p.\,7}, call a map
$\gamma\: (C,\,F^\bu C)\to (D,\,F^\bu D)$ in $\bb F_A$ {\it strict\/}
(or
 \indt{strict}
{\it strictly compatible with the filtrations\/}) if $\gamma(F^nC) =
\gamma(C)\cap F^nD$ for all $n$.

For example, if $\gamma\: C\to D$ is injective, then $\gamma$ is strict
iff $F^\bu C$ is {\it induced\/} by $F^\bu D$ in the sense that $F^nC =
\gamma^{-1}(F^nD)$.  Instead, if $\gamma\: C\to D$ is surjective, then
$\gamma$ is strict iff $F^\bu D$ is {\it induced\/} by $F^\bu C$ in the
sense that $F^nD = \gamma(F^nC)$.  In general, $\gamma$ is strict iff both
$F^\bu C$ and $F^\bu D$ induce the same filtration on $\gamma(C)$.

Let's see $\gamma\: (C,\,F^\bu C)\to (D,\,F^\bu D)$ {\it has a kernel
$\Ker(\gamma)$ and a cokernel\/} $\Cok(\gamma)$ in $\bb F_A$, whose
underlying $A$-modules are the usual $\Ker(\gamma)$ and $\Cok(\gamma)$
in $\bb M_A$, and whose filtrations are induced by $F^\bu C$ and $F^\bu
D$.
  Given $\beta\: (B,\, F^\bu B) \to (C,\,F^\bu C)$ with $\gamma\beta =
0$, there's a unique $\alpha\: B\to \Ker(\gamma)$ in $\bb M_A$ with
$\ker(\gamma)\alpha = \beta$ where $\ker(\gamma)\: \Ker(\gamma)\into C$.
But $\ker(\gamma)(\alpha(F^nB)) = \beta(F^nB) \subset F^nC$.  Thus, as
desired, $\alpha(F^nB)\subset \ker(\gamma)^{-1}(F^nC)$.  The analysis of
$\Cok(\gamma)$ is dual.

Consider the following sequence in $\bb F_A$ associated to $\gamma\:
(C,\,F^\bu C)\to (D,\,F^\bu D)$\/:
 \begin{equation*}\label{eqSm1}
 \Ker(\gamma)\xto{\ker(\gamma)} C\xto{\coim(\gamma)} \Coim(\gamma)
  \xto{\theta(\gamma)} \Im(\gamma) \xto{\im(\gamma)} D
    \xto{\cok(\gamma)} \Cok(\gamma)
 \end{equation*}
 where $\coim(\gamma) := \cok(\ker(\gamma))$ and $\im(\gamma) :=
\ker(\cok(\gamma))$.  Note $\theta(\gamma)$ is bijective in $\bb M_A$,
and it's an isomorphism in $\bb F_A$ iff $F^\bu C$ and $F^\bu D$ induce
the same filtration on $\Im(\gamma)$, so iff $\gamma$ is strict.  Thus,
since not every $\gamma$ is strict, $\bb F_A$ fails to be Abelian.

Given $(E,\, F^nE)$ in $\bb F_A$, let's see $\Hom_A(\bu,\,E)$ is left
exact in the sense that
 \begin{equation}\label{eqSm2}
 \Hom_A(\,\Cok(\gamma),\,E\,) = \Ker(\,\Hom_A(\gamma,\,E)\,)
 \tqn{in} \bb F_A.
 \end{equation}
 First, \eqref{eqSm2} holds in $\bb M_A$, as each side can be
canonically identified with the set of $A$-maps $\delta\: D\to E$ with
$\delta\gamma = 0$.  Second, \eqref{eqSm2} holds in $\bb F_A$, as the
$p$th term of the filtration of each side is identified with the subset
of all $\delta\: D\to E[p]$ in $\bb F_A$.

Finally, it is clear from the construction of $\Cok(\gamma)$ that it's
not only the cokernel in $\bb F_A$, but also in $\bb F_k$.  Hence
$\Cok(\gamma)^* = \Ker(\gamma^*)$ in $\bb F_k$ by \eqref{eqSm2} with $A
:= k$ and $E := k$.  In particular, given $n$, then $F^n(\Cok(\gamma)^*)
= F^n(\Ker(\gamma^*))$ in $\bb M_k$; so, this equation holds in $\bb
M_A$ too, as it plainly respects multiplication by each $a\in A$.  Thus
 \begin{equation}\label{eqSm3}
 \Cok(\gamma)^* = \Ker(\gamma^*)  \tqn{in} \bb F_A.
 \end{equation}
 \end{sbs}

\begin{sbs}[Rees modules]\label{sbRm}
 Fix a filtered $k$-algebra $(A,\,F^\bu A)$ and a variable $t$.  Form
the {\it (extended) Rees algebra} $\Rees(A) := \bigoplus_{n\in
 \indt{Rees algebra}
\Z}F^nA\cdot t^{-n}$ and the {\it Rees functor\/}
 \indt{Rees functor}
\begin{equation*}\label{eqRf}\ts
   \Rees\: \bb F_A \to \bb G_{\Rees(A)} \tq{by} (B,\,F^\bu B) \mapsto
        \Rees(B) := \bigoplus_{n\in \Z}F^n B\cdot t^{-n}.
 \end{equation*}
 Note that $\bb F_A$ and $\bb{EF}_A$ have the same image $\bb R_A$,
which (see \cite{LO96}*{p.\,vi}) is the full subcategory of graded
$\Rees(A)$-modules on which $t$ is {\it regular\/} (that is,
multiplication by $t$ is injective); an $M := \bigoplus M_n$ in $\bb
R_A$ is equal to the image of $ (B,\,F^\bu B)$ where $B := M/(1-t)M$ and
$F^nB := (M_n+(1-t)M )/ (1-t)M$ is the image of $M_n$ in $B$.  Moreover,
the Rees functor embeds $\bb{EF}$ as $\bb R_A$ in the Abelian category
$\bb G_{\Rees(A)}$.

Notice that $\bb R_A$ is closed under extensions; that is, given a short
exact sequence $0\to M\to M'\to M''\to 0$ in $\bb G_{\Rees(A)}$, if
$M,\ M'' \in \bb R_A$, then $M'\in \bb R_A$; indeed, apply the Snake
Lemma to this diagram, where the $\tau$'s denote multiplication by $t$,
 \begin{equation}\label{eqRm1}
 \begin{CD}
        0@>>> M  @>>>  M'  @>>>   M''  @>>> 0\\
        @. @V\tau VV @V\tau' VV @V\tau''VV\\
        0@>>> M  @>>>  M'  @>>>   M''  @>>> 0
 \end{CD}
 \end{equation}
 in order to see that $t$ is regular on $M'$ because it is on $M$ and
$M''$.

Thus $\bb R_A$ is the seminal example of an {\it exact category\/} in
 \indt{exact category}
\cite{Q}*{p.\,99}.  Accordingly, call a sequence $0\to (B,\,F^\bu B)
\xto\beta (C,\,F^\bu C) \xto\gamma (D,\,F^\bu D)\to 0$ in $\bb {EF}_A$
{\it exact\/} if its image in $\bb G_{\Rees(A)}$ is exact.  Plainly, the
latter condition holds iff each sequence $0\to F^nB\to F^nC\to F^nD\to
0$ is exact in $\bb M_A$; so iff $F^\bu B$ and $F^\bu D$ are induced by
$F^\bu C$; so iff $\beta = \ker(\gamma)$ and $\gamma = \cok(\beta)$ in
$\bb {EF}_A$, or equivalently in $\bb F_A$.

Given  $(B,\,F^\bu B)$ in $\bb F_A$, notice that
 \begin{equation}\label{eqRm2}
  G_\bu B = \Rees(B)\bigm/ t\cdot \Rees(B)
         = \Rees(B)\ox_{\Rees(A)}G_\bu A.
  \end{equation}
 Thus the functor $G_\bu\: \bb F_A \to \bb G_{G_\bu A}$ factors through
$\bu\ox G_\bu A\: \mathbb{G}_{\Rees(A)} \to \bb G_{G_\bu A}$.

The latter functor is, of course, right exact.  Moreover, its
restriction to $\bb R_A$ is exact owing to the 9-Lemma applied to
\eqref{eqRm1}.  Thus $G_\bu$ is exact on $\bb {EF}_A$ in the sense that
it preserves short exact sequences (compare with \cite{Bour}*{Prp.\,2,
p.\,25}).

Let's say that the restriction of $G_\bu$ to a subcategory $\bb S$ of
$\bb{EF}_A$ is {\it biexact\/} if a sequence $0\to (B,\,F^\bu B)
 \indt{biexact}
\xto\beta (C,\,F^\bu C) \xto\gamma (D,\,F^\bu D)\to 0$ in $\bb S$ with
$\gamma\beta = 0$ is exact in $\bb {EF}_A$ when its image under $G_\bu$
is exact in $\bb G_{G_\bu A}$.

By \eqref{prBiex} below, the restriction of $G_\bu$ to $\bb S$ is
biexact if $\bb S$ is {\it fully separated\/} in the sense that $\bb S$
 \indt{fully separated}
is closed under taking kernels and cokernels, and if $(B,\,F^\bu B)\in
\bb S$, then $F^\bu B$ is separated.  For example, $\bb {DEF}_A$ is,
plainly, fully separated.

For another example, take $A$ to be a {\it Zariski ring\/} (see
\cite{ZSvII}*{\S\,4, p.\,261} and \cite{LO96}*{p.\,83}); namely,
$\Rees(A)$ is a Noetherian ring, and $F^1A$ lies in the Jacobson radical
$\Jac(A$).  For instance, $A$ could be a weighted power series ring over
$k$.  Let $\bb{ZF}_A$ be the full subcategory of $\bb{EF}_A$ of
$(B,\,F^\bu B)$ such that $\Rees(B)$ is a finitely generated
$\Rees(A)$-module.
 Plainly, $\bb{ZF}_A$ is closed under taking kernels and cokernels.

 Given $(B,\,F^\bu B)$ in $\bb ZF_A$ and $b\in \bigcap_nF^nB$, let's see
$b=0$.  Inspired by Rees (see \cite{Sharp}*{p.\,563}), set $M^{(n)} :=
\Rees(A)\cdot bt^{-n} \subset \Rees(B)$.  The $M^{(n)}$ form an
increasing chain.  As $\Rees(B)$ is Noetherian, $M^{(p)}=M^{(p+1)}$ for
some $p$.  So $ubt^{-p} = bt^{-(p+1)}$ for some $u\in \Rees(A)$.  Say $u
= \sum u_it^{-i}$ with $u_i\in F^iA$.  Then $u_1b = b$.  So $(1-u_1)b =
0$.  But $F^1A\subset \Jac(A)$.  So $b=0$.  Thus $\bb{ZF}_A$ is fully
separated.
 \end{sbs}

\begin{proposition}\label{prBiex}
 Let $A$ be a filtered $k$-algebra, and $\bb S\subset \bb{EF}_A$ a fully
separated subcategory.  Then $G_\bu\: \bb{EF}_A \to \bb G_{G_\bu A}$ is
exact, and its restriction to $\bb S$ is biexact.
 \end{proposition}

\begin{proof}
 It remains to see that a sequence $0\to M\xto\mu M'\xto{\mu'} M''\to 0$
in the image of $\bb S$ in $\bb R_A$ is exact if $\mu'\mu = 0$ and if
$0\to M/tM \xto{\bar\mu} M'/tM'\xto{\bar\mu'} M''/tM''\to 0$ is exact in
$\bb G_{G_\bu A}$.  Consider the following commutative diagram with
exact rows:
 \begin{equation}\label{eqBiex1}
 \begin{CD}
        0@>>> M  @>\tau>>  M  @>>>  M/tM   @>>> 0\\
        @. @V\mu VV     @V\mu VV @V\bar\mu VV\\
        0@>>> M'  @>\tau'>> M' @>>> M'/tM' @>>> 0
 \end{CD}
 \end{equation}
 where  the $\tau$'s denote multiplication by $t$.

Set $L := \Ker(\mu)$.  Since $\bar\mu$ is injective, multiplication by
$t$ on $L$ is bijective by the Snake Lemma; so if say $L = \bigoplus
L_n$, then $L_{n+1}\risom L_n$ for all $n$.  But $L$ is in the image of
$\bb S$, and $\bb S$ is fully separated.  Thus $L = 0$.

Set $N := \Cok(\mu)$.  As $\mu'\mu = 0$, there's a commutative diagram
with exact rows
 \begin{equation}\label{eqBiex2}
 \begin{CD}
        0@>>> N @>\theta>> N @>>> N/tN @>>> 0\\
        @. @V\nu VV     @V\nu VV @V\bar\nu VV\\
        0@>>> M''  @>\tau''>> M'' @>>> M''/tM''' @>>> 0
 \end{CD}
 \end{equation}
 where $\theta$ and $\tau''$ denote multiplication by $t$.  Note that,
applied to \eqref{eqBiex1}, the Snake Lemma yields $\Cok(\bar\mu)= N/tN$.

By Hypothesis, $\Cok(\bar\mu)= M''/tM''$.  Hence $\bar\nu$ is bijective.
So applied to \eqref{eqBiex2}, the Snake Lemma yields that
multiplication by $t$ is bijective on $\Ker(\nu)$ and on $\Cok(\nu)$.
So arguing as for $L$ yields $\Ker(\nu)=0$ and $\Cok(\nu)=0$.  Thus
$\nu$ is a isomorphism.  Thus $0\to M\to M'\to M''\to 0$ is exact, as
desired.
 \end{proof}

\begin{sbs}[$k$-Artinian modules]\label{sbArt} Fix a $k$-algebra $A$.

Call an $A $-module $C$ {\it $k$-Artinian\/} if $C$ is locally free of
 \indt{kArt@$k$-Artinian}
constant finite rank over $k$.  Let $\bb {AM}_A\subset \bb M_A$ denote
 \indn{AMAbb@$\bb {AM}_A$}
the full subcategory of such $C$.  Note that $\bb {AM}_A$ is $A$-linear,
but not Abelian.  Given $C\in\bb {AM}_A$, note, plainly, $C^*\in\bb
{AM}_A$ too. Thus $*\:\bb M_A\to \bb M_A$ restricts to an $A$-linear
functor $*\:\bb {AM}_A\to \bb {AM}_A$.

When $k$ is a field, every $C\in \bb M_A$ is free.  So $C$ is
$k$-Artinian iff $\dim_kC <\infty$.

 Given a function $\mb h\:\Z\to \Z$, define $\mb h^*\:\Z\to \Z$ by $\mb
h^*(n) := \mb h(-n)$.  Also, set
 \indn{hmbzst@$\mb h^*(n)$}
 \begin{equation*}\label{eqinfsup}
   i(\mb h) := \inf\{\,n\mid \mb h(n)\neq 0\,\}\and
   s(\mb h) := \sup\{\,n\mid \mb h(n)\neq 0\,\}.
 \end{equation*}
 \indn{ihmb@$i(\mb h)$} \indn{shmb@$s(\mb h)$}
 Call $\mb h$ {\it finite\/} if $i(\mb h)$ and $s(\mb h)$ are finite (so
if $h\neq 0$, then $\{\,n\mid \mb h(n)\neq 0\,\}$ is finite).
 \indt{finite function}

Assume $A$ is graded.  Given $C\in \bb G_A$, say $C = \bigoplus C_n$.
Notice that, if $C$ is $k$-Artinian, then $C_n = 0$ for almost all $n$,
and each $C_n$ is locally free of finite, but not necessarily constant,
rank; the converse holds if also the rank of each $C_n$ is constant.
  Call $\mb h\:\Z\to \Z$ the {\it Hilbert function\/} of $C$ if each
 \indt{Hilbert function}
$C_n$ is locally free of (constant) rank $\mb h(n)$; let $\bb
{AG}_A^{\mb h}$ denote the collection of all such $C$.
 \indn{AGbbh@$\bb {AG}_A^{\mb h}$} \indn{AGbb@$\bb {AG}_A$}

Let $\bb {AG}_A\subset \bb G_A$ denote the full subcategory of all $C
\in \bb {AG}_A^{\mb h}$ for all finite $\mb h$.

Assume $A$ is just filtered.  Given $(C,\,F^\bu C)\in \bb F_A$, call $C$
{\it $k$-Artinian with Hilbert function $\mb h$\/} if $G_\bu C \in \bb
{AG}_A^{\mb h}$ and if $F^\bu C$ is finite; let $\bb {AF}_A^{\mb h}$
denote the collection of all such $(C,\,F^\bu C)$.  Now, $F^nC =
F^{n+1}C$ for $n < i(\mb h)$ and $n > s(\mb h)$; so $F^\bu C$ is finite
iff $F^nC = C$ for $n\le i(\mb h)$ and $F^nC = 0$ for $n> s(\mb h)$.
 \indn{AFbbA@$\bb {AF}_A^{\mb h}$}

Let $\bb {AF}_A\subset \bb F_A$ denote the full subcategory of all
$(C,\,F^\bu C)$ with $G_\bu C\in \bb {AG}_A$.  Trivially, the functor
$G_\bu \: \bb F_A \to \bb G_{G_\bu A}$ restricts to a functor $G_\bu \:
\bb {AF}_A \to \bb {AG}_{G_\bu A}$.  Trivially, $\bb {AF}_A$ is
exhaustive and discrete.  Thus, \eqref{prBiex} implies $G_\bu$ is
biexact.

As a consequence, a $(C,\,F^\bu C) \in \bb {AF}_A$ is generated over $A$
by a $k$-submodule $N$ with the induced filtration if $G_\bu N$
generates $G_\bu C$ over $G_\bu A$.  Indeed, first note $G_\bu N \subset
G_\bu C$ as $G_\bu$ is exact.  Now, form the multiplication map $\mu\:
A\ox_k N \to C$.  The induced map $G_\bu\mu\: G_\bu A\ox_k G_\bu N \to
G_\bu C$ is surjective by hypothesis.  Thus, as $G_\bu$ is biexact,
$\mu$ is surjective, as desired.

Finally, call $A$ {\it $k$-Gorenstein\/} if $A^*$ is a locally free of
 \indt{kGor@$k$-Gorenstein}
rank 1 (that is, invertible).
 \end{sbs}

\begin{proposition}\label{prFilt}
 Let $A$ be a filtered $k$-algebra, and $(C,\,F^\bu C)\in \bb{AF}_A^{\mb h}$.

\(1) Then $C$ and $C/F^nC$ and $F^mC$ and $F^mC/F^nC$ are, for all $m\le
n$, locally free of ranks $\sum_p\mb h(p)$ and $\sum_{p<n}\mb h(p)$ and
$\sum_{m\le p}\mb h(p)$ and $\sum_{m\le p<n}\mb h(p)$.

\(2) Then $F^n((C^*)^*) = F^nC$ for all $n$.\hfill

\(3) Then $G_n(C^*) = (G_{-n}C)^*$ for all $n$.

\(4) Then $C^*\in \bb{AF}_A^{\mb h^*}$.
 \end{proposition}

\begin{proof}
 {\bf For (1),} fix $m$.  If $m=n$, then $F^mC/F^nC = 0$.  Suppose $m<n$.
Form the exact sequence $0\to G_nC\to F^mC/F^{n+1}C\to F^mC/F^nC\to 0$.
By hypothesis, $G_nC$ is locally free of rank $\mb h(n)$.  By induction,
assume $F^mC/F^nC$ is locally free of rank $\sum_{m\le p<n}\mb h(p)$.
Thus $F^mC/F^{n+1}C$ is locally free of rank $\sum_{m\le p<n+1}\mb
h(p)$.  But $F^mC = C$ for $m\le i(\mb h)$ and $F^nC = 0$ for $n> s(\mb
h)$.  Thus (1) holds.

{\bf For (2),} note $F^n((C^*)^*) := (C^*/F^{1-n}(C^*))^*$ and
$F^{1-n}(C^*) := (C/F^nC)^*$.  So form the exact sequence $0\to F^nC\to
C\to C/F^nC\to 0$ of $k$-modules.  It splits, as $C/F^nC$ is projective
by (1).  So $0\to (C/F^nC)^* \to C^*\to (F^nC)^*\to 0$ is exact.  Hence
$(F^nC)^* = C^*/F^{1-n}(C^*)$.  Thus $F^n((C^*)^*) = ((F^nC)^*)^*$.  But
$((F^nC)^*)^* = F^nC$ as $F^nC$ is locally free of finite rank by (1).
Thus (2) holds.

{\bf For (3),} form the exact sequence $0\to G_nC\to C/F^{n+1}C\to
C/F^nC\to 0$. Note that $C/F^nC$ is projective by (1).  So replace $n$
by $-n$, and dualize.  Thus $0\to (C/F^{-n}C)^*\to (C/F^{1-n}C)^*\to
(G_{-n}C)^*\to 0$ is exact.  But the first term is equal to
$F^{n+1}(C^*)$, and the second, to $F^n(C^*)$.  Thus (3) holds.

{\bf For (4),} note (3) yields $G_\bu(C^*) \in \bb{AG}_A^{\mb h^*}$.  Also,
$F^n(C^*) := (C/F^{1-n}C)^*$ for all $n$ by \eqref{eqFm4}; so $F^\bu
(C^*)$ is finite as $F^\bu C$ is.
 \end{proof}

\begin{proposition}\label{prSes}
   Let $A$ be a filtered $k$-algebra, and $0 \to B \xto\beta C\xto\gamma
D \to 0$ an exact sequence in $\bb{EF}_A$.

\(1) Assume $D\in \bb{AF}_A$.  Then $B\in \bb{AF}_A$ iff $C\in
\bb{AF}_A$.

\(2) Assume $B,\,C,\,D\in \bb{AF}_A$.  Then $0 \to D^* \xto{\gamma^*}
C^* \xto{\beta^*} B^* \to 0$ is exact in $\bb{EF}_A$.
 \end{proposition}

\begin{proof}
 {\bf For (1),} given $n$, note $G_n(D )$ is locally free of constant finite
rank.  Hence, $G_n(B)$ is so iff $G_n(C)$ is so, since $0 \to G_n(B) \to
G_n(C) \to G_n(D )\to 0$ is exact by \eqref{prBiex}.  Also, $F^\bu B$ is
finite iff $F^\bu C$ is, since $F^\bu D$ is.  Thus (1) holds.

{\bf For (2),} given $n$, set $p := 1-n$, and form the following commutative
diagram:
 \begin{equation}\label{eqSes1}
 \begin{CD}
        0 @>>>   D^*    @>>>   C^*   @>>>   B^*  @>>> 0\\
        @.      @VVV          @VVV         @VVV\\
        0 @>>> (F^pD)^* @>>> (F^pC)^* @>>> (F^pB)^* @>>> 0
 \end{CD}
 \end{equation}
 The upper sequence is exact, since $0 \to B \to C\to D\to 0$ splits as
$D$ is projective by \eqref{prFilt}(1).  Similarly, the lower sequence
is exact.

The sequence $0\to F^pD\to D\to D/F^pD\to 0$ splits as $D/F^pD$ is
projective by \eqref{prFilt}(1).  So the left map in \eqref{eqSes1} is
surjective, and its kernel is $(D/F^pD)^*$, or $F^n(D^*)$ by
\eqref{eqFm4}.  Similarly, the middle and right maps are surjective, and
their kernels are $F^n(C^*)$ and $F^n(B^*)$.  Hence $0\to F^n(D^*)\to
F^n(C^*)\to F^n(B^*)\to 0$ is exact by the 9-Lemma.  Thus (2) holds.
 \end{proof}

\section{Dualities}\label{seDu}

\begin{setup}\label{se3}
  Keep the setup of \eqref{se2}.  For a $k$-algebra $K$ and a $k$-module
$C$, set $$C_K := C\ox_k K.$$
 \indn{CK@$C_K$}

Let $A$ be a filtered $k$-algebra, and $C \in \bb F_A$.  Given a
$k$-algebra $K$, note that $C_K$ carries an induced filtration; by
definition, $F^n(C_K)$ is the image of $(F^nC)_K$ in $C_K$.

Given a contravariant functor $\mb D$ from an additive category $\bb S$
to itself, call $\mb D$ {\it dualizing\/} if $\mb D^2$ is naturally
 \indt{dualizing}
isomorphic to the identity functor.  Trivially, $\mb D$ is an
equivalence of $\bb S$ with its opposite category.

Call an object $D$ of $\bb S$ {\it dualizing\/} if $\Hom_{\bb S}(C,\,
D)$ is in $\bb S$ for any $C\in\bb S$ and if the functor $C\mapsto
\Hom_{\bb S}(C,\, D)$ is dualizing.  More generally, given a forgetful
functor from $\bb S$ to another additive category $\bb T$, call an
object $D$ of $\bb S$ {\it dualizing for\/} $(\bb S,\,\bb T)$ if
$\Hom_{\bb T}(C,\, D)$ is in $\bb S$ for any $C\in\bb S$ and if
$C\mapsto \Hom_{\bb T}(C,\, D)$ is dualizing for $\bb S$.
 \end{setup}

\begin{proposition}\label{*duality}
 Let $A$ be a $k$-algebra. Then $*\:\bb{AM}_A\to \bb{AM}_A$ is a
dualizing functor.  If $A$ is filtered, then $*$ induces two more
dualizing functors: $*\:\bb{AF}_A\to \bb{AF}_A$ and\/ $* \:
\bb{AG}_{G_\bu A} \to \bb{AG}_{G_\bu A}$;  the latter two $*$'s commute
with $G_\bu \: \bb {AF}_A \to \bb {AG}_{G_\bu A}$.

Moreover, if $A$ is filtered, then for any $B,\,C\in \bb{AF}_A$, 
\begin{equation}\label{eqduality1}
 \Hom_A(B,\,C) = \Hom_A(C^*,\,B^*) \tqn{in} \bb F_A.
 \end{equation}

 \end{proposition}

\begin{proof}
 Given $C\in\bb{AM}_A$, note that the natural map $C\to (C^*)^*$ is an
isomorphism.  Thus $*\:\bb{AM}_A\to \bb{AM}_A$ is a dualizing functor.

 Assume $A$ is filtered.  Given $C\in\bb{AF}_A$, note $C\in\bb{AM}_A$ by
\eqref{prFilt}(1); moreover, $C^*\in \bb{AF}_A$ by \eqref{prFilt}(4).
Thus the functor $*\:\bb{AM}_A\to \bb{AM}_A$ induces a functor
$*\:\bb{AF}_A\to \bb{AF}_A$, which is dualizing by \eqref{prFilt}(2).

Given $B,\,C\in \bb{AF}_A$, note that, as $*\:\bb{AF}_A\to \bb{AF}_A$ is
dualizing,
 \begin{equation}\label{eqduality2}
  \Hom_{\,\bb F_A}(B,\,C) = \Hom_{\,\bb F_A}(C^*,\,B^*) \tqn{in} \bb F_A.
 \end{equation}
 Given $p$, note  $C^*[p] = (C[-p])^*$.  Hence, replacing $C$ by $C[p]$
in \eqref{eqduality2} yields
\begin{equation*}\label{eqduality3}
 F^p\Hom_A(B,\,C) = F^p\Hom_A(C^*,\,B^*).
 \end{equation*}
 But $F^\bu B$ and $F^\bu C$ are finite; so $F^\bu\Hom_A(B,\,C)$ and $
F^\bu\Hom_A(C^*,\,B^*)$ are finite too.  Thus \eqref{eqduality1} holds.

Given $C\in \bb {AG}_{G_\bu A}$, say $C = \bigoplus_p C_p$.  Note $C^* =
\bigoplus_q (C_{-q})^*$; so $C^*\in \bb{AG}_{G_\bu A}$.  Thus
$*\:\bb{AM}_A\to \bb{AM}_A$ restricts to a dualizing functor
$*\:\bb{AG}_A\to \bb{AG}_A$.  These two functors commute with $G_\bu$
owing to \eqref{prFilt}(3).
 \end{proof}

\begin{proposition}[Relative Matlis Duality]\label{prRMD}
 Let $A$ be a $k$-Artinian $k$-algebra.  Then $A^*$ is dualizing and
injective for $\bb{AM}_A$.

Assume $A\in \bb{AF}_A$.  Then $A^*$ is dualizing for
$(\bb{AF}_A,\,\bb{M}_A)$, and $(G_\bu A)^*$ is dualizing for
$(\bb{AG}_{G_\bu A},\,\bb{M}_{G_\bu A})$.  Moreover, $\Hom_{G_\bu
A}\bigl(G_\bu C, \,(G_\bu A)^*\bigr) = G_\bu \Hom_A(C, \,A^*)$ for any
$C$ in $\bb{AF}_A$, and $(G_\bu A)^* = G_\bu(A^*)$.  Lastly, if also
$G_0A=k$, then $k$ embeds in $A^*$ as the smallest nonzero submodule $D$
with $A^*/D\in \bb{AF}_A$.
 \end{proposition}

\begin{proof}
 Note $A^*\in \bb{AM}_A$ and $\Hom_{\bb{AM}_A}(C, \,A^*) =
\Hom_A(C,\,A^*)$ for $C\in \bb{AM}_A$; see \eqref{sbArt}.  However,
$\Hom_A(C,\,A^*) = C^*$ by \eqref{eqeqFm7} with $A$ and $C$ filtered
trivially, and $*$ is dualizing for $\bb{AM}_A$ by \eqref{*duality}.
Thus $A^*$ is a dualizing module for $\bb{AM}_A$.  In $\bb{AM}_A$,
moreover, $A^*$ is injective, as $C\mapsto C^*$ converts projectives
into injectives.

Assume $A\in \bb{AF}_A$.  Then, similarly, \eqref{*duality} and
\eqref{eqeqFm7} imply that $A^*$ is dualizing for $(\bb{AF}_A,\,\bb
M_A)$, and $(G_\bu A)^*$ is dualizing for $(\bb{AG}_{G_\bu A},\,\bb
M_{G_\bu A})$.

Moreover, $(G_\bu C)^* = G_\bu(C^*)$ by \eqref{prFilt}(3).  Apply
\eqref{eqeqFm7} to $G_\bu A$ and to $A$.  Thus $\Hom_{G_\bu A}(G_\bu C,
\,(G_\bu A)^*) = G_\bu \Hom_A(C, \,A^*)$.  On the other hand, taking $C
:= A$ yields $(G_\bu A)^* = G_\bu(A^*)$.

Lastly, let $\delta\: D\into A^*$ be the inclusion of a nonzero
submodule with $A^*/D$ in $\bb{AF}_A$.  Then $D\in \bb{AF}_A$ by
\eqref{prSes}(1).  Set $C:= D^*$.  Then $C\in \bb{AF}_A$ by
\eqref{prFilt}(4).  So $G_0C$ is locally free.  Moreover, $\delta^*\:
A\to C$ is surjective and $C$ has the induced filtration by
\eqref{prSes}(2) and \eqref{sbSm}.  So $G_0A\to G_0C$ is well defined
and surjective.  But $G_0C\neq0$, as $1\in G_0C$ and as $1\neq 0$ since
$G_\bu C$ had the same rank as $C$.

Assume $G_0A=k$.  Then, by the above, $G_0A\risom G_0C$. Hence the
quotient map $\kappa\: A\onto k$ factors via $\delta^*\: A\onto C$.  So
$\kappa^*: k\into A^*$ factors via $\delta\: D\into A^*$.  Now, $F^1A\in
\bb{AF}_A$; so $A^*/k = (F^1A)^*$ by \eqref{prSes}(2), and $(F^1A)^* \in
\bb{AF}_A$ by \eqref{prFilt}(4).  Thus $\kappa^*$ embeds $k$ in $A^*$ as
the smallest nonzero submodule $D$ with $A^*/D\in \bb{AF}_A$.
 \end{proof}

\begin{sbs}[Apolarity]\label{sbApl}
 Assume $A$ is a graded $k$-algebra.  Let $\mb b\:\Z\to \Z$ a function,
and $B$ a graded $A$-module whose graded pieces $B_p$ are locally free
of rank $\mb b(p)$ for all $p$.  Set $(B^\dg)_q := (B_{-q})^*$ and
 \indn{Badg@$B^\dg$}
$B^\dg := \bigoplus_q (B^\dg)_q$. Note that $B^* = \prod(B^\dg)_{q}$,
that $B^\dg$ is an $A$-submodule of $B^*$, that $B^\dg$ is graded, and
that $(B^\dg)_q$ is locally free of rank $\mb b^*(q)$ for all $q$.
Often, $B^\dg$ is called the {\it graded dual\/} of $B$.
 \indt{graded dual}

Given any $A$-submodule $I$ of $B$, its {\it apolar annihilator}
 \indt{apolar annihilator}
$(0:_{B^\dg}I)$ is defined by
 \begin{equation*}\label{eqApN}
 (0:_{B^\dg}I) := \{\, f\in B^\dg \mid \psi\cdot f = 0 \text{ for all }
\psi \in I \,\}.
 \end{equation*}
 \indn{0colI@$(0:_{B^\dg}I)$}
 Note that $(0:_{B^\dg}I)$ is an $A$-submodule of $B^\dg$.

Note $B^{\dg\dg} = B$.  Hence, given any $A$-submodule $D$ of
$B^\dg$,
 \begin{equation*}\label{eqApN2}
 (0:_BD) =  \{\, \psi\in B \mid \psi\cdot f = 0 \text{ for all } f \in D
\,\}.
 \end{equation*}  

 Equip $B$ and $B^\dg$ with the filtrations arising from their
gradings; equip every subquotient of $B$ and $B^\dg$ with the
induced filtration.  Fix a finite function $\mb h\:\Z\to \Z$.

Denote by $\Psi_B$ the set of $A$-submodules $I \subset B$ with $B/I
\in \bb{AM}_A$.  Denote by $\Delta_{B^\dg}$ the set of $A$-submodules
$D\subset B^\dg$ with $D \in \bb{AM}_A$ and with $B^\dg/D$ flat over
$k$.
 \indn{Psibm@$\Psi_B$} \indn{Delbm@$\Delta_{B^\dg}$}

Denote by $\mb F\Psi_B^{\mb h} \subset \Psi_B$ the subset of $I$
 \indn{FmbPsihm@$\mb F\Psi_B^{\mb h}$}
 \indn{FmbDelhstB@$\mb F \Delta_{B^\dg}^{\mb h^*}$}
 with $B/I \in \bb{AF}_A^{\mb h}$, and by $\mb F \Delta_{B^\dg}^{\mb
h^*} \subset \Delta_{B^\dg}$ the subset of $D$ with $D \in
\bb{AF}_A^{\mb h^*}$ and with $G_\bu(E)$ flat over $k$ where $E :=
B^\dg/D$.  Note that $E$ is flat over $k$ when $G_\bu(E)$ is, as $E =
\bigl(\bigoplus_{p\notin [i,\,s]} B^\dg_{-p}\bigr) \oplus
(F^{-s}E/F^{-i+1}E)$ with $i := i(\mb h)$ and $s := s(\mb h)$ and as
$F^{-s}E/F^{-i+1}E$ is flat by the proof of \eqref{prFilt}(1).

Denote by $\mb H\Psi_B^{\mb h} \subset \mb F\Psi_B^{\mb h}$ and
$\mb H\Delta_{B^\dg}^{\mb h^*} \subset \mb
F\Delta_{B^\dg}^{\mb h^*}$ the subset of homogeneous $I$ and the
subset of homogeneous $D$.
 \indn{HmbPsihm@$\mb H\Psi_B^{\mb h}$}
 \indn{HmbDelhstB@$\mb H\Delta_{B^\dg}^{\mb h^*}$}
 \end{sbs}

\begin{theorem}[Macaulay Duality]\label{thGMD3}
 Keep the setup of \eqref{sbApl}. 

\(1) Then $I\mapsto (0:_{B^\dg}I)$ gives a bijection $\Psi_B
\risom \Delta_{B^\dg}$, and $D\mapsto (0:_BD)$, its inverse.  Also,
$(0:_{B^\dg}I) = (B/I)^*$ and $B/(0:_BD) = D^*$.

Further, if $B = A$ and if $I$ and $D$ correspond, then $A/I$ is a
$k$-algebra, and $D$ is a dualizing module for $\bb{AM}_{A/I}$; also,
then $A/I$ is $k$-Gorenstein iff, for every maximal ideal $\go m$ of
$k$, there's a field $K$ containing $k/\go m$ such that $D_K$ is
$A_K$-cyclic.

\(2) The bijection in \(1) induces another, $\mb F\Psi_B^{\mb h}
\risom \mb F\Delta_{B^\dg}^{\mb h^*}$.  Further, if $A = B$ and
if $I$ and $D$ correspond, then $D$ is a dualizing module for
$(\bb{AF}_{A/I},\,\bb{AM}_{A/I})$.
 
The second bijection restricts to a third, $\mb H\Psi_B^{\mb h}
\risom \mb H\Delta_{B^\dg}^{\mb h^*} $.  These two bijections
commute with taking associated graded modules.  Furthermore, if $A = B$
and if $I$ and $D$ correspond, then $D$ is a dualizing module for
$(\bb{AG}_{A/I},\,\bb{AM}_{A/I})$.
 \end{theorem}

\begin{proof}
 {\bf For (1),} given $I\in \Psi_B$, set $C := B/I$.  Note $0\to I\to B\to
C\to 0$ splits over $k$ as $C\in \bb{AM}_A$.  So $0\to C^* \to B^* \to
I^*\to 0$ splits over $k$.  Correspondingly, there's a retraction
$\rho\: B^*\onto C^*$.

As $C\in \bb{AM}_A$, there are $m$ and $n$ with $G_pC = 0$ for $p< m$
and $p>n$.  Hence $B_p$ maps to $0$ in $C$ for $p< m$ and $p > n$.  So
$\bigoplus_{p=m}^n B_p$ surjects onto $C$.  Therefore, 
      $$\ts C^*\subset \bigoplus_{p=m}^n (B_p)^* \subset B^\dg
         \subset B^*.$$

Note $\rho|C^* = 1$.  So $\rho|B^\dg$ splits $0\to C^* \to B^\dg
\to B^\dg/C^*\to 0$ over $k$.  Thus $B^\dg/C^*$ is flat over
$k$, as $B^\dg$ is.  But $B^\dg/C^* \subset B^*/C^* = I^*$.  So
$f\in B^\dg$ maps to $0$ in $I^*$ iff $f\in C^*$; thus $C^* =
(0:_{B^\dg}I)$.

However, $C\in \bb{AM}_A$.  So plainly $C^*\in \bb{AM}_A$ .  Thus
$C^*\in \Delta_{B^\dg}$.

Conversely, given $D\in \Delta_{B^\dg}$, note $D$ is finitely
generated over $k$ as $D \in \bb{AM}_A$.  So there are $m,\,n$ with
$D\subset E:= \bigoplus_{p=m}^n (B^\dg)_p$.  Set $E' :=
\bigoplus_{p\notin[m,\,n]} (B^\dg)_p$, so $B^\dg = E\oplus E'$.
Then $B^\dg/D = (E/D) \oplus E'$.  But $B^\dg/D$ is flat.  So
$E/D$ is flat, so projective.  Thus the inclusion $D\into E$ induces a
surjection $E^*\onto D^*$.

Plainly the surjection $(B^\dg)^*\onto E^*$ restricts to a
surjection $(B^\dg)^\dg \onto E^*$.  But $(B^\dg)^\dg =
B$.  Thus the the inclusion $D\into B^\dg$ induces a surjection
$B\onto D^*$; also, it's not just a $k$-map, but an $A$-map.

Set $I := \Ker(B\onto D^*)$.  Then $I$ is an $A$-submodule of $B$, as
$B\onto D^*$ is an $A$-map.  Further, $B/I = D^*$.  But $D \in
\bb{AM}_A$, so $D^* \in \bb{AM}_A$. Thus $I\in \Psi_B$.  Also, $\psi
\in B$ maps to $0$ in $D^*$ iff $\psi\in I$; so $\psi\cdot f = 0$ for
all $f \in D$.  Thus $I = (0:_BD)$.

Next, let's check that the functions $I\mapsto (B/I)^*$ and $D\mapsto
\Ker(B\onto D^*)$ are inverse.  Given $I$, set $D:= (B/I)^*$.  Then $D^*
= B/I$ as $B/I \in \bb{AM}_A$.  Thus $\Ker(B\onto D^*) = I$.
Conversely, given $D$, set $I := \Ker(B\onto D^*)$.  Then $B/I = D^*$.
So $(B/I)^* = D^{**}$.  But $D^{**} = D$ as $D \in \bb{AM}_A$.  Thus
$(B/I)^* = D$, as desired.

Further, assume $A = B$ and $I$ and $D$ correspond.  Set $C:= A/I$.
Then $C$ is, plainly, a $k$-algebra, and $D$ is a dualizing module for
$\bb{AM}_C$ by \eqref{prRMD}, as $D = C^*$.

Also, by \eqref{prGor}, the $k$-algebra $C$ is $k$-Gorenstein iff, for
every $K$, the $C_K$-module $(C^*)_K$ is cyclic, or equivalently,
$(C^*)_K$ is $A_K$-cyclic.  But $D = C^*$.  Thus (1) holds,

{\bf For (2),} given $I \in \mb F\Psi_B^{\mb h}$, set $C:=B/I \in
\bb{AF}_A^{\mb h}$.  Note $C^*\in \bb{AF}_A^{\mb h^*}$ by
\eqref{prFilt}(4).

Fix $p$.  As $G_pC$ is projective, the quotient map $B_p\onto G_pC$ has
a right inverse.  So $G_pC$ is a direct summand of $B_p$; hence
$(B_p)^*/(G_pC)^*$ is one of $(B_p)^*$, so is projective.  But $(B_p)^*
=: (B^\dg)_{-p}$.  And $(G_pC)^* = G_{-p}(C^*)$ by
\eqref{prFilt}(3).  Also, $(B^\dg)_{-p} / G_{-p}(C^*) =
G_{-p}(B^\dg/C^*)$ by \eqref{prBiex}.  Thus $G_{-p}(B^\dg/C^*)$
is projective, so flat; hence, $G_{\bu}(B^\dg/C^*):= \bigoplus
G_{-p}(B^\dg/C^*)$ is flat too.  Thus $C^* \in \mb
F\Delta_{B^\dg}^{\mb h^*}$.

Conversely, given $D\in \mb F\Delta_{B^\dg}^{\mb h^*}$, note
$D^*\in \bb{AF}_A^{\mb h}$ by \eqref{prFilt}(4).  The first assertion of
(2) now follows from (1), and the second, from \eqref{prRMD}.

Plainly, if $I$ and $D$ are homogeneous, then so are $(0:_{B^\dg}I)$ and
$(0:_BD)$.  Thus the second correspondence restricts as asserted.
Finally, \eqref{prRMD} now yields the last two assertions of (2).  Thus
(2) holds.
 \end{proof}

 \begin{remark}\label{reMacI}
 If $B\in\bb{AM}_A$, but $A$ and $B$ are not graded, then \eqref{thGMD3}
is generalized by the statement below.  Its proof is omitted, as it is
similar and simpler.  If $A$ and $B$ are graded, then the statement
below is a corollary of \eqref{thGMD3}.

First, some notation.  Denote by $\mb Q_B$ the set of quotient
$A$-modules $C$ of $B$ such that $C\in\bb{AM}_A$.  If $A$ is filtered
and $B\in \bb{AF}_A$, denote by $\mb {FQ}_B^{\mb h}$ the subset of
$C\in \mb Q_B$ such that, with the induced filtration,
$C\in\bb{AF}_A^{\mb h}$.

Denote by $\mb S_B$ the set of $A$-submodules $D$ of $B$ with $B/D \in
\bb{AM}_A$.  If $A$ is filtered and $B\in \bb{AF}_A$, denote by $\mb
{FS}_B^{\mb h}$ the subset of $D\in \mb S_B$ such that, with the induced
filtrations, $D \in \bb{AF}_A^{\mb h}$ and $B/D \in \bb{AF}_A$.

 If $A$ is graded and $B\in\bb{AG}_A$, denote by $\mb {HQ}_B^{\mb h}$
and $\mb {HS}_B^{\mb h}$ the analogues of $\mb {FQ}_B^{\mb h}$ and $\mb
{FS}_B^{\mb h}$ where $C$ and $D$ are homogeneous.

Here's the statement in the corresponding two parts:{\it

\(1) Then $C\mapsto C^*$ gives a bijection $\mb Q_B\risom \mb S_{B^*}$,
and $D\mapsto D^*$, its inverse.

\(2) Assume $A$ is filtered and $B\in\bb{AF}_A$.  Then there's a
commutative square:
 \begin{equation*}\label{eqGMD2cd}
  \begin{CD}
  \mb {FQ}_B^{\mb h} @>\sim>> \mb {FS}_{B^*}^{\mb h^*}\\
      @VVV                       @VVV\\
   \mb {HQ}_{G_\bu B}^{\mb h} @>\sim>> \mb {HS}_{G_\bu B^*}^{\mb h^*}
  \end{CD}
 \end{equation*}
 The horizontal maps are, as indicated, bijective; they're given by
$C\mapsto C^*$, and their inverses, by $D\mapsto D^*$.  The vertical
maps are given by $C\mapsto G_\bu C$ and $D\mapsto G_\bu D$.}
 \end{remark}

\begin{example}\label{exPolyRg}
 Fix variables $X_1,\dotsc,X_r$; take $A := k[X_1,\dotsc,X_r]$.  Weight
$X_i$ by a positive integer $w_i$; grade the polynomial ring $A$ by
weighted degree.  Thus the graded piece $A_n$ is the free $k$-module on
the monomials $M$ of weighted degree $n$.

The graded dual $A^\dg$ can be usefully viewed as the free
$k$-module on all the Laurent monomials $1/M$, namely as
$k[X_1^{-1},\dotsc,X_r^{-1}]$, provided the product of a monomial $L\in
A$ with $1/M$ is taken to be $L/M$ if $L\mid M$ and to be $0$ if not.
Then the graded piece $(A^\dg)_{-n}$ is the free $k$-module on the
$1/M$ of weighted degree $-n$.

As an abstract $A$-module, $A^\dg$ is, plainly, independent of the
choice of the $w_i$.  Moreover (cf.\ \cite{KES}*{Subsec.\,3.2}), $A^\dg$
consists of the maps $A\to k$ that are continuous when $k$ is given the
discrete topology and $A$ is given the topology defined by the
filtration with $F^pA$ generated by the monomials of degree at least
$p$; notice that two choices of the $w_i$ give rise to cofinal
filtrations, and so the same topology on $A$.  Thus, a better name for
$A^\dg$ would be the {\it continuous dual}.
 \indt{continuous dual}

If $w_i=1$ for all $i$, then $A^\dg$ is equal to the $k$-algebra of
divided powers; see \cite{Eis}*{Sec.\,A2.4} or \cite{JPAA217}*{Sec.\,3,
pp.\,2258--59}.  If $k=\bb C$ too, then $A^\dg$ is the usual algebra of
polynomial differential operators on $A$; see \cite{Eis}*{Prp.\,A2.8}.
Remarkably, in this case, $A$ can as well be viewed as the algebra of
polynomial differential operators on $A^\dg$, viewed as a polynomial
ring in the $X_i^{-1}$; see \cite{Eis}*{Ex.\,21.7}.  These additional
structures certainly have their charm (and in \eqref{sbApl}, justify the
term ``apolarity''), but in the present work, they'd just add an
unnecessary layer of complication.

In \eqref{thGMD3}(2), take $B := A$.  Note that an $A$-submodule $D$ of
$B^\dg$ is finitely generated over $A$ iff it's finitely generated over
$k$.  Also, note that, if $0\neq C:= A/I$, then $I\subset F^1(A):= \la
X_1,\dotsc,X_r\ra$ (and conversely) because $G_0C = k$ since $G_0A\to
G_0C$ is surjective, $G_0A =k$, and $G_0C$ is locally free and is
nonzero as $1\in G_0C$ and $G_\bu C$ is of the same rank as $C$, so is
nonzero.  Thus \eqref{thGMD3}(2) recovers Macaulay Duality essentially
as is stated in \cite{Eis}*{Thm.\,21.6}, but $k$ is any Noetherian ring
in \eqref{thGMD3}(2), not just a field, and the weights $w_i$ needn't be
1.  In addition, \eqref{thGMD3}(2) asserts the compatibility of duality
with forming associated graded modules.

Finally, let $\vf\: C\to A^*$ be an $A$-map.  Say $A_n\subset I$ for $n>
n_0$.  Then, given any monomial $M\in A_n$ with $n> n_0$, note $M\cdot C
= 0$; hence, $M\cdot \vf(C) = 0$.  Now, given any $c\in C$, say $\vf(c)
= \prod_{n\ge0} \bigl(\bigoplus_i (a_{in}/N_{in})\bigr)$ where
$a_{in}\in k$ and $N_{in}\in A_n$.  If $a_{jp}\neq 0$, then $N_{jp}\cdot
\vf(c) \neq 0$, and so $p\le n_o$.  Thus $\vf(c)\in
\bigoplus_{n=0}^{n_0}(A_n)^*$.  Thus $\vf(C)\subset A^\dg$.  Thus
$\Hom_A(C,\,A^*) = \Hom_A(C,\,A^\dg)$.  But $C^* = \Hom_A(C,\,A^*)$ by
\eqref{eqeqFm7}.  Thus $C^* = \Hom_A(C,\,A^\dg)$, as shown in a
different way assuming $k$ is a field in the course of the proof of
\cite{Eis}*{Thm.\,21.6, p.\,531}.
 \end{example}

\section{The  {\it k}-socle}\label{seSoc}

\begin{setup}\label{se4}
  Keep the setup of \eqref{se3}

Fix a filtered $k$-algebra $A$ with $A/F^1A =k$; that is, the
composition $k\to A \to k$ is the identity.
Fix a finite function $\mb h \:\Z\to \Z$, and $C\in \bb{AF}_A^{\mb h}$.
Set
 \begin{equation*}\label{eqse4}
 i := i(\mb h) \and s:= s(\mb h) \and t:= \mb h(s).
 \end{equation*}
  Note $C\big/((F^1A)\cdot C) = C\ox_Ak$, and  $C\ox_A k$ has the
filtration induced from $C$.
 \indn{ihmb@$i(\mb h)$} \indn{shmb@$s(\mb h)$} \indn{tmi@$t$}

Given a $k$-algebra $K$, note that $C_K$ carries an induced filtration;
by definition, $F^n(C_K)$ is the image of $(F^nC)_K$ in $C_K$.
 \end{setup}

\begin{definition}\label{deSoc}
 \indt{kso@$k$-socle}
  The {\it $k$-socle\/} of $C$ and its induced filtration are given by
 \indt{kso@$k$-socle}
 \begin{gather*}\label{eqDefSoc}
 \Soc_k(C) := \Hom_A(k,\,C) = \{\,c\in C\mid (F^1A)\cdot c = 0\,\}
 =: (0:_C F^1A) \subset C,\\
  F^n(\Soc_k(C)) := F^n\Hom_A(k,\,C)
        = \Soc_k(C)\bigcap F^nC \quad\text{for all }n.
 \end{gather*}
 \end{definition}

\begin{lemma}\label{leSoc}
 In $\bb F_A$, there's a canonical isomorphism: $\Soc_k(C) = (C^*\ox_A
k)^*$.
 \end{lemma}

\begin{proof}
 Taking $B:=k$ in \eqref{eqduality1} and taking $B:= C^*$ and $C:=k$
in \eqref{eqeqFm6} yield
\begin{equation*}\label{eqSoc2}
 \Hom_A(k,\,C) = \Hom_A(C^*,\,k^*) = (C^*\ox_A k)^* \tqn{in} \bb F_A.
 \end{equation*}
 Thus \eqref{deSoc} yields the assertion.
  \end{proof}

\begin{lemma}\label{leSocLev}
 \(1) Then $G_iC \risom G_i(C\ox_A k)$.\hfill
  \(2) Then $G_s(\Soc_k(C)) = G_sC$.

\(3) Choose a $k$-submodule $N\subset C$ that maps isomorphically onto
$G_iC$ (such an $N$ exists because $G_iC$ is projective).  Then $A\cdot
N = C$ iff $C\ox_Ak = G_iC$.

\(4) Then \(a) $\Soc_k(C) = G_sC$ and $C^*\ox_A k\in \bb{AM}_A$ iff\/
\(b) $C^*\ox_Ak = G_{-s} (C^*)$.
 \end{lemma}

\begin{proof}
  {\bf For (1),} recall $C\ox_A k = C/(F^1A)C$.  Let $\kappa\:C\onto
C\ox_A k$ be the quotient map.  Then $F^n(C\ox_A k) := \kappa(F^nC)$ for
all $n$; so $\kappa^{-1}(F^n(C\ox_A k)) = F^nC+\Ker(\kappa)$.  But
$\Ker(\kappa) = (F^1A)C$ and $(F^1A)C \subset F^{1+i}C$.  So
$\kappa^{-1}(F^{1+i}(C\ox_A k)) = F^{1+i}C $.  Thus $C/ (F^{1+i}C)
\risom (C\ox_A k)/ (F^{1+i}(C\ox_A k)) $.  Thus (1) holds.

{\bf For (2),} note $F^{s+1}C=0$.  So $F^{s+1}(\Soc_k(C)) =0$.  But
$(F^1A)(F^sC)\subset F^{1+s}C$.  So $F^sC \subset \Soc_k(C)$.  So $F^sC
= F^s(\Soc_k(C))$.  Thus (2) holds.

{\bf For (3),} set $B := A/F^{s-i+1}A$.  Then $C$ is a $B$-module.  Set $\go N
:= F^1B$.  Then $\go N^{s-i+1} = 0$, and $B/\go N = k$.  So by
Nakayama's lemma, $B\cdot N= C$ iff $(B\cdot N)\ox_Bk$ maps onto
$C\ox_{B}k$, iff $N \to C\ox_{B} k$ is surjective.  However, $B\cdot N=
A\cdot N$ and $C\ox_B k = C\ox_A k$.  Thus $A\cdot N = C$ iff $N\to
C\ox_A k$ is surjective.

By hypothesis, $N\risom G_iC$.  But $G_iC = G_i(C\ox_A k)$ by (1).  So
$N\risom G_i(C\ox_A k)$.  The latter factors through $C\ox_A k$.  So
$N\to C\ox_A k$ is injective; so surjective iff it's an isomorphism; so
iff $C\ox_A k\to G_iC$ is an isomorphism.  Thus (3) holds.

{\bf For (4),} first assume (a).  Then $\Soc_k(C)^* = (G_sC)^*$ and
$(C^*\ox_Ak)^{**} = C^*\ox_Ak$.  Thus \eqref{leSoc} and
\eqref{prFilt}(3) yield (b).  Conversely, assume (b). Then $C^*\ox_A
k\in \bb{AM}_A$ as $G_{-s} (C^*)\in \bb{AM}_A$ by \eqref{prFilt}(4).
Also, $(C^*\ox_Ak)^* = (G_{-s} (C^*))^*$.  So \eqref{leSoc} and
\eqref{prFilt}(3) yield $\Soc_k(C)= ((G_sC)^*)^* = G_sC$.  Thus (a)
holds.  Thus (4) holds.
 \end{proof}

\begin{definition}\label{deType}
  Fix a finite function $\mb t\:\Z\to \Z$.

Call $\mb t$ the {\it local generator type\/} of $C$ if $C\ox_A k$ is in
 \indt{local generator type}
$\bb{AF}_A$ with Hilbert function $\mb t$.
 \indt{local generator type}

Call $\mb t$ the {\it $k$-socle type\/} of $C$ if $C^*$ has local
 \indt{ksot@$k$-socle type} \indt{type (level of)}
generator type $\mb t^*$.  (Then $\Soc_k(C)$ is in $\bb{AF}_A$ and has
Hilbert function $\mb t$ by \eqref{leSoc} and \eqref{prFilt}(4), but
perhaps not conversely.)

Call $C$ {\it level\/} if $C^*\ox_Ak = G_{-s} (C^*)$, or equivalently by
 \indt{level}
\eqref{leSocLev}(4), if $\Soc_k(C) = G_sC$ and $C^*\ox_A k\in
\bb{AM}_A$.  If $C$ is graded, then $C$ is level iff $(C^*)_{-s}$
generates $C^*$ by \eqref{leSocLev}(3).  If $C$ is level, call $s$ its
{\it $k$-socle degree\/} and $t$ its {\it type}.
 \indt{ksod@$k$-socle degree}
 \end{definition}

\begin{lemma}\label{leDiag}
 There's a canonical commutative diagram of surjective maps:
 \begin{equation}\label{eqprGAA1}
 \begin{CD}
 G_\bu(C^*) @>{\kappa_1}>> G_\bu(C^*)\ox_{G_\bu A} k @>{\lambda_1}>> G_{-s}(C^*) \\
   @VV=V                            @VV\mu V               @VV=V  \\
  G_\bu(C^*) @>{\kappa_2}>> G_\bu(C^*\ox_A k) @>{\lambda_2}>> G_{-s}(C^*) 
 \end{CD}
 \end{equation}
 \end{lemma}

\begin{proof}
  Plainly, in $G_\bu(C^*)$, there's this nested sequence of ideals:
\begin{equation*}\label{eqleDiag} \ts
 \bigl(\bigoplus_{j>0}G_jA\bigr)\cdot G_\bu(C^*) \subset
 G_\bu\bigl((F^1A)\cdot C^*\bigr) \subset \bigoplus_{j>-s} G_j(C^*).
 \end{equation*}
 A few applications of Noether's Third Isomorphism Theorem now yields
\eqref{eqprGAA1}.
 \end{proof}

\begin{proposition}\label{prGAA}
  Assume $G_\bu C$ is level.  Then so is $C$.
 \end{proposition}

\begin{proof}
 Consider\eqref{eqprGAA1}.  As $G_\bu C$ is level, $\lambda_1$ is
bijective.  So $\mu$ is injective.  But $\mu$ is surjective.  So $\mu$
is bijective.  So $\lambda_2$ is bijective.  So $G_j(C^*\ox k)=0$ for
$j\neq -s$ by \eqref{leSocLev}(1) with $C^*$ for $C$.  So $G_\bu(C^*\ox k) =
C^*\ox k$.  Thus $C^*\ox k = G_{-s}(C^*)$; that is, $C$ is level.
 \end{proof}

\begin{lemma}\label{leBC1}
 Let $V$ and $W$ be $k$-modules, with $V$ finitely generated.

\(1) Let $K$ be a $k$-algebra.  Assume $V$ is flat.  Then $\Hom(V,W)_K =
\Hom(V_K,W_K)$ and $(V^*)_K = (V_K)^*$.

\(2) Assume that, for every maximal ideal $\go m$ of $k$, there's a
field $K$ containing $k/\go m$ such that $V_K = 0$.  Then $V = 0$.

\(3) Let $\vf\: V\to W$ be a map.  Assume $W$ is flat and finitely
generated.  Assume, for every maximal ideal $\go m$ of $k$, there's a
field $K$ containing $k/\go m$ such that the induced map $\vf_K\: V_K\to
W_K$ is an isomorphism.  Then $\vf$ is an isomorphism.
 \end{lemma}

\begin{proof}
 {\bf In (1),} as $V$ is locally free of finite rank, $\Hom(V,W)_K\to
\Hom(V_K,W_K)$ is an isomorphism.  And so, taking $W:=k$ yields $(V^*)_K
= (V_K)^*$.

 {\bf For (2),} note $V = 0$ if, for every $\go m$, the localization $V_{\go
m}$ is 0.  By Nakayama's Lemma, $V_{\go m} = 0$ if $V_{\go m}\ox_{k_\go
m}(k_\go m/\go m k_\go m) = 0$; so, if $V\ox_k(k/\go m) = 0$;
so, if $V\ox_kK = 0$.  Thus (2) holds.

{\bf For (3),} set $Q := \Cok \vf$.  Then $V_K\xto{\vf_K} W_K\to Q_K\to 0$ is
exact for every $K$. But $\vf_K$ is an isomorphism.  So $Q_K = 0$.  But
$W$ is finitely generated; so $Q$ is too.  Thus (2) yields $Q = 0$.

Set $P := \Ker \vf$.  As $W$ is flat, $0\to P_K\to V_K\xto{\vf_K} W_K
\to 0$ is exact for every $K$.  But $\vf_K$ is an isomorphism. So $P_K =
0$.  But $V$ is finitely generated; so $P$ is too. So $P = 0$ by (2).
Thus (3) holds.
 \end{proof}

\begin{lemma}\label{leBC1a}
 Let $B$ and $K$  be $k$-algebras, and $M$ and $N$ two $B$-modules.  Then
 \begin{equation}\label{alBaseChg6}
 (M\ox_BN)_K = (M_K)\ox_{B_K}(N_K).
 \end{equation}
 \end{lemma}

\begin{proof}
   By Associativity, $(M\ox_BN)_K = M\ox_BN_K$.  But $N_K$ is a
$B_K$-module, and it's an $B$-module via restriction of scalars.  So
$M\ox_BN_K = (M\ox_BB_K)\ox_{B_K}N_K$ by Cancellation.  But $M\ox_BB_K =
M_K$ also by Cancellation.  Thus \eqref{alBaseChg6} holds.
 \end{proof}

\begin{lemma}\label{leBaseChg}
  Let $K$ be a $k$-algebra, and $n\in\Z$.

\smallskip

\(1) Then the defining surjection is an isomorphism: $(F^nC)_K
\risom F^n(C_K)$.

\(2) Then there's a canonical isomorphism:  $G_n(C_K) = (G_nC)_K$.

\(3) Then the conditions of \eqref{se4} hold for $A_K$ and $C_K$ in
place of $A$ and $C$.

\(4) Then there are canonical isomorphisms of filtered modules:
 \begin{align}
 (C^*)_K &= (C_K)^*                     \label{alBaseChga}\\
 (C^*\ox_Ak)_K &= (C_K)^*\ox_{A_K}K      \label{alBaseChgb}
 \end{align}
 \end{lemma}

\begin{proof}
 {\bf For (1),} tensor the sequence $0\to F^{n+1}C \to F^nC\to G_nC\to
0$ with $K$.  The resulting sequence is exact, because $G_nC$ is flat.
So $(F^{n+1}C)_K \subset (F^nC)_K$ and
 \begin{equation}\label{eqBaseChg3}
 (F^nC)_K \big/ (F^{n+1}C) _K = (G_nC)_K.
 \end{equation}
 But $F^nC = C$ for $n\ll0$.  So $(F^nC)_K \subset C_K$ for all
$n$.  Thus (1) holds.

{\bf Note (2)} follows from \eqref{eqBaseChg3} and (1).

{\bf For (3),} tensor $0\to F^1A\to A\to k\to 0$ with $K$, getting an exact
sequence.  So $(F^1A)_K = F^1(A_K)$ and $A_K\big/(F^1A)_K = K$, as
desired.

Recall $G_nC$ is locally free of constant rank $\mb h(n)$ over $k$;
hence, $G_n(C_K)$ is so over $K$ by (2).  Hence $C_K$ is in
$\bb{AF}_{A_K}$ and has Hilbert function $\mb h$. Thus (3) holds.

{\bf For (4),} note \eqref{alBaseChga} results from these canonical
isomorphisms, justified after:
 \begin{align}
 F^n((C^*)_K)   &= (F^n(C^*))_K \label{alBaseChg1}\\
                &= ((C/F^{1-n}C)^*)_K\label{alBaseChg2}\\
                &= ((C/F^{1-n}C)_K)^*\label{alBaseChg3}\\
                &= (C_K\big/(F^{1-n}C)_K)^*\label{alBaseChg4}\\
                &= (C_K\big/F^{1-n}(C_K))^*\label{alBaseChg4a}\\
                &= F^n((C_K)^*)\label{alBaseChg5}.
  \end{align}

 First, \eqref{alBaseChg1} holds by (1), as $C^* \in \bb{AF}_A$ by
\eqref{prFilt}(4).  Next, \eqref{alBaseChg2} and \eqref{alBaseChg5} hold
by \eqref{eqFm4}.  Next, $C/F^{1-n}C$ is locally free of finite rank by
\eqref{prFilt}(1); so \eqref{leBC1}(1) yields \eqref{alBaseChg3}. Next,
\eqref{alBaseChg4} holds by the right exactness of $\bu\ox_kK$.
Finally, \eqref{alBaseChg4a} holds by (1).  Thus \eqref{alBaseChga}
holds.

Take $B:=A$ and $N := k$ in \eqref{alBaseChg6}.  Note $N_K = k_K$ where
$k$ is viewed as a $k$-module via the composition $k\to A\to k$.  But
that's the identity.  Thus $N_K = K$.

Take $M := C^*$ also.  Then \eqref{alBaseChg6} becomes $(C^*\ox_Ak)_K =
(C^*)_K \ox_{A_K}K$.  This canonical isomorphism preserves the
filtrations, because the $n$th term on each side is, by definition, the
image of $(F^n(C^*))_K$.  Thus \eqref{alBaseChga} now yields
\eqref{alBaseChgb}.
  \end{proof}

\begin{proposition}\label{prFib}
 Given a function $\mb t\:\Z\to \Z$, these conditions are equivalent:
 \begin{enumerate}
 \item The $A$-module $C$ has $k$-socle type $\mb t$.

\item For every $k$-algebra $K$, the $A_K$-module $C_K$ has $K$-socle
type $\mb t$.

\item
    \begin{enumerate} \def\thenenumii{\alpha{enumii}}
    \item The $k$-module $G_n(C^*\ox_Ak)$ is flat for $-s< n \le -i$.

    \item For every maximal ideal $\go m$ of $k$, there's a field $K$
    containing $k/\go m$ such that $\Soc_K(C_K)$ has Hilbert function
    $\mb t$.
    \end{enumerate}
 \end{enumerate}
 \end{proposition}

\begin{proof}
 Assume (1).  By \eqref{leBaseChg}(3), the conditions of \eqref{se4}
hold for the pair $A_K,\,C_K$ and also, owing to (1), for the pair
$A_K,\,(C^*\ox_Ak)_K$ with $\mb t^*$ for $\mb h$.  Hence,
\eqref{alBaseChgb} implies that $(C_K)^*\ox_{A_K}K$ is in
$\bb{AF}_{A_K}$ and has Hilbert function $\mb t^*$.  Thus (2) holds.

Assume (2).  Then (3) holds as it's a special case of (2).

Assume (3).  Fix $n$.  Note $C^*\in \bb{AF}_A$ with Hilbert polynomial
$\mb h^*$ by \eqref{prFilt}(4).  So \eqref{prFilt}(1) with $C^*$ for $C$
implies $C^*$ is finitely generated over $k$.  So $C^*\ox_A k$ is too.
Thus $G_n(C^*\ox_A k)$ is.

If $n\le-s$, then $F^n(C^*) = C^*$, so $F^n(C^*\ox_Ak) = C^*\ox_Ak$.  If
$n>-i$. then $F^{n}(C^*) = 0$, so $F^{n}(C^*\ox_Ak) = 0$.  Thus
$G_n(C^*\ox_Ak) = 0$ if $n<-s$ or $n>-i$.

However, $G_{-s}(C^*\ox_Ak) = G_{-s}(C^*)$ by \eqref{leSocLev}(1); so
$G_{-s}(C^*\ox_Ak)$ is flat.  Thus (a) implies that, whatever $n$ is,
$G_n(C^*\ox_Ak)$ is flat, so locally free of finite rank.  It remains to
show that $G_n(C^*\ox_Ak)$ has constant rank $\mb t^*(n)$.

So given any $\go m$ and $K$ as in (b), we have to see
$\dim_K((G_n(C^*\ox_Ak))_K)= \mb t^*(n)$.  But $(G_n(C^*\ox_Ak))_K =
G_n((C^*\ox_Ak)_K)$ by the proof of \eqref{leBaseChg}(2) with
$C^*\ox_Ak$ for $C$, as that proof doesn't require the rank to be
constant.

So by \eqref{alBaseChgb}, we have to see $(C_K)^*\ox_{A_K}K$ has Hilbert
function $\mb t^*$, or by \eqref{prFilt}(3), equivalently,
$((C_K)^*\ox_{A_K}K)^*$ has Hilbert function $\mb t$.  Apply
\eqref{leSoc} with $ C_K$ for $C$, noting $C_K\in \bb{AF}_{A_K}$ as $K$
is a field.  Thus (b) now yields (1).
 \end{proof}

\begin{corollary}\label{CoLevel}
 The following conditions are equivalent:
 \begin{enumerate}
 \item The $A$-module $C$ is level of type $t$.

\item For every $k$-algebra $K$, the $A_K$-module $C_K$ is level of type
$t$.

\item For every maximal ideal $\go m$ of $k$, there is a field $K$
containing $k/\go m$ such that $$\dim_K(\Soc_K(C_K)) = t.$$
 \end{enumerate}
 \end{corollary}

\begin{proof}
 Define $\mb t\:\Z\to \Z$ by $\mb t(s) := t$ and $\mb t(j) := 0$ for
$j\neq s$.  Then \eqref{CoLevel} results from \eqref{prFib}, because
\eqref{CoLevel}(3) implies $G_n(C^*\ox_Ak)=0$ for $-s< n$, as shown
next.

We have to show $F^{1-s}(C^*\ox_Ak)=0$, or equivalently, $C^*\ox_Ak
\risom G_{-s}(C^*\ox_Ak)$.  But $G_{-s}(C^*\ox_Ak) = G_{-s}(C^*)$ by
\eqref{leSocLev}(1) with $C^*$ for $C$.  Thus there's a canonical surjection
$\vf\: C^*\ox_Ak \onto G_{-s}(C^*)$, and we have to show it's an
isomorphism.

So given any $\go m$ and $K$ as in (3), it suffices by \eqref{leBC1}(3)
to show $\vf\ox_k K$ is an isomorphism, as $G_{-s}(C^*)$ is projective.
Since $\vf\ox_k K$ is surjective and $K$ is a field, it suffices to show
its source and target have the same dimension over $K$.

The source is equal to $(C_K)^*\ox_{A_K}K$ by \eqref{alBaseChgb}.  Its
dimension is equal to that of $((C_K)^*\ox_{A_K}K)^*$, as $K$ is a
field; so to that of $\Soc_K(C_K)$ by \eqref{leSoc}; so to $t$ by (3).
Finally, the target is $(G_{-s}(C^*))_K$, and $G_{-s}(C^*)$ is locally
free of rank $\mb h^*(-s)$ by \eqref{prFilt}(4); so the dimension of the
target is $\mb h(s) =: t$, as desired.
 \end{proof}

\begin{sbs}[Multilevel modules]\label{sbMLev}
 First, assume $A$ and $C$ are graded.  Set $D :=C^*$,
\begin{equation}\label{eqMlev}\ts
 \Delta^mD := A\bigl(\bigoplus_{j= m}^sD_{-j}\bigr) \subset D
 \and \Lambda^mC := (\Delta^mD)^* \quad\text{for all }m.
 \end{equation}
  Fix $\mb h_m\:\Z\to \Z$ for all $m$.  If $\Delta^mD \in
\mb{H\Delta}_D^{\mb h_m^*}$ for all $m$, call $C$ \textit{multilevel}
and the $\mb h_m$ \textit{its Hilbert functions}.
 \indt{multilevel} \indn{hmbzmbm@$\mb h_m$} \indn{LamC@$\Lambda^mC$}

Given such a $C$, note each $\Lambda^mC$ is a flat quotient of $C$.
Moreover, $\Delta^mD = D$ for $m\le i$, and $\Delta^mD = 0$ for $m > s$;
so $\mb h_m = \mb h$ for $m\le i$, and $\mb h_m =0$ for $m > s$.  Also,
$C$ is level iff $h_m = h$ for $m\le s$, iff $\Lambda^mC = C$ for $m\le
s$.  Lastly, for all $m$ and $p \ge m$, plainly $(\Lambda^mC)_p =C_p$,
and so $\mb h_m(p) = \mb h(p)$.

Next, assume $C$ is filtered.  If the associated graded module $G_\bu C$
is multilevel (but not necessarily of the same $k$-socle type as $C$), call
$C$ \textit{multilevel}.  And if the $\mb h_m$ are the Hilbert functions
of $G_\bu C$, call them the \textit{Hilbert functions} of $C$ as well.

Notice that, if $C$ is level filtered, $C$ need not be multilevel, as
$G_\bu C$ may be of a different $k$-socle type.  For example, below, the
algebra $C$ in \eqref{exArt} is $k$-Gorenstein, so level of type 1 by
\eqref{prArt}, but $G_\bu C$ isn't $k$-Gorenstein when $C$ is filtered
via the standard weighting, although $G_\bu C$ is $k$-Gorenstein via a
different weighting.

Finally, given a graded $A$-module $B$, denote the sets of all
multilevel homogeneous and of all multilevel filtered quotients of $B$
by $\mb{H\Lambda}_B^{\{\mb h_n\}}$ and by $\mb{F\Lambda}_B^{\{\mb
h_n\}}$.
 \indn{HmbLamhBrm@$\mb{H\Lambda}_B^{\{\mb h_m\}}$}
 \indn{FmbLamhBrm@$\mb{F\Lambda}_B^{\{\mb h_m\}}$}
 \end{sbs}

\begin{proposition}\label{prMLev}
 In the setup of \eqref{sbMLev}, fix $C \in \mb{H\Lambda}_B^{\{\mb
h_n\}}$, fix a $k$-algebra $K$, and fix an integer $n$. For all $p$, set
$\mb t(p) := \mb h_p(p) - \mb h_{p+1}(p)$.

\(1) Then $C_K \in \mb{H\Lambda}_{B_K}^{\{\mb h_n\}}$ and
$\Lambda^m(C_K) = (\Lambda^mC)_K$ and $\Delta^m(C^*_K) =
(\Delta^mC^*)_K$.

\(2) Then $\mb t$ is the $k$-socle type of $C$.

\(3) Set $\mb h^\di_m := \mb h_n$ for $m\le n$ and $\mb h^\di_m := \mb h_m$
for $m\ge n$.  Then $\Lambda^nC \in \mb{H\Lambda}_{B_K}^{\{\mb h^\di_n\}}$.

 \(4) Assume $\mb t(n-1)\neq0$, and set $\?C := \bigl(\Delta^{n-1}D
\big/ \Delta^nD\bigr)^*$ and\/ $\?{\mb h} := \mb h_{n-1} - \mb h_n$.
Then $\?C$ is level, and $\?{\mb h}$ is its Hilbert function.
 \end{proposition}

\begin{proof}
 {\bf  In (1),} the assertions are easy to check.

{\bf For (2),} by \eqref{deType}, we have to show that $D_{-p}\big/(F^1A\cdot
D)_{-p}$ is locally free of rank $\mb t(p)$.  But $(F^1A\cdot D)_{-p} =
(\Delta^{p+1}D)_{-p}$, and $\Delta^{p+1}D \in \mb {H\Delta}_D^{\mb
h_{p+1}^*}$.  Thus (2) holds.

{\bf For (3),} note $\Delta^m\Delta^nD = \Delta^nD$ for $m\le n$ and
$\Delta^m\Delta^nD= \Delta^mD$ for $m\ge n$.  For $m \ge n$, form the
following exact sequence:
 \begin{equation}\label{eqprMLev}
  0 \to \Delta^nD \big/\Delta^mD \to D\big/ \Delta^mD \to D\big/ \Delta^nD \to 0.
  \end{equation}
 For all $p$, recall $\Delta^pD \in \smash{\mb{ H\Delta}}_D^{\mb
h_p^*}$.  Hence $\Delta^m\Delta^nD \in \smash{\mb
{H\Delta}}_{\Delta^nD}^{\mb h^{\di*}_m}$.  Thus (3) holds.

{\bf For (4),} note $\?C \in \bb {AG}_A^{\?{\mb h}}$ by \eqref{eqprMLev} with
$n-1$ for $n$ and with $n$ for $m$.  However, plainly,
$(\Delta^{n-1}D)_m = (F^1A\cdot \Delta^{n-2}D)_m$ for $m>-n+1$, and
$(\?C^*)_m = 0$ for $m<-n+1$; so $\?C^*\ox k = (\?C^*)_{-n+1}$.  Thus
(4) holds.
 \end{proof}

\begin{proposition}\label{prMLevCon}
  Keep the setup of \eqref{sbMLev}, and fix $n$.  Set $\?C := \bigl(D
\big/ \Delta^nD\bigr)^*$ and\/ $\?{\mb h} := \mb h - \mb h_n$.  Set $\mb
h'_m := \mb h_n$ for $m\le n$ and $\mb h'_m := \mb h_m$ for $m\ge n$.
Assume $\?C$ is level, and $\?{\mb h}$ is its Hilbert function.  Assume
$\Lambda^nC \in \mb{H\Lambda}_B^{\{\mb h'_n\}}$ and  $\mb h_m = \mb
h$ for $m< n$.  Then $C \in \mb{H\Lambda}_B^{\{\mb h_n\}}$.
 \end{proposition}

\begin{proof}
 We have to show $\Delta^mD \in \mb{H\Delta}_D^{\mb h_m^*}$ for all $m$.

Assume $m < n$.  Set $\?D := D/\Delta^nD = \?C^*$.  Since $\?C$ is
level, $\Delta^m\?D = \?D$.  But $\Delta^m\?D = \Delta^m D/\Delta^nD$.
Hence $\Delta^mD = D$.  By hypothesis, $\mb h_m = \mb h$; so $\mb h_m^*
= \mb h^*$.  Thus $\Delta^mD \in \mb{H\Delta}_D^{\mb h_m^*}$.

Assume $m\ge n$.  Then $\Delta^m\Delta^nD = \Delta^mD$.  But
$\Lambda^nC$ is multilevel, and the $\mb h'_m$, its Hilbert functions.
Thus $\Delta^nD \big/ \Delta^mD$ is flat, and $\Delta^mD$ has $\mb
h_m^*$ as Hilbert function.

As $\?C$ is level, $D\big/ \Delta^nD$ is flat.  But \eqref{eqprMLev} is
exact.  Thus $D/\Delta^mD$ is flat.
 \end{proof}

\section{{\it k}-Gorenstein algebras}\label{seGor}

\begin{setup}\label{se5}
 Keep the setup of \eqref{se3}.  In addition, recall from \eqref{sbArt},
that we call $A$ {\it k-Gorenstein\/} if $A^*$ is a locally free of
 \indt{kGor@$k$-Gorenstein}
rank 1 (that is, invertible).  \end{setup}

\begin{proposition}\label{prGor}
  Let $A$ be a $k$-Artinian $k$-algebra. Then these are equivalent:

\begin{enumerate}
 \item The $k$-algebra $A$ is $k$-Gorenstein.

\item For every $k$-algebra $K$, the $K$-algebra $A_K$ is
$K$-Gorenstein.

\item For every maximal ideal $\go m$ of $k$, there is a field $K$
containing $k/\go m$ such that the $K$-algebra $A_K$ is $K$-Gorenstein.

\item For every maximal ideal $\go m$ of $k$, there is a field $K$
containing $k/\go m$ such that the $A_K$-module $(A^*)_K$ is cyclic.
 \end{enumerate}
 \end{proposition}

\begin{proof}
  Assume (1).  Then $(A^*)_K$ is an invertible $A_K$-module for every
$k$-algebra $K$.  But $(A^*)_K = (A_K)^*$ by \eqref{leBC1}(1), and
plainly, $A_K$ is $K$-Artinian.  Thus (2) holds.

Trivially, (2) implies (3).

Assume (3).  Then, given $K$ as in (3), the $A_K$-module $(A_K)^*$ is
invertible.  But $A_K$ is semilocal.  Hence, $(A_K)^*$ is free of rank
1; so, cyclic.  But $(A^*)_K = (A_K)^*$ by \ref{leBC1}(1).  Thus (4)
holds.

Assume (4).  Given $\go m$, set $L := k/\go m$.  Below we find an $f\in
A^*$ whose image in $(A^*)_L$ generates it over $A_L$.  Given $f$,
define $\vf\: A\to A^*$ by $\alpha\mapsto \alpha f$.  Then the induced
$A_L$-map $\vf_L\: A_L\to (A^*)_L$ is surjective.  But $A$ and $A^*$ are
locally free of the same rank at $\go m$.  So $A_L$ and $(A^*)_L$ have
the same $L$-dimension.  So $\vf_L$ is an isomorphism.  So the
localization $\vf_\go m$ is an isomorphism by \eqref{leBC1}(3).  Thus
(1) holds.

To find $f$, let $\?K$ be the algebraic closure of $K$.  By
Cancellation, $A_{\?K} = A_K \ox_K\?K$ and $(A^*)_{\?K} = (A^*)_K
\ox_K\?K$.  Thus $(A^*)_{\?K}$ is a cyclic $A_{\?K}$-module.

Note $A_L$ is a finite product of Artinian local $L$-algebras $A_i$.  In
turn, each $A_i\ox_L\?K$ is a finite product of Artinian local
$\?K$-algebras $B_{ij}$.  As $\?K$ is algebraically closed, the residue
field of each $B_{ij}$ is $\?K$.

Note $A^*\ox_AB_{ij}$ is a quotient of $(A^*\ox_AA_i)\ox_L\?K$, which is
a quotient of $(A^*)_{\?K}$.  So $A^*\ox_AB_{ij}$ is a cyclic
$B_{ij}$-module.  So $(A^*\ox_AB_{ij})\ox_{B_{ij}}\?K$ is a
1-dimensional $\?K$-vector space.

  Fix $i$ and $j$.  Take $f_i\in A^*\ox_AA_i$ whose image in
$(A^*\ox_AB_{ij})\ox_{B_{ij}}\?K$ is nonzero, so generates it.  By
Nakayama's Lemma, the image of $f_i$ in $A^*\ox_AB_{ij}$ generates it.

Note $A_i\ox_L\?K$ is flat over $A_i$. Hence $B_{ij}$ is flat over
$A_i$, so faithfully flat since these rings are local.  Hence $f_i$
generates $A^*\ox_AA_i$ over $A_i$, as its image generates
$A^*\ox_AB_{ij}$ over $B_{ij}$.  So $\prod f_i$ generates $(A^*)_L$ over
$A_L$.  Lift $\prod f_i$ to some $f\in A^*$.  Then $f$ is the desired
element.
 \end{proof}

\begin{proposition}\label{prArt}
   Assume $A$ is $k$-Artinian with (finite) Hilbert function $\mb h$.

\(A) Then these five conditions are equivalent:
 \begin{enumerate}
 \item The $k$-algebra $A$ is $k$-Gorenstein (namely, the $A$-module
$A^*$ is invertible).

\item  The  $k$-module  $A^*\ox_Ak$ is invertible.

\item  The  $k$-module  $A^*\ox_Ak$ is flat, and the  $k$-module
$\Soc_k(A)$ is invertible.

\item The $A$-module $A$ is level of type $1$.

\item In $\bb{M}_A$, there's an isomorphism $A\cong L\ox_k A^*$ for
some $k$-module $L$.
 \end{enumerate}
 Moreover, if \(1)--\(5) hold, then $L$ is invertible, and $L\cong
\Soc_k(A) = G_sA$.

 \(B) If $A\in \bb {AG}_A$, say $A = \bigoplus A_n$, then \(1)--\(5) are
equivalent to this condition:
 \begin{enumerate}
 \item[$(6)$] In $\bb{AG}_A$, there's a canonical (homogeneous)
isomorphism $A = A_s\ox_kA^*$.
 \end{enumerate}
 (Here $(A_s\ox_kA^*)_n := A_s\ox_k(A^*)_{n-s}$ for all $n$.)  Moreover,
$\mb h(n) = \mb h(s-n)$.
 \end{proposition}

\begin{proof}
 Plainly, (1) implies (2).

Note (2) and (3) are equivalent owing to \eqref{leSoc}.  Indeed, in
general, an invertible $k$-module is flat, and its dual is invertible.
Conversely, in general, a locally free $k$-module of finite rank has
rank 1 everywhere if its dual does.

Assume (2).  Now, $G_{-s}(A^*) = G_{-s}(A^*\ox_Ak)$ by \eqref{leSocLev}(1)
with $ A^*$ for $C$.  However, $G_{-s}(A^*\ox_Ak)$ is a quotient of
$A^*\ox_A k$.  Let $\kappa\: A^*\ox_A k\onto G_{-s}(A^*)$ be the
quotient map.  Now, $G_{-s}(A^*)$ is locally free of rank $\mb h(s)$,
and $\mb h(s)\neq0$.  So (2) implies $\kappa$ is an isomorphism.  Thus
(4) holds.

Assume (4).  Then $\Soc_k(A) = G_sA$.  Also, $t=1$; that is,
$G_{-s}(A^*)$ is invertible.  But $G_{-s}(A^*) = (G_sA)^*$ by
\eqref{prFilt}(3).  Thus $(G_sA)\ox_k G_{-s}(A^*) \cong k$.

Choose a $k$-submodule $N\subset A^*$ that maps isomorphically onto
$G_{-s}(A^*)$.  As $A$ is level, \eqref{leSocLev}(3) implies $A\cdot N =
A^*$.  And, by the above, $(G_sA)\ox_k N\cong k$.

So the inclusion $N\into A^*$ yields a $k$-map $k\to (G_sA)\ox_k A^*$.
Say 1 maps to $\sum d_i\ox e_i$.  Define an $A$-map $\vf\: A\to
(G_sA)\ox_k A^*$ by $\vf(a) := \sum d_i\ox ae_i$.  Then $\vf$ is
surjective, as $A\cdot N = A^*$.  But $A$ and $A^*$ are locally free of
the same rank, and $G_{s}(A)$ is invertible.  So $\vf$ is an
isomorphism.  Thus (5) holds with $L:= G_sA$.

Assume (5).  Then $k\cong (L\ox_k A^*)\ox_Ak = L\ox_k (A^*\ox_Ak)$.  So
$L$ is invertible, and $L\cong (A^*\ox_Ak)^*$.  But $A\cong L\ox_k A^*$
in $\bb{M}_A$ by (5), and $L\ox_k A^*= (L\ox_kA)\ox_AA^*$.  Thus (1)
holds. Also, $L\cong \Soc_k(A)$ by \eqref{leSoc}.  Thus (A) holds.

In (B), trivially (6)\implies(5).  Conversely, take $N := (A^*)_{-s}$ in
the proof of (4)$\Rightarrow$(5).  Then $\vf\: A\risom A_s\ox_kA^*$ is
canonical, and $\vf(A_n) = A_s\ox_k (A^*)_{n-s}$.  Thus (4)\implies(6),
and as $A_s$ is invertible, $\mb h(n) = \mb h^*(n-s) = \mb h(s-n)$.
Thus (B) holds.
 \end{proof}

\begin{example}\label{exSymNonArt}
 Let's see how a nonstandard weighting can be used to show that a given
graded algebra is not $k$-Gorenstein, thus suggesting a general
question.

 Fix variables $X,\,Y$; take $A := k[X,\,Y]$.  Set $I := \la
X^2,\,XY,\,Y^3\ra$ and $C:= A/I$.  Let $x,\,y \in C$ denote the residues
of $X,\,Y$.  Plainly, $1,\,x,\,y,\,y^2$ form a free $k$-basis of $C$.
So $C$ is $k$-Artinian.  Notice that $I$ is monomial, so homogeneous
under both the standard and the nonstandard weightings.

Let's find $\Soc_k(C)$.  Given $a,\,b,\,c,\,d \in k$, set $\psi :=
a+bx+cy+dy^2$.  Then $x\psi = ax$ and $y\psi = ay+cy^2$.  So $\psi \in
\Soc_k(C)$ iff $a = 0$ and $c = 0$.  Thus $\Soc_k(C) = kx+ky^2$.  Thus
$C$ is, by \eqref{prArt}(A)(1)$\Rightarrow$(3), not $k$-Gorenstein.

First, give $X,\,Y$ their standard weights of 1.  Then the Hilbert
function $\mb h(n)$ has values $1,\,2,\,1$ for $n=0,\,1,\,2$; also,
$s(\mb h)=2$.  Thus $\mb h(n) = \mb h(s-n)$ although $C$ is not
$k$-Gorenstein.

Instead, give $X,\,Y$ weights of 3,\,2.  Then the Hilbert function $\mb
h(n)$ has values $1,\,0,\,1,\,1,\,1$ for $n=0,\,1,\,2,\,3,\,4$; also,
$s(\mb h)=4$.  Thus $\mb h(n) \neq \mb h(s-n)$ for $n=1$, and so $C$ is,
by \eqref{prArt}(B), not $k$-Gorenstein.

This example suggests the following question: given a $k$-Artinian
graded $k$-algebra $C$ with Hilbert function $\mb h(n)$ where $C$ is the
quotient of a polynomial ring by a monomial ideal, but $C$ is not
$k$-Gorenstein, then is there always a way to weight the variables so
that $\mb h(n) \neq \mb h(s-n)$ where $s := s(\mb h)$?

Likely the answer is yes, as Craig Huneke pointed out privately.
Namely, $\Soc_k(C)$ is generated by residues of monomials $M$, and there
must be at least two $M$ by \eqref{prArt}. Likely we can choose weights
so that, with $m := \max\{\deg(M)\}$, at least two $M$ have degree $m$;
then $m = s$ by \eqref{leSocLev}(2), and so $\mb h(s) > 1 = \mb h(0)$.
 \end{example}

\begin{example}\label{exArt}
 Here's a (fairly standard) example of a $k$-Gorenstein algebra $C$ such
that, under the standard weighting, $G_\bu C$ is not $k$-Gorenstein.
However, under a nonstandard weighting, $C$ is graded, and so $G_\bu C =
C$; thus $G_\bu C$ is $k$-Gorenstein.

 Keep the setup of \eqref{exPolyRg}.  Take $r:=2$, so $A:=
k[X_1,\,X_2]$.  Form the ideal $\go M := \la X_1,\,X_2\ra$, so $A/\go M = k$.
Identify $A^\dg$ with $k[X_1^{-1},\,X_2^{-1}]$.

Set $F:= X_1^{-2}+X_2^{-3}\in A^\dg$ and $D := AF\subset A^\dg$ and
$I := (0:_AD)\subset A$.  Note
 \begin{equation*}\label{eqArt1}\ts
 \bigl(\sum a_{ij}X_1^iX_2^j\bigr)\cdot F = a_{00}F + a_{10}X_1^{-1}
        + a_{01}X_2^{-2} + a_{02}X_2^{-1} + (a_{20} + a_{03}).
 \end{equation*}
 for any $a_{ij}\in k$.  
 So $\bigl(\sum a_{ij}X_1^iX_2^j\bigr)\cdot F = 0$ iff
$a_{00},\,a_{10},\,a_{01},\,a_{02},\,(a_{20} + a_{03}) = 0$.
Thus $I$ is the ideal generated by $X_1^{2}-X_2^{3},\ X_1X_2$.
 
Also $\go MF = kX_1^{-1} + kX_2^{-2} + kX_2^{-1} + k$ and $D = kF \oplus
\go MF$.  Denote the residues in $D/\go MD$ of $X_2^{-1}$ and $F$ by
$x_2^{-1}$ and $f$. Thus $D/\go MD = kf = kx_2^{-3}$.  But $D/\go MD =
D\ox_Ak$.  Thus $D\ox_Ak \cong k$.  Furthermore, $D\in \bb{AM}_A$.
However, $kF+kX_2^{-3} = kX_1^{-2} + kX_2^{-3}$.  Thus $A^\dg/D$ is
free over $k$, so projective.

Set $C:= A/I$.  Then $C = D^*\in \bb{AM}_A$ by \eqref{thGMD3}(1).  So $D
= D^{**} = C^*$.  However, $D := AF$.  So for every maximal ideal $\go
m$ of $k$ and for every field $K$ containing $k/\go m$, the $A_K$-module
$D_K$ is cyclic.  Thus by \eqref{thGMD3}(1), the $k$-algebra $C$ is
$k$-Gorenstein.

Since $D\ox_Ak\cong k$, if $C\in \bb{AF}_A$, then by
\eqref{prArt}(A)(2)$\Rightarrow$(1), again $C$ is $k$-Gorenstein.

To find $\Soc_k(C)$, set $E := k + kX_1 + kX_1^2 + kX_2 + kX_2^2$.  Then
$E\oplus I = A$.  Given $\psi := a+ bX_1+ cX_1^2+ dX_2+ eX_2^2\in E$,
note $X_1\psi\in I$ iff $a,\,b=0$ and $X_2\psi \in I$ iff $a,\,d,\,e =
0$.  For $i = 1,2$, let $x_i\in C$ be the residue of $X_i$.  Thus
whatever the weights $w_i$ of the $X_i$ may be, $\Soc_k(C)= kx_1^2=
kx_2^3$.  Thus, since $D\ox_Ak \cong k$, so is flat, by
\eqref{prArt}(A)(3)$\Rightarrow$(1), if $C\in \bb{AF}_A$, then, once
again, $C$ is $k$-Gorenstein.  Alternatively, $\Soc_k(C)= kx_2^3$ holds
because of \eqref{leSoc}, as $D/\go MD = kx_2^{-3}$.

First, weight the $X_i$ as usual: $w_1,\,w_2 :=1$.  From \eqref{sbArt},
recall $C\in \bb{AF}_A$ iff $G_\bu C\in \bb{AG}_A$.  Note $G_\bu C =
G_\bu A\big/G_\bu I$ by \eqref{prBiex}.  But $A$ is graded; so $G_\bu A
= A$.  Correspondingly, $G_\bu I$ is generated by the initial forms of
all the elements of $I$.  So $X_1X_2,\,X_1^2\in G_\bu I$.  But $X_2^4\in
I$; so $X_2^4\in G_\bu I$.

It's not immediately obvious that $X_1X_2,\,X_1^2,\,X_2^4$ generate
$G_\bu I$.  However, if so, then $G_\bu I$ is the $k$-span of the
monomials divisible by at least one of $X_1X_2,\,X_1^2,\,X_2^4$, and so
$G_\bu C\in \bb{AG}_A$.  Moreover, if so, then $x_1$ and $x_2^3$ lie in
$\Soc_k(G_\bu C)$, and they are independent over $k$; thus, $G_\bu C$
is, by \eqref{prArt}(A)(1)$\Rightarrow$(3), {\it not} $k$-Gorenstein.

Instead, consider $G_\bu D$.  Similarly, it's the $k$-span of the
initial forms of the elements of $D$.  So $G_\bu D = kX_2^{-3} +
kX_2^{-2} + (kX_1^{-1} + kX_2^{-1}) + k$.  Hence $D \in \bb{AF}_A$ and
$G_\bu(A^\dg/D)$ is free over $k$.  Thus \eqref{thGMD3}(2) yields
$G_\bu C\in \bb{AG}_A$.  So $C\in \bb{AF}_A$, and thus \eqref{prArt}(A)
implies, in two different ways, that $C$ is $k$-Gorenstein.

The expression for $G_\bu D$ above also shows that the residues in
$G_\bu D$ of $X_1^{-1}$ and $X_2^{-3}$ form a free basis of it.  Thus
\eqref{leSoc} yields $\Soc_k(G_\bu C) = kx_1\oplus kx_2^3$.

Also, $G_\bu I := (0 :_A G_\bu D)$.  But the monomials in $A$ and the
inverse (reciprocal) monomials in $A^\dg$ form dual bases.  So $G_\bu
I$ is the $k$-span of all the monomials of $A$ except
$1,\,X_1,\,X_2,\,X_2^2,\,X_2^3$.  Thus $X_1X_2,\,X_1^2,\,X_2^4$ do
generate $G_\bu I$, as expected.

In passing, note $\go M\cdot G_\bu D = kX_2^{-2} + kX_2^{-1} + k$. So
$G_\bu D\ox_Ak$ is free of rank 2 over $k$.  Thus, again, $G_\bu C$ is not
$k$-Gorenstein, this time by \eqref{prArt}(A)(1)$\Rightarrow$(2), So, by
\eqref{prGor}(4)$\Rightarrow$(1), although $D$ is, by construction,
generated by $F$, nonetheless $G_\bu D$ is not generated by the initial
form $X_2^{-3}$ of $F$, nor by any other single element.

Moreover, no isomorphism $C\cong D$ in $\bb{M}_A$ can preserve the
filtrations up to shift; else, $G_\bu C$ would be isomorphic in $\bb
G_A$ to a shift of $G_\bu D$.  In particular, the isomorphism $\vf\:
C\risom D$ given by $\vf(\tau):=\tau f$ doesn't do so.  In fact, 
        $$(F^2C)F = kX_2^{-1} + k \subsetneqq
           F^{-1}D = kX_1^{-1} + kX_2^{-1} + k.$$

Lastly, take $w_1 := 3$ and $w_2 := 2$.  Then $F$ is weighted homogeneous
of degree $-6$.  So the analysis simplifies, as there's no separate
associated graded case to consider in order to apply \eqref{prArt}.
Indeed, plainly, $(A^\dg)_{-6} = D_{-6} \oplus X_1^{-2}$ and
$(A^\dg)_{n} = D_{n}$ for $-4\le n \le 0$, with $D_{-5} = 0$ and
$D_{-1} = 0$.  So $D$ and $A^\dg/D$ are free over $k$, and $D\ox_Ak
= k$.  Thus \eqref{prArt}(A)(2)$\Rightarrow$(1) yields $C$ is
$k$-Gorenstein.

Alternatively, note that, for every maximal ideal $\go m$ of $k$ and for
every field $K$ containing $k/\go m$, the generators $X_1^{2}-X_2^{3},\
X_1X_2$ of $I$ map to a regular sequence in $K[X_1,\,X_2]$, because
$C_K$ is Artinian.  So $C_K$ is Gorenstein; see
\cite{Eis}*{Cor.\,21.19}.  Furthermore, $C$ is $k$-flat by
\cite{EGAIII0}*{(10.2.4), p.\,19} applied twice. So $C$ is $k$-Artinian.
Thus \eqref{prGor}(2)$\Rightarrow$(1) yields that $C$ is $k$-Gorenstein.
 \end{example}

\begin{proposition}[Linkage]\label{coLink}
   Assume $A$ is $k$-Artinian with Hilbert function $\mb h$, and is
$k$-Gorenstein.  Let $I\subset A$ be an ideal with $C := A/I\in
\bb{AM}_A$.  Form the ideal $I'\subset A$ {\bf directly linked} to $I$,
namely $I' := \Hom_A( C,\, A)$, and form the $A$-submodule of $A^*$ {\bf
Macaulay dual} to $C$, namely $C^*$.  Fix an isomorphism $A\cong
\Soc_k(A)\ox_k A^*$ in $\bb{M}_A$, as provided by \eqref{prArt}\(A).
Set $C':=A/I'$.

\(1) Then $I$ is directly linked to $I'$; that is, $I = \Hom_A( C',\,
A)$.

\(2) Then $I' = \Soc_k(A)\ox_k C^*$ and $C' = \Hom_A(I,\, A)$; moreover,
$C'\in \bb{AM}_A$.

\(3) Give $C$ the induced filtration, and assume $C\in \bb{AF}_A$.  Then
$C$ is $k$-Gorenstein iff there's an invertible $k$-module $L$ such that
$L\ox_k I'$ is a cyclic $A$-module.  If so, then necessarily $L \cong
\Soc_k(C)\ox_k\Soc_k(A)^*$ and $L\ox_k I' \cong C$.
 \end{proposition}

\begin{proof}
 {\bf For (1),} given any $k$-module $L$ and $A$-modules $M,\,N,\,P$,
form these two canonical maps:
 \begin{gather}
  \Hom_A( M,\,N) \To
  \Hom_A\bigl( L\ox_kM,\ L\ox_kN\bigr),\label{gab2}\\
 L\ox_k\Hom_A( P,\,N) \To \Hom_A( P,\ L\ox_kN).\label{gab3}
 \end{gather}
 Assume $L$ is invertible.  Then, plainly, both maps are isomorphisms.

Take $M := \Hom_A( P,\,N)$ in \eqref{gab2}.  Then \eqref{gab3} yields a
canonical isomorphism:
\begin{equation}\label{eqga4}
 \Hom_A\bigl( \Hom_A( P,\,N),\ N\bigr) =
  \Hom_A\bigl(\Hom_A( P,\,L\ox_kN),\ L\ox_kN\bigr).
 \end{equation}
 Plainly, \eqref{eqga4} respects the two canonical maps from $P$, one to
each side.

Given another $A$-module $Q$ and any $A$-isomorphism $\vf\:Q\risom
L\ox_kN$, note $\vf$ and $\vf^{-1}$ yield an isomorphism:
\begin{equation}\label{eqga5}
 \psi\: \Hom_A\bigl(\Hom_A( P,\,L\ox_kN),\ L\ox_kN\bigr)
  \risom \Hom_A\bigl( \Hom_A( P,\,Q),\ Q)\bigr) .
 \end{equation}
 Let $\nu_s$ and $\nu_t$ denote the canonical maps from $P$ to the
source and target of $\psi$.  Plainly, $\nu_t= \psi\nu_s$.  Hence, if
$\nu_s$ is an isomorphism, so is $\nu_t$. (Also, $\psi =
\nu_t\nu_s^{-1}$; so $\psi$ is independent of the choice of $\vf$.)

Assume $P\in \bb{AM}_A$; take $N:=A^*$.  Then $P=\Hom_A\bigl( \Hom_A(
P,\,N),\ N)\bigr)
$ by \eqref{prRMD}.  Take $L:=\Soc_k(A)$; it's
invertible by \eqref{prArt}(A)(1)\implies(3).  Then $\nu_s$ is an
isomorphism.

Take $Q:=A$ and $\vf\:Q\risom L\ox_kN$ to be the isomorphism fixed in
the hypotheses.  Then $\nu_t$ is an isomorphism; that is,
$P=\Hom_A\bigl( \Hom_A( P,\,A),\ A)\bigr)$.

Form the exact sequence $0\to I\to A\to C\to 0$.  Since $A,\,C\in
\bb{AM}_A$, plainly $I\in \bb{AM}_A$.  So taking $P:=I$ above yields $I
= \Hom_A\bigl( \Hom_A( I,\,A),\ A\bigr)$.

Note $0\to \Hom_A(C,\,A)\to \Hom_A(A,\,A)\xto\alpha \Hom_A(I,\,A)$ is
exact.  But below $\alpha$ is shown to be surjective.  Also
$\Hom_A(C,\,A) =: I'$ and $\Hom_A(A,\,A) =A$.  Thus $C' := A/I'
=\Hom_A(I,\,A)\in \bb{AM}_A$; hence, $\Hom_A(C',\, A) = I$.

The map $\Hom_A(A,\,A^*)\to \Hom_A(I,\,A^*)$ is surjective, as $A^*$ is
dualizing for $\bb{AM}_A$ by \eqref{prRMD}.  And it remains surjective
after applying $L\ox_k\bu$, as $L$ is invertible.  But \eqref{gab3} is
functorial in $P$.  Hence $\Hom_A(A,\,L\ox_kA^*)\to
\Hom_A(I,\,L\ox_k\,A^*)$ is surjective.  Hence the isomorphism
$\vf^{-1}\: L\ox_k\,A^*\risom A$ induces an isomorphism from that
surjection to $\alpha$.  Thus $\alpha$ is surjective, as asserted.  Thus
(1) holds.

{\bf For (2),} note that the quotient map $\kappa\: A\onto C$ induces
these two embeddings:
 \begin{gather*}\label{gs2}
 L\ox_k\Hom_A( C,\, A^*)\into L\ox_k\Hom_A( A,\, A^*) = L\ox_kA^*\\
 \Hom_A( C,\, L\ox_kA^*)\into \Hom_A( A,\, L\ox_kA^*) = L\ox_kA^*.
 \end{gather*}
 As \eqref{gab3} is an isomorphism that's compatible with $\kappa$, the
two images are the same submodule $G$ of $L\ox_kA^*$.

Set $F := \vf^{-1}G$.  Note $F$ is independent of the choice of $\vf$,
because any two choices differ by an automorphism of $A$, namely, by
multiplication by a unit of $A$ as $\Hom_A(A,A) = A$.  Now, $\Hom_A(
C,\, A^*) = C^*$ by \eqref{eqeqFm6}.  And $I' := \Hom_A( C,\, A)$.  Thus
\eqref{gab3} induces a canonical isomorphism $I' = \Soc_k(A)\ox_k C^*$,
which is independent of the choice of $\vf$,.  Thus (2) holds.

{\bf For (3),} assume $C$ is $k$-Gorenstein.  Then
\eqref{prArt}(A)\(1)\implies\(5) gives an isomorphism $C\cong
\Soc_k(C)\ox_kC^*$, and $\Soc_k(C)$ is invertible by
\eqref{prArt}(A)(1)\implies(3).  So (2) gives $C\cong L\ox_k I'$ with $L
:= \Soc_k(C)\ox_k\Soc_k(A)^*$.  Thus $L\ox_k I'$ is a cyclic $A$-module.

Conversely, assume there's an invertible $k$-module $L$ with $L\ox_k I'$
cyclic over $A$.  Then there's a surjection $A\onto L\ox_k I'$, and it
factors via a surjection $B\onto L\ox_k I'$, as $L\ox_k I'$ is a
$B$-module.  But $L\ox_k I' = M\ox_k B^* $ with $M:= L\ox_k\Soc_k(A)$
invertible by (2).  And $B^*$ is locally free over $k$ of the same rank
as $B$.  Hence $B\risom M\ox_kB^*$.  So $B$ is $k$-Gorenstein, and
$M\cong \Soc_k(B)$ by \eqref{prArt}(A).  Thus (3) holds.
 \end{proof}

\section{Basic Geometry}\label{seBasicGeo}

\begin{setup}\label{se6}
 Fix a nonempty Noetherian base scheme $S$, and a nonzero, finitely
generated, graded, quasi-coherent $\mc O_S$-algebra $\mc A$ with $\mc
A_p = 0$ for $p < 0$.  Given an $S$-scheme $T$, denote by $\mc A_T$ the
 \indn{AcalT@$\mc A_T$}
pullback of $\mc A$, and by $\bb M_{\mc A_T}$ the category of
quasi-coherent $\mc A_T$-modules.  For $\mc V\in \bb M_{\mc A_T}$, set
$\mc V^* := \sHom_T(\mc V, \,\mc O_T)$.

Likewise, generalize all the theory of Sections 2--4 in the standard way
using analogous notation, so that the old theory is recovered over each
open, affine subscheme $U :=\Spec(k)$ of $S$ with $A := H^0(U, \,\mc A)$
and over each open, affine subscheme $V :=\Spec(K)$ of $T$ mapping into
$U$.

Thus generalizing \eqref{sbFilt} and \eqref{sbArt} yields the categories
$\bb F_{\mc A_T}$ and $\bb G_{\mc A_T}$, their full subcategories
$\bb{AF}_{\mc A_T}$ and $\bb{AG}_{\mc A_T}$, the (associated-graded)
functor $G_\bu\: \bb F_{\mc A_T} \to \bb G_{\mc A_T}$, and its
restriction $G_\bu\: \bb{AF}_{\mc A_T} \to \bb{AG}_{\mc A_T}$.

Fix nonzero functions $\mb b,\,\mb h\:\Z\to \Z$ with $\mb b(p) \ge \mb
h(p)\ge 0$ for all $p$, and with $\mb h$ finite.  Set $s : = s(\mb h) :=
\sup\{\,p\mid \mb h(p)\neq 0\,\}$.

Fix a graded, locally finitely generated $\mc A$-module $\mc B$ with
graded pieces $\mc B_p$ locally free of rank $\mb b(p)$ for all $p$.
 \indn{bbmb@$\mb b$}
Then generalizing \eqref{sbApl} yields the graded dual $\mc B^\dg$ of
$\mc B$ and the following sets of subsheaves of $\mc B_T$ and $\mc
B^\dg_T$:
 \begin{equation}\label{eqse6}
 \mb F\Psi_{\mc B_T}^{\mb h}
  \index[notation]{FbmPsi@$\mb F\Psi_{\mc B_T}^{\mb h}$}
 \and \mb H\Psi_{\mc B_T}^{\mb h}
  \index[notation]{HmbPsi@$\mb H\Psi_{\mc B_T}^{\mb h}$} 
 \and \mb F\Delta_{\mc B^\dg_T}^{\mb h^*}
   \index[notation]{FmbDel@$\mb F\Delta_{\mc B^\dg_T}^{\mb h^*}$}
 \and \mb H\Delta_{\mc B^\dg_T}^{\mb h^*}.
   \index[notation]{HmbDel@$\mb H\Delta_{\mc B^\dg_T}^{\mb h^*}$}
 \end{equation}
 \indn{FbbPsit@$\bb F\Psi_{ B}^{\mb h,\mb t}$}
 \indn{HbbPsit@$\bb H\Psi_{ B}^{\mb h,\mb t}$}
 \indn{FbbDeltst@$\bb F\Delta_{ B^\dg}^{\mb h^*, \mb t}$}
 \indn{HbbDeltst@$\bb H\Delta_{ B^\dg}^{\mb h^*, \mb t}$}
 For convenience, view the elements of $\mb F\Psi_{\mc B_T}^{\mb h}$
and $\mb H\Psi_{\mc B_T}^{\mb h}$, {\it not\/} as certain subsheaves
$\mc I$ of $\mc B_T$, but as certain quotients $\mc C$, via the
canonical bijection $\mc I \mapsto \mc C := \mc B_T/\mc I$.

As $\mc B$ is locally finitely generated over $\mc A$ and as $S$ is
quasi-compact, a finite set of local generators belongs to some
$F^{b_0}\mc B$.  Then $F^p\mc B = \mc B$ for all $p\le b_0$.  So $F^p\mc
B^\dg = 0$ for all $p> -b_0$.  Fix $b_0$ as large as possible:
 \begin{equation}\label{eqeqse6
.2}
 b_0 := \sup\{\,p \mid F^p\mc B = \mc B \,\}
         = \inf\{\,p \mid b(p) \neq 0 \,\}.
 \end{equation}

From \eqref{prLevS} on, assume $\mc A/F^1\mc A = \mc O_S$, and
generalize the notion of level module in \eqref{deType} to sheaves.
Note that $\mc A_T/F^1\mc A_T = \mc O_T$ for any $T/S$ by
\eqref{leBaseChg}(3).  

Finally, given a finite function $\mb t\:\Z\to\Z$, define 
\begin{equation*}\label{eqse6
t}
 \mb F\Psi_{\mc B_T}^{\mb h, \mb t}
         \subset \mb F\Psi_{\mc B_T}^{\mb h}
 \and \mb H\Psi_{\mc B_T}^{\mb h, \mb t}
         \subset \mb H\Psi_{\mc B_T}^{\mb h}
 \end{equation*} as the subsets of sheaves of $T$-socle type $\mb t$.
In particular, define
 \begin{equation*}
 \mb F\Lambda_{\mc B_T}^{\mb h}
         \subset \mb F\Psi_{\mc B_T}^{\mb h}
 \and \mb H\Lambda_{\mc B_T}^{\mb h}
         \subset \mb H\Psi_{\mc B_T}^{\mb h}
 \end{equation*}
 as the subsets of level sheaves.  So $\mb F\Lambda_{\mc B_T}^{\mb h}
= \mb F\Psi_{\mc B_T}^{\mb h, \mb t}$ and $\mb H\Lambda_{\mc
B_T}^{\mb h} = \mb H\Psi_{\mc B_T}^{\mb h, \mb t}$ where $\mb t(s)
:= \mb h(s)$ and $\mb t(p) := 0$ for $p \neq s$.  By \eqref{prFib}, these
sets behave functorially in $T$.
 \end{setup}

The statements discussed in the rest of this paper about $\mc B$ and its
quotients and about $\mb h$ and $\mb t$ are plainly equivalent to
similar statements about their shifts by $b_0$ (to the left).
Therefore, to lighten the notation and to make some statements feel more
familiar to many readers, from now on assume $b_0 = 0$.

\begin{theorem}\label{thRep}
  Set $\ell := \sum_p\mb h(p)$ and $X := \Spec(\mc A)$.  Denote the $\mc
O_X$-module associated to $\mc B$ by $\wt{\mc B}$, and set $Q :=
\Quot_{\wt{\mc B}/X/S}^\ell$.  Then the sets in \eqref{eqse6} form
functors in $T$, which are representable by (possibly empty) subschemes
of $Q$, say $\bb F\Psi_{\mc B}^{\mb h}$ and $\bb H\Psi_{\mc B}^{\mb h}$,
 and $\bb F\Delta_{\mc B^\dg}^{\mb h^*}$ and $\bb H\Delta_{\mc
B^\dg}^{\mb h^*}$.  Moreover, Macaulay Duality gives canonical
isomorphisms $\bb F\Psi_{\mc B}^{\mb h} = \bb F\Delta_{\mc B^\dg}^{\mb
h^*}$ and\/ $\bb H\Psi_{\mc B}^{\mb h} = \bb H\Delta_{\mc B^\dg}^{\mb
h^*}$.
 \end{theorem}
 \indn{FbbLamb@$\bb F\Lambda_{ B}^{\mb h}$}
 \indn{FbbPsi@$\bb F\Psi_B^{\mb h}$}
 \indn{HbbPsi@$\bb H\Psi_B^{\mb h}$}

\begin{proof}
 To begin, fix an $S$-map $T\to R$.  Let $\mc C$ be an $\mc
A_R$-quotient of $\mc B_R$ with the induced filtration.  By right
exactness of pullback, for all $q$, these two sequences
 \begin{equation*}\label{eqRep3}
 (F^q(\mc B_R))_T \to (F^q\mc C)_T \to 0 \and 
 (F^q\mc C)_T \to \mc C_T \to (\mc C/F^{q}\mc C)_T\to 0
 \end{equation*}
 are exact.  But $(F^q(\mc B_R))_T
  = \bigl(\bigoplus_{p\ge q}(\mc B_p)_R\bigr)_T
  = \bigoplus_{p\ge q} (\mc B_p)_T = F^q(\mc B_T)$.  Therefore,
 \begin{equation}\label{eqRep4}
 \mc C_T / F^q(\mc C_T) = (\mc C/F^{q}\mc C)_T.
 \end{equation}
 But $0\to G_q(\mc C_T) \to \mc C_T / F^{q+1}(\mc C_T) \to \mc C_T / F^q(\mc
C_T) \to 0$ is exact.  Thus, so is
 \begin{equation}\label{eqRep5}
 0\to G_q(\mc C_T) \to (\mc C/F^{q+1}\mc C)_T \to (\mc
  C/F^{q}\mc C)_T \to 0.
 \end{equation}

If $\mc C/F^{q}\mc C$ is flat over $R$, then $0\to G_q\mc C \to \mc
C/F^{q+1}\mc C \to \mc C/F^q\mc C \to 0$ remains exact on pullback to
$T$, and so \eqref{eqRep5} yields
 \begin{equation}\label{eqRep6}
 G_q(\mc C_T) = (G_q\mc C)_T.
 \end{equation}

Suppose now $\mc C\in \mb F\Psi_{\mc B_R}^{\mb h}$.  Then $\mc C\in
\bb{AF}_{\mc A_R}^{\mb h}$.  So \eqref{prFilt}(1) implies $\mc C$ and
$\mc C/F^p\mc C$ are flat over $R$.  Moreover, $\mc C_T\in \bb{AF}_{\mc
A_T}^{\mb h}$ as $G_q(\mc C_T)$ is locally free of rank $\mb h(q)$ for
all $q$ by \eqref{eqRep5}.  Thus $\mc C_T \in \mb F\Psi_{\mc
B_T}^{\mb h}$.  Plainly, if $\mc C \in \mb H\Psi_{\mc B_R}^{\mb h}$,
then $\mc C_T \in \mb H\Psi_{\mc B_T}^{\mb h}$.  Thus as $T$ varies,
$ \mb F\Psi_{\mc B_T}^{\mb h}$ and $\mb H\Psi_{\mc B_T}^{\mb h}$
form functors.

Next, form $X\x_SQ = \Spec(\mc A_Q)$.  Let $\mc Q$ be the $\mc
A_Q$-quotient module of $\mc B_Q$ whose associated $\mc
O_{X\x_SQ}$-module is the universal quotient of $\wt{\mc B}_Q$.  Give
$\mc Q$ the induced filtration.  For each $q$, set $\ell_q := \sum_{p <
q}\mb h(p)$.  By the theory of flattening stratification
(\cite{SB221}*{Lem.\,3.6, p.\,15}), there's a subscheme $Q_q$ of $Q$
such that an $S$-map $T\to Q$ factors through $Q_q$ iff $(\mc Q/F^{q}\mc
Q)_T$ is locally free of rank $\ell_q$.  Set
 \begin{equation*}\label{eqRep10}\ts
  \bb F\Psi_{\mc B}^{\mb h} := \bigcap_{q} Q_q \subset Q.
 \end{equation*}
 We have to prove that its set of $T$-points is (canonically) equal to
$\mb F\Psi_{\mc B_T}^{\mb h}$.

Given $\mc Q'\in \mb F\Psi_{\mc B_T}^{\mb h}$, note $\mc
Q' \in \bb{AF}_{\mc A_T}^{\mb h}$.  So \eqref{prFilt}(1) implies $\mc
Q'$ and $\mc Q'/F^{q}\mc Q'$ are locally free of ranks $r$ and $\ell_q$
for all $q$.  So there's a unique map $T\to Q$ such that $\mc Q' = \mc
Q_T$.  Hence \eqref{eqRep4}, with $Q$ and $\mc Q$ for $R$ and $\mc C$,
implies $(\mc Q/F^{q}\mc Q)_T$ is locally free of rank $\ell_q$ for all
$q$.  Thus $T\to Q$ factors through $\bb F\Psi_{\mc B}^{\mb h}$.

Conversely, suppose $T\to Q$ factors through $\bb F\Psi_{\mc B}^{\mb
h}$.  Then $\mc Q_T \in \bb{AF}_{\mc A_T}^{\mb h}$ owing to
\eqref{eqRep4} with $Q$ and $\mc Q$ for $R$ and $\mc C$.  Thus $\mc C
\in \mb F\Psi_{\mc B_T}^{\mb h}$.  Thus $\mb F\Psi_{\mc B_T}^{\mb
h}$ is equal to the set of $T$-points of $\bb F\Psi_{\mc B}^{\mb h}$, as
desired.

Let's now construct $\bb H\Psi_{\mc B}^{\mb h}$.  For convenience, set
$Z := \bb F\Psi_{\mc B}^{\mb h}$.  Define $\mc S\subset \mc B$ by
$B_Q/\mc S := \mc Q$; give $\mc S$ the induced filtration.  Consider the
pullback $\mc Q_Z$ of $\mc Q$ and its associated graded quotient $G_\bu
(\mc Q_Z)$.  Note $\mc S_Z\subset \mc B_Z$ and $\mc B_Z/\mc S_Z = \mc
Q_Z$ as $\mc Q$ is flat.  So $G_\bu(\mc S_Z) \subset \mc B_Z$ and $\mc
B_Z/G_\bu(\mc S_Z) = G_\bu (\mc Q_Z)$ by \eqref{prBiex}.  Set $\mc R :=
{\mc B}_Z/(\mc S_Z + G_\bu(\mc S_Z))$.

The theory of flattening stratification yields a subscheme $\bb H\Psi_{
B}^{\mb h}$ of $Z$ such that, given a map $T\to Z$, it factors through
$\bb H\Psi_{\mc B}^{\mb h}$ iff $\mc R_T$ is locally free of rank
$\ell$.  We have to prove that the set of $T$-points of $\bb H\Psi_{\mc
B}^{\mb h}$ is (canonically) equal to $\mb H\Psi_{\mc B_T}^{\mb h}$.

Recall $\mc S\subset \mc B$ and $B_Q/\mc S: = \mc Q$. But $\mc Q$ is
flat; so $\mc S_T \subset \mc B_T$ and $\mc B_T / \mc S_T = \mc Q_T$.
So \eqref{prBiex} yields $G_\bu(\mc S_T) \subset \mc B_T$ and $\mc
B_T/G_\bu(\mc S_T) = G_\bu (\mc Q_T)$.  In particular, $G_\bu(\mc S_Z)
\subset \mc B_Z$ and $\mc B_Z/G_\bu(\mc S_Z) = G_\bu (\mc Q_Z)$.  Now,
$\mc Q_Z/F^{q}\mc Q_Z$ is flat for all $q$ by the definition of $Z$; so
$(G_\bu (\mc Q_Z))_T = G_\bu ((\mc Q_Z)_T)$ by \eqref{eqRep6}, and
$G_\bu (\mc Q_Z)$ is flat by \eqref{eqRep5} with $Z$ for $T$.
Hence $(G_\bu(\mc S_Z))_T \subset \mc B_T$ and $\mc S_T \subset \mc
B_T$.  Furthermore,
 $$\mc B_T\big/(G_\bu(\mc S_Z))_T = (G_\bu (\mc Q_Z))_T
     = G_\bu ((\mc Q_Z)_T) = G_\bu (\mc Q_T) = \mc B_T/G_\bu(\mc S_T).$$
 Thus $(G_\bu (\mc S_Z))_T = G_\bu(\mc S_T)\subset \mc B_T$.

Note the following right exact sequence:
 $\bigl(\mc S_Z + G_\bu(\mc S_Z)\bigr)_T \to \mc B_T \to \mc R_T \to 0$.
But $(\mc S_Z)_T = \mc S_T \subset \mc B_T$ and $(G_\bu (\mc S_Z))_T =
G_\bu(\mc S_T)\subset \mc B_T$.  Thus $\mc R_T = \mc B_T/(\mc S_T +
G_\bu(\mc S_T))$.

Given $\mc Q'\in \mb H\Psi_{\mc B_T}^{\mb h} \subset \mb
F\Psi_{\mc B_T}^{\mb h}$, there's a unique map $T\to Z$ with $\mc Q'
= (\mc Q_Z)_T$.  But $(\mc Q_Z)_T = \mc Q_T = \mc B_T/\mc S_T$ and
$G_\bu (\mc Q_T) = \mc B_T/G_\bu(\mc S_T)$.  Further, $G_\bu \mc Q' =
\mc Q'$ as $\mc Q'\in \mb H\Psi_{\mc B_T}^{\mb h}$.  Hence $\mc S_T =
G_\bu(\mc S_T) \subset \mc B_T$.  So $\mc R_T = \mc B_T/\mc S_T$.  So
$\mc R_T$ is locally free of rank $\ell$.  Thus $T\to Z$ factors through
$\bb H\Psi_{\mc B}^{\mb h}$, as desired.

Conversely, given a map $T\to \bb H\Psi_{\mc B}^{\mb h}$, note $\mc R_T$
is locally free of rank $\ell$ by definition of $\bb H\Psi_{ B}^{\mb
h}$.  Form the natural surjections $\mc B_T/\mc S_T \onto \mc R_T$ and
$\mc B_T/G_\bu(\mc S_T) \onto \mc R_T$.  Recall $\mc B_T/\mc S_T = Q_T$,
and as $T\to \bb H\Psi_{\mc B}^{\mb h}\subset \mb H\Psi_{\mc
B_T}^{\mb h}$, also $\mc B_T/G_\bu(\mc S_T) = G_\bu ((\mc Q_Z)_T)$.
Moreover, both $Q_T$ and $G_\bu ((\mc Q_Z)_T)$ are locally free of rank
$\ell$.  So both surjections are isomorphisms.  So $Q_T = G_\bu (Q_T)$.
Thus $\mc Q_T\in \mb H\Psi_{\mc B_T}^{\mb h}$, as desired.

Alternatively, $\bb H\Psi_{\mc B}^{\mb h}$ may be constructed as follows.
Set $\mc C_p := \Im\bigl( (\mc B_p)_Q \to \mc Q\bigr)$ for all $p$.  By
the theory of flattening stratification, there's a subscheme $Y_p$ of
$Q$ such that a map $T\to Q$ factors through $Y_p$ iff $(\mc Q/\mc
C_p)_T$ is locally free of rank $\ell-\mb h(p)$.  Set $\bb H\Psi_{\mc
B}^{\mb h} := \bigcap_{p} Y_p \subset Q$.  We have to prove its set of
$T$-points is equal to $\mb H\Psi_{\mc B_T}^{\mb h}$.

Suppose $T\to Q$ factors through $\bb H\Psi_{\mc B}^{\mb h}$.  Then
$(\mc Q/\mc C_p)_T$ is locally free of rank $\ell-\mb h(p)$.  Set $\mc
C_p' := \Im\bigl( (\mc B_p)_T \to \mc Q_T\bigr)$.  By right exactness of
pullback, the sequences
 \begin{equation*}\label{eqRep12}
 (\mc B_p)_T \to (\mc C_p)_T \to 0 \and
 (\mc C_p)_T \to \mc Q_T \to (\mc Q/\mc C_p)_T \to 0
 \end{equation*}
 are exact.  So $\mc Q_T/\mc C_p' = (\mc Q/\mc C_p)_T$.  So $\mc Q_T/\mc
C_p'$ is locally free of rank $\ell-\mb h(p)$.  Thus $\mc C_p'$ is locally
free of rank $\mb h(p)$.

So $\bigoplus_p \mc C_p'$ is locally free of rank $\ell$.  However, the
canonical map $\bigoplus_p \mc C_p'\to \mc Q_T$ is surjective; so it's
an isomorphism.  Thus $\mc Q_T$ is a homogeneous quotient of $\mc B_T$,
and $\mc Q_T \in \bb{AG}_{\mc A_T}^{\mb h} \subset \bb{AF}_{\mc
A_T}^{\mb h}$.  Thus $\mc C \in \mb H\Psi_{\mc B_T}^{\mb h}$.

Conversely, given $\mc Q' \in \mb H\Psi_{\mc B_T}^{\mb h}$, note $\mc
Q_p' := \Im\bigl( (\mc B_p)_T \to \mc Q'\bigr)$.  As $\mc Q' \in \mb
F\Psi_{\mc B_T}^{\mb h}$, there's a unique map $T\to Q$ such that
$\mc Q' = \mc Q_T$.  As above, $\mc Q_T/\mc Q_p' = (\mc Q/\mc C_p)_T$.
But $\mc Q_p'$ is locally free of rank $\mb h(p)$ as $\mc Q' \in \mb
H\Psi_{\mc B_T}^{\mb h}$.  Hence $(\mc Q/\mc C)_T$ is locally free of
rank $\ell-\mb h(p)$.  Thus $T\to Q$ factors through $\bb H\Psi_{\mc
B}^{\mb h}$.  Thus $\mb H\Psi_{\mc B_T}^{\mb h}$ is equal to the set
of $T$-points of $\bb H\Psi_{\mc B}^{\mb h}$, as desired.

Finally, owing to \eqref{thGMD3}, $\mc C \mapsto \mc C^*$ defines a
bijection $\mb F\Psi_{\mc B_T}^{\mb h} \risom \mb F\Delta_{\mc
(B_T)^\dg}^{\mb h^*}$, which restricts to a bijection $\mb H\Psi_{\mc
B_T}^{\mb h} \risom \mb H\Delta_{\mc (B_T)^\dg}^{\mb h^*}$.  These
bijections are functorial in $T$ owing to \eqref{alBaseChga}.  Take $\bb
F\Delta_{\mc B^\dg}^{\mb h^*} := \bb F\Psi_{\mc B}^{\mb h}$ and $\bb
H\Delta_{\mc B^\dg}^{\mb h^*} := \bb H\Psi_{\mc B}^{\mb h}$.  Thus $\mb
F\Delta_{\mc B^\dg_T}^{\mb h^*}$ and $\mb H\Delta_{\mc
B^\dg_T}^{\mb h^*}$ form functors in $T$, which are representable by
subschemes of $Q$.
 \end{proof}

\begin{corollary}\label{coIncRetr}
   There's a canonical pair of maps consisting of a closed embedding $\bb
H\Psi_{\mc B}^{\mb h} \into \bb F\Psi_{\mc B}^{\mb h}$ and a retraction
$\bb F\Psi_{\mc B}^{\mb h}\onto \bb H\Psi_{\mc B}^{\mb h}$.  Moreover,
there's an analogous pair of maps\/ $\bb H\Delta_{\mc B}^{\mb h}\into
\bb F\Delta_{\mc B}^{\mb h}$\varstrut and $\bb F\Delta_{\mc B}^{\mb h}
\onto\bb H\Delta_{\mc B}^{\mb h}$.  Both pairs respect Macaulay duality.
 \end{corollary}

\begin{proof}
 For each $S$-scheme $T$, by definition, $\mb H\Psi_{\mc B_T}^{\mb
h}$ is a subset of $\mb F\Psi_{\mc B_T}^{\mb h}$; plainly, the
inclusion is functorial in $T$.  Furthermore, taking associated graded
modules gives a retraction $\mb F\Psi_{\mc B_T}^{\mb h} \to \mb
H\Psi_{\mc B_T}^{\mb h}$; owing to \eqref{eqRep6}, it's functorial in
$T$.

As the functors are representable by quasi-projective $S$-schemes by
\eqref{thRep}, the maps between the functors are representable by
separated maps of schemes, say $\iota\: \bb H\Psi_{\mc B}^{\mb h} \to
\bb F\Psi_{\mc B}^{\mb h}$ and $\gamma\: \bb F\Psi_{\mc B}^{\mb h} \to
\bb H\Psi_{\mc B}^{\mb h}$.  

Moreover, $\gamma\iota = 1$ as $\gamma\iota$ represents the identity map
of the functor $T\mapsto \mb H\Psi_{\mc B_T}^{\mb h}$.  Thus $\iota$
is a closed embedding by \cite{EGAI}*{(I,\,5.2.9)} as $\gamma\iota$ is
one and $\gamma$ is separated.

Similarly, there's an analogous pair $\bb H\Delta_{\mc B}^{\mb h}\into
\bb F\Delta_{\mc B}^{\mb h}$ and $\bb F\Delta_{\mc B}^{\mb h} \onto\bb
H\Delta_{\mc B}^{\mb h}$.

Finally, owing to \eqref{thGMD3}, the two pairs respect Macaulay
duality.
 \end{proof}

\begin{proposition}\label{prHflg}
 Assume $\mc B$ is an $S$-Artinian and $S$-Gorenstin quotient of $\mc A$
of socle degree $b$.  For all $p$, set $\mb h'(p) := \mb b(p) - \mb
h(b-p)$.  Define a (Gorenstein direct) {\bf linkage} map $\lambda^{\mb
h} \: \bb F\Psi_{\mc B}^{\mb h} \to \bb F\Psi_{\mc B}^{\mb h'}$ on
$T$-points $\mc C \in \mb F\Psi_{\mc B_T}^{\mb h}$ by $\lambda^{\mb
h}(\mc C) := \mc B/\mc I' \in \mb F\Psi_{\mc B_T}^{\mb h'}$ where
$\mc I' := \sHom_{\mc B_T}( \mc C,\, \mc B_T)$.

\(1) Then $\lambda^{\mb h}$ is an isomorphism, $(\lambda^{\mb h})^{-1} =
\lambda^{\mb h'}$, and $\lambda^{\mb h} (\bb H\Psi_{\mc B}^{\mb h}) =
\bb H\Psi_{\mc B}^{\mb h'}$.

\(2) Fix $R:=\Spec(K)/S$ where $K$ is a field, and fix $\mc C \in \mb
H\Psi_{\mc B_R}^{\mb h}$.  Assume the ideal in $\mc A_R$ of\/ $\mc
B_R$ is generated by a homogeneous regular sequence.  Then $\bb
H\Psi_{\mc A}^{\mb h}$ is $S$-smooth at $\mc C$ iff\/ $\bb H\Psi_{\mc
A}^{\mb h'}$ is $S$-smooth at $\lambda^{\mb h}(\mc C)$.
 \end{proposition}

\begin{proof}
 {\bf For (1),} note that $\lambda^{\mb h}$ is a well-defined map of
functors, as follows.  First, $\lambda^{\mb h}(\mc C)\in \bb{AM}_{\mc
B_T}$ owing to \eqref{coLink}(2).  But plainly, $\mb h'$ is the Hilbert
function of $\lambda^{\mb h}(\mc C)$.  Thus $\lambda^{\mb h}(\mc C)\in
\mb F\Psi_{\mc B_T}^{\mb h'}$.  Now, both $\mc C$ and $\mc B_T$ are
locally free of finite rank; so, given $T'/T$, the base-change map
$\sHom_{\mc B_T}( \mc C,\, \mc B_T)_{T'} \To \sHom_{\mc B_{T'}}( \mc
C_{T'},\, \mc B_{T'})$ is an isomorphism.  Thus $\lambda^{\mb h}(\mc
C)_{T'} = \lambda^{\mb h}(\mc C_{T'})$.  Thus $\lambda^{\mb h}$ is
well-defined.

By virtue of \eqref{thRep}, the functors $T\mapsto \mb F\Psi_{\mc
B_T}^{\mb h}$ and $T\mapsto \mb H\Psi_{\mc B_T}^{\mb h}$ are
representable by the $S$-schemes $\bb F\Psi_{\mc B}^{\mb h}$ and $\bb
H\Psi_{\mc B}^{\mb h}$.  Thus the map of functors $\lambda^{\mb h}$ is
representable by a map of $S$-schemes $\lambda^{\mb h}$.

Owing to \eqref{coLink}(1), the ideal of $\mc C$ is equal to $\sHom_{\mc
B_T}( \lambda^{\mb h}(\mc C),\, \mc B_T)$.  Thus $\lambda^{\mb
h'}\circ\lambda^{\mb h}$ is the identity.  Similarly, so is
$\lambda^{\mb h'}\circ\lambda^{\mb h}$.  Thus $\lambda^{\mb h}$ is an
isomorphism, and $(\lambda^{\mb h})^{-1} = \lambda^{\mb h'}$.

Plainly, $\lambda^{\mb h} (\bb H\Psi_{\mc B}^{\mb h}) \subset \bb
H\Psi_{\mc B}^{\mb h'}$.  Similarly, $\lambda^{\mb h'} (\bb H\Psi_{\mc
B}^{\mb h'}) \subset \bb H\Psi_{\mc B}^{\mb h}$.  But $\lambda^{\mb h}$
and $\lambda^{\mb h'}$ are inverse isomorphisms.  Thus $\lambda^{\mb h}
(\bb H\Psi_{\mc B}^{\mb h}) = \bb H\Psi_{\mc B}^{\mb h'}$. Thus (1)
holds.

 {\bf For (2),} form the Hilbert-flag scheme $\bb H\Psi_{\mc A}^{\mb b,
\,\mb h}$.  Its $T$-points are the pairs $(\mc B^\di,\, \mc C^\di)$
where $\mc B^\di \in \mb H\Psi_{\mc A_T}^{\mb b}$ and $\mc C^\di \in
\mb H\Psi_{\mc B^\di_T}^{\mb h}$.  Thus $\bb H\Psi_{\mc A}^{\mb b,\,
\mb h} = \bb H\Psi_{\mc B^\sp}^{\mb h}$ where $\mc B^\sp$ is the
universal quotient on $\bb H\Psi_{\mc A}^{\mb b}$; moreover, $\bb
H\Psi_{\mc A}^{\mb b}$ and $\bb H\Psi_{\mc B^\sp}^{\mb h}$ exist by
virtue of \eqref{thRep}.  Form the two forgetful maps $\vf^{\mb b} \: \bb
H\Psi_{\mc A}^{\mb b, \,\mb h} \to \bb H\Psi_{\mc A}^{\mb b}$ and
$\vf^{\mb h} \: \bb H\Psi_{\mc A}^{\mb b, \,\mb h} \to \bb H\Psi_{\mc
A}^{\mb h}$.  They represent the maps of functors given by $\vf^{\mb
b}(\mc B^\di,\, \mc C^\di) : = \mc B^\di$ and $\vf^{\mb h}(\mc B^\di,\,
\mc C^\di) := \mc C^\di$.

Let's prove that $\vf^{\mb h}$ is smooth at $(\mc B_R,\,\mc C)$.  To do
so, let's use the Infinitesimal Criterion \cite{EGAIV}*{(17.14.1),
p.\,98}: Given a local $S$-scheme $Y$, a closed subscheme $X$ of $Y$, an
$S$-map $\psi\: Y \to \bb H\Psi_{\mc A}^{\mb h}$ that carries the closed
point of $Y$ to $\mc C$, and an $S$-map $\xi\: X\to \bb H\Psi_{\mc
A}^{\mb b, \,\mb h}$ that carries the closed point of $X$ to $(\mc
B_R,\,\mc C)$ and that lifts $\psi|X$, we must lift $\psi$ to an $S$-map
$\chi\: Y \to \bb H\Psi_{\mc A}^{\mb b, \,\mb h}$ that restricts to
$\xi$.

Let $\mc I$ be the ideal of $\mc B^\sp$ in the pullback of $\mc A$ to
$\bb H\Psi_{\mc A}^{\mb b}$.  Then $\mc I_R$ is the ideal of $\mc B_R$
in $\mc A_R$.  So by hypothesis, $\mc I_R$ is generated by a homogeneous
regular sequence, say $f_1,\dotsc,f_r$.  But $\mc B_R$ is Artinian.  So
the $f_i$ may be replaced by any $r$ homogeneous generators.  But say
$K_0$ is the residue field of the scheme point of $\bb H\Psi_{\mc
A}^{\mb b}$ that represents $\mc B_R$, and set $R_0 := \Spec(K_0)$.
Then $\mc I_R = \mc I_{R_0}\ox_{K_0} K$.  So take the $f_i$ to form a
minimal set of homogeneous generators of $ \mc I_{R_0}$.  Plainly the
$f_i$ generate $\mc I_R$.  If they no longer form a minimal set, then
some $f_i$ is a $K$-linear combination of monomials in the other $f_i$.
But $K$ and $K_0$ are fields.  So $f_i$ is a $K_0$-linear combination of
the same monomials, contrary to the minimality of the $f_i$.  Thus the
$f_i$ form a regular sequence of homogeneous generators of $\mc
I_{R_0}$.

Let $K_1$ be the residue field of $X$, and set $R_1 := \Spec(K_1)$.  The
composition $\vf^{\mb b} \circ \xi\: X\to \bb H\Psi_{\mc A}^{\mb b}$
induces an injection $K_0 \into K_1$.  Via it, the $f_i$ form a regular
sequence of homogeneous generators of $\mc I_{R_1}$, which is the ideal
of $\mc B^\sp_{R_1}$ in $\mc A_{R_1}$.  Since $X$ is a local scheme,
each $f_i$ is the residue of some homogeneous $F_i \in \Gamma(\mc I_X)$.
By Nakayama's lemma, the $F_i$ generate $\mc I_X$.  By the local
criterion of flatness, the $F_i$ form a regular sequence.

Let $\mc J$ be the ideal of the universal quotient on $\bb H\Psi_{\mc
A}^{\mb h}$ of the pullback of $\mc A$.  The map $\xi\: X\to \bb
H\Psi_{\mc A}^{\mb b, \,\mb h}$ yields an inclusion $\mc I_X \subset \mc
J_X$.  But $X$ is a closed subscheme of $Y$.  So $\mc J_X$ is a quotient
of $\mc J_Y$, where $\mc J_Y$ is the pullback of $\mc J$ via $\psi\: Y
\to \bb H\Psi_{\mc A}^{\mb h}$.  Since $\mc J_Y$ is homogeneous, each
$F_i$ is the residue of some homogeneous $\wt F_i \in \Gamma(\mc J_Y)$.
Let $\wt {\mc I}$ be the ideal the $\wt F_i$ generate in $\mc J_Y$, and
set $\wt{\mc B} := \mc A_Y/ \wt {\mc I}$.  Then $\wt{\mc B}_X = \mc
B^\sp_X$.  By the local criterion of flatness, $\wt{\mc B}$ is $Y$-flat.
Hence $\wt{\mc B} \in \mb H\Psi_{\mc A_Y}^{\mb b}$.  Set $\wt{\mc C}
= \mc A_Y / \mc J_Y$.  Then $\wt{\mc C} \in \mb H\Psi_{\wt{\mc
B}}^{\mb h}$.  Thus the pair $(\wt{\mc B},\,\wt{\mc C})$ defines a map
$Y\to \bb H\Psi_{\mc A}^{\mb b, \,\mb h}$ that restricts to $\xi$, as
desired.  Thus $\vf^{\mb h} \: \bb H\Psi_{\mc A}^{\mb b, \,\mb h} \to
\bb H\Psi_{\mc A}^{\mb h}$ is smooth at $(\mc B_R,\,\mc C)$.

Finally, suppose $\bb H\Psi_{\mc A}^{\mb h}$ is $S$-smooth at $\mc C$.
Now, $\vf^{\mb h} \: \bb H\Psi_{\mc A}^{\mb b, \,\mb h} \to \bb
H\Psi_{\mc A}^{\mb h}$ is smooth at $(\mc B_R,\,\mc C)$.  Hence $\bb
H\Psi_{\mc A}^{\mb b, \,\mb h}$ is $S$-smooth at $(\mc B_R,\,\mc C)$.
But $\bb H\Psi_{\mc A}^{\mb b, \,\mb h} = \bb H\Psi_{\mc B^\sp}^{\mb
h}$, as noted early in this proof; similarly, $\bb H\Psi_{\mc A}^{\mb b,
\,\mb h'} = \bb H\Psi_{\mc B^\sp}^{\mb h'}$.  So let $U$ be the largest
open subscheme of $\bb H\Psi_{\mc A}^{\mb b}$ on which $\mc B^\sp$ is
$U$-Gorenstein.  Then $\lambda^{\mb h} \: \bb H\Psi_{\mc B^\sp_U}^{\mb
h} \risom \bb H\Psi_{\mc B^\sp_U}^{\mb h'}$ owing to (1) with $U$ for
$S$.  By hypothesis, $\mc B_R$ is an $R$-point of $U$.  Thus $\bb
H\Psi_{\mc B^\sp}^{\mb h'}$ is smooth at $\lambda^{\mb h}(\mc B_R,\,\mc
C)$, which is $(\mc B_R,\,\lambda^{\mb h}(\mc C))$.  But $\vf^{\mb h'}
\: \bb H\Psi_{\mc A}^{\mb b, \,\mb h'} \to \bb H\Psi_{\mc A}^{\mb h'}$
is smooth at $(\mc B_R,\,,\lambda^{\mb h'}(\mc C))$, just as $\vf^{\mb
h} \: \bb H\Psi_{\mc A}^{\mb b, \,\mb h} \to \bb H\Psi_{\mc A}^{\mb h}$
is smooth at $(\mc B_R,\,\mc C)$.  Thus $\bb H\Psi_{\mc A}^{\mb h'}$ is
$S$-smooth at $\lambda^{\mb h}(\mc C)$, as desired.  Similarly, if\/
$\bb H\Psi_{\mc A}^{\mb h'}$ is $S$-smooth at $\lambda^{\mb h}(\mc C)$,
then $\bb H\Psi_{\mc A}^{\mb h}$ is $S$-smooth at $\lambda^{\mb
h'}(\lambda^{\mb h}(\mc C))$.  But $\lambda^{\mb h'} = (\lambda^{\mb
h})^{-1}$ by (1).  Thus (2) holds.
 \end{proof}

\begin{remark}\label{reHflg}
 Assume $S:=\Spec(k)$ where $k$ is any algebraically closed field, and
$K$= k.  Then \eqref{prHflg}, statement and proof, recover and extend
the result of \cite{Trento}*{(2.8), (2.10), p.\,148}: it asserts, in the
setup of \eqref{prHflg}, that the map $\lambda^{\mb h} \: \bb H\Psi_{\mc
B^\sp_U}^{\mb h} \to \bb H\Psi_{\mc B^\sp_U}^{\mb h'}$ is an isomorphism
locally at $(\mc B,\,\mc C)$ if $\mc A$ is Gorenstein.  Of
course,\eqref{prHflg} proves $\lambda^{\mb h}$ is an isomorphism
globally on $U$ without assuming $\mc A_U$ is Gorenstein; moreover,
\eqref{prHflg}(1) asserts a more general result where $\mc B$ needn't be
the universal sheaf $\mc B^\sp$.  (However, \cite{Trento} treats
non-Artinian $\mc B$ and $\mc C$ as well.)

Also, \eqref{prHflg}(2) recovers the last assertion in
\cite{PGor}*{Prop.\,1.7, p.\,611}, and extends it to an $\mc A$ that
isn't a polynomial ring.  In both proofs, the key is the smoothness of
the forgetful map $\vf^{\mb h} \: \bb H\Psi_{\mc A}^{\mb b, \,\mb h} \to
\bb H\Psi_{\mc A}^{\mb h}$ at $(\mc B,\,\mc C)$, but in \eqref{prHflg}
the proof of the key is more direct.  Note that the proof of Prop.\,1.7
doesn't use the blanket hypothesis that $\car(k) \neq 2$ made on
p.\,607; thus Prop.\,1.7 holds in arbitrary characteristic.

Set $B := \Gamma(\mc B)$.  Then Prop.\,1.7 of \cite{PGor} also asserts
the following useful formula:
 \begin{align}\label{eqInfDim}
 \dim_{\mc C}\bb H\Psi_{\mc A}^{\mb h}
 &+ \dim_k\Hom_{B}\bigl( J/J^2,\,I\,\bigr)_0 \notag\\
 &= \dim_{\lambda^{\mb h}(\mc C)}\bb H\Psi_{\mc A}^{\mb h'}
  + \dim_k\Hom_{B}\bigl(J/J^2,\,I'\bigr)_0.
 \end{align}
 where $J$, $I$ and $I'$ are the ideals of $B$ in $\Gamma(\mc A)$, of
$\Gamma(\mc C)$ in $B$ and of $\Gamma({\lambda^{\mb h}(\mc C)})$ in $B$.
Since $J$ is generated by a regular sequence, $J/J^2$ is a free
$B$-module, and so the two Hom's can be easy to compute.

An infinitesimal analysis, made in \cite{PGor}, shows the Hom's in
\eqref{eqInfDim} are the tangent spaces to the fibers of $\vf^h$ and
$\vf^{h'}$ at $(\mc B,\,\mc C)$ and at $(\mc B,\,\lambda^{\mb h}(\mc
C))$.  But, at those points, $\vf^{\mb h}$ and $\vf^{\mb h'}$ are
smooth.  Thus the local dimensions of their fibers are the two
$\dim_k$'s.  In general, the dimension of a fiber over a point can be
computed by base change to the algebraic closure of the point's residue
field.  Thus those two $\dim_k$'s are the local dimensions of the fibers
of $\vf^{\mb h}$ and $\vf^{\mb h'}$ in the generality of
\eqref{prHflg}(2).

A good example of the use of \eqref{prHflg}(2) and \eqref{eqInfDim},
with the aid of Macaulay~2, was made in \cite{LevAlg}*{Ex.\,49, p.\,697}
to give an affirmative answer to Iarrobino's Ques.\,1.3ii on p.\,276 in
\cite{Pen2005}: Is there an $\mb h$ such that $\bb H\Psi_{\mc A}^{\mb
h}$, where $A$ is a polynomial ring in three variables, has two
different components whose general quotients are level of type 2?
(Other such $\mb h$ were found independently by Boij and Iarrobino, but
left unpublished according to the abstract of \cite{RedLev}).  In
\cite{LevAlg}, one component is generically smooth of dimension 47, as a
certain quotient can be connected via a sequence of direct linkages to a
quotient whose ideal is generated by a certain regular sequence.  The
second component is generically smooth of dimension 46, as was shown by
direct computation and also by conceptual means.
 \end{remark}

\begin{lemma}\label{leLev}
 Fix an $S$-map $T\to R$, and a coherent $\mc O_R$-module $\mc K$.

\(1) Then $\mc K_T$ vanishes iff $T\to R$ factors through $U := R -
\Supp(\mc K)$.

\(2) Given a surjection $\lambda\: \mc L\onto \mc M$ of coherent $\mc
O_R$-modules, where $\mc M$ is flat, take $\mc K := \Ker\lambda$.  Then
$\lambda_T\: \mc L_T\risom \mc M_T$ iff $T\to R$ factors through $U := R
- \Supp(\mc K)$.
 \end{lemma}

\begin{proof}
 {\bf For (1),} given a (scheme) point $t\in T$, denote by $r\in R$ its image,
and by $k(r)/k(t)$ the extension of residue fields.  By Nakayama's
Lemma, the stalk $\mc K_{T,t}$ vanishes iff the fiber $\mc K_{T,t}\ox
k(t)$ vanishes; similarly, $\mc K_r$ vanishes iff $\mc K_r\ox k(r)$
vanishes.  But $\mc K_r\ox k(r)$ vanishes iff $r\in U$.  Thus (1) holds.

{\bf For (2),} note that $0\to \mc K\to \mc L \xto\lambda \mc M\to 0$
remains exact on pullback to $T$ as $\mc M$ is flat.  So $\lambda_T\:
\mc L_T\risom \mc M_T$ iff $\mc K_T$ vanishes.  Thus (1) implies (2).
 \end{proof}

\begin{proposition}\label{prLevS}
 As the function $\mb t\: \Z\to \Z$ varies, the functors $T\mapsto \mb
F\Psi_{\mc B_T}^{\mb h, \mb t}$ are representable by (possibly empty)
subschemes, say $\bb F\Psi_{\mc B}^{\mb h, \mb t}$, of\/ $\bb F\Psi_{\mc
B}^{\mb h}$ stratifying it, and the $T\mapsto \mb H\Psi_{\mc
B_T}^{\mb h, \mb t}$ are representable by the intersections $\bb
H\Psi_{\mc B}^{\mb h, \mb t} := \bb F\Psi_{\mc B}^{\mb h, \mb t} \cap
\bb H\Psi_{\mc B}^{\mb h}$, which stratify $\bb H\Psi_{\mc B}^{\mb h}$.
  Further, $T\mapsto \mb F\Lambda_{\mc B_T}^{\mb h}$ and $T\mapsto
\mb H\Lambda_{\mc B_T}^{\mb h}$ are representable, say by $\bb
F\Lambda_{\mc B}^{\mb h}$ and $\bb H\Lambda_{\mc B}^{\mb h}$; in fact,
$\bb F\Lambda_{\mc B}^{\mb h}$ is an open subscheme of\/ $\bb F\Psi_{\mc
B}^{\mb h}$, and $\bb H\Lambda_{\mc B}^{\mb h} = \bb F\Lambda_{\mc
B}^{\mb h}\cap \bb H\Psi_{\mc B}^{\mb h}$.
 \end{proposition}
 \indn{FbbLam@$\bb F\Lambda_{ B}^{\mb h}$}
 \indn{HbbLam@$\bb H\Lambda_{ B}^{\mb h}$}

\begin{proof} 
 Set $R:=\bb F\Psi_{\mc B}^{\mb h}$ for short.  Let $\mc C \in \mb
F\Psi_{\mc B_R}^{\mb h}$ denote the universal quotient.  Set $\mc R
:= \mc C^*\ox_{\mc A_R}\mc O_R$.  It is easy to adapt the construction
of $\bb F\Psi_{\mc B}^{\mb h}$ in the proof of \eqref{thRep} with $R$
and $\mc R$ in place of $Q$ and $\mc Q$ to construct $\bb F\Psi_{\mc
B}^{\mb h, \mb t}$ for each $\mb t$.  Plainly, every scheme-point of $R$
lies in some $\bb F\Psi_{\mc B}^{\mb h, \mb t}$, because, for every
$T/S$ where $T$ is the Spec of a field, $\mc C_T$ has some $T$-socle type
$\mb t$; thus, the $\bb F\Psi_{\mc B}^{\mb h, \mb t}$ stratify $R$.

Recall from \eqref{coIncRetr} that $\bb H\Psi_{\mc B}^{\mb h}$ is a
(closed) subscheme of $\bb F\Psi_{\mc B}^{\mb h}$.  So the $\bb
H\Psi_{\mc B}^{\mb h, \mb t}$ are well defined.  Plainly, they represent
the $T\mapsto \mb H\Psi_{\mc B_T}^{\mb h, \mb t}$ and stratify $\bb
H\Psi_{\mc B}^{\mb h}$.

 Further, therefore, $T\mapsto \mb F\Lambda_{\mc B_T}^{\mb h}$ and
$T\mapsto \mb H\Lambda_{\mc B_T}^{\mb h}$ are representable, say by
$\bb F\Lambda_{\mc B}^{\mb h}$ and $\bb H\Lambda_{\mc B}^{\mb h}$, and
also $\bb H\Lambda_{\mc B}^{\mb h} = \bb F\Lambda_{\mc B}^{\mb h}\cap
\bb H\Psi_{\mc B}^{\mb h}$, as $\mb F\Lambda_{\mc B_T}^{\mb h} = \mb
F\Psi_{\mc B_T}^{\mb h, \mb t}$ and $\mb H\Lambda_{\mc B_T}^{\mb
h} = \mb H\Psi_{\mc B_T}^{\mb h, \mb t}$ where $\mb t(s) := \mb
h(s)$, and $\mb t(p) := 0$ for $p \neq s$; see the end of \eqref{se6}.

Here's an alternative construction of $\bb F\Lambda_{\mc B}^{\mb h}$,
which, moreover, shows that it's an open subscheme of $R$.  Given an
$S$-map $T\to R$, notice $\mc C_T \in \bb{AF}_{\mc A_T}^{\mb h}$.  Also,
there's a canonical surjection $\lambda(T)\: \mc R_T \onto G_{-s}(\mc
C_T^*)$ owing to \eqref{leSocLev}(1).  And $\lambda(T)$ is an
isomorphism iff $\mc C_T$ is, by definition \eqref{deType}, level.

Note $\mc C^* \in \bb{AF}_{\mc A_R}^{\mb h^*}$ by \eqref{prFilt}(4); so
$G_{-s}(\mc C^*)$ is flat.  Also, $\lambda(T) = \lambda(R)_T$ by
\eqref{leBaseChg}(4),\,(2).  Set $\bb F\Lambda_{\mc B}^{\mb h} :=
R-\Supp(\Ker\lambda(R))$.  Then $\mb F\Lambda_{\mc B_T}^{\mb h}$ is,
by \eqref{leLev}(2), the set of $T$-points of $\bb F\Lambda_{\mc B}^{\mb
h}$, which is an open subscheme of\/ $R := \bb F\Psi_{\mc B}^{\mb h}$,
as desired.
 \end{proof}

\begin{sbs}[Another construction]\label{sbLev}
 Set $t := \mb h(s)$.  Note $0 < t \le \mb b(s)$ by \eqref{se6}.  Set
$\bb G := \Grass_t(\mc B_s)$, the Grassmann\-ian of rank-$t$ locally
free quotients of $\mc B_s$.  Let $\mc U$ be the universal quotient of
$(\mc B_s)_\bb G$ on $\bb G$; so $\mc U^*\subset (\mc B_\bb
G^\dg)_{-s}$.  Form the natural map:
 \begin{equation}\label{eqLev1}
 \mu_p\: (\mc A_\bb G)_{p+s}\ox \mc U^* \to (\mc B_\bb G^\dg)_p 
 \tqn{for} p\ge -s. 
 \end{equation}

Denote by $L_p\subset \bb G$ the subscheme where $\rank(\mu_p) = \mb
h^*(p)$; that is, an $S$-map $T\to \bb G$ factors through $L_p$ iff
$\Cok((\mu_p)_T)$ is locally free of rank $\mb b^*(p) -\mb h^*(p)$.
But, by right exactness of pullback, $\Cok((\mu_p)_T) =
(\Cok(\mu_p))_T$.  So by the theory of flattening stratification, $L_p$
exists, although it might be empty.  Set
 \begin{equation*}\label{eqLdv2}\ts
 \bb{HL}_{\mc B}^{\mb h}
         := \bigcap_{p= -s}^{-i(\mb h)} L_p \subset \bb G 
  \text{\quad where } i(\mb h) := \inf\,\{\, p \mid \mb h(p) \neq 0\,\}.
  \end{equation*}

In passing, notice, if $\mc A_{p+s}$ is locally free, say of rank $\mb
a(p+s)$, then $\mu_p$ is locally represented by a $\mb b^*(p)$ by $t\mb
a(p+s)$ matrix.  So locally $L_p$ is defined by the vanishing of its
minors of size $\mb h^*(p)+1$ and the nonvanishing of those size $\mb
h^*(p)$; moreover,
 \begin{equation}\label{eqeqLev2}
 \mb h^*(p) \le \min\{\,t\,\mb a(p+s),\ \mb b^*(p)\,\}.
 \end{equation}

Given a scheme point $g\in \bb G$, let $K$ be its residue field.  Then
$g\in \bb{HL}_{\mc B}^{\mb h}$ iff the map of $K$-vector spaces
$\mu_p\ox K$ is of rank $\mb h^*(p)$ for $p\ge -s$.  Thus as $\mb h$
varies, the $\bb{HL}_{\mc B}^{\mb h}$ stratify $\bb G$; that is,
set-theoretically their disjoint union $\coprod_{ h}\bb{HL}_{\mc
B}^{\mb h}$ is $\bb G$.

Assume $S$ is reduced and irreducible.  Assume $L := \bb{HL}_{\mc
B}^{\mb h}$ contains the generic point of $\bb G$.  Then $L$ is open,
since it's the intersection of an open subscheme and a closed subscheme,
and the latter must be all of $\bb G$.  Thus $L$ is nonempty, reduced,
and irreducible, and it's covered by open subschemes, each one
isomorphic to an open subscheme of the affine space over $S$ of fiber
dimension $t(\mb b(s) - t)$.  Also, by lower semicontinuity of rank, if
$\bb{HL}_{\mc B}^{\mb h^\di} \neq \emptyset$, then $\mb h^\di(p) \le \mb
h(p)$ for all $p$.  Notice $\sum_p\mb h^\di(p) \le \sum_p\mb h(p)$, with
equality iff $\mb h^\di = \mb h$.  So call $\mb h$ {\it maximal} for
 \indt{maximal (Hilbert polynomial)}
$\mc B$ and $t$.
 \end{sbs}

\begin{remark}\label{rmLev1}
 In one important case, let's see that the map $\mu_p$ of \eqref{eqLev1}
can be represented locally by a ``catalecticant'' matrix; cf.\
\cite{MaxGor}*{Thm.\,11, proof, p.\,3146}.

Take $t:=1$.  Then $\bb G = \IP(\mc B_s)$ and $\mc U^* = \mc O_\bb G(-1)
\subset (\mc B^\dg_{\bb G})_{-s}$.

Assume $S:=\Spec(k)$ for some (Noetherian) ring $k$.  Let $X$ be a
vector of $r$ weighted variables, set $A := k[X]$ and take $\mc B:= \mc
A := \wt A$.  Then $A_s$ is the $k$-span of the monomials $X^w$ where
$w$ is a multiindex of weight $s$.  As in \eqref{exPolyRg}, view
$(A^\dg)_{-s}$ as the $k$-span of the Laurent monomials $X^{-w}$.

Let $R$ denote the polynomial ring over $k$ on the $X^w$ viewed as
independent variables.  Then $\bb G = \Proj(R)$.  Moreover, the
universal surjection $(\mc A_s)_{\bb G} \onto \mc U$ arises from the
tautological $R$-map $R\ox_kA_s \to R[1]$ sending $1\ox X^w$ to $X^w$ in
$R_1 = A_s$.  So the inclusion $\mc U^* \into (\mc A^\dg_s)_{\bb G}$
arises from the $R$-dual map $R[-1] \to R\ox_k(A^\dg)_{-s}$ sending
1 to $\sum X^w\ox X^{-w}$ in $R_1\ox_k(A^\dg)_{-s}$, because the
canonical isomorphism $A_s \ox_k(A^\dg)_{-s} = \Hom_k(A_s,\,A_s)$
identifies $\sum X^w\ox X^{-w} $ with 1.

Fix $w_0$, and set $x_w := X^w/X^{w_0}$.  Then the principal open set
$D_+(X^{w_0})$ is equal to $\Spec(k[\{x_w\}])$.  Fix an affine
$S$-scheme $T := \Spec(K)$ and a $T$-point of $D_+(X^{w_0})$ with
coordinates $(a_w)$; the latter is the $S$-map $T\to D_+(X^w)$ defined
by $x_w\mapsto a_w$.  Then $\mc U^*_T = \mc O_T\cdot f \subset (\mc
A^\dg_T)_{-s}$ where $f := \sum a_w X^{-w}$ in $(A^\dg_K)_{-s} =
\Gamma((\mc A^\dg_T)_{-s})$.

Note that $\mu_p$ arises from the map $(A_K)_{p+s}\to (A_K^\dg)_p$
of multiplication by $f$ for $p\ge -s$.  As bases of $(A_K)_{p+s}$ and
$(A_K^\dg)_p$, take the monomials $X^v$ and $X^{-u}$ of weights
$p+s$ and $p$.  Then $\mu_p$ is represented by the ``catalecticant''
matrix $((b_{uv}))$ where $b_{uv} := a_w$ with $w := u+v$.
 \end{remark}

\begin{theorem}\label{thRk}
 The scheme $\bb{HL}_{\mc B}^{\mb h}$ of \eqref{sbLev} represents
$T\mapsto \mb H\Lambda_{\mc B_T}^{\mb h}$; in other words, if
non\-empty, $\bb{HL}_{\mc B}^{\mb h}$ carries a universal level
homogeneous quotient, $\mc {L}_{\mc B}^{\mb h}$ say.
 \end{theorem}

 \begin{proof}
 Given $T$, we have to see that the set of $T$-points of $\bb{HL}_{\mc
B}^{\mb h}$ is equal to $\mb H\Lambda_{\mc B_T}^{\mb h}$.  In
particular, taking $T:=\bb{HL}_B^{\mb h}$ yields $\mc L_{\mc B}^{\mb
h}$.

Given a map $T\to \bb{HL}_{\mc B}^{\mb h}$, set $\mc D_p :=
\Im((\mu_p)_T)$ and $\mc D := \bigoplus_{p\ge-s} \mc D_p$.  Plainly,
$\mc D$ is an $\mc A_T$-submodule of $(\mc B^\dg)_T$.  Also, $0\to \mc
D_p \to (\mc B_T^\dg)_p \to \Cok((\mu_p)_T) \to 0$ is exact.  But
$\Cok((\mu_p)_T)$ is locally free of rank $\mb b^*(p) -\mb h^*(p)$.  So
$\mc D_p$ is locally free of rank $\mb h^*(p)$.  Thus $\mc D\in
\bb{AG}_{\mc A_T}^{\mb h^*}$.

Note $\mc B^\dg_T/\mc D = \bigl(\bigoplus_{p<-s}(\mc B_T^\dg)_p\bigr)
\oplus \bigl(\bigoplus_{p\ge-s} \Cok((\mu_p)_T)\bigr)$.  So $\mc
B_T^\dg/\mc D$ is flat.  Thus $\mc D \in \mb H\Delta_{\mc B_T^\dg}^{\mb
h^*}$.  Set $\mc C:= \mc D^*$.  Then $\mc C \in \mb H\Psi_{B_T}^{\mb
h}$ and $\mc C^* = \mc D$ by \eqref{thGMD3}.  But $\mc D$ is generated
by $\mc D_{-s}$ by definition of $\mu_p$.  So $\mc C$ is level.  Thus
$\mc C \in \mb H\Lambda_{\mc B_T}^{\mb h}$.

Conversely, given $\mc C \in \mb H\Lambda_{\mc B_T}^{\mb h}$, set
$\mc D := \mc C^*$.  As $\mc C$ is level, $\mc D_{-s}$
generates $\mc D$.  Also, $\mc D \in \mb H\Delta_{B_T^\dg}^{\mb
h^*}$ by \eqref{thGMD3}; hence, $(\mc B_T^\dg)_p/\mc D_p$ is locally
free of rank $\mb b^*(p) -\mb h^*(p)$.

Note that $\mc D_{-s}$ is locally free of rank $t$ as $\mc C \in \mb
H\Lambda_{\mc B_T}^{\mb h}$.  Hence there's a unique map $T\to \bb G$
with $\mc U_T = \mc C_s$.  So $ (\mc U_T)^* = \mc D_{-s}$.  But $\mc
D_{-s}$ generates $\mc D$.  So $ \Im(\mu_p)_T = \mc D_p$.  So
$\Cok((\mu_p)_T)$ is locally free of rank $\mb b^*(p) -\mb h^*(p)$.
Thus $T\to \bb G$ factors through all $L_p$.  Thus the set of $T$-points
of $\bb{HL}_{\mc B}^{\mb h}$ is equal to $\mb H\Lambda_{\mc B_T}^{\mb
h}$.
 \end{proof}

\begin{corollary}\label{coLevS}
 There's a canonical isomorphism of schemes:\enspace $\bb H\Lambda_{\mc
B}^{\mb h} = \bb{HL}_{\mc B}^{\mb h}$.
 \end{corollary}

\begin{proof}
 By \eqref{prLevS} and \eqref{thRk}, both $\bb H\Lambda_{\mc B}^{\mb h}$
and $ \bb{HL}_{\mc B}^{\mb h}$ represent the functor $T\mapsto \mb
H\Lambda_{\mc B_T}^{\mb h}$.  Thus $\bb H\Lambda_{\mc B}^{\mb h}$ and
$ \bb{HL}_{\mc B}^{\mb h}$ are canonically isomorphic.
 \end{proof}

\section{Schemes of Multilevel Homogeneous Quotients}

\begin{setup}\label{se7}
 Keep the setup of \eqref{se6} with this exception: don't fix the
``potential Hilbert function'' $\mb h$ a priori; so don't set $s :=
s(\mb h)$, but define $s$ as in (3) below.  Note however that, given a
 \indn{s@$s$}
filtered $S$-Artinian $\mc A$-module with Hilbert function $\mb h$ and
$S$-socle type $\mb t$, plainly \eqref{leSocLev}(2) yields $s(\mb h) =
s(\mb t)$.

Assume $\mc A/F^1\mc A = \mc O_S$.  Fix a ``socle type,'' a nonzero
finite function $\mb t\: \Z\to \Z$.

With \eqref{sbMLev}--\eqref{prMLevCon} in mind, call a set $\{\mb h_m\}$
of functions {\it fit\/} for $\mb t$ if
 \indt{fit}
 \begin{enumerate}
 \item\enspace $\mb t(p) = \mb h_p(p) - \mb h_{p+1}(p)$\enspace for all
$p$, and

\item\enspace $\mb b(p) \ge \mb h_m(p) \ge \mb h_{m+1}(p) \ge 0$\enspace
 for all  $m$ and $p$, and

\item\enspace $\mb h_m = 0$\enspace
 for all $m > s := s(\mb t) := \sup\{\,p\mid \mb t(p) > 0\,\}$.
 \end{enumerate}
 Notice that (1) and (2) imply $\mb b(p) \ge \mb t(p)\ge 0$ for all $p$.

Set $\bar s := i(\mb t) := \inf\,\{\, p \mid \mb t(p) \neq 0\,\}$.  To
 \indn{sbar@$\bar s$}
facilitate recursion on $(s-\bar s)$, set
 \begin{equation*}\label{eqFac}
  \mb t'(p):= \begin{cases}  0 & \text{for } p = \bar s,\\
                            \mb  t(p) & \text{for }p \neq \bar s,
                     \end{cases} \and
 \mb h'_m := \begin{cases}  \mb h_{\bar s+1} & \text{for }m\le \bar s+1,\\
                            \mb h_m & \text{for }m\ge \bar s+1.
                     \end{cases}
 \end{equation*}
 Call $\mb t'$ and $\{\mb h'_m\}$ the {\it attendants\/} to $\mb t$ and
 \indn{tmbpr@$\mb t'$}  \indn{hmbzprm2@$\mb h'_m$} \indt{attendant}
$\{\mb h_m\}$.  Note $\mb t$ and $\mb t'$ are fixed through out this
section, but $\{\mb h_m\}$ and $\{\mb h'_m\}$ are not.

Suppose temporarily $\{\mb h_m\}$ is fit for $\mb t$.  If $\mb t' = 0$,
or equivalently $\bar s = s$, note that (3) yields $\mb h'_m = 0$ for all
$m$.  If $\mb t' \neq 0$, or equivalently $\bar s < s$, set $\bar s' := i(\mb
t')$ and note that $\bar s < \bar s' \le s$ and that $\{\mb h'_m\}$ is fit for
$\mb t'$.
 \indn{sbarpr@$\bar s'$}

Finally, generalize the notion of multilevel module in \eqref{sbMLev} to
sheaves, and for any $T/S$ and any $\{\mb h_m\}$, denote the sets of all
multilevel graded and all multilevel filtered quotients of $\mc B_T$
with the $\mb h_m$ as Hilbert functions by $\mb{H\Lambda}_{\mc
B_T}^{\{\mb h_m\}}$ and $\mb{F\Lambda}_{\mc B_T}^{\{\mb h_m\}}$.
 \indn{HmbLamhBTb@$\mb{H\Lambda}_{\mc B_T}^{\{\mb h_m\}}$}
 \indn{FmbLamhBTb@$\mb{F\Lambda}_{\mc B_T}^{\{\mb h_m\}}$}
 Similarly, generalize the operator $\Delta^n$ and $\Lambda^n$ of
\eqref{sbMLev} to sheaves $\mc D \in \mb{H\Delta}_{B^\dg_T}^{\{\mb
h_m^*\}}$ and $\mc C \in \mb{H\Lambda}_{\mc B_T}^{\{\mb h_m\}}$.  By
\eqref{prMLev}(1), all these objects are functorial in $T$.
 \indn{HmbDelBT@$\mb{H\Delta}_{B^\dg_T}^{\{\mb h_m^*\}}$}
 \indn{HmbLamB@$\mb{H\Lambda}_{\mc B_T}^{\{\mb h_m\}}$}

If $\{\mb h_m\}$ is fit for $\mb t$, then $\mb t(p) = \mb h_p(p) - \mb
h_{p+1}(p)$ by (1); so by \eqref{prMLev}(2), for any $T/S$, every $\mc C
\in \mb{H\Lambda}_{\mc B_T}^{\{\mb h_m\}}$ is of $T$-socle type $\mb t$.
The converse holds, as noted below.

Assume $\mc C \in \mb{H\Lambda}_{\mc B_T}^{\{\mb h_m\}}$ is of $T$-socle
type $\mb t$.  Then $\mc C =\Lambda^m\mc C$ for $m \le \bar s$ by
\eqref{sbMLev} as $\mc C^*$ is locally generated in degree at most
$-\bar s$.  So $\mb h_m = \mb h_{\bar s}$ for $m \le \bar s$, and $\mb
h_{\bar s}$ is the Hilbert function of $\mc C$.  Similarly, $\mc C_p =
0$ for $p > s$ as $\mc C^*$ is locally generated in degree at least
$-s$.  Thus $s = s(\mb h_{\bar s})$.  And $\mb b(p) = 0$ for $p< {0}$ by
\eqref{se6}; thus
 \begin{equation}\label{eqLMlbd} \ts
 \rank (\mc C) = \sum_p \mb h_{\bar s}(p) \le \sum_{p=0}^s \mb b(p).
 \end{equation}
  Note $\{\mb h_m\}$ is fit for $\mb t$, since \eqref{prMLev}(2) gives
(1), and (2)--(3) follow from \eqref{sbMLev}.  So $\{\mb h'_m\}$ is fit
for $\mb t'$ by the above.  Also, $\Lambda^{\bar s+1}\mc C \in
\mb{H\Lambda}_{\mc B_T}^{\{\mb h'_m\}}$ by \eqref{prMLev}(3) with $n :=
\bar s+1$.  Thus $\Lambda^{\bar s+1}\mc C$ is of $T$-socle type $\mb
t'$.
 \end{setup}

\begin{proposition}\label{prMLevS}
  Fix\/ $\mb h\: \Z\to \Z$.  As $\{\mb h_m\}$ varies keeping $\mb h_0
= \mb h$, the functors $T\mapsto \mb H\Lambda_{\mc B_T}^{\{\mb
h_m\}}$ and $T\mapsto \mb F\Lambda_{\mc B_T}^{\{\mb h_m\}}$ are
representable by (possibly empty) subschemes $\bb H\Lambda_{\mc
B}^{\{\mb h_m\}}$ and $\bb F\Lambda_{\mc B}^{\{\mb h_m\}}$ of\/ $\bb
H\Psi_{\mc B}^{\mb h}$ and $\bb F\Psi_{\mc B}^{\mb h}$, stratifying
them.  Moreover, $\bb F\Lambda_{\mc B}^{\{\mb h_m\}}$ is the preimage of
$\bb H\Lambda_{\mc B}^{\{\mb h_m\}}$ under the retraction $\bb
F\Psi_{\mc B}^{\mb h} \onto \bb H\Psi_{\mc B}^{\mb h}$.
 \end{proposition}
 \indn{HbbLamhm@$\bb H\Lambda_{B}^{\{\mb h_m\}}$}
 \indn{FbbLamhm@$\bb F\Lambda_{B}^{\{\mb h_m\}}$}

\begin{proof} 
  Set $H := \bb H\Psi_{\mc B}^{\mb h}$\,, and let $\mc D\subset \mc
B_H^\dg$ be the universal subsheaf with Hilbert function $\mb h^*$.
Much as in the proof of \eqref{thRep}, the theory of flattening
stratification yields, for every $n$, a subscheme $H_m$ of $H$ such that
an $S$-map $T\to H$ factors through $H_m$ iff $\Delta^n\mc D_T \in
\mb{H\Delta}_{\mc D_T}^{\mb h_m^*}$.  Set $\bb H\Lambda_{\mc B}^{\{\mb
h_m\}} := \bigcap_mH_m$.  Then $\bb H\Lambda_{\mc B}^{\{\mb h_m\}}$
represents $T\mapsto \mb H\Lambda_{\mc B_T}^{\{\mb h_m\}}$.  Plainly,
every scheme-point of $H$ lies in $\bb H\Lambda_{\mc B}^{\{\mb h_m\}}$
for some choice of $\{\mb h_m\}$.  Thus, the $\bb H\Lambda_{\mc
B}^{\{\mb h_m\}}$ stratify $H$.  Finally, it's immediate from the
definitions that the preimage of $\bb H\Lambda_{\mc B}^{\{\mb h_m\}}$ in
$\bb F\Psi_{\mc B}^{\mb h}$ represents the functor $T\mapsto \mb
F\Lambda_{\mc B_T}^{\{\mb h_m\}}$.
 \end{proof}

\begin{sbs}[A recursive construction]\label{sbMLSch}
 Fix $\{\mb h_n\}$, and let's recursively construct a scheme $L :=
\bb{HL}^{\{\mb h_m\}}_{\mc B}$, and if $L \neq \emptyset$, also a
tautological multilevel homogeneous quotient $\mc L^{\{\mb h_m\}}_{\mc
B}$ in $\mb H\Lambda_{\mc B_L}^{\{\mb h_m\}}$ of $L$-socle type $\mb
t$.  If $L = \emptyset$, let's leave $\mc L^{\{\mb h_m\}}_{\mc B}$
undefined.

If there exist a $T/S$ and a $\mc C \in \mb{H\Lambda}_{\mc B_T}^{\{\mb
h_m\}}$, then $\{\mb h_m\}$ is fit for $\mb t$ if $\mc C$ is of
$T$-socle type $\mb t$, or equivalently, if $\{\mb h_m\}$ satisfies
\eqref{se7}(1); see the end of \eqref{se7}.  If not, set $L :=
\emptyset$.  So let's assume $\{\mb h_n\}$ is fit for $\mb t$, and let
$\mb t'$ and $\{\mb h'_m\}$ be their attendants.  Recall that either
$\bar s = s$ or $\bar s < \bar s' \le s$.  Let's proceed by recursion on
$(s-\bar s)$

If $\bar s = s$, set $L' := S$ and $\mc L' = 0$.  If $\bar s < s$, set $L' :=
\bb{HL}^{\{\mb h'_m\}}_{\mc B}$, and if $L'\neq \emptyset$, set $\mc L'
:= \mc L^{\{\mb h'_m\}}_{\mc B}$; by recursion, $L'$ and $\mc L'$ are
defined.  If $L'= \emptyset$, set $L :=\bb{HL}^{\{\mb h_m\}}_{\mc B} :=
\emptyset$.

Assume $L'\neq \emptyset$.  Set $\?{\mb h} := \mb h_{\bar s} - \mb
h_{\bar s+1}$.  Set $\mc B' := \Ker(\mc B_{L'}\onto \mc L')$, and let $\mb
b' $ be its Hilbert function.  Now, $\mc L'$ has Hilbert function $\mb
h'_{\bar s'}$.  But $\mb h'_{\bar s'} = \mb h'_{\bar s+1}$ as $\bar s+1\le \bar s'$.
And $ \mb h'_{\bar s+1} = \mb h_{\bar s+1}$ by definition of the $\mb h'_m$.
So $\mb b' = \mb b - \mb h_{\bar s+1}$.  Thus $\mb b'(p) \ge \?{\mb
h}(p)\ge 0$ for all $p$.  So using \eqref{sbLev}, define
 \begin{equation}\label{eqMLSch3}
            L := \bb{HL}^{\{\mb h_m\}}_{\mc B} := L_{\mc B'}^{\?{\mb h}}.
 \end{equation}

If $L\neq \emptyset$, define $\mc L^{\{\mb h_m\}}_{\mc B}$ via the
following pushout diagram on $L$:
 \begin{equation}\label{eqMLSch4}
 \begin{CD}
 0 @>>>    \mc B'_L  @>>>  \mc B_L  @>>> \mc L'_L @>>> 0\\
 @.              @VVV                @VVV\              @VV1V\\\
 0 @>>> \mc L_{\mc B'}^{\?{\mb h}} @>>> \mc L^{\{\mb h_m\}}_{\mc B}
        @>>> \mc L'_L @>>> 0
 \end{CD}
 \end{equation}
 Note $ \mc L'$ is in $\mb H\Lambda_{\mc B_{L'}}^{\{\mb h'_m\}}$.  So
$\mc L'_L$ is in $\mb H\Lambda_{\mc B_L}^{\{\mb h'_m\}}$ by
\eqref{prMLev}(1).  By construction, $\mc L_{\mc B'}^{\?{\mb h}}$ is in
$\mb H\Lambda_{\mc B_L'}^{\?{\mb h}}$.  Thus \eqref{prMLevCon} yields
$\mc L_{\mc B}^{\{\mb h_m\}} \in \mb H\Lambda_{\mc B_L}^{\{\mb
h_m\}}$, as desired.
 \end{sbs}

\begin{theorem}\label{thMLRpr}
 The scheme $L := \bb{HL}^{\{\mb h_m\}}_{\mc B}$ of \eqref{sbMLSch}
represents $T\mapsto \smash{\mb{H\Lambda}_{\mc B_T}^{\{\mb h_m\}}}$, and
if $L\neq \emptyset$, then $\mc L_{\mc B}^{\{\mb h_m\}}$ is the
universal multilevel homogeneous quotient of $\mc B_L$.
 \end{theorem}

\begin{proof}
 If $L\neq \emptyset$, set $\mc L := \mc L_{\mc B}^{\{\mb h_m\}}$.
Given $T \neq \emptyset$ and a multilevel sheaf $\mc C \in
\mb{H\Lambda}_{\mc B_T}^{\{\mb h_m\}}$, we have to show $L\neq
\emptyset$ and find a unique map $\tau\:T\to L$ with $\mc C =\tau^* \mc
L$.  To do so, as in \eqref{sbMLSch}, let's proceed by recursion on
$(s-\bar s)$.

First, suppose $L\neq \emptyset$ and $\tau$ exists; let's check $\tau$
is unique.  Note $L'\neq \emptyset$ as $L\neq \emptyset$.  Let
$\lambda\: L\to L'$ be the structure map, and set $\tau' :=
\lambda\tau$.  So $\tau'\: T\to L'$.

If $\bar s = s$, then $L' = S$, and so $\tau'\: T\to S$ is the structure
map, so unique.

Assume $\bar s < s$.  Then by recursion $\tau'^*\mc L'$ determines $\tau'$.
But $\tau'^*\mc L = \tau^*\mc L'_L$.  By construction, $\mc L'_L =
\Lambda^{\bar s+1}\mc L$.  So $\tau^*\mc L'_L = \Lambda^{\bar s+1}(\tau^*\mc
L) = \Lambda^{\bar s+1}\mc C$.  So $\mc C$ determines $\tau'$.

For any $\bar s$, view $\tau$ as a lifting of $\tau'$.  Then $\tau$ is
determined by the level quotient $\tau^*\?{\mc L}$ of $\mc B_T'$ by
\eqref{thRk}.  Pull back \eqref{eqMLSch4} under $\tau$.  The resulting
diagram plainly has surjective vertical maps, and as $\mc L'_L$ is
flat, it also has exact rows.  But, as was just observed, the surjection
$\tau^*\mc L \onto \tau^* \mc L'_L$ is equal to $\mc C \onto
\Lambda^{\bar s+1}\mc C$.  But $\tau^*\?{\mc L}$ is the kernel of $\tau^*\mc L
\onto \tau^* \mc L'_L$.  Hence $\tau^*\?{\mc L}$ is determined by $\mc
C$.  Thus $\tau$ is too, as desired.

It remains to show $L\neq \emptyset$ and find $\tau$.  Set $\mc C' :=
\Lambda^{\bar s+1}\mc C$.   Proceed by recursion.

If $\bar s = s$, then $L' = S$.  So define $\tau'\: T\to L'$ to be the
structure map.

Assume $\bar s < s$.  Then $C' \in \mb{H\Lambda}_{\mc B_T}^{\{\mb h'_m\}}$
by \eqref{prMLev}(3) with $n := \bar s+1$.  So by recursion, $L'\neq
\emptyset$ and there's a map $\tau'\: T\to L'$ with $\mc C' = \tau'^*\mc
L'$.

For any $\bar s$, set $\?{\mc C} = \Ker(\mc C\to \mc C')$.  Then $\?C \in
\mb{H\Lambda}_{\mc B_T}^{\?{\mb h}}$ by \eqref{prMLev}(4) with $n :=
\bar s+1$.  So $L\neq \emptyset$ and $\tau'$ lifts to a map $\tau\:T\to L$
with $\?{\mc C} = \tau^* \?{\mc L}$ by \eqref{thRk}.  Pull back
\eqref{eqMLSch4} under $\tau$.  As noted above, the resulting diagram
has exact rows and surjective vertical maps.  But $\?{\mc C} = \tau^* \?{\mc
L}$ and $\mc C' = \tau^*\mc L'_L$.  Thus $\mc C = \tau^*\mc L$, as
desired.
 \end{proof}

\begin{corollary}\label{coMLevS}
 There's a canonical isomorphism: $\bb H\Lambda_{\mc
B}^{\{\mb h_m\}} = \bb{HL}_{\mc B}^{\{\mb h_m\}}$.
 \end{corollary}

\begin{proof}
 By \eqref{prMLevS} and \eqref{thMLRpr}, both $\bb H\Lambda_{\mc
B}^{\{\mb h_m\}}$ and $\bb{HL}_{\mc B}^{\{\mb h_m\}}$ represent
$T\mapsto \mb H\Lambda_{\mc B_T}^{{\{\mb h_m\}}}$.  Thus $\bb
H\Lambda_{\mc B}^{\{\mb h_m\}}$ and $\bb{HL}_{\mc B}^{\{\mb h_m\}}$ are
canonically isomorphic.
 \end{proof}

\begin{sbs}[Recursively maximal]\label{sbRecMax}
  Call a set $\{\mb h_m\}$ {\it recursively maximal\/} for\ $\mb t$ and
 \indt{recursively maximal}
$T/S$ if these three conditions hold, with $\mb t'$ and $\{\mb h'_m\}$
the attendants to $\mb t$ and $\{\mb h_m\}$:
 \begin{enumerate}
 \item if $\bar s < s$, then  $\{\mb h'_m\}$ is recursively maximal
for $\mb t'$ and $T/S$;

\item given any $T^\di/T$ and any $\{\mb h^\di_m\}$ and any $\mc C^\di
\in \mb{H\Lambda}_{\mc B_{T^\di}}^{\{\mb h^\di_m\}}$ of $T^\di$-socle
type $\mb t$, with $\Lambda^{\bar s+1}\mc C^\di \in \mb{H\Lambda}_{\mc
B_{T^\di}}^{\{\mb h'_m\}}$ if $\bar s < s$, necessarily $\mb h^\di_m(p)
\le \mb h_m(p)$ for all $m,\ p$;

\item there exist $T^\sh/T$ and $\mc C^\sh \in \mb{H\Lambda}_{\mc
B_{T^\sh}}^{\{\mb h_m\}}$ of $T^\sh$-socle type $\mb t$.
 \end{enumerate}

Given (1) and (3), the inequality $\mb h^\di_m(p) \le \mb h_m(p)$ in (2)
holds for all $m$ if it holds just for $m = \bar s$.  Indeed, (1)
implies it holds for $m > \bar s$ by recursion.  But both $\mc C^\di$
and $\mc C^\sh$ are of socle type $\mb t$; so $\mb h^\di_m = \mb
h^\di_{\bar s}$ and $\mb h_m = \mb h_{\bar s}$ for $m \le \bar s$ by the
end of \eqref{se7}.

Notice, in (2), it suffices to use $T^\di$ of the form $\Spec K$ where
$K$ is a field.  Indeed, given any $T^\di$, just take $K$ to be one of
its residue fields.  Ditto for (3).

Notice (3) implies $\{\mb h_m\}$ is fit for $\mb t$ by the end of
\eqref{se7}.

Notice, if $\{\mb h_m\}$ and $\{\mb h^\di_m\}$ are both recursively
maximal for $\mb t$ and $T/S$, then $\{\mb h_m\}= \{\mb h^\di_m\}$.
Indeed, let $\mc C^\di \in \mb{H\Lambda}_{\mc B_{T^\di}}^{\{\mb
h^\di_m\}}$ be provided by (3).  Now, if $\bar s < s$, then $\{\mb
h'_m\}$ is equal to the attendant set to $\{\mb h^\di_m\}$ by (1) and
recursion; hence, $\Lambda^{\bar s+1}\mc C^\di \in \mb{H\Lambda}_{\mc
B_{T^\di}}^{\{\mb h'_m\}}$ by \eqref{prMLev}(3) with $n := \bar s+1$.
So $\mb h^\di_m(p) \le \mb h_m(p)$ for all $m$ and $p$ by (2).  By
symmetry, $\mb h_m(p) \le \mb h^\di_m(p)$ for all $m$ and $p$.  Thus
$\{\mb h^\di_m\}= \{\mb h_m\}$.
 \end{sbs}

\begin{sbs}[Quasi-permissible]\label{sbRecPer}
  Call $\mb t$ {\it quasi-permissible\/} if either $\bar s = s$ and then
 \indt{quasi-permissible}
$\mb b(s) \ge \mb t(s)$, or else $\bar s < s$ and then these two conditions
hold:
 \begin{enumerate}
 \item the attendant $\mb t'$ to $\mb t$ is quasi-permissible;

\item given any $T^\ft/S$ and any $\{\mb h^\ft_m\}$ recursively maximal for
$\mb t'$ and $T^\ft$, necessarily
         $$\mb b(\bar s) - \mb h^\ft_{\bar s+1}(\bar s) \ge \mb t(\bar s).$$
 \end{enumerate}
 \end{sbs}

\begin{theorem} \label{thMax}
 If there exists a recursively maximal set for $\mb t$ and $T/S$, then
$\mb t$ is quasi-permissible.  If $\mb t$ is quasi-permissible and $S$
is reduced and irreducible, then there exists a recursively maximal set
$\{\mb h_m\}$ for $\mb t$ and $T/S$; moreover, for any such set $\{\mb
h_m\}$, then $\bb H\Lambda_{B}^{\{\mb h_m\}}$ is nonempty, reduced, and
irreducible, and it's covered by open subschemes, with each one
isomorphic to an open subscheme of the affine space over $S$ of fiber
dimension $\mb H$ where $\mb H :=\sum_p\mb t(p)\bigl(\mb b(p) - \mb
h_{\bar s}(p)\bigr)$.
 \end{theorem}

\begin{proof}
 First, assume there exists a recursively maximal set $\{\mb h_m\}$ for
$\mb t$ and $T/S$. Note \eqref{sbRecMax}(3) provides $\mc C^\sh \in
\mb{H\Lambda}_{\mc B_{T^\sh}}^{\{\mb h_m\}}$ of $T^\sh$-socle type $\mb
t$.  So $\mb b(\bar s) \ge \mb h_{\bar s}(\bar s)$, and $\mb t(\bar s) =
\mb h_{\bar s}(\bar s) - \mb h_{\bar s+1}(\bar s)$ by \eqref{prMLev}(2).
Thus $\mb b(\bar s) - \mb h_{\bar s+1}(\bar s) \ge \mb t(\bar s)$.

Suppose $\bar s = s$.  Then $\mb h_{\bar s+1} = 0$ by \eqref{sbMLev}.
Thus $\mb b(s) \ge \mb t(s)$ as required by \eqref{sbRecPer}.

Suppose $\bar s < s$.  Then $\{\mb h'_m\}$ is recursively maximal for
$\mb t'$ and $T/S$ by \eqref{sbRecMax}(1).  So by recursion,
\eqref{sbRecPer}(1) holds.  Let's prove \eqref{sbRecPer}(2).  Note
$\{\mb h^\ft_m\} = \{\mb h'_m\}$ by the last paragraph of
\eqref{sbRecMax}.  But $\mb h'_{\bar s+1} = \mb h_{\bar s+1}$.  Thus
$\mb h^\ft_{\bar s+1} = \mb h_{\bar s+1}$, as desired.

Next, assume $\mb t$ is quasi-permissible and $S$ is reduced and
irreducible. Roughly follow \eqref{sbMLSch}, proceeding by recursion on
$(s-\bar s)$.  Set $\?t := \mb t(\bar s)$.

Suppose $\bar s = s$.  Again set $L' := S$, set $\mc L' = 0$, and set
$\mb h'_m:=0$ for all $m$.  Then $L'$ is nonempty, reduced, and
irreducible by hypothesis; trivially, it's of fiber dimension 0.  As
$\mb h'_{\bar s+1}:=0$ and $\mb t$ is quasi-permissible, $\mb b(\bar s)
- \mb h'_{\bar s+1}(\bar s) \ge \?t > 0$.

Suppose $\bar s < s$.  By recursion assume that there's a recursively
maximal set $\{\mb h'_m\}$ for $\mb t'$ and $S$, and that $L' :=
\bb{HL}^{\{\mb h'_m\}}_{\mc B}$ is nonempty, reduced, and irreducible,
and covered by open subschemes, each one isomorphic to an open subscheme
of the affine space over $S$ of fiber dimension $\mb H'$ where $\mb H'
:= \sum_p\mb t'(p)\bigl(\mb b(p) - \mb h'_{\bar s'}(p)\bigr)$.  So $\mb
h'_m = \mb h'_{\bar s'}$ for $m \le \bar s'$ by \eqref{sbMLev} with
\smash{$\mc L^{\{\mb h'_m\}}_{\mc B}$} for $\mc C$.  But $\bar s+1\le
\bar s'$.  Thus $\mb h'_m = \mb h'_{\bar s+1}$ for $m \le \bar s+1$.
Also, $\mb t$ is quasi-permissible; so \eqref{sbRecPer}(2) gives $\mb
b(\bar s) - \mb h'_{\bar s+1}(\bar s) \ge \?t > 0$.

For any $\bar s$, set $\mc B' := \Ker(\mc B_{L'}\onto \mc L')$; let $\mb
b'$ be its Hilbert function.  Now, $\mc L'$ has Hilbert function $\mb
h'_{\bar s'}$.  But $\mb h'_{\bar s'} = \mb h'_{\bar s+1}$ and $\mb
b(\bar s) - \mb h'_{\bar s+1}(\bar s) \ge \?t > 0$. Thus $\mb b' = \mb b
- \mb h'_{\bar s+1}$, and so $\mb b'(\bar s) \ge \?t > 0$.  Hence, the
end of \eqref{sbLev} with $S := L'$ yields a maximal Hilbert function,
$\bar{\mb h}$ say, for $\mc B'$ and $\?t$.  For use below, note that the
construction in \eqref{sbLev} yields $\bar t = \bar{\mb h}(\bar s)$.

Set $\mb h_m := \?{\mb h} + \mb h'_{\bar s+1}$ for $m \le \bar s$; set $\mb
h_m := \mb h'_m$ for $m > \bar s$.  Recall $\mb h'_m = \mb h'_{\bar s+1}$ for
$m \le \bar s+1$; see above.  Hence $\{\mb h'_m\}$ is the attendant to
$\{\mb h_m\}$.  Thus \eqref{sbRecMax}(1) holds.

The end of \eqref{sbLev} also yields a nonempty, reduced and irreducible
$L := \bb{HL}_{\mc B}^{\?{\mb h}}$.  By \eqref{thRk},  it represents
$T\mapsto \mb H\Lambda_{\mc B_T}^{\?{\mb h}}$, and correspondingly
carries a universal quotient $\mc {L}_{\mc B}^{\?{\mb h}}$.
 Set $\bb{HL}_{\mc B}^{\{\mb h_m\}} := L$ and $\mc L := \mc {L}_{\mc
B}^{\?{\mb h}}$.

Given any $T^\di/S$ and any $\{\mb h^\di_m\}$ and any $\mc C^\di \in
\mb{H\Lambda}_{\mc B_{T^\di}}^{\{\mb h^\di_m\}}$ of $T^\di$-socle type
$\mb t$, with $\Lambda^{\bar s+1}\mc C^\di \in \mb{H\Lambda}_{\mc
B_T^\di}^{\{\mb h'_m\}}$ if $\bar s < s$, there's a unique map $\tau'\:
T^\di\to L'$ with $\Lambda^{\bar s+1}\mc C^\di = \tau'^*\mc L'$.  Also,
$\tau'$ lifts to a unique map $\tau\:T^\di\to L$ with $\mc C^\di =
\tau^*\mc L$.  But $\?{\mb h}$ is maximal.  So $\mb h^\di_{\bar s}(p) -
\mb h^\di_{\bar s+1}(p) \le \?{\mb h}(p)$ for all $p$.  Thus $\mb
h^\di_{\bar s}(p) \le \mb h_{\bar s}(p)$ for all $p$.  Thus
\eqref{sbRecMax}(2) holds.  And, if $T^\di :=L$ and $\mc C^\di := \mc
L$, then $\mb h^\di_m = \mb h_m$ for all $m$; thus \eqref{sbRecMax}(3)
holds.

By the end of \eqref{sbLev} too, $L$ is covered by open subschemes, each
one isomorphic to an open subscheme of the affine space over $L'$ of
fiber dimension $\?t(\mb b'(\bar s)-\?t)$.  From the above, recall $\mb
b' = \mb b - \mb h'_{\bar s+1}$ and $\mb h_{\bar s}(\bar s) := \?{\mb
h}(\bar s) + \mb h'_{\bar s+1}(\bar s)$ and $\?t = \?{\mb h}(\bar s)$
and $\?t := \mb t(\bar s)$.  Thus $\?t(\mb b'(\bar s)-\?t) = \mb t(\bar
s)\bigl(\mb b(\bar s) - \mb h_{\bar s}(\bar s))$.

 However, $\mb t'(p) = 0$ for $p \le \bar s$, and $\mb t'(p) = \mb t(p)$
for $p > \bar s$; furthermore, $\mb h'_{\bar s'} = \mb h_{\bar s+1}$, and $\mb
h_{\bar s+1}(p) = \mb h_{\bar s}(p)$ for $p > \bar s$; see \eqref{se7} and
\eqref{sbMLev}.  Hence \begin{equation*}\label{eqMax1}\ts
  \?t(\mb b'(\bar s)-\?t) + \sum_p\mb t'(p)\bigl(\mb b(p)
                                         - \mb h'_{\bar s'}(p)\bigr)
   = \sum_p\mb t(p)\bigl(\mb b(p) - \mb h_{\bar s}(p)\bigr).
 \end{equation*}
 Thus $L$ is covered by open subschemes, each one isomorphic to an open
subscheme of the affine space over $S$ of fiber dimension $\mb H$, as
desired.
 \end{proof}

\begin{sbs}[Recursively compressed]\label{sbRC}
  Given $\mc C \in \mb{H\Lambda}_{\mc B_T}^{\{\mb h_m\}}$ for some
$\{\mb h_m\}$ and some $T/S$, call $\mc C$ {\it recursively compressed
for\/} $\mb t$ if $\mc C$ is of $T$-socle type $\mb t$ and if the
following two conditions hold, where $\mb t'$ is the attendant to $\mb
t$:
 \begin{enumerate}
 \item if $\bar s < s$, then $\Lambda^{\bar s+1}\mc C$ is recursively
compressed for $\mb t'$;
 \indt{recursively compressed}

\item for any $T^\di/S$ and $\{\mb h^\di_m\}$ and $\mc C^\di \in
\mb{H\Lambda}_{\mc B_{T^\di}}^{\{\mb h^\di_m\}}$ of $T^\di$-socle type
$\mb t$, with $\Lambda^{\bar s+1}\mc C^\di$ recursively compressed for
$\mb t'$ if $\bar s < s$, necessarily $\rank (\mc C^\di) \le \rank (\mc
C)$.
 \end{enumerate}

Notice, for every recursively compressed $\mc C$ for $\mb t$, the number
$\rank (\mc C)$ is the same.  Moreover, if $\mc C$ is recursively
compressed, then so is $\mc C_R$ for any $R/T$.
 \end{sbs}

\begin{theorem} \label{thRC}
 Assume $\mb t$ is quasi-permissible.

 \(1) In some $\mb{H\Lambda}_{\mc B_T}^{\{\mb h_m\}}$, there exists a
recursively compressed $\mc C$ for $\mb t$ .

\(2) Let $\{\mb h_m\}$ be recursively maximal for $\mb t$ and some
$T/S$.  Let $\mc C^\sh \in \smash{\mb{H\Lambda}_{\mc B_{T^\sh}}^{\{\mb
h^\sh_m\}}}$ be of $T^\sh$-socle type $\mb t$.  Then $\{\mb h_m\} =
\{\mb h^\sh_m\}$ iff\/ $\mc C^\sh$ is recursively compressed for $\mb
t$.
 \end{theorem}

\begin{proof}
 Let's prove (1) and (2) together.  By recursion on $(s-\bar s)$, assume
(2) holds with $\mb t$ and $\{\mb h_m\}$ replaced by their attendants
$\mb t'$ and $\{\mb h'_m\}$ if $\bar s < s$.

{\bf For (1),} form the set $\mb S$ of all triples $\bigl(T/S,\,\{\mb
h_m\},\,\mc C\bigr)$ with $\mc C \in \mb{H\Lambda}_{\mc B_T}^{\{\mb
h_m\}}$ of $T$-socle type $\mb t$, and $\Lambda^{\bar s+1}\mc C$
recursively compressed for $\mb t'$ if $\bar s < s$.  Let's see $\mb S$
is nonempty.

Let $K$ be a residue field of $S$, and set $T := \Spec(K)$.  Then $T$ is
reduced and irreducible.  So \eqref{thMax} yields a recursively maximal
set $\{\mb h_m\}$ for $\mb t$ and $T$.  So \eqref{sbRecMax}(3) yields a
$T^\sh/T$ and a $\mc C^\sh \in \mb{H\Lambda}_{\mc B_{T^\sh}}^{\{\mb
h_m\}}$ of $T^\sh$-socle type $\mb t$.  If $\bar s = s$, then
$\bigl(T^\sh,\,\{\mb h_m\},\,\mc C^\sh\bigr) \in \mb S$.

Suppose $\bar s < s$.  Let $\{\mb h'_m\}$ be the attendant to $\{\mb
h_m\}$.  Then $\Lambda^{\bar s+1}\mc C^\sh \in \mb{H\Lambda}_{\mc
B_{T^\sh}}^{\{\mb h'_m\}}$ by \eqref{prMLev}(3) with $n := \bar s+1$.  But
$\{\mb h'_m\}$ is recursively maximal for $\mb t'$ and $T^\sh$ by
\eqref{sbRecMax}(1) and \eqref{prMLev}(1).  Apply (2) with $\mb t'$ and
$\{\mb h'_m\}$ for $\mb t$ and $\{\mb h_m\}$; thus $\Lambda^{\bar s+1}\mc
C^\sh$ is recursively compressed for $\mb t'$.  Hence
$\bigl(T^\sh,\,\{\mb h_m\},\,\mc C^\sh\bigr) \in \mb S$.  Thus $\mb
S\neq \emptyset$ for any $\bar s$.

For $\bigl(T/S,\,\{\mb h_m\},\,\mc C\bigr)\in \mb S$, the numbers $\rank
(\mc C)$ are bounded owing to \eqref{eqLMlbd}.  So there exists
$\bigl(T/S,\,\{\mb h_m\},\,\mc C\bigr)\in \mb S$ with $\rank (\mc C)$
maximal.  Thus (1) holds.

{\bf For (2),} first assume $\mc C^\sh$ is recursively compressed for $\mb t$.
Suppose $\bar s < s$.  Then $\Lambda^{\bar s+1}\mc C^\sh$ is recursively
compressed for $\mb t'$ by \eqref{sbRC}(1).  Apply (2) with $\mb t'$ and
$\{\mb h'_m\}$ for $\mb t$ and $\{\mb h_m\}$; thus $\Lambda^{\bar s+1}\mc
C^\sh \in \mb{H\Lambda}_{\mc B_T^\sh}^{\{\mb h'_m\}}$.  Hence
\eqref{prMLev}(3) yields $\mb h^\sh_m = \mb h_m$ for $m>\bar s$.

For any $\bar s$, note that \eqref{sbRecMax}(2) yields $\mb h^\sh_m(p) \le
\mb h_m(p)$ for all $m$ and $p$.  Hence $\sum_p \mb h_m^\sh(p) \le
\sum_p \mb h_m(p)$ for all $m$. with equality iff $\{\mb h^\sh_m\} =
\{\mb h_m\}$.  Thus, owing to the preceding paragraph, it remains to
prove $\sum_p \mb h_m^\sh(p) \ge \sum_p \mb h_m(p)$ for $m\le \bar s$.

Note \eqref{sbRecMax}(3) gives $T/S$ and $\mc C \in \mb{H\Lambda}_{\mc
B_T}^{\{\mb h_m\}}$.  So $\Lambda^{\bar s+1}\mc C \in \mb{H\Lambda}_{\mc
B_T}^{\{\mb h'_m\}}$ by \eqref{prMLev}(3), and $\Lambda^{\bar s+1}\mc C$
has $T$-socle type $\mb t'$ by \eqref{prMLev}(2) with $\Lambda^{\bar
s+1}\mc C$ for $\mc C$.  So \eqref{sbRC}(1) holds by (2) with $\mb t'$
and $\{\mb h'_m\}$ for $\mb t$ and $\{\mb h_m\}$.  But, by assumption,
$\mc C^\sh$ is recursively compressed for $\mb t$.  Thus
\eqref{sbRC}(2), with $\mc C$ and $\mc C^\sh$ for $\mc C^\di$ and $\mc
C$, yields $\rank (\mc C) \le \rank (\mc C^\sh)$.

Note $\rank (\mc C) = \sum_p \mb h_{\bar s}(p)$ and $\rank (\mc C^\sh) =
\sum_p \mb h_{\bar s}^\sh(p)$ by \eqref{eqLMlbd} as both $\mc C$ and
$\mc C^\sh$ are of socle type $\mb t$.  Moreover, $\mb h_{\bar s} = \mb
h_m$ and $\mb h_{\bar s}^\sh = \mb h^\sh_m$ for $m \le \bar s$ by the
end of \eqref{se7}.  But, $\rank (\mc C) \le \rank (\mc C^\sh)$.  Thus
$\sum_p \mb h_m(p) \le \sum_p \mb h^\sh_m(p)$ for $m\le \bar s$, as
desired.

Conversely, assume $\{\mb h_m\} = \{\mb h^\sh_m\}$.  Then $\mc C^\sh \in
\mb{H\Lambda}_{\mc B_{T^\sh}}^{\{\mb h_m\}}$.  Two paragraphs above,
take $C$ to be $C^\sh$.  Thus \eqref{sbRC}(1) holds and it remains to
prove \eqref{sbRC}(2), both for $C^\sh$.

By (1), there exists a recursively compressed $\mc C^\ft$ for $\mb t$ in
some $\mb{H\Lambda}_{\mc B_{T^\ft}}^{\{\mb h^\ft_m\}}$.  So $\rank (\mc
C^\di) \le \rank (\mc C^\ft)$ by \eqref{sbRC}(2) for $\mc C^\ft$.
However, it's proved above that, if $\mc C^\sh$ is recursively
compressed for $\mb t$, then $\{\mb h^\sh_m\} = \{\mb h_m\}$; thus
$\{\mb h^\ft_m\} = \{\mb h_m\}$.  So $\rank (\mc C^\ft) = \sum_p \mb
h_{\bar s}(p)$ by \eqref{eqLMlbd}.  However, $\{\mb h_m\} = \{\mb
h^\sh_m\}$ by assumption; so $\rank (\mc C^\sh) = \sum_p \mb h_{\bar
s}(p)$ too by \eqref{eqLMlbd}.  Thus $\rank (\mc C^\di) \le \rank (\mc
C^\sh)$, as desired.
 \end{proof}

\section{Permissible Socle Types}

\begin{setup}\label{se8}
 Keep the setup of \eqref{se7}.  Also, assume each graded piece $\mc
A_p$ is locally free, say of rank $\mb a(p)$; recall $\mc A/F^1\mc A =
 \indn{ambp@$\mb a(p)$}
\mc O_S$, so $\mb a(0) = 1$.  Next, for all $m$ and $p$, set
 \begin{align}
 \mb g_m(p) &:= \ts \sum_{q = m}^s\mb t(q)\,\mb a(q-p),
          \label{eqsbImax}\\
 \mb h^{\rm I}_m(p)  &:= \min\{\,\mb g_m(p),\, \mb b(p)\,\},
          \label{eqsbImaxh}\\
  \beta^{\rm I}_m &:= \ts \sum_p \mb h^{\rm I}_m(p)
  \and \beta^{\rm I} := \beta^{\rm I}_{\bar s}\label{eqbetaI}.
 \indn{hmbzIm@$\mb h^{\rm I}_m(p)$}  \indn{gmp@$g_m(p)$}
 \indn{babetaIm@$\beta^{\rm I}_m$}
 \end{align}

 Call $\{\mb g_m\}$ the {\it pre-{\rm I}-set\/} of $\mb t$, and $\{\mb
h^{\rm I}_m\}$ the I-{\it set\/} of $\mb t$.
 \indt{preIs@pre-I-set}
 \indt{I-set}

Assume $\mb t(p) \ge 0$ for all $p$.  Then $\mb g_m(p) \ge 0$ and $\mb
h^{\rm I}_m(p) \ge 0$ for all $m$.

 Notice that the inequalities in \eqref{prhImax} below may become
stricter when $\mc A$ is replaced by $\mc A/(0:\mc B)$, so that $(0:\mc
B) = 0$, and that this replacement may be necessary if the bound is to
be achieved.  Thus, from a practical standpoint, it is reasonable to
make this replacement.  However, from a logical standpoint, this
replacement has no effect on the theory in this section and the
following ones.
  \end{setup}

\begin{proposition}[Maximality of $\mb h^{\rm I}_m$] \label{prhImax}
 Fix $\mc C \in \mb{H\Lambda}_{\mc B}^{\{\mb h_m\}}$ of $S$-socle type
$\mb t$.  Then $\mb h_m(p) \le \mb h^{\rm I}_m(p)$ for all $m$ and $p$;
also $\rank(\Lambda^m\mc C) \le \beta^{\rm I}_m$, with equality iff\/ $\mb
h_m = \mb h^{\rm I}_m$.

 \end{proposition}
 \begin{proof}
  Note $\mc C^*$ is of local generator type $\mb t^*$ by \eqref{deType}.
So locally there's a set of homogeneous generators of $\mc C^*$ with
$\mb t(q)$ elements in degree $-q$.  Hence locally $\mc
(\Delta^mC^*)_{-p}$ is generated by $\mb g_m(p)$ elements.  So $\mb
h_m(p) := \rank\mc (\Delta^mC^*)_{-p} \le \mb g_m(p)$.  But $\mc
(\Delta^mC^*)_{-p} \subset \mc B^\dg_{-p}$; so $\mb h_m(p) \le \mb
b(p)$.  Thus $\mb h_m(p) \le \mb h^{\rm I}_m(p)$.

Since $\rank(\Lambda^m\mc C) = \sum_p\mb h_m(p)$, the second assertion
follows directly.
 \end{proof}

\begin{sbs}[Permissible socle types]\label{sbPerm}
 Call $\mb t$ {\it permissible\/} if there's $v \ge 1$ with
 \indt{permissible}
 \begin{enumerate}
 \item[(a)]\enspace $\mb t(p) = 0$ for $p < v-1$, and
 \item[(b)]\enspace $\mb b(v) > \mb g_v(v)$, and
 \item[(c)]\enspace $\mb t(v-1)
                 = \max\,\{\, 0,\ \mb b(v-1) - \mb g_v(v-1)\,\}$, and
 \item[(d)]\enspace $\mb b(p) > \mb g_m(p)$ for $v_m \le p \le s$ and
all $m$ with $v_m := \inf\{\, p \mid \mb b(p) > \mb g_m(p) \,\}$.
 \end{enumerate}
 \indn{vm@$v_m$} \indn{v@$v$}
 This definition is virtually Iarrobino's Def.\,2.2 in
\cite{Iar84}*{p.\,344}, but here $s$ replaces $j$. 

Iarrobino assumed $S = \Spec k$ with $k$ an infinite field, $\mc A$ a
polynomial ring in two or more variables over $\mc O_S$, and $\mc B :=
\mc A$.  His proof of Prop.\,3.6 on p.\,356 in \cite{Iar84} and his
discussion in \S\,4E on p.\,372 yield the existence of $\mc C$ in
$\mb{H\Lambda}_{\mc A}^{\{\mb h^{\rm I}_m\}}$, provided $\mb t$ is
permissible.  He assumed $v\ge2$, but it's easy to handle the case $v
=1$; see the end of this subsection.  The existence of $\mc C$ was
reproved by Fr\"{o}berg and Laksov \cite{F-L83}*{Thm.\,14, p.\,142}; see
also Laksov \cite{JPAA217}*{Cor.\,5.5, p.\,2262}.  In the level case
($\bar s = s$) Boij, in Prop.\,3.5 on p.\,368 in \cite{BoijJA226},
removed the restriction that $k$ be infinite, and generalized the
existence to $\mc B := \mc A^{\oplus n}$ for $n \ge 1$.

In Iarrobino's setup, (d) doesn't appear explicitly, as it's automatic.
Plainly, it holds whenever the difference $\mb b(p) - \mb g_m(p)$ is
nondecreasing in $p$.  The latter holds whenever $\mb a(p)$ and $\mb
b(p)$ are nondecreasing in $p$, because \eqref{eqsbImax} yields
 \begin{equation*}\label{eqDiffg} \ts
  \mb g_m(p) -  \mb g_m(p+1)
        = \sum_{q=m}^s \mb t(q)\bigl( \mb a(q-p) - \mb a(q-(p+1)) \bigr)
                \ge 0.
 \end{equation*}

Notably, (a)--(d) can hold if $\mc A$ or $\mc B$ is $\mc O_S$-Artinian.

By \eqref{leperm}(1) below, $v$ is unique; in fact, $v = v_0$.
Moreover, as Iarrobino asserted in Def.\,2.2 in \cite{Iar84} and is
proved in \eqref{leperm}(1) in full generality, $v$ is the {\it initial
value\/} of any $\mc B_T/\mc I$ in $\mb{H\Lambda}_{\mc B_T}^{\{\mb
h^{\rm I}_m\}}$; that is, $\mc I_p = 0$ for $p < v$ but $\mc I_v \neq
0$.

Suppose $\mb t$ is permissible.  Then so is its attendant $\mb t'$ if
$\bar s < s$ by \eqref{leperm}(3).  Moreover, Iarrobino's proof of
Prop.\,3.6 on p.\,356 in \cite{Iar84} is set up to proceed by recursion
on $\bar s$ much as in \eqref{sbMLSch} and \eqref{thMLRpr}.  Thus there
exist $\mc C$ in $\mb{H\Lambda}_{\mc B}^{\{\mb h^{\rm I}_m\}}$.

Suppose $\mb t$ is permissible.  Then $\mb t$ is quasi-permissible by
\eqref{leperm}(4), and $\{\mb h^{\rm I}_m\}$ is fit for $\mb t$ by
\eqref{leperm}(5).  Thus all the theory developed in
\eqref{se7}--\eqref{thRC} applies to $\mb t$ and $\{\mb h^{\rm I}_m\}$.
In particular, by \eqref{se7}, for any $T/S$, every $\mc C \in
\mb{H\Lambda}_{\mc B_T}^{\{\mb h^{\rm I}_m\}}$ is of $T$-socle type $\mb
t$.

In Iarrobino's setup, a simpler formulation of permissibility was given
by Fr\"{o}berg and Laksov in \cite{F-L83}, and presented more clearly by
Zanello in Def.\,2.1 on p.\,183 and Prop.\,2.6 on p.\,184 in
\cite{ZanI}.  Namely, $\mb t$ is permissible iff $\mb t(p) = 0$ for $p <
b$, where
 \indn{bbrm@$b$}
 \begin{equation}\label{eqsbImaxc}
  b := \inf\{\,p \ge 0\mid \mb b(p) \ge \mb g_p(p)\,\}
 \end{equation}

The equivalence of the two formulations isn't proved in either
\cite{F-L83} or \cite{ZanI}, but is in \eqref{leperm}(2) below after $b$
is replaced by $b_1$ where
 \begin{equation}\label{eqdfb1}
 b_1 := \begin{cases} v_0 - 1  & \text{if\enspace}
             \mb b(v_0 - 1) = \mb g_{v_0 - 1}(v_0-1),\\
             v_0 & \text{if\enspace}
              \mb b(v_0 - 1) < \mb g_{v_0 - 1}(v_0 - 1).
     \end{cases}
 \indn{bbrm1@$b_1$}
 \end{equation}
 In this connection, note $\mb g_0(p) = \mb g_p(p)$ for $p \ge 0$ by
\eqref{eqsbImax}; hence, (d) above yields
 \begin{equation}\label{eqv0}
 v_0:= \inf\{\,p \ge 0\mid \mb b(p) > \mb g_p(p)\,\}.
 \end{equation}
 Consequently, $\mb b(p) \le \mb g_p(p)$ for $p \le v_0-1$.
In particular, $b_1$ is well defined.

 Plainly, $b_1 =b$ iff $\mb b(p) < \mb g_p(p)$ for $0 \le p \le v_0-2$.
The latter holds, notably, in Iarrobino's setup; here's why.  Set $\mb
f(p) := \mb g_p(p) - \mb g_{p+1}(p+1)$.  Then \eqref{eqsbImax} yields
 \begin{equation*}\label{eqb1b}\ts
   \mb f(p) =  \mb t(p)\mb a(0) 
      + \sum_{q=p+1}^s \mb t(q)\bigl( \mb a(q-p) - \mb a(q-(p+1))\bigr).
 \end{equation*}
 But $\mb a(p)$ is strictly increasing for $p \ge 0$.  So $\mb f(p) \ge
\mb t(s)\bigl( \mb a(s-p) - \mb a(s-(p+1))\bigr) > 0$.  However, $\mb b
= \mb a$.  Thus $\mb b(p) - \mb g_p(p)$ is strictly increasing for $p
\ge 0$.  However, $\mb b(v_0 - 1) \le \mb g_{v_0 - 1}(v_0-1)$. Thus $\mb
b(p) < \mb g_p(p)$ for $0 \le p \le v_0-2$, as desired.

In Iarrobino's setup, Zanello in Rmk.\,2.2 on p.\,183 in \cite{ZanI}
observed that, if $\mb t$ is permissible, then $b\ge s/2$.  Let's see
that his proof yields $b_1\ge s/2$ whenever $\mb b(p) \le \mb t(s)\mb
a(p)$ and $\mb a(p) < \mb a(s-p)$ for $p < s/2$.  Indeed, suppose $b_1 <
s/2$.  Now, \eqref{eqsbImax} yields $\mb g_{b_1}(b_1) \ge \mb t(s)\mb
a(s-b_1)$.  So $ \mb g_{b_1}(b_1) > \mb t(s)\mb a(b_1) \ge \mb b(b_1)$.
So \eqref{eqv0} implies $b_1 \neq v_0$ and also $b_1 \neq v_0-1$ if
$g_{v_0-1}(v_0-1) = b(v_0-1)$.  Hence \eqref{eqdfb1} is contradicted.
Thus $b_1\ge s/2$.

Finally, in Iarrobino's setup, suppose $\mb t$ is permissible but $v=1$.
Then $v=v_0$ by \eqref{leperm}(1) below, and $b_1 \ge s/2$ as shown
above.  But $v_0 \ge b_1$ by \eqref{eqdfb1}.  Thus $1\ge s/2$, so
$s\le2$.  Now, $\mb b(1) > \mb g_1(1)$ by (b) above as $v = 1$, and $\mb
b = \mb a$.  But $\mb g_1(1) = \mb t(1) + \mb t(2)\mb a(1)$ by
\eqref{eqsbImax}.  So $\mb a(1) >\mb t(2)\mb a(1)$.  So $\mb t(2) = 0$.
So $s=1$.  Thus $\mb t$ is level of socle degree 1.  Set $t := \mb
t(1)$. Say the variables of $\mc A$ are $X_1,\dotsc,X_r$.  Since $r =
\mb b(1) > \mb g_1(1) =t$, the ideal
$(X_1,\dotsc,X_t)^2+(X_{t+1},\dotsc,X_r)$ gives a possible $\mc C \in
\mb{H\Lambda}_{\mc A}^{\{\mb h^{\rm I}_m\}}$, as desired.
 \end{sbs}

\begin{lemma}\label{leperm}
 \(1) Assume $\mb t$ is permissible.  Then $v = v_0 \le s+1$, and $v$ is
the initial value of any $\mc I$ with $\mc B_T/\mc I$ in
$\mb{H\Lambda}_{\mc B_T}^{\{\mb h^{\rm I}_m\}}$ for any $T/S$.

\(2) Then $\mb t$ is permissible iff $\mb t(p) = 0$ for $p < b_1$ and
\eqref{sbPerm}\(d) holds.

\(3) Assume $\mb t$ is permissible.  Then so is its attendant $\mb t'$
provided $\mb t'\neq 0$.

\(4) Assume $\mb t$ is permissible.  Then $\mb t$ is quasi-permissible.

\(5) Assume $\mb t$ is permissible.  Then  $\{\mb h^{\rm I}_m\}$ is fit for $\mb t$.

\(6) If $\mb t$ is permissible, then for any $T/S$, each $\smash{\mc C
\in \mb{H\Lambda}_{\mc B_T}^{\{\mb h^{\rm I}_m\}}}$ is of $T$-socle type
$\mb t$.
  \end{lemma}

\begin{proof}
 First, recall $\mb t(p) \ge 0$ for all $p$ from \eqref{se8}.  Hence
\eqref{eqsbImax} yields
 \begin{equation}\label{eqleperm1}
 \mb g_m(p) - \mb g_{m+1}(p) = \mb t(m)\,\mb a(m-p) \ge 0.
 \end{equation}
 for all $m$ and $p$.  Now, recall $\mb a(0) = 1$ from \eqref{se8}.  So
taking $m := p$ in \eqref{eqleperm1} yields
\begin{equation}\label{eqleperm3}
                \mb g_p(p) - \mb g_{p+1}(p) = \mb t(p) \ge 0.
 \end{equation}

 {\bf For (1),} note $v_0 \le v$ owing to \eqref{sbPerm}(b).  Further,
$\mb b(p) > \mb g_p(p)$ for $v_0 \le p \le s$ by hypothesis.  So to show
$v_0 = v \le s+1$, let's show $v-1 \le s$ and $\mb b(v-1) \le \mb
g_{v-1}(v-1)$, because then $v-1 < v_0$, as well as $v_0 \le v$.

Note $\mb t(p) = 0$ for $p < v-1$ by \eqref{sbPerm}(a).  But $\mb t(s)
\neq 0$.  Thus $v-1 \le s$, as desired.

Note that taking $p := v-1$ in \eqref{eqleperm3} yields
 \begin{equation}\label{eqleperm2}
 \mb b(v-1) - \mb g_v(v-1)
  = \bigl(\mb b(v-1) - \mb g_{v-1}(v-1)\bigr) + \mb t(v-1).
 \end{equation}
 But $\mb t(v-1) \ge 0$ by \eqref{se8}.  Hence \eqref{sbPerm}(c) gives
$\mb b(v-1) \le \mb g_{v-1}(v-1)$, as desired.

For use in proving (2) below, note moreover that, if $\mb b(v-1) - \mb
g_{v-1}(v-1) < 0$, then \eqref{eqleperm2} and \eqref{sbPerm}(c) imply
$\mb t(v-1) = 0$.  But $v = v_0$, as was just proved.  Thus
\eqref{eqdfb1} yields that $b_1 = v_0$ implies $\mb t(v-1) = 0$.

Given $\mc B/\mc I \in \mb{H\Lambda}_{\mc B}^{\{\mb h^{\rm I}_m\}}$,
note $(\mc B/\mc I)_p$ is locally free of rank $\mb h^{\rm I}_p(p)$ for
all $p$ by definition; see \eqref{sbMLev}.  But $\mb b(p) \le \mb
g_p(p)$ for $p < v$ as $v = v_0$ by the above; so $\mb b(p) = \mb h^{\rm
I}_p(p)$ for $p < v$ by \eqref{eqsbImaxh}.  Hence $(\mc B/\mc I)_p$ is
of rank $\mb b(p)$ for $p < v$.  Thus $\mc I_p = 0$ for $p < v$.

Finally, $\mb g_v(v) < \mb b(v)$ by \eqref{sbPerm}(b).  So $\mb h^{\rm
I}_v(v) < \mb b(v)$ by \eqref{eqsbImaxh}.  Hence $(\mc B/\mc I)_v$ is of
rank strictly less than $\mb b(v)$.  Thus $\mc I_v \neq 0$.  Thus (1)
holds.

{\bf For (2),} first assume $\mb t$ is permissible.  Note $b_1 \le v_0$
by \eqref{eqdfb1}.  Recall $v_0 = v$ by (1), and its proof notes $\mb
t(v-1) = 0$ if $b_1 = v_0$.  But $\mb t(p) = 0$ for $p < v-1$ by
\eqref{sbPerm}(a).  Thus $\mb t(p) = 0$ for $p < b_1$, whether $b_1 \le
v-1$ or $b _1= v$.

Conversely, assume $\mb t(p) = 0$ for $p < b_1$.  Notice
\eqref{eqsbImax} does not involve $v$.  So set $v := v_0$.  Then
\eqref{eqdfb1} yields \eqref{sbPerm}(a),\,(b).

For \eqref{sbPerm}(c), first assume $b_1 = v-1$.  Then $\mb b(v-1) = \mb
g_{v-1}(v-1)$ by \eqref{eqdfb1}.  Thus \eqref{eqleperm2} reduces
\eqref{sbPerm}(c) to the tautology $\mb t(v-1) = \mb t(v-1)$.

Finally, assume $b_1 = v$.  Then $\mb b(v-1) < \mb g_{v-1}(v-1)$ by
\eqref{eqdfb1}.  But $\mb t(v-1) = 0$ by hypothesis.  Thus
\eqref{eqleperm2} yields \eqref{sbPerm}(c).  Thus (2) holds.

{\bf For (3),} from \eqref{se7}, recall $\mb t'(\bar s) = 0$ and $\mb t'(p)
= \mb t(p)$ for $p\neq \bar s$.

Define $\mb g'_m$ via \eqref{eqsbImax} with $\mb t'$ for $\mb t$.
However, $\mb t(p) \ge 0$ for all $p$ by \eqref{se8}.  Thus $\mb
g'_m(p)\le \mb g_m(p)$ for all $p$.  Define $v'_m$ via \eqref{sbPerm}(d)
with $\mb g'_m$ for $\mb g_m$.  Thus $v'_m \le v_m$.

Define $b'_1$ via \eqref{eqdfb1} with $v'_0$ and $\smash{\mb g'_{v'_0}}$
for $v_0$ and $\mb g_{v_0}$.  Let's prove $b'_1 \le b_1$.  First, if
$b_1 = v_0$, then $b'_1 \le v'_0 \le v_0 =b_1$.  Second, if $b_1 =
v_0-1$ and $b'_1 = v'_0-1$, then $b'_1 \le v'_0 \le v_0 =b_1$.  Third,
assume $b_1 = v_0-1$ and $b'_1 = v'_0$.  Then \eqref{eqdfb1} yields $\mb
b(v_0-1) = \mb g_{v_0 - 1}(v_0-1)$ and $\mb b(v'_0 - 1) < \mb g'_{v'_0 -
1}(v'_0 - 1)$.  But $\mb g'_m(p)\le \mb g_m(p)$ for all $p$.  Hence
$v'_0 \neq v_0$.  But $v'_0 \le v_0$.  So $b'_1 = v'_0 \le v_0-1 =b_1$.
Thus always $b'_1 \le b_1$.

However, $\mb t'(p) = 0$ for $p < b_1$ by (2).  Thus $\mb t'(p) = 0$ for
$p < b_1'$.

Finally, \eqref{sbPerm}(d) holds for $\mb t'$ as plainly $g'_m = g_m$
for $m > \bar s$ and $g'_m = g_{\bar s+1}$ for $m \le \bar s$.  Thus (3)
holds.

{\bf For (4),} note $\mb t(p) = 0$ for $p < b_1$ by (2).  But $\mb
t(\bar s) \neq 0$.  So $\bar s \ge b_1$.  If $\bar s = b_1$ and $b_1 =
v_0-1$, then $\mb b(\bar s) = \mb g_{\bar s}(\bar s)$ by \eqref{eqdfb1}.
Otherwise, $\bar s \ge v_0$.  But $\bar s \le s$.  So then $\mb b(\bar
s) > \mb g_{\bar s}(\bar s)$ by \eqref{eqv0}.  In either case,
\eqref{eqleperm3} gives $\mb b(\bar s) - \mb g_{\bar s+1}(\bar s) \ge
\mb t(\bar s)$.

If $\bar s = s$, then $\mb g_{\bar s+1}(\bar s) = 0$ by \eqref{eqsbImax}.  Thus
$\mb b(s) \ge \mb t(s)$, as required by \eqref{sbRecPer}.

Suppose $\bar s < s$.  Then $\mb t'$ is permissible by (3).  So by
recursion, assume \eqref{sbRecPer}(1).

Given any $T^\ft/S$ and any $\{\mb h^\ft_m\}$ recursively maximal
for $\mb t'$ and $T^\ft$ with $\mb t'$ the attendant to $\mb t$, note
\eqref{sbRecMax}(3) gives $T^\sh/T^\ft$ and $\mc C^\sh \in
\mb{H\Lambda}_{\mc B_{T^\sh}}^{\{\mb h^\ft_m\}}$ of $T^\sh$-socle type
$\mb t'$.

Define $\mb g'_m$ via \eqref{eqsbImax} with $\mb t'$ for $\mb t$.  So
$\mb h^\ft_{\bar s+1}(\bar s) \le \mb g'_{\bar s+1}(\bar s)$ by
\eqref{prhImax} and \eqref{eqsbImaxh}.  But $\mb g'_{\bar s+1}(\bar s) =
\mb g_{\bar s+1}(\bar s)$ by \eqref{eqsbImax}, as $\mb t'(\bar s) = 0$
and $\mb t'(p) = \mb t(p)$ for $p \neq \bar s$ by \eqref{se7}.  Hence
$\mb b(\bar s) - \mb h^\ft_{\bar s+1}(\bar s) \ge \mb b(\bar s) - \mb
g_{\bar s+1}(\bar s)$.  But $\mb b(\bar s) - \mb g_{\bar s+1}(\bar s)
\ge \mb t(\bar s)$ as noted above.  Thus $\mb b(\bar s) - \mb
h^\ft_{\bar s+1}(\bar s) \ge \mb t(\bar s)$, as required by
\eqref{sbRecPer}(2).  Thus $\mb t$ is quasi-permissible.

{\bf For (5),} recall $\mb g_m(p) \ge \mb g_{m+1}(p)$ from
\eqref{eqleperm1}.  Recall $\mb g_{m+1}(p) \ge 0$ from \eqref{se8}.
Thus \eqref{eqsbImaxh} gives $\mb b(p) \ge \mb h^{\rm I}_m(p) \ge \mb
h^{\rm I}_{m+1}(p) \ge 0$; that is, \eqref{se7}(2) holds.

Next, note \eqref{eqsbImax} gives $\mb g_m = 0$ for $m > s$.  So
\eqref{eqsbImaxh} gives $\mb h^{\rm I}_m = 0$ for $m > s$; that is,
\eqref{se7}(3) holds.

It remains to show \eqref{se7}(1), or $\mb h^{\rm I}_p(p) - \mb h^{\rm
I}_{p+1}(p) = \mb t(p)$ for all $p$.

First, suppose $p \ge b_1$.  Let's prove $\mb b(p) \ge \mb g_p(p)$.  If
$p = b_1 = v_0-1$, then $\mb b(p) = \mb g_p(p)$ by \eqref{eqdfb1}.
Otherwise, $p \ge v_0$.  If also $p \le s$, then $\mb b(p) > \mb g_p(p)$
by hypothesis.  Finally, if $p > s$, then $\mb g_p = 0$ by
\eqref{eqsbImax}; so again $\mb b(p) \ge \mb g_p(p)$.

Recall $\mb g_p(p) - \mb g_{p+1}(p) = \mb t(p) \ge 0$ from
\eqref{eqleperm3}.  So $\mb b(p) \ge \mb g_{p+1}(p)$ too.  So
\eqref{eqsbImaxh} gives $\mb h^{\rm I}_p(p) = \mb g_p(p)$ and $\mb
h^{\rm I}_{p+1}(p) = \mb g_{p+1}(p)$.  Thus $\mb h^{\rm I}_p(p) - \mb
h^{\rm I}_{p+1}(p) = \mb t(p)$.

Finally, suppose $p < b_1$.  Then $\mb t(p) = 0$ by (2).  So $\mb g_p(p)
= \mb g_{p+1}(p)$ by \eqref{eqleperm3}.  But $p < b_1 \le v_0$ by
\eqref{eqdfb1}; hence, \eqref{eqsbImaxc} yields $\mb b(p) \le \mb
g_p(p)$.  Hence \eqref{eqsbImaxh} gives $\mb b(p) = \mb h^{\rm I}_p(p) =
\mb h^{\rm I}_{ p+1}(p)$.  Thus $\mb h^{\rm I}_p(p) - \mb h^{\rm
I}_{p+1}(p) = 0 = \mb t(p)$.  Thus (5) holds.

 {\bf Notice (6)} results from (5) and  \eqref{se7}.
 \end{proof}

\begin{lemma}\label{lehsI}
 \(1) Assume \eqref{sbPerm}\(d).  Then $\mb h^{\rm I}_{\bar s}(p) = \mb
g_{\bar s}(p)$ for $p \ge b_1$.

 \(2) Always  $\mb h^{\rm I}_{\bar s}(p) = \mb b(p)$ for  $p < b_1$. 
 \end{lemma}

\begin{proof}
 {\bf  In (1),} if $p = b_1 = v_0-1$, then $\mb b(p) = \mb g_p(p)$ by
\eqref{eqdfb1}.  Otherwise, $p \ge v_0$.  If $v_0 \le p \le s$, then
$\mb b(p) > \mb g_p(p)$ by \eqref{sbPerm}(d) and \eqref{eqv0}.  If $p >
s$, then $ g_p = 0$ by \eqref{eqsbImax}; so $\mb b(p) \ge \mb
g_p(p)$.  But $\mb g_p(p) = \mb g_{\bar s}(p)$ owing to
\eqref{eqsbImax}.  Thus in any case, $\mb b(p) \ge \mb g_{\bar s}(p)$.
Thus \eqref{eqsbImaxh} yields $\mb h^{\rm I}_{\bar s}(p) = \mb g_{\bar
s}(p)$.

{\bf For (2),} note $p \le v_0$ by \eqref{eqdfb1}.  So $ \mb b(p) \le \mb
g_p(p)$ by \eqref{eqv0}.  But again, $\mb g_p(p) = \mb g_{\bar s}(p)$
owing to \eqref{eqsbImax}.  Thus \eqref{eqsbImaxh} yields $\mb h^{\rm
I}_{\bar s}(p) = \mb b(p)$.
 \end{proof}

\begin{proposition}\label{prthenperm}
 Assume $\mb{H\Lambda}_{\mc B_T}^{\{\mb h^{\rm I}_m\}} \neq \emptyset$
for some $T/S$, and $\mc A_1\mc B^\dg_{-p} = \mc B^\dg_{1-p}$ for $p \le
b_1$.  Then $\mb t(p) = 0$ for $p < b_1$, and if \eqref{sbPerm}\(d) holds,
then $\mb t$ is permissible.
 \end{proposition}

\begin{proof}
 Note replacing $T$ by $S$, we may assume $T = S$.  Fix $\mc C \in
\mb{H\Lambda}_{\mc B_T}^{\{\mb h^{\rm I}_m\}}$.

  Given $p \le b_1$ with $\mc B_p = \mc C_p$, let's see $\mc B_{p-1} =
\mc C_{p-1}$ and $\mb t(p-1) = 0$.  Note
 \begin{equation*}\label{eqtp1}
 \mc B^\dg_{1-p} = \mc A_1\mc B^\dg_{-p} = \mc A_1\mc C^*_{-p}
        \subset \mc C^*_{1-p} \subset \mc B^\dg_{1-p}.
 \end{equation*} So $\mc A_1\mc C^*_{-p} = \mc C^*_{1-p} = \mc
B^\dg_{1-p}$.  Thus $\mc B_{p-1} = \mc C_{p-1}$.

Note $\mc A_1\mc C^*_{-p} \subset (F^1\mc A C^*)_{1-p} \subset \mc
C^*_{1-p}$.  So $(F^1\mc A\cdot\mc C^*)_{1-p} = \mc C^*_{1-p}$ as $\mc
A_1\mc C^*_{-p} = \mc C^*_{1-p}$.  But $\mb t^*$ is the Hilbert function
of $\mc C^*\ox_{\mc A}\mc O_S$; that is,
 \begin{equation}\label{eqtp2}
   \mb t^*(q) = \rank\big( \mc C^*_q\big/ (F^1\mc A\cdot\mc C^*)_q\big)
    \text{\quad for all }q.
 \end{equation}
 So $\mb t^*(1-p) = 0$.  Thus $\mb t(p-1) = 0$.

Note $\mc C \in \mb{H\Lambda}_{\mc B}^{\{\mb h^{\rm I}_m\}}$ by
hypothesis, and $\mb h^{\rm I}_{\bar s}(p) = \mb b(p)$ for $p < b_1$ by
\eqref{lehsI}(2).  So $\mc B_p = \mc C_p$ for $p < b_1$.  Thus, by the
above, $\mb t(p-1) = 0$ for $p < b_1$.

It remains to prove $\mb t(b_1-1) = 0$.  There are two cases.  First, if
$b_1 = v_0-1$, then \eqref{eqdfb1} gives $\mb b(b_1) = \mb
g_{b_1}(b_1)$.  But $\mb g_{b_1}(b_1) = \mb g_{\bar s}(b_1)$  owing to
\eqref{eqsbImax}.  Hence  $\mc B_{b_1} = \mc C_{b_1}$.
Thus, by the above, $\mb t(b_1-1) = 0$.

Second, suppose $b_1 = v_0$.  Then \eqref{eqdfb1} gives $\mb b(b_1-1) <
\mb g_{b_1-1}(b_1-1)$.

Proceeding by contradiction, suppose also $\mb t(b_1-1)\neq 0$.  Then
\eqref{eqtp2} gives $(F^1\mc A\cdot\mc C^*)_{1-b_1} \subsetneq \mc
C^*_{1-b_1}$.  But $\mc C^*_{1-b_1}\subset \mc B^\dg_{1-b_1}$.  Thus
$\rank (F^1\mc A\cdot\mc C^*)_{1-b_1} < \mb b(b_1-1)$.

Note $(F^1\mc A\cdot\mc C^*)_{1-b_1} = (\Delta^{b_1}C^*)_{1-b_1}$ owing
to \eqref{eqMlev}.  But $\mc C \in \mb{H\Lambda}_{\mc B}^{\{\mb h^{\rm
I}_m\}}$; so $\rank (\Delta^{b_1}C^*)_{1-b_1} = \mb h^{\rm
I}_{b_1}(b_1-1)$.  Therefore, $\mb h^{\rm I}_{b_1}(b_1-1) < \mb
b(b_1-1)$.  So \eqref{eqsbImaxh} gives $\mb h^{\rm I}_{b_1}(b_1-1) = \mb
g_{b_1}(b_1-1)$.  Thus $\rank (F^1\mc A\cdot\mc C^*)_{1-b_1} = \mb
g_{b_1}(b_1-1)$.

Hence \eqref{eqtp2} gives $\mb g_{b_1}(b_1-1) + \mb t(b_1-1) = \rank \mc
C^*_{1-b_1}$.  But \eqref{eqleperm3} gives $\mb g_{b_1}(b_1-1) + \mb
t(b_1-1) = \mb g_{b_1-1}(b_1-1)$.  So $ \mb g_{b_1-1}(b_1-1) = \rank \mc
C^*_{1-b_1} \le \mb b(b_1-1)$, a contradiction.  Thus $\mb t(b_1-1) =
0$, as desired.

Finally, if \eqref{sbPerm}\(d) holds, then \eqref{leperm}(2) now implies $\mb
t$ is permissible.
 \end{proof}

\begin{proposition}\label{prImax}
 Assume $\mb t$ is permissible.  Fix $T/S$ and $\mc C \in
\mb{H\Lambda}_{\mc B_T}^{\{\mb h^{\rm I}_m\}}$.  Then $\{\mb h^{\rm
I}_m\}$ is recursively maximal for $\mb t$ and $S$, and $\mc C$ is
recursively compressed for $\mb t$.
 \end{proposition}

\begin{proof}
Consider \eqref{sbRecMax}(1).  Assume $\bar s < s$.  Let $\mb t'$ and
$\{(\mb h^{\rm I}_m)'\}$ be the attendants to $\mb t$ and $\{\mb h^{\rm
I}_m\}$ Set $\mc C' := \Lambda^{\bar s+1}\mc C$.  Then $\mc C' \in
\mb{H\Lambda}_{\mc B_T}^{\{(\mb h^{\rm I}_m)'\}}$ by \eqref{prMLev}(3).
Note that $ (\mb h^{\rm I}_m(p))' = \min\{\,\mb g'_m(p),\, \mb b(p)\,\}$
where $\mb g'_m(p) := \ts \sum_{q = m}^s\mb t'(q)\,\mb a(q-p)$.  So
$\{(\mb h^{\rm I}_m)'\}$ is is the I-set of $\mb t'$.  Moreover, $\mb
t'$ is permissible by \eqref{leperm}(3).  So by recursion, $\{(\mb
h^{\rm I}_m)'\}$ is recursively maximal for $\mb t'$; that is,
\eqref{sbRecMax}(1) holds.

As to \eqref{sbRecMax}(2), it holds by \eqref{prhImax} applied to $\mc
C^\di \in \mb{H\Lambda}_{\mc B_{T^\di}}^{\{\mb h^\di_m\}}$.

By hypothesis, \eqref{sbRecMax}(3) holds, as $\mb t$ is the $T$-socle
type of $\mc C$ by \eqref{leperm}(6).  Thus $\{\mb h^{\rm I}_m\}$ is
recursively maximal for $\mb t$ and $S$ by \eqref{sbRecMax}.  Finally,
$\mb t$ is quasi-permissible by \eqref{leperm}(4); so \eqref{thRC} now
yields that $\mb C$ is recursively compressed for $\mb t$.
 \end{proof}

\begin{corollary}\label{coImax}
 Assume $\mb t$ is permissible, and $\mc A$ is the symmetric algebra on
a locally free $\mc O_S$-module.  Assume $\mc B = \mc A$, or in the
level case (namely, if $\bar s = s$), assume $\mc B = \mc A^{\oplus n}$
for some $n \ge 1$.  Then the {\rm I}-set $\{\mb h^{\rm I}_m\}$ for $\mb
t$ is recursively maximal for $\mb t$ and $S$, and there exist a $T/S$
and a $\mc C \in \mb{H\Lambda}_{\mc B_T}^{\{\mb h^{\rm I}_m\}}$ that's
recursively compressed for $\mb t$.  Moreover, let $\mc C^\sh \in
\mb{H\Lambda}_{\mc B_{T^\sh}}^{\{\mb h^\sh_m\}}$ be of $T^\sh$-socle
type $\mb t$ for any $T^\sh/S$ and $\{\mb h^\sh_m\}$; then $\{\mb h^{\rm
I}_m\} = \{\mb h^\sh_m\}$ iff\/ $\mc C^\sh$ is recursively compressed
for $\mb t$.
 \end{corollary}

\begin{proof}
 Pick a scheme point of $S$ and an infinite field $K$ containing its
residue field.  Set $T := \Spec(K)$.  Then, as recalled in
\eqref{sbPerm}, Iarrobino's work yields recursively compressed quotients
of $\mc A_T$ for $\mb t$, and in the level case, Boij's work yields such
quotients of $\mc A_T^{\oplus n}$ for any $n \ge 1$.  Thus
\eqref{prImax} yields the first assertion.

Moreover, the second assertion now holds by \eqref{thRC}(2).
 \end{proof}

\begin{proposition}\label{prCpdArt}
 Fix $\mb h$ and $\mc C \in \mb H\Psi_{\mc B}^{\mb h, \mb t}$.

\(1) Set $t := \mb t(s)$.  If $\mb t$ is permissible and if $\mb b(p)
\le t\mb a(p)$ and $\mb a(p) \le \mb a(s-p)$ for $p \le s/2$, then $v >
s/2$ and $\mb h(p) = \mb b(p)$ for $p \le s/2$.

\(2) If $\mc B = \mc A$, if $\mc C$ is $S$-Gorenstein, and if $\mb h(p)
= \mb a(p)$ for $p \le s/2$, then $\mb h = \mb h^{\rm I}_{\bar s}$ and
$\mb a(p) \le \mb a(s-p)$ for $p \le s/2$.

\(3) Assume $\mc B = \mc A$ and $\mc C^* = \mc Af$ with $f \in
\Gamma(\mc A^\dg_{-s})$.  Define $\vf \: \mc A(s) \to \mc A^\dg$ locally
by $\vf(a) := af$.  Fix $p \le s/2$.  Then the following three
conditions are equivalent:
 $$  \(a)\quad \mb h(p) = \mb a(p);
      \qquad \(b)\quad\vf_{-p}\text{ is surjective;}
      \qquad \(c)\quad \vf_{p-s}\text{ is injective.}   $$
 \end{proposition}

\begin{proof}
{\bf For (1),} assume $\mb t$ is permissible.  Then $\mb h(p) = \mb b(p)$ for
$p \le s/2$ iff $v > s/2$ owing to \eqref{leperm}(1).  So to prove $v >
s/2$, suppose $v \le s/2$.  Now, $g_v(v) \ge t\mb a(s-v)$ by
\eqref{eqsbImax}.  Assume $\mb b(p) \le t\mb a(p)$ and $\mb a(p) \le \mb
a(s-p)$ for $p \le s/2$.  It follows that $g_v(v) \ge \mb b(v)$,
contradicting \eqref{sbPerm}(b).  Thus $v > s/2$.  Thus (1) holds.

{\bf For (2),} assume $\mc B = \mc A$ and $\mc C$ is Gorenstein.  Then
$\mb t(s) =1$ and $\mb t(p) = 0$ for $p \neq s$.  Fix any $p$.  Then
$\mb g_{\bar s}(p) = \mb a(s-p)$ by \eqref{eqsbImax}.  So $\mb h^{\rm
I}_{\bar s}(p) = \min\{\mb a(s-p),\,\mb a(p)\}$ and $\mb h^{\rm I}_{\bar
s}(s-p) = \min\{\mb a(p),\,\mb a(s-p)\}$ by \eqref{eqsbImaxh}.  Thus
$\mb h^{\rm I}_{\bar s}(s-p) = \mb h^{\rm I}_{\bar s}(p) $.

Assume $p \le s/2$ and $\mb h(p) = \mb a(p)$.  Now, $\mc C$ is Gorenstein;
so $\mb h(s-p) = \mb h(p)$ by \eqref{prArt}(B).  And $\mc C$ is a
quotient of $\mc A$; so $\mb h(s-p) \le \mb a(s-p)$.  Thus $\mb a(p) \le
\mb a(s-p)$, as desired.  So $\mb h^{\rm I}_{\bar s}(p) = \min\{\mb
a(s-p),\,\mb a(p)\} = \mb a(p)$.  Thus $\mb h^{\rm I}_{\bar s}(p) = \mb
h(p)$.

If $p \ge s/2$, then $s-p \le s/2$, and so $\mb h^{\rm I}_{\bar s}(p) =
\mb h^{\rm I}_{\bar s}(s-p) = \mb h(s-p) = \mb h(p)$.  Thus $\mb h = \mb
h^{\rm I}_{\bar s}$, as desired.

{\bf For (3),} notice that $\vf$ factors as follows: $\vf \: \mc A(s)
\onto \mc C^* \into \mc A^\dg$. Thus (b) holds iff $(\mc C^*)_{-p} =
(\mc A^*)_{-p}$, hence, iff (a) holds, as $\mc A^\dg/\mc C^*$ is flat.
Finally, (c) holds iff $\mc A(s)_{p-s} = (\mc C^*)_{p-s}$, iff $\mb a(p)
= \mb h(s-p)$.  But $\mc C$ is Gorenstein as $\mc C^*$ is an invertible
$\mc C$-module; so $\mb h(s-p) = \mb h(p)$ by \eqref{prArt}(B).  Thus
(c) holds iff (a) does.
 \end{proof}

\begin{corollary}\label{coPowSum}
 Let $k$ be a field, $R := k[X_1,\dotsc,X_r]$ the polynomial ring.  Set
$\mb r(p) := \binom{p+ r-1}{r-1} =\dim R_p$.  Fix integers $a,\ q,\ s$
with $1\le s \le a$ and $1\le q\le \mb r(\lceil \frac s2\rceil)$.  Fix a
nonzero $g \in R_{a-s}$.  Fix $a_{ij} \in k$ for $1\le i\le q$ and $1\le
j\le r$.  For $p\ge 0$, set
 \begin{align}
 L_i^{[p]} &:= \sum_{p_j\ge0 \text{ and } \sum p_j=p} a_{i1}^{p_1}
      \dotsb a_{ir}^{p_r}X_1^{-p_1}\dotsb X_r^{-p_r}
         \quad\text{for }1\le i\le q,\label{eqPS1}\\
  V_p &:= \bigl\{\, h\in R_p \bigm| h(a_{i1},\dotsc,a_{ir})=0 \text{ for }
      1\le i\le q\, \bigr\}.\label{eqPS2}
 \end{align}
 Assume the $a_{ij}$ are general, meaning: \(a) $\dim V_p = \max\{\,\mb
r(p) - q,\, 0\,\}$ for $p \le s/2$ and $a-s/2 \le p \le a$, and \(b)
$g(a_{i1},\dotsc,a_{ir}) \neq 0$ for all $i$.  Fix nonzero $a_i\in k$
for $1\le i\le q$. Set $f := \sum_{i=1}^qa_iL_i^{[a]}$.  Set $A :=
(Rf)^*$ and $C := (Rgf)^*$; denote their Hilbert functions by $\bf a$
and $\bf h$.  Then $\mb a(p) = \mb h(p) = \min\{\,q,\,\mb r(p)\,\}$ for
$p\le s/2$, and $\mb h = \mb h_{\bar s}^{\rm I}$ where $\mb t(p) = 0$
for $p\neq s$ and $\mb t(p) = 1$ for $p = s$ and so $\bar s =s$.
 \end{corollary}

\begin{proof}
 The construction here is distilled and adapted from
\cite{LNM1721}*{pp.\,9--14}, where the $L_i^{[p]}$ are known as {\it
power sums}.

  Note $X_jL_i^{[p]} = a_{ij}L_i^{[p-1]}$ for $p\ge 1$.  So induction on
$w\le p$ yields the following key formula
(cf.\ \cite{LNM1721}*{(1.1.10), p.\,9}), provided $h$ is a monomial:
 \begin{equation}\label{eqPS3}
 hL_i^{[p]} = h(a_{i1},\dotsc,a_{ir})L_i^{[p-w]}
          \text{\quad for any } h\in R_w \text{ and any } i.
 \end{equation}
 Hence, linearity in $h$ yields \eqref{eqPS3} for any $h\in R_w$  with
$1\le w\le p$.

For all $p\ge0$, set $M_p := \sum_{i=1}^q kL_i^{[p]}$.  Then
$(0:_{R_p}M_p) = V_p$ owing to \eqref{eqPS3} with $w := p$
(cf.~\cite{LNM1721}*{Lem.\,1.15ii, p.\,12}).  But the pairing $R_p\times
R^\dg_p \to k$ is perfect.  So (a) yields $\dim M_p = \min\{\,q,\, \mb
r(p)\,\}$ if $p \le s/2$ or if $a-s/2 \le p \le a$.  Thus if also $q \le
\mb r(p)$, then the $L_i^{[p]}$ are linearly independent
(cf.~\cite{LNM1721}*{Lem.\,1.15iii, p.\,12}).

Note $V_p \subset (0:_{R_p}f)$ for all $p$ owing to \eqref{eqPS3}.
Conversely, assume $p \le s/2$ and $q \le \mb r(a-p)$.  Given $h\in R_p$
with $hf = 0$, note $\sum_i a_ih(a_{i1},\dotsc,a_{ir})L_i^{[a-p]}= 0$
owing to \eqref{eqPS3}.  But the $L_i^{[a-p]}$ are linearly independent
by the above with $a-p$ for $p$; also, $a_i \neq 0$ by hypotheses.
Therefore, $h(a_{i1},\dotsc,a_{ir}) = 0$ for all $i$.  Thus $h\in V_p$.

 By hypothesis, $1\le q\le \mb r(\lceil \frac s2\rceil)$.  But $\lceil
\frac s2 \rceil \le a - \lfloor \frac s2\rfloor \le a-p$, and $\mb r$
is increasing.  Hence $q \le \mb r(a-p)$.  Thus the above yields $V_p =
(0:_{R_p}f)$ for $p \le s/2$ (cf.~\cite{LNM1721}*{Lem.\,1.15iv,
p.\,12}).

Assume $p \le s/2$.  Note $R_p\big/ (0:_{R_p}f) = A_p$.  So by the
above, (a) gives $\mb a(p) = q$.  Thus $\mb a(p) = \min\{\,q,\, \mb
r(p)\,\}$, as desired (cf.~\cite{LNM1721}*{Lem.\,1.17, p.\,13}).

Set $b_i := a_ig(a_{i1},\dotsc,a_{ir})$ for all $i$.  Note that $b_i\neq
0$ by (b) as $a_i\neq 0$.  Now, $gf = \sum_i b_iL_i^{[s]}$ owing to
\eqref{eqPS3}.  Hence, with $gf$ for $f$, the argument above yields $\mb
h(p) = \min\{\,q,\, \mb r(p)\,\}$ for $p \le s/2$, as well.  Thus $\mb
h(p) = \mb a(p)$ for $p \le s/2$.  Note $C$ is $k$-Gorenstein of  $k$-socle
type $s$.  Thus \eqref{prCpdArt}(2) yields $\mb h = \mb h_{\bar s}^{\rm
I}$, as desired.
  \end{proof}

\begin{corollary}\label{cogenAG}
 Fix positive integers $a,\ r,\ s$ with $s\le a$.  Let $\mc R$ be the
symmetric algebra on a locally free $\mc O_S$-module of rank $r$; set
$\mb r(p) := \binom{p+ r-1}{r-1}$.  Set $\bb P := \bb P(\mc R_a)$.  Fix
a $g \in \Gamma(\mc R_{a-s})$, and form the map $\gamma\: \mc R_{s} \to
\mc R_{a}$ of multiplication by $g$; assume that $\gamma$ is injective
and that $\Cok\gamma$ is locally free.  Define $\mb a$ and $\mb h$ by
 \begin{align}
 \mb a(p) &:= \min\{\,\mb r(p),\, \mb r(a-p)\,\} \quad and \label{eqHFs}\\
 \mb h(p) &:= \min\{\,\mb r(p),\, \mb r(s-p)\,\}\quad \text{for all
  }p \notag.
 \end{align}

 Then there's an open subscheme $U$ of $\bb P$ such that \(a) $U$ has a
$K$-point for every infinite field $K$ containing a residue field of
$S$, and \(b) given $T/S$, the $T$-points of $U$ represent the
$T$-Artinian, $T$-Gorenstein quotients $\mc A$ of $\mc R_T$ for which
$\mc C := (g\mc A^*)^*$ is such a quotient of $\mc A$, and the Hilbert
functions of $\mc A$ and $\mc C$ are $\mb a$ and $\mb h$.  Moreover,
$\mb h = \mb h_{\bar s}^{\rm I}$ where $\mb t(p) = 0$ for $p\neq s$ and
$\mb t(p) = 1$ for $p = s$ and so $\bar s =s$.
 \end{corollary}

\begin{proof}
 Set $\Lambda := \bb{HL}_{\mc R}^{\mb a}$.  Recall $\Lambda$ is the
scheme parameterizing the level quotients of $\mc R$ with Hilbert
function $\mb a$; notice that these are the Gorenstein quotients of
socle degree $a$.  Note that \eqref{sbLev} implies $\Lambda$ is a
subscheme of $\bb P$.  So $\Lambda$ is a closed subscheme of some open
subscheme $V$ of $\bb P$.  Let's prove $\Lambda = V$.  Given $\lambda\in
\Lambda$, set $R := \Spec \mc O_{\Lambda,\lambda}$ and $T:= \Spec \mc
O_{V,\lambda}$.  It suffices to prove $T = R$.

Let $f \in (\mc R^\dg_{-a})_T$ represent the inclusion $T\into \bb P$.
For all $p$, form the map $\mu_p\: (\mc R_p)_T \to (\mc R^\dg_{p-a})_T$
of multiplication by $f$; cf.\ \eqref{sbLev}.  As $R\subset \Lambda$ and
as $\mb a$ is defined by \eqref{eqHFs}, it follows that $(\mu_p)_R$ is
surjective for $p\ge a/2$ and that it's injective with free cokernel for
$p\le a/2$.  Hence, so is $\mu_p$ by Nakayama's lemma and by the local
criterion of flatness.  Hence $T\into \bb P$ factors through $R$.  Thus
$T = R$.  Thus $\Lambda$ is an open subscheme of $\bb P$.

Similarly, set $\Lambda' := \bb{HL}_{\mc R}^{\mb h}$; then $\Lambda'$ is
an open subscheme of $\bb P(\mc R_s)$.  Fix any field $K$ containing a
residue field of $S$.  Then by virtue of the work of Iarrobino if $K$ is
infinite and of Boij if $K$ is finite\emdash see \eqref{sbPerm}\emdash
both $\Lambda'$ and $\Lambda$ have $K$-points; also, if $K$ is infinite,
then \eqref{coPowSum} with $s = a$ and $q = \mb r(\lfloor \frac
s2\rfloor)$ provides such points.

Set $\bb V := \bb V(\Cok\gamma)$ and $W := \bb P - \bb V$.  By
hypothesis, $\gamma$ is injective, and $\Cok\gamma$ is locally free.  Thus
$\bb V$ is a vector bundle (of fiber dimension $\mb r(a) - \mb r(s)$),
and $\gamma$ defines a central projection $\pi\: W \onto \bb P(\mc
R_s)$.
 
Set $U := \Lambda\cap \pi^{-1}\Lambda'$.  Both $\Lambda_K$ and
$\pi^{-1}_K\Lambda'_K$ are nonempty open subschemes of $\bb P_K$; so
$U_K$ is one too.  But, over an infinite base field, every nonempty
subscheme of a projective space has a rational point.  So if $K$ is
infinite, then $U_K$ has a $K$-point.  Thus $U$ satisfies (a).
Furthermore, by construction, $U$ satisfies (b).

Finally, \eqref{prCpdArt}(2) immediately yields $\mb h = \mb h_{\bar
s}^{\rm I}$, as desired.
 \end{proof}

\begin{example}\label{exCT} Often, \eqref{prCpdArt} gives a simple way
to see if, given an $S$-Gorenstein $\mc C$ in $\mb H\Psi_{\mc A}^{\mb
h, \mb t}$ for some $\mb h$, necessarily $\mb h = \mb h^{\rm I}_{\bar
s}$.  Here are some examples based on the literature; in all cases, $A$
is a quotient of a polynomial ring in $r$ variables over an
algebraically closed field $k$, and unless otherwise said, $\car(k) = 0$
or $\car(k) > s$.  Furthermore, $C$ is a general Gorenstein quotient of
$A$.

First, Iarrobino and Kanev \cite{LNM1721}*{Thm.\,4.16, p.\,111} take $A$
to be a complete intersection that's either general or monomial, and
take $s \le d-r$ where $d$ is the product of the degrees of the forms
defining $A$.  By deforming $A$ and $C$ so that the forms become
monomials, they conclude separately that $C_p = A_p$ for $p\le s/2$ and
that the Hilbert function of $C$ is $\mb h^{\rm I}_{\bar s}$; however,
owing to \eqref{prCpdArt}, the former implies the later in much greater
generality.

Second, Iarrobino and Kanev \cite{LNM1721}*{Lem.\,6.1, p.\,209} take $A$
to be the homogeneous coordinate ring of a finite subscheme $Z$ of the
affine $(r-1)$-space.  They assume that $Z$ is smooth or that $Z$ is
supported at a single point $p$ and the ideal of $Z$ in $\mc O_p$ is
homogeneous.  They show that, if $s \ge 2\tau$ where $\tau :=
\inf\{\,q\mid \dim_kA_q = \length Z\,\}$, then $C_p = A_p$ for $p\le
s/2$.

Cho and Iarrobino state in \cite{C-IJA241}*{Thm.\,2.1, a., p.\,753} and
prove in \cite{CI-JA366}*{Thm.\,3.3(i), p.\,68} that $Z$ need only be
Gorenstein.  Boij \cite{BoijBLMS31}*{Prp.\,2.4, p.\,12} proved that, if
$Z$ is smooth, then $k$ may be infinite of any characteristic;
furthermore, he specified the meaning of ``general.''

Again, \eqref{prCpdArt} yields $\mb h = \mb h^{\rm I}_{\bar s}$.  Cho
and Iarrobino note $\mb h = \mb h^{\rm I}_{\bar s}$ above Lem.\,2.3 on
p.\,753 in \cite{C-IJA241}, but as proof, they say only that it holds
since $\mb a(p)$ is nondecreasing.

 \end{example}

\section{I-compressed Homogeneous Quotients}

\begin{setup}\label{se9}
 Keep the setup of \eqref{se8}.  So $\{\mb g_m\}$ and $\{\mb h^{\rm
I}_m\}$ are the pre-I-set and I-set of $\mb t$.  Also, $\bar s
:= \inf\,\{\, p \mid \mb t(p) \neq 0\,\}$ and $s := \sup\{\, p \mid \mb
t(p) \neq 0\,\}$.  Finally, recall that the notion of ``permissible
socle type'' is defined in \eqref{sbPerm}.

Assume $\mb t$ is permissible.  Given $T/S$ and a multilevel quotient
$\mc C$ of $\mc B_T$, call $\mc C$ I-{\it compressed\/} if
 \indt{Icomp@I-compressed} its set of Hilbert functions is $\{\mb h^{\rm
I}_m\}$.  If $\mc C$ is homogeneous and I-compressed, then so is
$\Lambda^{\bar s+1}C$, as its set of Hilbert functions is the attendant
to $\{\mb h^{\rm I}_m\}$ by \eqref{prMLev}(3), so plainly it's equal to
the I-set of the attendant $\mb t'$ to $\mb t$.

Note, by definition, a homogeneous $\mc C$ is I-compressed iff $\mc C
\in \mb{H\Lambda}_{\mc B_T}^{\{\mb h^{\rm I}_m\}}$ Thus \eqref{coImax}
yields, under its hypotheses, that $\mc C$ is I-compressed iff it's
recursively compressed for $\mb t$.
  \end{setup}

\begin{proposition}\label{prMB}
 Assume $\mb t$ is permissible, $\mc A$ is a polynomial ring over $\mc
O_S$, and $S = \Spec k$ with $k$ a field.  Set $\mb h(p) := \mb h^{\rm
I}_{\bar s}(p+s)$ for all $p$ and $\mc B^\sh := \bigoplus_{q} \mc
A(q-s)^{\oplus \mb t(q)}$.

\(1) \(Cf.\ Miri \cite{MiriCA21}*{Prp.\,2.8, p.\,2846}) Assume $k$
infinite, $v\ge2$, and $\mc B := \mc A$.  Set $\wt{\mb t}(p) := \mb
t(s-p)$.  Then, in $\mb{H\Psi}_{\mc B^\sh}^{\mb h^*}$, there's $\mc D$
level of type $1$ and of local generator type $\wt{\mb t}$.

\(2) \(Cf.\ Boij \cite{BoijJA226}*{Prp.\,3.5, p.\,368}) Fix $n\ge 1$.
Assume $k$ arbitrary, and $\mc B := \mc A^{\oplus n}$.  Then, in
$\mb{H\Psi}_{\mc B^\sh}^{\mb h^*}$, there's $\mc D$ level of type $n$
with $\mc D^*$ level of type $\mb t(s)$.
 \end{proposition}

\begin{proof}
  As recalled in \eqref{sbPerm}, by Iarrobino's work in the setup of
(1), and by Boij's work in the setup of (2), there's
$\mc C$ in $\mb{H\Lambda}_{\mc B}^{\{\mb h^{\rm I}_m\}}$ of $S$-socle
type $\mb t$.

Set $\mc D := \mc C^*(-s)$.  Then $\mc D^* = \mc C(s)$.  So $\mc D^*$ is
of $S$-socle type $(\wt{\mb t})^*$.  In (2), thus $\mc D^*$ is level of
type $\mb t(s)$.  And in (1), thus $\mc D$ is of local generator type
$\wt{\mb t}$.

Note $\mc C$ is a quotient of $\mc B$, and $\mc B := \mc A^{\oplus n}$,
with $n = 1$ in (1).  Hence $\mc D$ is level of $S$-socle type $n$.
Now, $\mb h^{\rm I}_{\bar s}$ is the Hilbert function of $\mc C$ by the
end of \eqref{se7}.  So $\mb h^*$ is the Hilbert function of $\mc D$.

Note $\mc C$ is of $S$-socle type $\mb t$.  So $\mc D(s)$ is of local
generator type $\mb t^*$.  But $k$ is a field.  So $\mc D(s)$ has a set
of homogeneous generators with $\mb t(p)$ elements in degree $-p$.  Use
them to express $\mc D(s)$ as a homogeneous quotient of $\mc B^\sh(s)$.
Thus $\mc D \in \mb{H\Psi}_{\mc B^\sh}^{\mb h^*}$.
 \end{proof}

\begin{theorem}\label{prSmOp}
  Assume $\mb t$ permissible.  Set $\mb H := \sum_p\mb
t(p)\bigl(\mb b(p) - \mb h^{\rm I}_{\bar s}(p)\bigr)$.

\(1) Then $\bb H\Lambda_{\mc B}^{\{\mb h^{\rm I}_m\}}$ is an $S$-smooth
open subscheme of\/ $\bb H\Psi_{\mc B}^{\mb h_{\bar s}}$, but
possibly empty.

\(2) Then $\bb H\Lambda_{\mc B}^{\{\mb h^{\rm I}_m\}}$ is covered by
open subschemes, each one isomorphic to an open subscheme of the affine
space over $S$ of fiber dimension $\mb H$.
 \end{theorem}

\begin{proof}
 {\bf For (1),} set $\Lambda := \bb H\Lambda_{\mc B}^{\{\mb h^{\rm
I}_m\}}$.  To prove $\Lambda$ is $S$-smooth, use the Infinitesimal
Criterion \cite{EGAIV}*{(17.14.1), p.\,98}.  Given a local $S$-scheme
$T$, a closed subscheme $R$ of $T$, and an $S$-map $\rho\: R \to
\Lambda$, we must lift $\rho$ to an $S$-map $\tau\: T \to \Lambda$.

Say $\rho$ corresponds to $\mc C \in \mb H\Lambda_{\mc B_R}^{\{\mb
h^{\rm I}_m\}}$.  Set $\mc D := \mc C^* \subset \mc B_R^\dg$ and $\mc E
:= \bigoplus_{q\in\Z} \mc A(q)^{\oplus \mb t(q)}$.  Since $\mb t$ is
permissible, $\mc C$ is of $R$-socle type $\mb t$ by \eqref{leperm}(6).
So $\mc D$ is of local generator type $\mb t^*$.  But $R$ is a local
scheme.  So $\mc D$ has a set, $\Gamma$ say, of homogeneous generators
with precisely $\mb t(p)$ elements in degree $-p$.  Use these generators
to define an $\mc A_R$-map $u\: \mc E_R \to \mc B_R^\dg$ of degree $0$
with image $\mc D$.

For each $m$ and $p$, form the induced map $u^m_p\: (\Delta^m \mc
E_R)_{-p} \to (\mc B_R^\dg)_{-p}$.  Notice $\Im(u^m_p) = (\Delta^m \mc
D)_{-p}$.  But $\mc C \in \mb{H\Lambda}_{\mc B_R}^{\{\mb h^{\rm I}_m\}}$; so
$(\Delta^m \mc D)_{-p}$ is $R$-free of rank $\mb h^{\rm I}_m(p)$, and is a
direct summand of $(\mc B_R^\dg)_{-p}$.  Also $(\Delta^m \mc E_R)_{-p}$
is free of rank $\mb g_m(p)$, and $(\mc B_R^\dg)_{-p}$ is free of rank
$\mb b(p)$.  By hypothesis, \eqref{eqsbImaxh} holds.  Thus either $\mb
h^{\rm I}_m(p) = \mb b(p)$ and $u^m_p$ is surjective, or $\mb h^{\rm I}_m(p) = \mb
g_m(p)$ and $u^m_p$ is injective (as $(\Delta^m \mc E_R)_{-p} \onto
\Im(u^m_p)$ is a surjection between free modules of the same rank) and
$\Cok u^m_p$ is $R$-free.

As $R$ is closed in the local scheme $T$, the elements of $\Gamma$ lift
to homogeneous sections of $\mc B_T^\dg$.  Use them to define an $\mc
A_T$-map $v\: \mc E_T \to \mc B_T^\dg$.  For each $m$ and $p$, form the
induced map $v^m_p\: (\Delta^m \mc E_T)_{-p} \to (\mc B_T^\dg)_{-p}$.
Plainly $v^m_p|R = u^m_p$.  Thus, by Nakayama's Lemma, if $u^m_p$ is
surjective, so is $v^m_p$, and by the Local Criterion of Flatness for
finitely generated modules, if $u^m_p$ is injective and $\Cok u^m_p$ is
$R$-free, then $v^m_p$ is injective and $\Cok v^m_p$ is $T$-free.

Set $\mc D' := \Im v \subset \mc B_T^\dg$.  Then $(\Delta^m\mc D')_{-p}
= \Im v^m_p$.  Hence, the conclusion of the preceding paragraph implies
$(\Delta^m \mc D')_{-p}$ is $T$-free of rank $h^{\rm I}_m(p)$.  Moreover, $\mc
(B_T^\dg)_{-p} \big/ (\Delta^m \mc D')_{-p}$ is $T$-free.  But $\mc D'
=\Delta^m\mc D'$ if $m\ll 0$.  So the exact sequence
\begin{equation*}\label{eqSmO1}
 0 \to \mc D'_{-p} \big/ (\Delta^m \mc D')_{-p}
     \to \mc (B_T^\dg)_{-p} \big/ (\Delta^m \mc D')_{-p}
      \to \mc (B_T^\dg)_{-p} \big/ \mc D'_{-p} \to 0
 \end{equation*}
 shows that $\mc D'_{-p} \big/ (\Delta^m \mc D')_{-p}$ is free.  Thus
$\mc D'^* \in \mb H\Lambda_{\mc B_T}^{\{\mb h^{\rm I}_m\}}$ and $\mc D'^*|R =
\mc C$.

Say $\mc D'^*$ corresponds to $\tau\: T \to \Lambda$.  Then $\tau|R =
\rho$.  Thus $\Lambda$ is smooth.

Next, recall $\Lambda$ is a subscheme of $\bb H\Psi_{\mc B}^{\mb
h^{\rm I}_{\bar s}}$ by \eqref{prMLevS}.  So $\Lambda$ is a closed subscheme of
some open subscheme $U$ of $\bb H\Psi_{\mc B}^{\mb h^{\rm I}_{\bar s}}$.  Let's
prove $\Lambda = U$.  Given $\lambda\in \Lambda$, set $R := \Spec \mc
O_{\Lambda,\lambda}$ and $T:= \Spec \mc O_{U,\lambda}$.  It suffices to
prove $T = R$.

Let $\mc C' \in \mb {H\Psi}_{\mc B_T}^{\mb h^{\rm I}_{\bar s}}$ correspond to the
inclusion $T \into \bb H\Psi_{\mc B}^{\mb h^{\rm I}_{\bar s}}$.  It suffices to
prove $\mc C' \in \mb H\Lambda_{\mc B_T}^{\{\mb h^{\rm I}_m\}}$.  Indeed,
then the inclusion $T \into \bb H\Psi_{\mc B}^{\mb h^{\rm I}_{\bar s}}$ factors
through $\Lambda$.  However, $T\cap \Lambda = R$.  Thus $T = R$, as
desired.

Set $\mc C := \mc C'|R$.  Then $\mc C$ corresponds to the inclusion $R
\into \Lambda$ because $\Lambda \subset \bb H\Psi_{\mc B}^{\mb
h^{\rm I}_{\bar s}}$.  Hence $\mc C \in \mb H\Lambda_{\mc B_R}^{\{\mb h^{\rm I}_m\}}$.
Set $\mc D' := \mc C'^*$ and $\mc D := \mc C^*$.

As above, $\mc D$ has a set $\Gamma$ of homogeneous generators with
precisely $\mb t(p)$ elements in degree $-p$.  Note $\mc D' |R = \mc D$.
So the elements of $\Gamma$ lift to homogeneous sections of $\mc D'
\subset \mc B_T^\dg$.  Use them to define an $\mc A_T$-map $v\: \mc E_T
\to \mc B_T^\dg$ whose image is $\mc D'$.

For each $m$ and $p$, form the induced map $v^m_p\: (\Delta^m \mc
E_T)_{-p} \to (\mc B_T^\dg)_{-p}$.  As above, either $\mb h^{\rm I}_m(p) = \mb
b(p)$ and $v^m_p$ is surjective, or $\mb h^{\rm I}_m(p) = \mb g_m(p)$ and
$v^m_p$ is injective and $\Cok v^m_p$ is $T$-free.  But $\Im v^m_p =
(\Delta^m\mc D')_{-p}$.  Thus $\mc C' \in \mb H\Lambda_{\mc
B_T}^{\{\mb h^{\rm I}_m\}}$.   Thus (1) holds.

 {\bf {\bf For (2),}} first note $\bb H\Lambda_{\mc B}^{\{\mb h^{\rm
I}_m\}} = \bb{HL}_{\mc B}^{\{\mb h^{\rm I}_m\}}$ and $\bb H\Lambda_{\mc
B}^{\{(\mb h^{\rm I}_m)'\}} = \bb{HL}_{\mc B}^{\{(\mb h^{\rm I}_m)'\}}$
by \eqref{coMLevS}.  So let's adapt the proof of \eqref{thMax}, and so
also that of \eqref{sbMLSch}; let's use their notation.  Recall $L' :=
S$ if $\bar s = s$, and $L' := \bb{HL}^{\{(\mb h^{\rm I}_m)'\}}_{\mc B}$
if $\bar s < s$.  Note $L'\neq \emptyset$ as $\bb H\Lambda_{\mc
B}^{\{\mb h^{\rm I}_m\}}\neq \emptyset$.  Thus, by recursion, $L'$ is
covered by open subschemes, each one isomorphic to an open subscheme of
an affine space over $S$.

From \eqref{eqMLSch3}, recall $\bb{HL}^{\{\mb h^{\rm I}_m\}}_{\mc B} :=
L_{\mc B'}^{\?{\mb h}}$ where $\mc B' := \Ker(\mc B_{L'}\onto \mc L')$
and where $\?{\mb h} := \mb h^{\rm I}_{\bar s} - \mb h^{\rm I}_{\bar
s+1}$.  By construction, $L_{\mc B'}^{\?{\mb h}}$ is a closed subscheme
of an open subscheme, $U$ say, of $\bb G := \Grass_{\mb t(\bar s)}(\mc
B'_s)$.  Thus, to prove $\bb{HL}^{\{\mb h^{\rm I}_m\}}_{\mc B}$ is
covered by open subschemes, each one isomorphic to an open subscheme of
an affine space over $S$, it suffices to prove $L_{\mc B'}^{\?{\mb h}} =
U$ assuming $ L_{\mc B'}^{\?{\mb h}} \neq \emptyset$.

Let $\mc I$ be the ideal of $L_{\mc B'}^{\?{\mb h}}$ in $\bb G$.  Then
$L_{\mc B'}^{\?{\mb h}} = U$ if $\mc I = 0$ when $L_{\mc B'}^{\?{\mb h}}
\neq \emptyset$.

Note $\mc I$ is $S$-flat, as $L_{\mc B'}^{\?{\mb h}}$ is $S$-flat by
(1).  So $\mc I_T$ is the ideal of $L_{\mc B'_T}^{\?{\mb h}}$ in $U_T$
for any $T/S$ with $L_{\mc B'_T}^{\?{\mb h}} \neq \emptyset$.  Thus by
Nakayama's Lemma, to prove $\mc I = 0$, it suffices to prove $\mc I_T =
0$ when $T := \Spec K$ where $K$ is any residue field of $S$ with
$L_{\mc B'_T}^{\?{\mb h}} \neq \emptyset$.

Note $\Spec K$ is reduced and irreducible.  Note $\mb t$ is
quasi-permissible by \eqref{leperm}(3).  Also note $\{\mb h^{\rm I}_m\}$
is the (unique) recursively maximal set for $\mb t$ and $S$ as follows.
Take a point in $\bb H\Lambda_{\mc B}^{\{\mb h^{\rm I}_m\}}$.  Let $K'$
be an algebraically closure of its residue field, and set $T' := \Spec
K'$.  The map $T'\to S$ corresponds to some $\mc C'$ in
$\mb{H\Lambda}_{\mc B_{T'}}^{\{\mb h^{\rm I}_m\}}$.  Then $\mc C'$ is of
$T'$-socle type $\mb t$ by \eqref{leperm}(4).  Thus $\{\mb h^{\rm
I}_m\}$ is, by \eqref{prImax}, the recursively maximal set.

Thus \eqref{thMax} applies.  It's proof shows $\bb G_T$ is reduced and
irreducible, and contains $L_{\mc B'_T}^{\?{\mb h}}$ as a nonempty open
subscheme.  But $L_{\mc B'_T}^{\?{\mb h}}$ is closed in $U_T$, and
$U_T$ is open in $\bb G_T$.  So $L_{\mc B'_T}^{\?{\mb h}}$ is nonempty,
open and closed in $U_T$, which is reduced and irreducible, So $L_{\mc
B'_T}^{\?{\mb h}} = U_T$.  Thus $\mc I_T=0$, as desired.

Finally, \eqref{thMax} implies $\bb{HL}^{\{\mb h^{\rm I}_m\}}_{\mc B_T}
$ is of fiber dimension $\mb H$.  Thus (2) holds.
 \end{proof}

\begin{proposition}\label{prlift}
 Given $x \in S$, let $K$ be  its residue field $k(x)$.
Then, given $\mc C$ in $\mb H\Lambda_{\mc B_K}^{\{\mb h^{\rm I}_m\}}$ of
$K$-socle type $\mb t$, there's an open neighborhood $T$ of $x$ and a
$\mc C' $ in $\mb H\Lambda_{\mc B_T}^{\{\mb h^{\rm I}_m\}}$ of $T$-socle
type $\mb t$ with $\mc C'_K = \mc C$.
 \end{proposition}

\begin{proof}
 Proceed as in the first part of the proof of \eqref{prSmOp}(1).  Set
$\mc D := \mc C^* \subset \mc B_K^\dg$ and $\mc E := \bigoplus_{q\in\Z}
\mc A(q)^{\oplus \mb t(q)}$.  By hypothesis, $\mc C$ is of $K$-socle
type $\mb t$; so $\mc D$ is of generator type $\mb t^*$.  But $K$ is a
field.  So $\mc D$ has a set of homogeneous generators with precisely
$\mb t(p)$ elements in degree $-p$.  Lift these generators to sections
of $\mc B^\dg$ over some open subscheme $T$ of $S$.  Use these sections
to define an $\mc A_T$-map $v\: \mc E_T \to \mc B_T^\dg$.

For each $m$ and $p$, form the induced map $v^m_p\: (\Delta^m \mc
E_T)_{-p} \to (\mc B_T^\dg)_{-p}$.  Essentially as in the proof of
\eqref{prSmOp}(1), it can be shown that, after $T$ is replaced by a
smaller open subscheme, either $\mb h^{\rm I}_m(p) = \mb b(p)$ and $v^m_p$ is
surjective, or $\mb h^{\rm I}_m(p) = \mb g_m(p)$ and $v^m_p$ is injective and
$\Cok v^m_p$ is $T$-free.

Set $\mc D' := \Im v \subset \mc B_T^\dg$.  Then $(\Delta^m\mc D')_{-p}
= \Im v^m_p$.  So $(\Delta^m \mc D')_{-p}$ is $T$-free of rank $h^{\rm
I}_m(p)$.  Also, $\mc D'_{-p} \big/ (\Delta^m \mc D')_{-p}$ is free.
Set $\mc C' := \mc D'^*$.  Thus $\mc C' \in \mb H\Lambda_{\mc
B_T}^{\{\mb h^{\rm I}_m\}}$ and $\mc C'_K = \mc C$.  By construction,
$\mc D'$ is of generator type $\mb t^*$; so $\mc C'$ is of $T$-socle
type $\mb t$.
 \end{proof}

\begin{example}\label{exNotAll}
 Even if $\mb t$ is permissible, $\{\mb h^{\rm I}_m\}$ is its I-set, $S
= \Spec k$ with $k$ any field, and $\mc A$ is the sheaf on $S$
associated to a polynomial ring $A$, nevertheless $\bb H\Psi_{\mc
A}^{\mb h^{\rm I}_{\bar s}, \mb t} - \bb H\Lambda_{\mc A}^{\{\mb h^{\rm
I}_m\}}$ may be nonempty, even contain $k$-points.  Here's an example.

Fix variables $X$, $Y$; set $A := k[X,Y]$.  So $\mb a(p) = p+1$.  Take
$s := 5$ and $\bar s := 4$; take $\mb t(s) := 1$ and $\mb t(\bar s) :=
1$.  Then the vectors of nonzero values of $\mb h^{\rm I}_{\bar s}$ and
$\mb h^{\rm I}_s$ are
 \begin{equation*}\label{eqNotAll1}
  \mb h^{\rm I}_s\,:\,(1,2,3,3,2,1)
    \and \mb h^{\rm I}_{\bar s}\,:\,(1,2,3,4,3,1).
 \end{equation*}

To see $\mb t$ is permissible, let's use \eqref{leperm}(2).  For $p \le
4$, note that $\mb g_p = \mb g_4$.  Hence $\mb g_p(p) = \mb a(4-p) + \mb
a(5-p) = 11-2p$.  So $\mb a(p) - \mb g_p(p) = 3p - 10$.  Thus $b = 4$.
But $\mb t(p) = 0$ for $p < 4$.  Finally, recall from \eqref{sbPerm}
that \eqref{sbPerm}(d) is automatic in the current setup.  Thus $\mb t$
is permissible.

 As in \eqref{exPolyRg}, view $A^\dg$ as the $k$-span of the Laurent
monomials $X^{-m}Y^{-n}$.  Set
 \begin{equation*}\label{eqNotAll2}
 F := X^{-5}+Y^{-5}  \and  G := X^{-2}Y^{-2}.
 \end{equation*}
Set $D := AF+AG \subset A^\dg$.  It's easy to check that
 \begin{equation*}\label{eqNotAll3}
               D = kF + (kX^{-4} + kY^{-4} + kG)
                     + A^\dg_{-3} + A^\dg_{-2} + A^\dg_{-1}+ k.
  \end{equation*}
 So $D$ has Hilbert function $(\mb h^{\rm I}_{\bar s})^*$.  Thus $D^*
\in \mb H \Psi_{\mc A}^{\mb h^{\rm I}_{\bar s}}$.  Also, $D^*$ is of
$k$-socle type $\mb t$.
 
 Note $\Delta^sD = AF$.  So it's easy to check that
 \begin{equation*}\label{eqNotAll4} \ts
 \Delta^sD = kF + \sum_{i=1}^4 (kX^{-i} + kY^{-i}) + k.
  \end{equation*}
 So $\Delta^sD$ doesn't have $(\mb h^{\rm I}_s)^*$ as Hilbert function.  Thus
$D^*\notin \mb H \Lambda_{\mc A}^{\{\mb h^{\rm I}_m\}}$, as desired.

In passing, note $ \bb H\Lambda_{\mc A}^{\{\mb h^{\rm I}_m\}}$ is smooth
of fiber dimension $1\cdot (5-3) + 1\cdot (6-1)$, or 7, by
\eqref{prSmOp}(2).  In fact, it contains $k$-points.  For instance, set
$F' := F + X^{-4}Y^{-1}$ and $D' := AF'+AG$.  It's easy to check that
the Hilbert functions of $D'$ and $AF'$ are $\mb h^{\rm I}_{\bar s}$ and
$\mb h^{\rm I}_s$.  But $\Delta^sD' = AF'$.  Thus $D'^*$ is in $\mb H
\Lambda_{\mc A}^{\{\mb h^{\rm I}_m\}}$, as desired.
 \end{example}

\begin{remark}\label{reErr}
 Consider Prop.\,3.6 with $V=0$ on p.\,356 of \cite{Iar84} and its dual
formulation, Thm.\,IIB on p.\,351.  Let's correct an error in their
assertions.

For a moment, let $\{\mb h_m\}$ be an arbitrary set of functions, and
let $\{\mb h_m'\}$ and $\mb t'$ be the attendants; see \eqref{se7}.
Then there's a ``forgetful" map
 \begin{equation*}\label{eqErr1}
                \pi\: \bb H\Lambda_{\mc B}^{\{\mb h_m\}}
                         \to \bb H\Lambda_{\mc B}^{\{\mb h'_m\}};
 \end{equation*}
 it is given on $T$-points by mapping a $\mc C \in \mb {H\Lambda}_{\mc
B_T}^{\{\mb h_m\}}$ to $\Lambda^{\bar s+1}\mc C$.  Dually, $\pi$ maps
$\mc C^*$ to $\Delta^{\bar s+1}(\mc C^*)$; that is, $\pi$ ``forgets" the
local minimal generators of $\mc C^*$ of degree $-\bar s$.

Assume $\mb t$ is permissible, and $\{\mb h_m\}$ is the I-set $\{\mb
h^{\rm I}_m\}$.  Then the proof of \eqref{prSmOp}(2) shows $\pi$ is
smooth.

In addition, assume $\mc A$ is the symmetric algebra on a locally free
$\mc O_S$-module.  And assume $\mc B = \mc A$, or in the level case
(namely, if $\bar s = s$), assume $\mc B = \mc A^{\oplus n}$ for some $n
\ge 1$.  Then, applied fiberwise, \eqref{coImax} implies each fiber of
$\bb H\Lambda_{\mc B}^{\{\mb h^{\rm I}_m\}}\big/ S$ is nonempty.  Thus
the proof of \eqref{prSmOp}(2) shows $\pi$ is surjective.

Finally, $\bb H\Lambda_{\mc B}^{\{\mb h^{\rm I}_m\}}$ is an open
subscheme of $\bb H\Psi_{\mc B}^{\mb h^{\rm I}_{\bar s}, \mb t}$ by
\eqref{prSmOp}(1) and \eqref{prLevS}; similarly, $\bb H\Lambda_{\mc
B}^{\{(\mb h^{\rm I}_m)'\}}$ is an open subscheme of $\bb H\Psi_{\mc
B}^{\mb h^{\rm I}_{\bar s+1}, \mb t'}$.  The results in \cite{Iar84}
cited above assert that $\mc C \mapsto \Lambda^{\bar s+1}\mc C$ defines
a surjective map $\Pi\: \bb H\Psi_{\mc B}^{\mb h^{\rm I}_{\bar s}, \mb
t} \onto \bb H\Psi_{\mc B}^{\mb h^{\rm I}_{\bar s+1}, \mb t'}$ when $S$
is the Spec of an infinite field $k$.  This assertion is false; indeed,
$\Pi$ is not defined at the $k$-point representing the module $D^*$ of
\eqref{exNotAll}.  However, the smooth surjection $\pi$ is defined at
the $k$-point representing the module $D'^*$ of \eqref{exNotAll}.

Thus, to correct the error in \cite{Iar84}*{Thm.IIB}, simply replace
$G(E)$ by $\bb H\Lambda_{\mc B}^{\{\mb h^{\rm I}_m\}}$; that is, replace
arbitrary compressed quotients by recursively compressed quotients.
 \end{remark}

\begin{example}\label{exNotAll2}
 Here's another example where $\bb H\Psi_{\mc A}^{\mb h^{\rm I}_{\bar
s}} - \bb H\Lambda_{\mc A}^{\{\mb h^{\rm I}_m\}}$ is nonempty.
However, this time the socle type changes.

Take $S$, $\mc A$, and $A$ as in \eqref{exNotAll}, but $\bar s := s := 5$
and $\mb t(5) = 1$.  Then $\mb g_p(p) = 6-p$; so $\mb g_p(p) \le \mb
a(p)$ for $p \ge 3$.  Thus $\mb t$ is permissible by \eqref{leperm}(2).

Let $\{\mb h^{\rm I}_m\}$ be the I-set.  Then the vector of nonzero
values of $\mb h^{\rm I}_5$ is
\begin{equation*}\label{eqNotAll2a}
 \mb h^{\rm I}_5 \, :\, (1,2,3,3,2,1).
 \end{equation*}
As in \eqref{exPolyRg}, view $A^\dg$ as the $k$-span of the Laurent
monomials $X^{-m}Y^{-n}$.  Set
 \begin{equation*}\label{eqNotAll2b}
 F := X^{-5} + Y^{-5} \and G := X^{-2}Y^{-1}.
 \end{equation*}
Set $D := AF+AG \subset A^\dg$.  It's easy to check that
 \begin{equation*}\label{eqNotAll2c}
        D = kF + (kX^{-4} + kY^{-4}) + (kX^{-3} + kY^{-3} + kG)
                + A^\dg_{-2} + A^\dg_{-1}+ k.
 \end{equation*}
  So $D$ has $(\mb h^{\rm I}_{\bar s})^*$ as Hilbert function. So $D^*
\in \mb H \Psi_{\mc A}^{\mb h^{\rm I}_{\bar s}}$.  But $D$ requires
two generators.  So $D^*$ isn't of $k$-socle type $\mb t$.  Thus
\eqref{leperm}(6) yields $D^*\notin \mb H \Lambda_{\mc A}^{\{\mb
h^{\rm I}_m\}}$, as desired.

In passing, note $\bb H\Lambda_{\mc A}^{\{\mb h^{\rm I}_m\}}$ is smooth
of fiber dimension $1\cdot (6-1)$, or 5, by \eqref{prSmOp}(2).  In fact,
it contains $k$-points.  For instance, set $F' := F + X^{-4}Y^{-1}$ and
$D' := AF'$.  It's easy to check that $D'$ has Hilbert function $\mb
h^{\rm I}_5$.  But $\Delta^mD' = D'$ and $\mb h^{\rm I}_m = \mb h^{\rm
I}_5$ for $m \le 5$.  Thus $D'^*$ is in $\mb H \Lambda_{\mc A}^{\{\mb
h^{\rm I}_m\}}$, as desired.
 \end{example}

\begin{proposition}\label{prIrr}
 Assume $\mb t$ is permissible, and $S$ is irreducible.

\(1) Assume $\bb H\Psi_{\mc B}^{\mb h^{\rm I}_{\bar s}, \mb t}$ is
nonempty.  Then it's irreducible.

\(2) Assume $\bb H\Lambda_{\mc B}^{\{\mb h^{\rm I}_m\}}$ is
nonempty. Then it's an open, dense subscheme of\/ $\bb H\Psi_{\mc B}^{\mb
h^{\rm I}_{\bar s}, \mb t}$.
 \end{proposition}

\begin{proof}
  {\bf For (1),} for $\bar s\le q\le s$, set $P_q := \bb P(\mc B_q)^{\times
\mb t(q)}$ with $P_q := S$ if $\mb t(q) = 0$.  Set $P:= \prod_qP_q$.
For $1\le i\le \mb t(q)$, let $\mc L_{q,i}$ denote the pullback to $P$,
under the projection to the $i$th factor of $P_q$, of $O_{\bb P(\mc
B_q)}(1)$; so $\mc L_{q,i}^* \subset (\mc B_P^\dg)_{-q}$.  Form the
natural multiplication map
 \begin{equation}\label{eqIrr1}\ts
 \mu_p\: \bigoplus_{q,i}(\mc A_P)_{p+q}\ox_{\mc O_P}
       \mc L^*_{q,i} \to (\mc B_P^\dg)_{p} \tqn{for} -s\le p\le -\bar s.
 \end{equation}

Given a finite function $\mb h\: \Z\to\Z$, denote the subscheme where
$\rank(\mu_p) = \mb h^*(p)$ by $L_p\subset P$; that is, an $S$-map $T\to
P$ factors through $L_p$ iff $\Cok((\mu_p)_T)$ is locally free of rank
$\mb b^*(p) -\mb h^*(p)$.  But $\Cok((\mu_p)_T) = (\Cok(\mu_p))_T$ by
right exactness of pullback.  So by the theory of flattening
stratification, $L_p$ exists, although it might be empty.  Set
 \begin{equation*}\label{eqIrr2}\ts
 L^{\mb h} := \bigcap_{p=-s}^{-\bar s} L_p \subset P.
  \end{equation*}

Let $x\in P$ be a scheme point, $K$ its residue field.  Then $x\in
L^{\mb h}$ iff the map of $K$-vector spaces $(\mu_p)_K$ is of rank $\mb
h^*(p)$ for $0\ge p\ge -s$.  Thus as $\mb h$ varies, the $L^{\mb h}$
stratify $P$; that is, set-theoretically their disjoint union
$\coprod_{\mb h}L^{\mb h}$ is $P$.

 By hypothesis, $S$ is irreducible, and each $\mc B_q$ is locally free
of finite rank.  Hence $P$ is irreducible too.  Its generic point lies
in $L^{\mb h^{\rm J}}$ for some $\mb h^{\rm J}$.  Set $L := L^{\mb
h^{\rm J}}$.  Then the support of $L$ is an open subset of $P$, since
$L$ is the intersection of an open subscheme and a closed subscheme, and
the latter's support must be all of $P$.  Thus $L$ is (nonempty and)
irreducible.  Also, by lower semicontinuity of the rank of $\mu_r$, if
$L^{\mb h} \neq \emptyset$ for some $\mb h$, then $\mb h(r) \le
\mb h^{\rm J}(r)$ for all $r$.

Given $x' \in \bb H\Psi_{\mc B}^{\mb h^{\rm I}_{\bar s} \mb t}$, let
$K'$ be its residue field, and $C$ the corresponding quotient of $\mc
B_{K'}$, Then $C^*$ is of (local) generator type $\bf t^*$.  Pick a
minimal set of generators.  For $0\ge p\ge -s$, this set defines a map
$\bigoplus_q(\mc A_{K'})_{p+q}^{\oplus\mb t(q)} \to (\mc B_{K'}^\dg)_p$
of rank $\bigl(\mb h^{\rm I}_{\bar s}\bigr)^*(p)$, and so a $K'$-point
of $L^{\mb h^{\rm I}_{\bar s}}$.  Take $\mb h := \mb h^{\rm I}_{\bar s}$
above.  Thus $\mb h^{\rm I}_{\bar s}(r) \le \mb h^{\rm J}(r)$ for all
$r$.

Set $\mc D:= \Im\bigl(\bigoplus_p(\mu_p)_L\bigr)$. Fix $p$.  Then $\rank
\mc D_p = \bigl(\mb h^{\rm J}\bigr)^*(p)$.  But \eqref{eqIrr1} yields
 \begin{equation*}\label{eqIrr3}
  \ts  \rank \mc D_p
   \le \rank\Bigl(\bigoplus_q(\mc A_\bb L)_{p+q}\ox (\mc U^*_q)_L\Bigr)
    = \mb g^*_{\bar s}(p).
 \end{equation*}
 Also, $\rank \mc D_p \le b^*(p)$.  Thus $\mb h^{\rm J}(-p) \le \mb
h^{\rm I}_{\bar s}(-p)$.  But $\mb h^{\rm J}(-p) \ge \mb h^{\rm I}_{\bar
s}(-p)$ by the above paragraph.  Thus $\mb h^{\rm J} = \mb h^{\rm I}
_{\bar s}$. But $L := L^{\mb h^{\rm J}}$. Thus $L := L^{\mb h^{\rm
I}_{\bar s}}$.

Notice \eqref{eqIrr1} yields a (homogeneous) surjection
 $\tau\:\bigoplus_{q,i}\mc A\ox_{\mc O_L}(\mc L^*_{q,i})_L \onto \mc
D$.  Apply $\bu\ox_{\mc A_L}\mc O_L$ to get the surjection $\bar\tau\:
\bigoplus_{q,i} (\mc L^*_{q,i})_L \onto \mc D\ox_{\mc A_L}\mc O_L$.  
 Fix $q < b_1$.  Then $\mb t(q) = 0$ by \eqref{leperm}(2).  So there are
no $\mc L^*_{q,i}$.  Thus $\bigl(\mc D\ox_{\mc A_L}\mc O_L\bigr)_{-q} =
0$.  Now, fix $q \ge b_1$.  Then $\tau_{-q}$ is an isomorphism; indeed,
its target is locally free of rank $\mb h^{\rm I}_{\bar s}(q)$, whereas
its source is locally free of rank $g_{\bar s}(q)$; however, $g_{\bar
s}(q) = \mb h^{\rm I}_{\bar s}(q)$ by \eqref{leperm}(1).
So $\bar \tau_{-q}$ is an isomorphism.  Thus $\bigl(\mc D\ox_{\mc
A_L}\mc O_L\bigr)_{-q}$ is locally free of rank $\mb t(q)$.  Thus $\mc
D$ is of local generator type $\mb t^*$.

Therefore, $\mc D^*$ defines a map $L \to \bb H\Psi_{\mc B}^{\mb h^{\rm
I}_{\bar s}, \mb t}$ owing to the universal property of $\bb H\Psi_{\mc
B}^{\mb h^{\rm I}_{\bar s}, \mb t}$.  This map is surjective, because,
as shown above, every $x' \in \bb H\Psi_{\mc B}^{\mb h^{\rm I}_{\bar s}
\mb t}$ gives rise to a point of $L$, and plainly the latter maps to
$x'$.  But $L$ is irreducible. Thus $\bb H\Psi_{\mc B}^{\mb h^{\rm
I}_{\bar s}, \mb t}$ is too.  Thus (1) holds.

 {\bf For (2),} recall $\bb H\Lambda_{\mc B}^{\{\mb h^{\rm I}_m\}}$ is
an open subscheme of $\bb H\Psi_{\mc B}^{\mb h^{\rm I}_{\bar s}, \mb t}$
by \eqref{prSmOp}(1).  Thus (1) yields (2).
 \end{proof}

\begin{corollary}\label{coIrrmax}
 Assume $\mb t$ is permissible, and $S$ is irreducible.  Assume $\mc A$
is the symmetric algebra on a locally free $\mc O_S$-module.  Finally,
assume $\mc B = \mc A$, or in the level case (namely, if $\bar s = s$),
assume $\mc B = \mc A^{\oplus n}$ for some $n \ge 1$.

Then $\bb H\Lambda_{\mc B}^{\{\mb h^{\rm I}_m\}}$ is a nonempty, open
subscheme of $\bb H\Psi_{\mc B}^{\mb h^{\rm I}_{\bar s}, \mb t}$, which
is irreducible.
 \end{corollary}

\begin{proof}
 The assertions follow immediately from \eqref{coImax} and
\eqref{prIrr}(1),\,(2).
 \end{proof}

\begin{remark}\label{regenstype}
 Assume $\mb t$ is permissible, and $S$ is irreducible.  Then a minor
modification of the proof of \eqref{prIrr} yields a bit more than
stated.  First, let's see that $P$ parameterizes, in a weak sense, the
quotients of $\mc B$ of socle type bounded by $\mb t$.

Let $x\in P$ be a point, $K$ a field containing its residue field.  Consider
\eqref{eqIrr1}, and set $D := \Im\bigl(\bigoplus_p (\mu_p)_K\bigr)$.
Then $D$ is generated, possibly not minimally, by homogeneous elements
of $\mc B^\dg_K$ of which $\mb t(q)$ are of degree $-q$; that is, the
(local) generator type of $D$ is bounded by $\mb t^*$.  Thus $D^*$ is a
quotient of $\mc B_K$ of socle type bounded by $\mb t$.

Conversely, let $K$ be a field containing a residue field of $S$, and
$C$ a homogeneous quotient of $\mc B_K$ of socle type bounded by $\mb
t$.  Then $C^*$ is generated, possibly not minimally, by homogeneous
elements of $\mc B^\dg_K$ of which $\mb t(q)$ are of degree $-q$.  Order
them to define a map $\bigoplus_q(\mc A_{K})_{p+q}^{\oplus\mb t(q)} \to
(\mc B_{K}^\dg)_p$, and so a $K$-point of $P$ representing $C$.  Thus
$P$ parameterizes the various $C$.  Notice however that different
$K$-points may represent the same $C$.  Furthermore, $P$ carries no
universal flat family.

Second, consider the map $L \to \bb H\Psi_{\mc B}^{\mb h^{\rm I}_{\bar
s}, \mb t}$.  Recall that $L$ is open in $P$, and it parameterizes, in
the same weak sense, the various compressed $C$ of socle type $\mb t$.
Furthermore, as the map is surjective and as $\bb H\Lambda_{\mc
B}^{\{\mb h^{\rm I}_m\}}$ is open in $\bb H\Psi_{\mc B}^{\mb h^{\rm
I}_{\bar s}, \mb t}$ by \eqref{prIrr}(2), the preimage $\Lambda$ of $\bb
H\Lambda_{\mc B}^{\{\mb h^{\rm I}_m\}}$ is open in $L$, and
parameterizes, in the same sense, the various recursively compressed $C$
of socle type $\mb t$; also, $\Lambda$ is nonempty under the hypotheses
of \eqref{coIrrmax}.

Third, consider the locus $T\subset P$ just of the $C$ of socle type
$\mb t$.  Note forming $\Im\bigl(\bigoplus_p\mu_p\bigr)$ needn't commute
with base change, but forming $\mc D^{\mb h} :=
\Im\bigl(\bigoplus_p(\mu_p)_{L^{\mb h}}\bigr)$ does.  For $1\le i\le \mb
t(q)$, form the flatting stratification of $(\mc D^{\mb h}\ox_{\mc
A_P}\mc O_P)_{-q}$, and denote the stratum of rank $\mb t(q)$ by $T^{\mb
h}_q$.  Set $T^{\mb h} :=\bigcap_q T^{\mb h}_q$.  Then $T = \amalg_{\mb
h} T^{\mb h}$.  Thus $T$ is constructible, but possibly not locally
closed.
 \end{remark}

 \section{I-Compressed Filtered Quotients}
 \begin{setup}\label{se10}
   Keep the setup of \eqref{se9}.  So $\{\mb g_m\}$ and
   $\{\mb h^{\rm I}_m\}$ are the pre-I-set and I-set of $\mb t$.  Also,
   $\bar s := \inf\,\{\, p \mid \mb t(p) \neq 0\,\}$ and
   $s := \sup\{\, p \mid \mb t(p) \neq 0\,\}$.  Recall that
   ``I-Compressed'' and``permissible socle type'' and $v_0$ and $b_1$
   and $\beta^{\rm I}_{\bar s}$ are defined in \eqref{se9} and
   \eqref{sbPerm} and \eqref{sbPerm}(d) and \eqref{eqdfb1} and
   \eqref{eqbetaI}.
 \end{setup}

\begin{sbs}[Stability and integrity]\label{sbIS}
  Given a graded $\mc A$-module $\mc C$, call $\mc C$ $n$-{\it stable}
 \indt{nsta@$n$-stable}
  if $\mc A_1\cdot \mc C_{q-1} = \mc C_q$ for $q > n$.  Say that $\mc C$
  has {\it integrity\/} $m$ if $(0:_{\mc C_p}\mc A_1) = 0$ for $p < m$.
 \indt{integrity@integrity $m$}

  For example, suppose $\mc A_1$ generates $\mc A$ over $\mc O_S$.  Then
  $\mc C$ has integrity $m$ iff $\Soc_k(\mc C) \subset F^m\mc C$, since
  $(0:_{\mc C}\mc A_1) = \Soc_k(\mc C)$ as
  $\Soc_k(\mc C) = (0:_{\mc C}F^1\mc A)$.  So for instance, if $\mc A$
  is the symmetric algebra on a locally free $\mc O_S$-module, then
  plainly $\mc A(q)^{\oplus r}$ has integrity $m$ for all $m$, all $q$,
  and all $r$; furthermore, $\mc A^\dg(q)^{\oplus r}$ is $m$-stable for
  all $m$, all $q$, and all $r$ either by direct computation or by the
  following proposition.
\end{sbs}

\begin{proposition}\label{prIntSt}
  Let $\mc C$ be a graded $\mc A$-module, $m\in \Z$.  Assume $\mc C_p$
  for $p\le m$ and $\mc A_1$ are locally free of finite rank.  Then the
  pullback $\mc C_T$ has integrity $m$ for all $T/S$ iff $\mc C^\dg$ is
  $(-m)$-stable.
\end{proposition}

\begin{proof}
  Fix $p<m$.  Define
  $\mu \: \mc C_p \to \Hom_{\mc O_S}(\mc A_1,\,\mc C_{p+1})$ using
  multiplication.  Adjoint associativity gives
  $\Hom_{\mc O_S}(\mc A_1,\,\mc C_{p+1}) = \Hom_{\mc O_S}( \mc A_1
  \ox_{\mc O_S} (\mc C_{p+1})^*, \,\mc O_S)$.  So dualizing $\mu$ gives
  $\mu^* \: \mc A_1 \ox_{\mc O_S} \mc C^\dg_{-p-1} \to \mc C^\dg_{-p}$.

  By the local criterion, every pullback $\mu_T$ is injective iff $\mu$
  is injective and $\Cok \mu$ is flat, so iff $\mu$ is locally split.
  But $\mu$ is locally split iff $\mu^*$ is, iff $\mu^*$ is surjective.
  Also, $\Ker(\mu_T) = (0:_{\mc C_{T,p}}\mc A_{T,1})$.  Thus every
  $\mc C_T$ has integrity $m$ iff $\mc C^\dg$ is $(-m)$-stable.
\end{proof}

\begin{proposition}\label{prIcprsd1}
  Assume $\mb t$ is permissible.  Fix
  $\mc F \in \mb F\Psi_{\mc B}^{\mb h, \mb t}$ for some $\mb h$.

  \(1) Then $\rank(\mc F) \le \beta^{\rm I}_{\bar
  s}$, with equality if $\mb h = \mb h^{\rm I}_{\bar s}$.

  \(2) Assume that $\rank(\mc F) = \beta^{\rm I}_{\bar s}$, that $\mc
  A$ has integrity $s-b_1$, and that $\mc B^\dg$ is
  $(1-b_1)$-stable.  Then $\mb h = \mb h^{\rm I}_{\bar
  s}$ and $G_{\bu}\mc F \in \mb H\Psi_{\mc B}^{\mb h, \mb
  t}$ and
  $G_{\bu}(\mc F^*) \ox_{\mc A}\mc O_S = G_{\bu}(\mc F^* \ox_{\mc A}\mc
  O_S)$.
\end{proposition}

\begin{proof}
  Each assertion holds if it holds on some neighborhood of each point of
  $S$.  Now, $\mb t^*$ is the local generator type of $\mc
  F^*$; so each point of $S$ has a neighborhood on which $\mc
  F^*$ has a set of generators with precisely $\mb
  t(q)$ generators in $F^{-q}\mc F^* - F^{1-q}\mc
  F^*$.  Thus shrinking
  $S$, we may assume these generators exist globally on $S$.

  Set $\mc E := \bigoplus_q \mc A(q)^{\oplus \mb t(q)}$.  The generators
  define a map $\ve\: \mc E \to \mc F^*$; it's surjective and filtered
  of degree 0.  Let
  $\?\ve\: \mc E/F^{1-b_1}\mc E \to \mc F^*/F^{1-b_1}\mc F^*$ be the
  induced map.  It is surjective.  Plainly
  $\rank (\mc E/F^{1-b_1}\mc E) = \sum_{p\ge b_1} \mb g_{\bar s}(p)$
  owing to \eqref{eqsbImax}.  However,
  $\mb g_{\bar s}(p) = \mb h^{\rm I}_{\bar s}(p)$ for $p\ge b_1$ by
  \eqref{lehsI}(1).  Thus
  $\rank(\mc F^* / F^{1-b_1}\mc F^*) \le \sum_{p\ge b_1} \mb h^{\rm
  I}_{\bar s}(p)$.

  On the other hand, $F^{1-b_1}\mc F^* \subset F^{1-b_1}\mc B^\dg$, and
  $F^{1-b_1}\mc B^\dg / F^{1-b_1}\mc F^*$ is flat.  Hence
  $\rank (F^{1-b_1}\mc F^*) \le \sum_{p< b_1} \mb b(p)$.  But
  $\mb b(p) = \mb h^{\rm I}_{\bar s}(p)$ for $p< b_1$ by
  \eqref{lehsI}(2).  Therefore,
  $\rank (F^{1-b_1}\mc F^*) \le \sum_{p< b_1} \mb h^{\rm I}_{\bar
  s}(p)$.  Hence $\rank(\mc F^*) \le \sum_p \mb h^{\rm I}_{\bar s}(p)$.
  Thus $\rank(\mc F) \le \beta^{\rm I}_{\bar s}$.

  Lastly, note $\rank(\mc F) = \rank(G_\bu\mc F)$.  But
  $G_\bu\mc F \in \mb H\Psi_{\mc B}^{\mb h}$ since
  $\mc F \in \mb F\Psi_{\mc B}^{\mb h, \mb t}$; hence,
  $\rank(G_\bu\mc F) = \sum_p \mb h(p)$.  Thus
  $\rank(\mc F) = \beta^{\rm I}_{\bar s}$ if
  $\mb h = \mb h^{\rm I}_{\bar s}$.  Thus (1) holds.

  {\bf For (2),} set
  $\beta^{\rm I}_1 := \sum_{p\ge b_1} \mb h^{\rm I}_{\bar s}(p)$ and
  $\beta^{\rm I}_2 := \sum_{p < b_1} \mb h^{\rm I}_{\bar s}(p)$.  Then
  $\beta^{\rm I}_1 + \beta^{\rm I}_2 = \beta^{\rm I}_{\bar s}$.  And
  $\rank(\mc F^* / F^{1-b_1}\mc F^*) \le \beta^{\rm I}_1$ and
  $\rank (F^{1-b_1}\mc F^*) \le \beta^{\rm I}_2$; see the proof of (1).
  But $\rank(\mc F^*) = \beta^{\rm I}_{\bar s}$ now.  Thus
  $\rank(\mc F^* / F^{1-b_1}\mc F^*) = \beta^{\rm I}_1$ and
  $\rank (F^{1-b_1}\mc F^*) = \beta^{\rm I}_2$.

  Note $\rank(F^{1-b_1}\mc B^\dg) = \beta^{\rm I}_2$ and
  $F^{1-b_1}\mc B^\dg / F^{1-b_1}\mc F^*$ is flat.  So
  $F^{1-b_1}\mc F^* = F^{1-b_1}\mc B^\dg$.  But
  $F^q\mc F^* = \mc F^* \cap F^q\mc B^\dg$ for all $q$; indeed, as
  $\mc F$ has the induced filtration, the quotient map
  $\mc B \onto \mc F$ is strict, and so its dual, the inclusion
  $\mc F^* \into \mc B^*$, is strict too by \eqref{sbSm}.  Thus
  $G_{-p}\mc F^* =\mc B^\dg_{-p}$ for $p < b_1$.  So
  $\mb h(p) = \mb b(p)$ for $p < b$.  Recall
  $\mb b(p) = \mb h^{\rm I}_{\bar s}(p)$ for $p< b_1$ by
  \eqref{lehsI}(2).  Thus $\mb h(p) = \mb h^{\rm I}_{\bar s}(p)$ for
  $p< b_1$.

  By the above,
  $\?\ve\: \mc E/F^{1-b_1}\mc E \onto \mc F^*/F^{1-b_1}\mc F^*$ is a
  surjection of locally free sheaves of the same rank.  So $\?\ve$ is an
  isomorphism.  So its restriction to $F^{-b_1}\mc E / F^{1-b_1}\mc E$
  is injective.  Thus
  $G_{-b_1}\ve \: G_{-b_1} \mc E \to G_{-b_1} \mc F^*$ is injective.

  Proceed by induction: given $p\ge b_1$, assume
  $G_{-p}\ve \: G_{-p} \mc E \to G_{-p} \mc F^*$ is injective.  Given
  $\sigma\in S$ and $y \in G_{-p-1} \mc E_\sigma$ with
  $(G_{-p-1} \ve) y = 0$, take any $x \in \mc A_{1,\sigma}$.  Then
  $(G_{-p} \ve) (xy) = x((G_{-p-1} \ve) y) = 0$.  So $xy = 0$.  But
  $\mc A$ has integrity $s-b_1$, and $p+1 > b_1$.  So $y=0$.  Thus
  $G_{-p}\ve \: G_{-p} \mc E \to G_{-p} \mc F^*$ is injective for all
  $p\ge b_1$.

  By \eqref{sbArt}, the functor $\mc M \mapsto G_\bu\mc M$ is biexact on
  the filtered $\mc O_S$-Artinian modules $\mc M$.  Thus
  $\?\ve\: \mc E / F^{1-b_1}\mc E \to \mc F^* / F^{1-b_1}\mc F^*$ is a
  strict injection.

  But $\?{\ve}$ is surjective. So $\?\ve$ is a strict isomorphism.  Thus
  $G_\bu\?\ve$ is an isomorphism.

  Hence $\mb g_{\bar s}(p) = \mb h(p)$ for $p\ge b_1$.  But
  $\mb g_{\bar s}(p) = \mb h^{\rm I}_{\bar s}(p)$ for $p\ge b_1$ by
  \eqref{lehsI}(1).  Thus $\mb h(p) = \mb h^{\rm I}_{\bar s}(p)$ for
  $p\ge b_1$, so for all $p$ owing to the conclusion above, as desired.

Note $G_{\bu}\mc F \in \mb H\Psi_{\mc B}^{\mb h, \mb t}$ if
$G_{\bu}\mc F ^*$ is of local generator type $\mb t^*$, or $((G_\bu\mc F
^*) \ox_{\mc A}\mc O_S)_{-p}$ is locally free of rank $\mb t(p)$ for all
$p$.  The latter holds for $p \ge b_1$, as $G_\bu\?\ve$ is bijective and
as $\mc E$ is plainly of local generator type $\mb t^*$.  But $\mb t$ is
permissible; so $\mb t(p) = 0$ for $p < b_1$ by \eqref{leperm}.  Thus it
remains to show $((G_\bu\mc F ^*) \ox_{\mc A}\mc O_S)_{-p} = 0$ for $p <
b_1$.

Recall $G_{-p}\mc F^* = \mc B^\dg_{-p}$ for $p < b_1$.  But, by
hypothesis, $\mc B^\dg$ is $(1-b_1)$-stable; that is, $\mc A_1\cdot \mc
B^\dg_{-p} = \mc B^\dg_{1-p}$ for $p < b_1$.  Thus $((G_\bu\mc F^*)
\ox_{\mc A}\mc O_S)_{-p} = 0$ for $p < b_1-1$.

Recall $\?\ve\: \mc E / F^{1-b_1}\mc E \to \mc F^* / F^{1-b_1}\mc F^*$
is injective and $\ve\: \mc E \to \mc F^*$ is surjective.  Hence $\ve$
restricts to a surjection $F^{1-b_1}\mc E \onto F^{1-b_1}\mc F^*$ by the
Snake Lemma.  Thus $(G_{1-b_1}\ve) (\mc E_{1-b_1}) = G_{1-b_1}\mc F^*$.

Recall $\mc E$ is of local generator type $\mb t^*$ and $\mb t(b_1-1) =
0$.  So $(\mc E \ox_{\mc A}\mc O_S)_{1-b_1} = 0$.  Hence
$\bigoplus_{q>0}\mc A_q\cdot \mc E_{1-b_1-q} = G_{1-b_1}\mc E$.
Therefore,
 $$\ts \bigoplus_{q>0}\mc A_q\cdot
         \big((G_{1-b_1-q}\ve) (\mc E_{1-b_1-q})\big)
                = (G_{1-b_1}\ve) (G_{1-b_1}\mc E).$$
  But $(G_{1-b_1-q}\ve) (\mc E_{1-b_1-q})\subset G_{1-b_1-q}\mc F^*$,
and $(G_{1-b_1}\ve) (\mc E_{1-b_1}) = G_{1-b_1}\mc F^*$.  Hence
$\bigoplus_{q>0}\mc A_q\cdot G_{1-b_1-q}\mc F^* = G_{1-b_1}\mc F^*$.  So
$((G_\bu\mc F ^* )\ox_{\mc A}\mc O_S)_{1-b_1} = 0$.  Thus $G_{\bu}\mc F
\in \mb H\Psi_{\mc B}^{\mb h, \mb t}$.

Finally, by \eqref{leDiag} there's a canonical surjection $G_{\bu}(\mc
F^*) \ox_{\mc A}\mc O_S \onto G_{\bu}(\mc F^* \ox_{\mc A}\mc O_S)$.
It's an isomorphism as both its source and target are locally free with
the same Hilbert function, namely $\mb t^*$.  Thus (2) holds.
 \end{proof}

\begin{proposition}\label{prIcprsd2}
   Assume $\mb t$ is permissible.  Given $\mc F \in \mb F\Psi_{\mc
B}^{\mb h,\mb r}$ for some $\mb h$ and $\mb r$, assume $G_{\bu}\mc F \in
\mb H\Psi_{\mc B}^{\mb h^{\rm I}_{\bar s}, \mb t}$.  Then $\mb h =
\mb h^{\rm I}_{\bar s}$ and $\mb r = \mb t$ and $\rank(\mc F) =
\beta^{\rm I}_{\bar s}$ .
 \end{proposition}

\begin{proof}
  As $\mc F \in \mb F\Psi_{\mc B}^{\mb h,\mb r}$ and $G_{\bu}\mc F
\in \mb H\Psi_{\mc B}^{\mb h^{\rm I}_{\bar s}, \mb t}$, the
definitions yield $\mb h = \mb h^{\rm I}_{\bar s}$.

Note that $G_\bu \mc F \in \mb H\Psi_{\mc B}^{\mb h^{\rm I}_{\bar s},
\mb t}$.  Hence $\rank(G_\bu\mc F) = \beta^{\rm I}_{\bar s}$ by
\eqref{prIcprsd1}(1).  However, $\rank(G_\bu \mc F) = \rank(\mc F)$.
Thus $\rank(\mc F) = \beta^{\rm I}_{\bar s}$.

By definition, $\mb r^*$ is the generating function of $G_\bu (\mc F^*
\ox_{\mc A}\mc O_S)$, and $\mb t^*$ is that of $(G_\bu\mc F^*)\ox_{\mc
A}\mc O_S$.  But there's a canonical surjection $G_\bu (\mc F^*)\ox_{\mc
A}\mc O_S \onto G_\bu (\mc F^* \ox_{\mc A}\mc O_S)$ by \eqref{leDiag}.
Thus $\mb r(p) \le \mb t(p)$ for all $p$.  Define $\alpha^{\rm I}$ like
$\beta^{\rm I}_{\bar s}$ in \eqref{eqbetaI} using $\mb r$ in place of
$\mb t$.  Thus $\alpha^{\rm I} \le \beta^{\rm I}_{\bar s}$, and
$\alpha^{\rm I} = \beta^{\rm I}_{\bar s}$ iff $\mb r = \mb t$.  But
$\rank(\mc F) \le \alpha^{\rm I}$ by \eqref{prIcprsd1}(1).  So
$\beta^{\rm I}_{\bar s} = \rank(\mc F) \le \alpha^{\rm I} \le \beta^{\rm
I}_{\bar s}$.  So $\alpha^{\rm I} = \beta^{\rm I}_{\bar s}$.  Thus $\mb
r = \mb t$.
 \end{proof}

\begin{lemma}\label{lePotLift}
 Fix $\mb h$ and $\mc C\in \mb H\Psi_{\mc B}^{\mb h, \mb t}$.  Set
$\mc N := \mc C^*\ox_{\mc A}\mc O_S$ and $\mc W := F^{-s}\mc B^\dg/\mc
C^*$.  Fix an affine $S$-scheme $T$.  Then given a homogeneous $\mc
O_T$-section $\nu\:\mc N_T \into \mc C_T^*$ of $\mc C_T^*\onto \mc N_T$,
and one $\omega \: \mc W_T \into F^{-s}\mc B_T^\dg$ of $F^{-s}\mc
B_T^\dg \onto \mc W$ and given $\mc F \in \mb F\Psi_{\mc B_T}^{\mb
h}$ with $G_\bu \mc F = \mc C_T$, there's a unique $\gamma \in
F^1\Hom_{\mc O_T}(\mc N_T,\,\mc W_T)$ with $\mc F^* := \mc A_T \cdot(\nu
+\omega\gamma)(\mc N_T)$.
 \end{lemma}

\begin{proof}
 Note that $G_\bu \mc F^*$ is a finitely generated and flat $\mc
O_T$-module.  So, as $T$ is affine, each projection $F^p\mc F^* \onto
G_p\mc F^*$ splits over $\mc O_T$; say $F^p\mc F^* = \mc F'_p \oplus
F^{p+1}\mc F^*$.  But $G_\bu \mc F^* = \mc C_T^*$.  Thus $F^p\mc B_T^\dg
= \mc F'_p \oplus \omega(\mc W_T)_p \oplus F^{p+1}\mc B_T^\dg$.

Recall from \eqref{se6} that $F^q\mc B = \mc B$ for all $q\le 0$.  Hence
$F^q\mc B^\dg = 0$ for all $q\ge 1$.  It follows that $F^p\mc F^* =
\bigoplus_{q\ge p} \mc F'_q$ and that $F^p\mc B_T^\dg = F^p\mc F^*
\oplus \omega(F^p\mc W_T)$.  Denote the corresponding projections by
$\vf_p\: F^p\mc B_T^\dg \to F^p\mc F^*$ and $\pi_p\: F^p\mc B_T^\dg \to
F^p\mc W_T$.

Recall that $G_\bu \mc F^* = \mc C_T^*$.  So for every $p$, there's a
corresponding projection $F^p\mc F^* \onto {(\mc C_T^*)}_p$.  Fix a
lifting $\nu'_p \: (\mc N_T)_p \into F^p\mc F^*$ of $\nu_p\: ({\mc
N_T})_p \into (\mc C_T^*)_p$.  Then $\nu'_p - \nu_p \in \Hom_{\mc
O_S}(\mc N_p,\, F^{p+1}\mc B^\dg)$.  Set $\gamma_p := \pi_{p+1}(
\nu'_p - \nu_p) \in \Hom_{\mc O_S}( ({\mc N_T})_p,\,F^{p+1}\mc W_T)$
and $\gamma := \sum \gamma_p \in F^1\Hom_{\mc O_S}( \mc N_T,\,\mc W_T)$.

Note that $\vf_p + \omega_p\pi_p = 1_{F^p\mc B_T^\dg }$.  Therefore,
$\omega_p\gamma_p = (\nu'_p - \nu_p) - \vf_{p+1}(\nu'_p - \nu_p)$
in $\Hom_{\mc O_S}( (\mc N_T)_p ,\, F^{p+1}\mc B_T^\dg)$.  Furthermore,
 $$\nu_p + \omega_p\gamma_p  =  \nu'_p
    - \vf_{p+1}(\nu'_p - \nu_p) \text{\quad in\quad}
    \Hom_{\mc O_S}( (\mc N_T)_p ,\, F^p\mc F^*).$$ 
 Hence $(\nu +\omega\gamma)(\mc N_T) \subset \mc F^*$ and $G_\bu (\nu
+\omega\gamma)(\mc N_T) = \mc N_T$.  But $G_\bu \mc F^* = \mc C_T^*$ and
$\mc A_T\cdot \mc N_T = \mc C_T^*$.  Apply the last paragraph of
\eqref{sbArt}.  Thus $\mc A_T\cdot (\nu +\omega\gamma)(\mc N_T) = \mc
F^*$.

Lastly, given $\gamma^\ft\in \Hom_{\mc O_T}(\mc N_T,\,\mc W_T)$, set
$\delta := \gamma^\ft- \gamma$.  Note $\delta \in \Hom_{\mc O_T}(\mc
N_T,\,\mc W_T)$.  So $\delta(\mc N_T) \subset \mc W_T$.  So
$(\omega\delta)(\mc N_T) \subset \omega(\mc W_T)$.  Now, assume $(\nu
+\omega\gamma^\ft)(\mc N_T) \subset \mc F^*$.  Note $(\nu
+\omega\gamma^\ft) - (\nu +\omega\gamma) = \omega\delta$.  Hence
$(\omega\delta)(\mc N_T) \subset \mc F^*$.  Thus $(\omega\delta)(\mc N_T)
\subset \mc F^*\cap \omega(\mc W_T)$.  But $F^p\mc F^* \cap
\omega(F^p\mc W_T) = 0$ for all $p$.  So $(\omega\delta)(\mc N_T) = 0$.
So $\omega\delta = 0$.  But $\omega$ is injective.  So $\delta = 0$.  Thus
$\gamma^\ft = \gamma$, as desired.
 \end{proof}

\begin{definition}\label{deEff}
 Fix $\mc C\in \mb H\Psi_{\mc B}^{\mb h}$ for some $\mb h$.  Given a
filtered quotient $\mc F$ of $\mc B_T$ for some $T/S$, call $\mc F$ a
{\it lifting\/} of $\mc C$ if $G_\bu \mc F = \mc C_T$ as quotients of
 \index[terminology]{lifting}
$\mc B_T$ (so $\mc F \in \mb F\Psi_{\mc B_T}^{\mb h}$).

Set $\mc W := F^{-s}\mc B^\dg/\mc C^*$ and $\mc N := \mc C^*\ox_{\mc
A}\mc O_S$.  Call $\mc F$ a {\it potential lifting\/} of $\mc C$ if
 \indt{potential lifting@potential lifting of $\mc C$}
there exist homogeneous sections $\nu\:\mc N_T \into \mc C_T^*$ of $\mc
C_T^*\onto \mc N_T$ and $\omega \: \mc W_T \into F^{-s}\mc B_T^\dg$ of
$F^{-s}\mc B_T^\dg \onto \mc W$ and a map $\gamma \in F^1\Hom_{\mc
O_T}(\mc N_T,\,\mc W_T)$ with $\mc F^* := \mc A_T \cdot(\nu
+\omega\gamma)(\mc N_T)$.

Call $\mc C$ {\it effective\/} if any potential lifting $\mc F$ on any
 \indt{effective@effective $\mc C$}
(Noetherian) affine $T$ is, in fact, a lifting.  Call $\mc C$ {\it very
effective\/} if, in addition, $\mc F$ has the same socle type $\mb t$ as
$\mc C$.
 \indt{very effective@very effective $\mc C$}
 \end{definition}

\begin{proposition}\label{prIveryeff}
 Every $\mc C\in \mb H\Psi_{\mc B}^{\mb h^{\rm I}_{\bar s}, \mb t}$
with $\mb t$ permissible is very effective.
 \end{proposition}

\begin{proof}
 Fix a potential lifting $\mc F$ of $\mc C$, and use the setup of
\eqref{deEff}.  Set $\mc E := \mc A\ox_{\mc O_S}\mc N$.  Define $\mu \:
\mc E_T \to \mc B_T^\dg$ by $\mu(a\ox f) := a\cdot (\nu
+\omega\gamma)(f)$.  Then $\mu(\mc E_T) = \mc F^*$.  And $G_\bu\mu\: \mc
E_T \to B_T^\dg$ is given by $G_\bu\mu(a\ox f) := a\cdot \nu(f)$.  But
$\mc A_T\cdot \nu(\mc N_T) = \mc C_T^*$.  Thus $G_{-p}\mu$ restricts to
a surjection $(\mc E_T)_{-p} \onto (\mc C_T^*)_{-p}$ for all $p$; also
$\mc C_T^*\subset G_\bu \mc F^*$, but equality cannot yet be asserted as
it's not obvious $\mu$ is strict.

For all $p$, note $\mc E_{-p} = \bigoplus_{q+r=p}\mc A_{-q}\ox_{\mc
O_S}\mc N_{-r}$.  Thus \eqref{eqsbImax} gives $\rank\mc E_{-p} = \mb
g_{\bar s}(p)$.

For $p\ge b_1$, note $\mb h^{\rm I}_{\bar s}(p) = \mb g_{\bar s}(p)$ by
\eqref{lehsI}(1) as $\mb t$ is permissible.  But $\mc C^* \in \mb
H\Delta_{\mc B^\dg}^{\mb h^*}$ where $\mb h := \mb h^{\rm I}_{\bar
s}$.  So $\mb h^*(-p) = \mb g_{\bar s}(p)$.  Thus $(\mc
E_T)_{-p} \onto (\mc C_T^*)_{-p}$ is an isomorphism.

Form the induced map $\bar\mu\: \mc E_T/ F^{1-b_1}\mc E_T \to\mc B_T^\dg
/ F^{1-b_1}\mc B_T^\dg$.  Then $G_{\bu}\bar\mu$ is injective, because it
factors as $\bigoplus_{p\ge b}(\mc E_T)_{-p} \risom \bigoplus_{p\ge
b}(C^*_T)_{-p} \into \mc B_T^\dg$.  But owing to \eqref{sbArt}, the
functor $\mc M \mapsto G_\bu \mc M$ is biexact on the filtered $\mc
O_T$-Artinian modules $\mc M$.  Thus $\?{\mu}$ is a strict injection.

Note $\bar\mu$ factors as $\mc E_T/ F^{1-b_1}\mc E_T \onto \mc F^*/
F^{1-b_1}\mc F^* \into \mc B_T^\dg / F^{1-b_1}\mc B_T^\dg$ since $\mc
F^*$ has the induced filtration.  So $\bar\mu$ gives a strict
isomorphism $\mc E_T/ F^{1-b_1}\mc E_T \risom \mc F^*/ F^{1-b_1}\mc
F^*$.  But $\?{\mu}$ is a strict injection.  So $G_{\bu}\bar\mu$ gives
an isomorphism $(\mc E_T)_{-p} \risom G_{-p}\mc F^*$ for $p\ge b_1$.
Thus $(\mc C_T^*)_{-p} = G_{-p}\mc F^* \subset (\mc B_T^\dg)_{-p}$ for
$p\ge b_1$.

For $p< b_1$, note $\mb h^{\rm I}_{\bar s}(p) = \mb b(p)$ by
\eqref{lehsI}(2).  But $\mc C^* \in \mb H\Delta_{\mc B^\dg}^{\mb
h^*}$ where $\mb h := \mb h^{\rm I}_{\bar s}$.  So $\mb h^*(-p) = \mb
b(p)$.  But $(\mc C^*_T)_{-p} \subset G_{-p}\mc F^* \subset (\mc
B_T^\dg)_{-p}$.  Further, $\mc B^\dg_{-p} \big/ \mc C^*_{-p}$ is flat,
because $\mc C^* \in \mb H\Delta_{\mc B^\dg}^{\mb h^*}$.  Thus $(\mc
C_T^*)_{-p} = G_{-p}\mc F^*= (\mc B_T^\dg)_{-p}$.  Thus $\mc C_T^* =
G_\bu \mc F^*$.

Finally, form the following commutative diagram, which is induced by
$\mu$:
 \begin{equation}\label{eqFHGH}
 \begin{CD}
 (F^{1-b_1}\mc E_T)\ox_{\mc A_T}\mc O_T  @>>>  \mc E_T\ox_{\mc A_T}\mc O_T
   @>>> \bigl(\mc E_T \big/ F^{1-b_1}\mc E_T\bigr)\ox_{\mc A_T}\mc O_T
    @>>>  0 \\
              @VVV                  @VVV
                                        @VVV \\
 (F^{1-b_1}\mc F^*)\ox_{\mc A_T}\mc O_T   @>>> \mc F^*\ox_{\mc A_T}\mc O_T
   @>>> \bigl( \mc F^* \big/ F^{1-b_1}\mc F^*\bigr) \ox_{\mc A_T}\mc O_T
   @>>> 0
 \end{CD}
 \end{equation}
 Its rows are right exact.  Its vertical maps are surjective; in fact,
the right map is an isomorphism, as the isomorphism $\mc E_T/
F^{1-b_1}\mc E_T \risom \mc F^*/ F^{1-b_1}\mc F^*$ was proved above.

Note $\mc E\ox_{\mc A}\mc O_S = (\mc A\ox_{\mc O_S}\mc N)\ox_{\mc A}\mc
O_S = \mc N\ox_{\mc O_S}\mc O_S = \mc N$.  But $F^{1-b_1}\mc N = 0$ as
$\mc C^*$ has local generator type $\mb t^*$ and as $\mb t$ is
permissible.  In \eqref{eqFHGH}, the upper left map preserves the
filtrations.  So it's 0.  So the lower left map is 0 too, as the left
vertical map is surjective.  Thus \eqref{eqFHGH} yields $\mc N_T \risom
\mc F^*\ox_{\mc A_T} \mc O_T$.  But this isomorphism is strict, as $\mc
E/ F^{1-b_1}\mc E \risom \mc F^*/ F^{1-b_1}\mc F^*$ is strict.  Thus
$\mc N_T \risom G_\bu (\mc F^*\ox_{\mc A_T} \mc O_T)$.  Thus $\mc F$ has
the same socle type $\mb t$ as $\mc C$.  Thus $\mc C$ is very effective.
 \end{proof}

\begin{proposition}\label{prAffBdl}
 Fix $\mb h$.  Set $r_p := \sum_{q<p}(\mb b(q) - \mb h(q))$ and $\mb R :=
\sum_p \mb t(p)r_p$.  Given an effective $\mc C\in \mb H\Psi_{\mc
B_T}^{\mb h, \mb t}$ for some $T/S$, form the associated map $T\to \bb
H\Psi_{\mc B}^{\mb h}$.  Then $T\times_{\bb H\Psi_{\mc B}^{\mb h}}
\bb F\Psi_{\mc B}^{\mb h}$ is an affine-space bundle over $T$ of fiber
dimension $\mb R$.
 \end{proposition}

 \newbox\FHbox \setbox\FHbox=\hbox{$F^1\Hom_{\mc O_R}(\mc N_R,\,\mc W_R)$}
 \newbox\gambox \setbox\gambox=\hbox to \wd\FHbox{\hfil$\gamma$\hfil}
 \begin{proof}
 Set $\bb F := T\times_{\bb H\Psi_{\mc B}^{\mb h}} \bb F\Psi_{\mc
B}^{\mb h}$.  We must prove each point $x \in T$ has a neighborhood
$U_x$ equipped with an isomorphism $\vf_x \: \bb F_{U_x} \risom \bb
A^{\mb R}_{U_x}$ such that, on each overlap $U_{xy} := U_x\cap U_y$, the map
$(\vf_y | U_{xy})(\vf_x | U_{xy}) ^{-1}$ is an affine transformation of
$\bb A^{\mb R}_{U_{xy}}$.  Notice that the cocycle condition automatically
holds.

Set $\mc N := \mc C^*\ox_{\mc A_T} \mc O_T$ and set $\mc W := \mc
B_T^\dg / \mc C_T^*$.  Fix $p\le 0$.  For any $R/T$, set $\mb V_p(R) :=
\Hom_{\mc O_R}\bigl( (\mc N_p)_R,\ (F^{p+1}\mc W)_R\bigr)$.  As $R$
varies, the $\mb V_p(R) $ form a functor; by \cite{EGAI}*{(9.6.1),
p.\,377}, it's representable, say by $\bb V_p / T$.  Also, note $\rank
(\mc N_p) = \mb t(p)$ and $\rank (F^{p+1}\mc W) = r_p$.  Thus $\dim (\bb
V_p / T) =\mb t(p)r_p$.

Set $\mb V(R) := F^1\Hom_{\mc O_R}( \mc N_R,\,\mc W_R)$.  Note $\mb V(R)
= \bigoplus_p \mb V_p(R)$.  Thus, as $R$ varies, the $\mb V(R)$ form a
functor, and it's representable by the fiber product, say $\bb V / T$,
of the $\bb V_p / T$.  Furthermore, $\dim (\bb V / T) = \mb R$.

Given $x \in T$, choose a neighborhood $U_x$ such that $\mc N_{U_x}$ and
$\mc W_{U_x}$ are free over $\mc O_{U_x}$.  Choosing bases yields an
isomorphism $\psi_x \: \bb V_{U_x} \risom \bb A^{\mb R}_{U_x}$.

Also, then there exist homogeneous sections $\nu\:\mc N_{U_x} \into \mc
C_{U_x}^*$ of $\mc C_{U_x}^*\onto \mc N_{U_x}$ and $\omega \: \mc
W_{U_x} \into F^{-s}\mc B_{U_x}^\dg$ of $F^{-s}\mc B_{U_x}^\dg \onto \mc
W_R$.  But $\mc C$ is effective for $T$.  So for any affine $R/U_x$,
there exists a map of sets
 \begin{align*}
 \copy\FHbox &\longrightarrow \big\{\,\mc F \in
  \mb F\Psi_{B_R}^{\mb h}\ \big| \ G_\bu \mc F = \mc C_R \,\big\} \\
  \copy\gambox & \longmapsto  \hphantom{\big\{\,}
            \mc F^* := \mc A_R\cdot(\nu_R +\omega_R\gamma)(\mc N_R).
 \end{align*}
 Owing to \eqref{lePotLift}, this map is bijective.  Plainly, it's
functorial in $R$.  So it defines an isomorphism $\theta_x\: \bb F_{U_x}
\risom \bb V_{U_x}$.  Set $\vf_x := \psi_x\theta_x$, so that $\vf_x \:
\bb F_{U_x} \risom \bb A^{\mb R}_{U_x}$.

The isomorphisms $\psi_x$ and $\theta_x$ depend on the choices of
several $\mc O_{U_x}$-maps, but two different choices of the same map
differ by an appropriate $\mc O_{U_x}$-isomorphism.  Thus although
$\vf_y | U_{xy}$ and $\vf_x | U_{xy}$ needn't be the same isomorphism
from $\bb F_{U_{xy}}$ to $\bb A^{\mb R}_{U_{xy}}$, nevertheless $(\vf_y
| U_{xy})(\vf_x | U_{xy}) ^{-1}$ is an affine transformation of $\bb
A^{\mb R}_{U_{xy}}$.
  \end{proof}

\begin{theorem}\label{thFHdim}
  Set $r_p := \sum_{q<p}(\mb b(q) - \mb h^{\rm I}_{\bar s}(q))$ and $\mb
R := \sum_p \mb t(p)r_p$.  Assume that $\mc A$ has integrity $s-b_1$,
and $\mc B^\dg$ is $(1-b_1)$-stable\emdash for example, that $\mc A$ is
the symmetric algebra on a locally free $\mc O_S$-module and that $\mc B
= \mc A$.  Assume $\mb t$ is permissible and $\bb H\Psi_{\mc B}^{\mb
h^{\rm I}_{\bar s}, \mb t} \neq \emptyset$.  Then $\bb F\Psi_{\mc
B}^{\mb h^{\rm I}_{\bar s}, \mb t}$ is an affine-space bundle over $\bb
H\Psi_{\mc B}^{\mb h^{\rm I}_{\bar s}, \mb t}$ of fiber dimension $\mb
R$, and $\bb F\Lambda_{\mc B}^{\{\mb h^{\rm I}_m\}}$ is one over $\bb
H\Lambda_{\mc B}^{\{\mb h^{\rm I}_m\}}$.
 \end{theorem}

\begin{proof}
 Set $\mb h := \mb h^{\rm I}_{\bar s}$.  Set $T := \bb H\Psi_{\mc
B}^{\mb h, \mb t}$ and $\bb F := T\times_{\bb H\Psi_{\mc B}^{\mb h}} \bb
F\Psi_{\mc B}^{\mb h}$.  Let $\mc C$ be the universal quotient of $\mc
B_T$.  Then $\mc C$ is (very) effective by \eqref{prIveryeff}.  So $\bb
F$ is an affine-space bundle over $T$ of fiber dimension $r$ by
\eqref{prAffBdl}.  Thus we must prove $\bb F = \bb F\Psi_{\mc B}^{\mb h,
\mb t}$.

Fix $R/S$.  An $R$-point of $\bb F$ is given by an $\mc F \in \mb
F\Psi_{\mc B_R}^{\mb h}$ with $G_\bu \mc F \in \mb H\mb \Psi_{\mc
B_R}^{\mb h, \mb t}$.  So $\mc F \in \mb F \Psi_{\mc B_R}^{\mb h, \mb
t}$ by \eqref{prIcprsd2} or \eqref{prIveryeff}.  Thus $\mc F$ gives an
$R$-point of $\bb F\Psi_{\mc B}^{\mb h, \mb t}$.  Conversely, an
$R$-point of $\bb F\Psi_{\mc B}^{\mb h, \mb t}$ is given by an $\mc F
\in  F\Psi_{\mc B_R}^{\mb h, \mb t}$.  Then $G_\bu \mc F \in \mb H\mb
\Psi_{\mc B_R}^{\mb h, \mb t}$ by \eqref{prIcprsd1}(1),\,(2) and
\eqref{sbIS}.  Thus $\bb F$ and $\bb F\Psi_{\mc B}^{\mb h, \mb t}$ have
the same $R$-points, so they're equal as schemes.

The last assertion follows because $F\Lambda_{\mc B}^{\{\mb h^{\rm
I}_m\}} = \bb H\Lambda_{\mc B}^{\{\mb h^{\rm I}_m\}}\times_{\bb
H\Psi_{\mc B}^{\mb h}} \bb F\Psi_{\mc B}^{\mb h}$ essentially by
defintion; see \eqref{sbMLev}.
 \end{proof}

\begin{corollary}\label{coFIrr}
  Assume $\mb t$ is permissible, and $S$ is irreducible.

\(1) Assume $\bb F\Psi_{\mc B}^{\mb h^{\rm I}_{\bar s}, \mb t}$ is nonempty.
Then it is irreducible. 

\(2)  Assume  $\bb F\Lambda_{\mc B}^{\{\mb h^{\rm I}_m\}}$ is nonempty. Then
it is open and dense in $\bb F\Psi_{\mc B}^{\mb h^{\rm I}_{\bar s}, \mb
t}$.
 \end{corollary}

\begin{proof}
  {\bf For (1),} recall that $\bb H\Psi_{\mc B}^{\mb h^{\rm I}_{\bar s},
\mb t}$ is irreducible by \eqref{prIrr}(1).  So $\bb F\Psi_{\mc B}^{\mb
h^{\rm I}_{\bar s}, \mb t}$ is irreducible, as it's an
affine space bundle over $\bb H\Psi_{\mc B}^{\mb h^{\rm I}_{\bar s},
\mb t}$ by \eqref{thFHdim}.  Thus (1) holds

The proof of (2) is similar with \eqref{prIrr}(2) in place of
\eqref{prIrr}(1).
 \end{proof}

\begin{corollary}\label{coFdim}
 Set $\mb F := \sum_p \mb t(p)\big(\sum_{q\le p}(\mb b(q) - \mb h^{\rm
I}_{\bar s}(q))\bigr)$.  Assume that $\mc A$ has integrity $s-b_1$ and
that $\mc B^\dg$ is $(1-b_1)$-stable\emdash for example, that $\mc A$ is
the symmetric algebra on a locally free $\mc O_S$-module and that $\mc B
= \mc A$.  Assume that $\mb t$ is permissible%
.

\(1) Then $\bb F\Lambda_{\mc B}^{\{\mb h^{\rm I}_m\}}$ is covered by
(possibly empty) open subschemes, each one isomorphic to an open
subscheme of the affine space over $S$ of fiber dimension $\mb F$.

\(2) If $S$ is irreducible, then so are $\bb H\Lambda_{\mc B}^{\{\mb
h^{\rm I}_m\}}$ and $\bb F\Lambda_{\mc B}^{\{\mb h^{\rm I}_m\}}$, but
they may be empty.
 \end{corollary}

\begin{proof}
 First, note that $\bb H\Lambda_{\mc B}^{\{\mb h^{\rm I}_m\}}$ is
nonempty iff $\bb F\Lambda_{\mc B}^{\{\mb h^{\rm I}_m\}}$ is so.
Moreover, both (1) and (2) are trivial if these schemes are empty.  So
assume they're nonempty.

 {\bf For (1),} note $\bb H\Lambda_{\mc B}^{\{\mb h^{\rm I}_m\}}$ and
$\bb H\Psi_{\mc B}^{\mb h^{\rm I}_{\bar s}, \mb t}$ are both subschemes
of $\bb H\Psi_{\mc B}^{\mb h^{\rm I}_{\bar s}}$ by \eqref{prMLevS} and
\eqref{prLevS}.  In fact, $\bb H\Lambda_{\mc B}^{\{\mb h^{\rm I}_m\}}
\subset \bb H\Psi_{\mc B}^{\mb h^{\rm I}_{\bar s}, \mb t}$, since the
universal quotient on $\bb H\Lambda_{\mc B}^{\{\mb h^{\rm I}_m\}}$ is of
$S$-socle type $ t$ by \eqref{leperm}(6).  So $\bb F\Lambda_{\mc
B}^{\{\mb h^{\rm I}_m\}}$ is an affine-space bundle over $\bb
H\Lambda_{\mc B}^{\{\mb h^{\rm I}_m\}}$ of fiber dimension $\mb R$ by
\eqref{thFHdim}.  Thus (1) \eqref{prSmOp}(2) yields (1).

{\bf For (2),} note $\{\mb h^{\rm I}_m\}$ is recursively maximal by
\eqref{coImax}.  Assume $S$ is irreducible. Then $\bb H\Lambda_{\mc
B}^{\{\mb h^{\rm I}_m\}}$ is irreducible by \eqref{thMax} applied with
$S_{\rm red}$ for $S$.  Hence $\bb F\Lambda_{\mc B}^{\{\mb h^{\rm
I}_m\}}$ is irreducible, as it's an affine-space bundle over $\bb
H\Lambda_{\mc B}^{\{\mb h^{\rm I}_m\}}$ by \eqref{thFHdim}.  Thus (2)
holds.

  Alternatively, (2) follows immediately from \eqref{prIrr} and
\eqref{coFIrr}.
 \end{proof}

\begin{example}[smoothable quotients, elementary components]
       \label{exElemCpts}
 Fix a polynomial ring $A := k[X_1,\dotsc,X_r]$ with $k$ a field.  Let's
see how \eqref{coFdim} can help in finding examples of $k$-Artinian
quotients $C$ of $A$ that aren't {\it smoothable\/} (deformable to
$k$-etale quotients) and examples of {\it elementary\/} components of
 \indt{smoothable quotient}
 \indt{elementary component} $\Hilb_{\bb A_{k}^r/k}^d$ where $d :=
\dim_kC$ (absolutely irreducible components whose $K$-points for every
field $K\supset k$ are quotients $C'$ of $A_K$ with $\Spec(C')$
absolutely irreducible).  Notably, we are going to find a new {\it
small\/} elementary component (one of dimension
 \indt{small elementary component}
 smaller than $dr$, which is the dimension of the {\it principal
component\/} of $\Hilb_{\bb A_k^r/k}^d$, which parameterizes all
smoothable $C$).
 
To simplify notation, set $\mb h_m := \mb
h^{\rm I}_m$ for all $m$.  Define the {\it translation map}
 \indt{translation map}
\begin{equation}\label{eqTr}\ts
 \theta\:\bb F\Lambda_A^{\{\mb h_m\}}\x\bb A_k^r \to \Hilb_{\bb A_k^r/k}
 \end{equation}
 as follows (cf.\ Jelisiejew's definition \cite{JJ-JLMS2019}*{p.\,258}).
It suffices to define the action of $\theta$ on $K$-points for an
arbitrary $k$-algebra $K$.  A $K$-point of $\bb A_k^r$ is given by a
vector $\mb x := (x_1,\dotsc,x_r)$ with $x_i \in K$.  Define a
$K$-automorphism $\alpha_{\mb x}\: A_K \risom A_K$ by $\alpha_{\mb
x}(X_i) := X_i - x_i$.  Given $A_K/I \in \bb F\Lambda_{A_K}^{\{\mb
h_m\}}$, set $\theta((A_K/I),\,\mb x) := A_K/\alpha_{\mb x}(I)$.  This
assignment is functorial in $K$.  Thus it defines the desired map of
schemes $\theta$.

Note that $\theta$ is injective on $K$-points if $K$ is a field, as $\mb
x = \Supp(\alpha_{\mb x}(I))$ and $I = \alpha_{\mb -x}(\alpha_{\mb x}(I)
)$.  
 Note that the definition of $\theta$ and its injectivity remain valid
if $k$ is replaced by an arbitrary Noetherian ring; this added
generality is used in \eqref{prLift}.

Iarrobino and Emsalem, in their pioneering paper \cite{IE78}, proved the
first version of \eqref{coFdim}.  It treats I-compressed filtered
Gorenstein $C$ over an arbitrary algebraically closed field. In
particular, when the Hilbert vector of these $C$ is $(1,r,r,1)$, then
$\Im(\theta)$ is of dimension $r^3/6 + r^2 + 5r/6 -1$.  But the
principal component of $\Hilb_{\bb A_k^r/k}^{2r+2}$ is of dimension
$r(2r+2)$; so it can't contain $\Im(\theta)$ for $r\ge 8$.  Thus, they
proved (see \cite{IE78}*{p.\,152, mid}) that most such $C$ aren't
smoothable if $r\ge 8$.

Similarly, suppose the socle vector is $(0,0,1,1)$.  Then the nonzero
Hilbert vectors of an I-compressed $C$ are $(1,r,r+1,1)$ and
$(1,r,r,1)$; see \eqref{se9}.  So \eqref{coFdim} yields
 \begin{align*}\label{eq}\ts
 \mb F &= \mb t(3)\Bigl(\Bigl(\tbinom{r+2}{3} -1\Bigr)
             + \Bigl(\tbinom{r+1}{2} -(r+1)\Bigr)\Big)
                  +  \mb t(2)\Bigl(\tbinom{r+1}{2} -(r+1)\Bigr)\\
   &= r^3/6 + 3r^2/2 - 2r/3 - 3.
 \end{align*}
 The principal component is of dimension $(2r+3)r$.  But $\mb F+r >
(2r+3)r$ if $r \ge 7$, and $\mb F+r < (2r+3)r$ if $r\le 6$.  Thus most
$C$ in $\mb F\Lambda_A^{\{h_m\}}$ aren't smoothable if $r\ge 7$.

Iarrobino and Emsalem devoted most of \cite{IE78} to another novel
method, known as the {\it method of  small tangent spaces} or {\it
 \indt{method of  small tangent spaces}
 \indt{method of trivial negative tangents}
trivial negative tangents\/}.  It grew out of Schlessinger's work with
the induced grading on the tangent space $T = \Hom_A(I,C)$ to
$\Hilb_{\bb A_k^r/k}^d$ at a homogeneous quotient $C$ with ideal $I$
(his setting is more general); see p.\,151, top.  Iarrobino and Emsalem
described the method in terms of $T$ and the functor $T^1(\bu)$ of
deformation theory.  In essence, see \cite{IE78}*{p.\,155 f.}, they
showed that, under mild restrictions on $\car(k)$, if $I$ is generated
by several general quadrics and if $T^1(C)_{-1} = 0$, then $T$ is small
in the sense that $\dim_kT_{-1} \le r$ and $T_v = 0$ for $v \le -2$, and
so $C$ belongs to an elementary component.  Unfortunately, they falsely
concluded, due to a computational error, see \cite{LNM1721}*{p.\.221,
bot.}, that there's an elementary component when the Hilbert vector is
$(1,r,r,1)$ for $r = 4,5$.

However, the method itself is sound, and was put in definitive form for
filtered $C$ by Jelisiejew \cite{JJ-JLMS2019}*{Thm.\,4.5, p.\,259 and
Thm.\,4.9, p.\,260}.  At the top of p.\,250, he made this definition:
$C$ has {\it trivial negative tangents\/} if $T^1(C)\big/F^0T^1(C) = 0$.
 \indt{trivial negative tangents}
Moreover, by Cor.\,4.7 on p.\,260, if $C$ is homogeneous, then it has
trivial negative tangents iff $\sum_{v<0} \dim_k T_v = r$ (caution,
there's a typo: $\Hom(I_R,O_R)$ should be $\Hom(I_R,O_R)_{<0}$ in the
statement and in l.\,2 of the proof; cf.\ arXiv:1710.06124v5.)

Further, Thm.\,4.5 asserts that, if $C$ has trivial negative tangents,
then $\theta$ restricts to an isomorphism between neighborhoods of
$(C,0)$ and $C$.  Hence, cf.\ \cite{JJ-JLMS2019}*{ Cor.\,4.6, proof,
p.\.259}, if $C$ is also I-compressed, then \eqref{coFdim} implies $C$
is a simple point of $\Hilb_{\bb A_k^r/k}^d$, so lies on a unique
elementary component, which is of dimension $\mb F+r$.  Conversely, by
\cite{JJ-JLMS2019}*{Thm.\,4.9, p.\,260}, if $\car(k) = 0$ and $C$ is a
simple point of $\Hilb_{\bb A_k^r/k}^d$ lying on an elementary
component, then $C$ has trivial negative tangents.  Versions of these
important results also hold by \eqref{prLift}, which is a self-contained
alternative to \cite{JJ-JLMS2019}, but was greatly inspired by it; these
versions suffice here.

In the case of Hilbert vector $(1,r,r,1)$, Iarrobino and Kanev
\cite{LNM1721}*{Lem.\,6.21, p.\,221} asserted that an elementary
component exists for $6 \le r \le 12$, but $r \neq 7$.  To prove it,
they reformulated the condition to have trivial negative tangents in
terms of the Hilbert function of $A/I^2$; this reformulation is valid
essentially because $\Hom_A(I,C) = (I/I^2)^*(-3)$ as $C$ is Gorenstein.
Then they used Macaulay 2 to verify that, over several finite fields,
specific $C$ have trivial negative tangents; they said that therefore
they ``consider'' the existence of such $C$ true when $\car(k) = 0$.

Recently, Szafarczyk \cite{RSz} about settled the case of $(1,r,r,1)$.
He showed that a certain well-chosen $C$ has a small tangent space.  He
did so in Section 3 for $r\ge 18$ and arbitrary $k$ via a lengthy
elementary hand-calculation.  He did so in Section~4 with Macaulay~2 for
$6 \le r \le 17$, but $r \neq 7$, and $k = \bb Q,\ \bb Z/2,\ \bb Z/3$.
He concluded in Proposition 2.20 on p.\,7 that this $C$ is a simple
point of an elementary component.

Independently, the present authors used Macaulay 2 to show that, for
$8\le r \le 19$ and $k = \bb Z/101$, a randomly generated homogeneous
$C$ has trivial negative tangents, so is a simple point of $\Hilb_{\bb
A_k^r/k}^d$ and lies on a unique elementary component, which is of
dimension $\mb E := r^3/6 + r^2 + 5r/6 -1$.
  Hence, over $\bb Q$ too, for $8\le r \le 19$, there exists a
generically smooth elementary component of dimension $\mb E$, by virtue
of \eqref{prLift} with $k := \bb Z$.

  Notice that it suffices to show $\dim T \le \mb E$ in one positive
characteristic.  Notice also that \eqref{prLift} provides a general
method for establishing the existence of an elementary component with
given permissible socle type in characteristic 0 by producing one
example of a $C$ with trivial negative tangents in positive
characteristic (and its base field needn't be prime or even finite).
Thus \eqref{prLift} illustrates an advantage gained from having
\eqref{coFdim} with a base $S$ that's not the Spec of a field.

In the case of socle vector $(0,0,a,1)$, the authors' computations with
Macaulay 2 have yielded randomly generated I-compressed homogeneous $C$
with trivial negative tangents for $r = 5,7,8,9$ and $a=1$, for $r = 6,
7, 8$ and $2 \le a \le 5$, for $r = 7, 8$ and $6 \le a \le 8$, and for
$r = 8$ and $9 \le a \le 10$, all over the prime field $k$ of
characteristic 101 and several others (31, 701, 3001) for good measure
when $a=1$.  Similarly, they found such $C$ with Hilbert vector
$(1,r,2r,2)$ for $r = 6,7,8$ and $(1,r,2r+1,2)$ for $r = 7,8$ in
characteristic 101.  More precisely, the computations show that $\dim
T_{-1} = r$ and $T_{-2} = 0$ and that the ideal $I$ of $C$ is generated
by quadrics, say $n$ in number, giving a surjection $A(-2)^{\oplus n}
\onto I$; but $T := \Hom_A(I,C)$, so $T\into C(2)^{\oplus n}$, thus $T_v
= 0$ for $v < -2$.

Thus in those cases, each $C$ lies in the smooth locus of a unique
elementary component, which is of dimension $\mb F + r$.
 Hence, \eqref{prLift} yields, in each case, a generically smooth
elementary component of dimension $\mb F + r$ over $\bb Q$.
Independently, in infinitely many related cases, Satriano and Staal
\cite{SS2}*{Thm.\,1.3, p.\,3} produced on 25 Oct.\ 2022 a generically
smooth elementary component over an algebraically closed field of
characteristic 0 by conceptual means.

The case of socle vector $(0,0,1,1)$ and $r = 5$ is of particular
interest.  Indeed, $\mb F+r = 57$ and $(2r+3)r = 65$; thus the component
is small.  But it's not in Jelisiejew's \cite{JJ-JLMS2019}*{Thm.\,1.4,
p.\,251, and Ex.\,6.9 and Rmk.\ 6.10, p.\.269}, which include all small
elementary components known in 2019; also, it's not in Satriano and
Staal's preprint \cite{SS1} of 2 Dec.\,2021.  For $r = 5$, moreover,
Macaulay 2 worked over $\bb Q$, thus corroborating \eqref{prLift}.

In the case of socle vector $(0,0,1,1)$ and $r = 6$ and $k = \bb Q$,
most I-compressed $C$ are smoothable, even though this case is
sandwiched between the cases $r = 5$ and $r =7$, where, as noted above,
most aren't.  This smoothability was just proved via a Macaulay 2
calculation by Jelisiejew, who kindly informed the authors about it in
an email.  Thus the situation is analogous to that of Gorenstein $C$
with Hilbert vector $(1,r,r,1)$.  For $r=6$ and $r = 8$ most such $C$
aren't smoothable as noted above.  But, for $r = 7$, all such $C$ are
smoothable by Bertone, Cioffi, and Roggero's \cite{BCR}*{Thm.\,3.7};
also, Jelisiejew said he had recovered this result via a Macaulay~2
calculation like the one for $(1,6,7,1)$.

The authors found no small elementary component among the examples with
socle vector $(0,0,a,1)$ computed above with $a \ge 1$ and $r \ge 6$.
Similarly, they found none for $(0,0,0,2)$ and $(0,0,1,2)$.
Furthermore, the computations suggest that a general enough $C$ with
Hilbert vector $(1,r,br+a,b)$ lies in a unique elementary component for
any $a \ge 0$, any $b\ge1$, and any $r\gg0$.

Iarrobino and Kanev conjectured in \cite{LNM1721}*{Conj.\,6.30, p.\,226}
that a general homogeneous Gorenstein $C$ of odd socle degree $s$ has
trivial negative tangents if $r=4$ and $s\ge 15$, if $r=5$ and $s\ge 5$,
or if $r\ge6$ and $s\ge 3$ but $(r,s) \neq (7,3)$.  The case $s=3$ was
proved by Szafarczyk, as noted above.  So the authors used Macaulay 2 to
show that, in the 18 cases $r \ge 5$ and $r + s \le 16$ and $(s,r) =
(5,12),\ (11, 6)$ when $k = \bb Z/101$, a randomly generated Gorenstein
$C$ is I-compressed and has trivial negative tangents.  Hence by
\eqref{prLift}(5)--(6), the conjecture holds in these cases when $k =
\bb Z/101$ and also when $k = \bb Q$.

If $s$ is even and $C$ is Gorenstein and $I$-compressed, it's not hard
to see, by looking at $A/I^2$ and noting $(I^2)_{s+1} = 0$, that $\dim
T_{-1} = \dim A_{s+1}$, and so $C$ never has trivial negative tangents.
Nevertheless, if $s = 4$ and $r \ge 8$, or $s = 6$ and $r \ge 5$, or $s
= 8$ and $r \ge 5$, or $s \ge 10$ and $r \ge 4$, then $\mb F + r > rd$,
because $\mb F = \binom{r+s}{r} - d$ and $d = 2\binom{r+n-1}{r}+
\binom{r+n-1}{r-1}$; thus then a general $C$ isn't smoothable.
 \end{example}

\begin{proposition}\label{prLift}
 Let $A$ be a polynomial ring in $r$ variables over a Noetherian ring
$k$, and $K$ a field containing a residue field $R$ of $k$.  Set $\mb
h_m := \mb h^{\rm I}_m$ for all $m$.  Set $\bb P := \bb
F\Lambda_A^{\{\mb h_m\}}\x_k\bb A_k^r$ and $\bb H := \Hilb_{\bb
A_k^r/k}$.  Let $\theta\: \bb P \to \bb H$ be the translation map
\eqref{eqTr}.  Set $\mb F := \sum_p \mb t(p)\big(\sum_{q\le
p}\bigl(\binom{q+r-1}{r-1} - \mb h_{\bar s}(q)\bigr) \bigr)$.  Assume
$\mb t$ is permissible.

Fix $C \in \mb F\Lambda_{A_K}^{\{\mb h_m\}}$.  Let $T$ be the tangent
space at the $K$-point $C$ to $\bb H_K$.  Let $Z$ be the
scheme-theoretic image of $\theta$ in $\bb H_R$; let $Z\sm$ be the
smooth locus of $Z$, and $Z\0\sm$ the intersection of $Z\sm$ with the
smooth locus of $\bb H_R$.  Given a subscheme $G$ of a Noetherian scheme
$H$, call $G$ a {\rm component} if $Z = \Spec(\mc O_H/\mc Q)$
where $\mc Q$ is the primary ideal of a minimal prime in the primary
decomposition of $(0)$.

\(1) Then $Z$ is absolutely integral of dimension $\mb F+r$.  Moreover,
$Z$ is a component, necessarily elementary, iff $Z$ has a point at which
$\bb H_R$ is smooth of dimension $\mb F+r$.

\(2) If either $\dim_K(T/F^0T) \le r$ or $\dim_K(T) \le \mb F+r$, then
both are equalities.

\(3) Assume $C$ has trivial negative tangents.  Then $\dim_K(T/F^0T)\le
r$.

\(4) Assume $\dim_K(T/F^0T) \le r$.  Then $\bb H_R$ is smooth of dimension
$\mb F+r$ at $C$.

\(5) Assume either \(a) $C$ has trivial negative tangents or \(b)
$\car(K) = 0$ and $Z$ is a component of\/ $\bb H_R$.  Then $\theta$
restricts to an isomorphism $\theta^{-1}Z\sm \risom Z\sm$.  Moreover, a
$K$-point $C'$ of $Z$ has trivial negative tangents iff $C'$ lies in
$Z\0\sm$.

\(6) Assume that $k$ is a Dedekind domain with fraction field $L$ and
that $Z$ is a component of\/ $\bb H_R$.  Then the scheme-theoretic image
$Y$ of\/ $\theta_L$ is a component of\/$\bb H_L$.
 \end{proposition}

\begin{proof}
  {\bf For (1),} note that $\bb F\Lambda_{A_R}^{\{\mb h_m\}}$ is smooth
and absolutely irreducible of dimension $\mb F$ owing to \eqref{coFdim}.
So $\bb P_R$ is smooth and absolutely irreducible of dimension $\mb
F+r$.  Thus $Z$ is absolutely integral.  But $\theta$ is injective on
field-valued points; hence $\theta$ cannot have a fiber of positive
dimension.  Thus $\dim Z =\mb F+r$.

Suppose $Z$ is a component.  Then necessarily it’s elementary, as the
$K'$-points of $\bb P$ for every algebraically closed field $K'\supset
k$ are the quotients $C'$ of $A_{K'}$ such that $\Spec(C')$ is supported
at a $K'$-point, so is absolutely irreducible.  Moreover, $Z$ contains a
nonempty open subscheme $U$ that's open in $\bb H_R$.  Since $Z$ is
integral of dimension $\mb F+r$, so is $U$.  Thus $U$ contains a point
at which $\bb H_R$ is smooth of dimension $\mb F+r$.

Conversely, suppose $Z$ contains a scheme point $z$ at which $\bb H_R$
is smooth of dimension $\mb F+r$.  Then $z$ lies in a unique integral
component $Z'$ of $\bb H_R$ of dimension $\mb F+r$.  But $Z$ is
integral; so $Z$ lies in some irreducible component $Z''$ of $\bb H$.
But $z\in Z$; so $Z''= Z'$.  But $\dim Z = \dim Z'$.  Thus $Z = Z'$.
Thus (1) holds.

 {\bf For (2),} let $T_{\bb F}$ be the tangent space at the $K$-point
$C$ to $\bb F\Lambda_{A_K}^{\{\mb h_m\}}$.  First, let's prove $T_{\bb
F}= F^0T$.

Recall that $\bb H \Lambda_{A}^{\{\mb h_m\}}$ is an open subscheme of\/
$\bb H \Psi_{A}^{\mb h_{\bar s}}$ by \eqref{prSmOp}(1).  But $\bb F
\Psi_{A}^{\mb h_{\bar s}}$ is the preimage of\/ $\bb H
\Lambda_{A}^{\{\mb h_m\}}$ in $\bb F \Psi_{A}^{\mb h_{\bar s}}$ under
the retraction map by definition; see \eqref{sbMLev}.  Thus $T_{\bb F}$
is also the tangent space at $C$ to $\bb F \Psi_{A}^{\mb h_{\bar s}}$.
Note $T_{\bb F}\subset T$.

Let $I$ be the ideal of $C$.  Fix a standard basis $\{f_i\}$ of $I$.  A
$\tau\in T = \Hom_A(I,C)$ corresponds to a flat deformation of $C$ over
the dual numbers $K[e]$; say its ideal has basis $\{f_i+g_ie\}$ with
$g_i\in A$.  Then $\tau \in T_{\bb F}$ iff the initial terms of the
$f_i+g_ie$ define a flat deformation of $G_\bu C$.  The latter holds iff
the values $v(f_i)$ and $v(f_i+g_ie)$ are equal for all $i$ by Bruce
Bennett's \cite{Bennett}*{Thm.\,(2.13). p.\,29}; so iff $v(f_i) \le
v(g_i)$ for all $i$.  But $g_i$ reduces to $\tau(f_i)\in C$.  So if
$\tau \in T_{\bb F}$, then $\tau \in F^0T$.  Conversely, if $\tau \in
F^0T$, then replace $g_i$ with another lifting of $\tau(f_i)$ so that
$v(f_i) \le v(g_i)$.  Thus $T_{\bb F}= F^0T$.

Note that $\bb F\Lambda_{A}^{\{\mb h_m\}}$ is smooth and irreducible of
fiber dimension $\mb F$ by \eqref{coFdim}.  So $\dim F^0T = \mb F$.  So
$\dim_K(T/F^0T) + \mb F = \dim_K(T)$.  Thus if $\dim_K(T/F^0T) \le r$,
then $\dim_K(T) \le \mb F+r$, and if $\dim_K(T) = \mb F+r$, then
$\dim_K(T/F^0T) = r$.

Note $\dim Z =\mb F+r$ by (1).  Hence $\bb H_K$ is of dimension at least
$\mb F+r$ at the $K$-point $C$.  Thus $\dim_K(T) \ge \mb F+r$.  Hence,
if $\dim_K(T) \le \mb F+r$, then $\dim_K(T) = \mb F+r$, and so
$\dim_K(T/F^0T) = r$.  Thus (2) holds.

{\bf For (3),} consider the standard right-exact sequence
\begin{equation}\label{eqLift1}
\Der_K(A_K,C)\big/F^0\Der_K(A_K,C) \xto{\ d\ } T\big/F^0T
                     \to T^1(C)\big/F^0T^1(C)\to 0.
  \end{equation}
 So $C$ has trivial negative tangents, or $T^1(C)/F^0T^1(C) = 0$, iff
$d$ is surjective.  But $\Omega^1_{A/k} = A(-1)^{\oplus r}$.  So
$\Der_K(A_K,C) = C[1]^{\oplus r}$.  So the first term in \eqref{eqLift1}
is of $K$-dimension $r$.  Thus (3) holds.

{\bf For (4),} note $\dim_K T = \mb F+r$ by (2).  But (1) imples $\bb
H_K$ is of dimension at least $\mb F+r$ at $C$.  So $\bb H_K$ is smooth
of dimension $\mb F+r$ at $C$.  Hence, $\bb H_R$ is too by the Jacobian
criterion.  Thus (4) holds.

{\bf For (5),} note $Z$ and $\bb P_R$ are integral of dimension $\mb
F+r$ by (1) and its proof.  Recall $\theta$ is injective on field-valued
points; that is, $\theta$ is radicial, see \cite{EGAIihes}*{(3.5.4),
p.\,115}.  So $\theta_R$ is purely inseparable by
\cite{EGAIihes}*{(3.5.8), p.\,116}.  Thus (b) implies $\theta$ induces a
birational map $\theta'_R\: \bb P_R \to Z$.  Let's prove (a) too implies
$\theta'_R$ is birational.

Assume (a).  Then the map $d$ in \eqref{eqLift1} is surjective.  Its
source is plainly equal to $\Der_K(A_K,K)$.  And under this
identification, $d$ becomes the tangent map of the restriction of
$\theta$ to the fiber over $C$ of the projection from $\bb P$ onto its
first factor.  The proof of (1) yields $T_{\bb F}= F^0T$.  Hence the
tangent map $d\theta'_R$ is surjective at $(C,0)$.  But its source is of
dimension $\mb F+r$, as $\bb P_R$ is smooth and irreducible of dimension
$\mb F+r$ by the proof of (1).  And $\dim_K(T) = \mb F+r$ by (2) and
(3).  Thus $d \theta_R$ is an isomorphism of $K$-vector spaces at
$(C,0)$.

Hence $\theta_K$ is smooth at $(C,0)$ by d)\,\implies\,a) of
\cite{EGAIV}*{(17.11.1), p.\,82}, as the scheme points of $P_K$ and $\bb
H_K$ representing $(C,0)$ and $C$ have the same residue field, namely
$K$.  So $\theta_R$ is smooth at its $K$-point $(C,0)$.  But $\theta_R$
is purely inseparable, as noted above.  Thus $\theta'_R\: \bb P_R \to Z$
is birational.

Hence, $\theta^{-1}Z\sm \to Z\sm$ is bijective and birational.  But
$Z\sm$ is smooth, so normal.  Thus \cite{EGAIII0}*{(4.4.9), p.\,137}
implies $\theta^{-1}Z\sm \to Z\sm$ is an isomorphism.

Finally, replace $K$ by its algebraic closure.  Then $C'$ is supported
at a $K$-point of $\bb A_K^r$.  Change coordinates so that this point is
the origin.  First, assume $C'$ has trivial negative tangents. Then
$H_R$ is smooth of dimension $\mb F+r$ at $C'$ by (3) and (4) with $C'$
for $C$.  Hence, $Z$ is a component of $\bb H_R$ by (1), and it's the
only component containing $C'$.  Thus $C'$ lies in $Z\0\sm$.

Conversely, assume $C'$ lies in $Z\0\sm$.  Then $d\theta$ yields an
isomorphism from the tangent space of $\bb P_K$ at $(C',0)$ onto that,
$T'$ say, of $\bb H_K$ at $C'$, since $\theta^{-1}Z\sm \to Z\sm$ is an
isomorphism.  So let $T_{\go f}$ be the tangent space of the fiber
through $C'$ of the projection from $\bb P$ onto its first factor; then
$d\theta$ yields an isomorphism from $T_{\go f}$ onto $T'/F^0T'$, see
the reasoning above for $C$.  Furthermore, with $C'$ for $C$, the map
$d$ in \eqref{eqLift1} is therefore an isomorphism.  Thus
$T^1(C')\big/F^0T^1(C') = 0$; that is, $C'$ has trivial negative
tangents.  Thus (5) holds.

{\bf In (6),} $k$ is a domain; so $\bb F\Lambda_A^{\{\mb h_m\}}$ is
integral by \eqref{coFdim}.  So $\bb P$ is too. Thus $\Im(\theta)$ is
integral, so lies in an irreducible component $X$ of $\bb H$.

 By hypothesis, $Z$ is a component of\/ $\bb H_R$.  Also, $Z$ is
integral by (1).  Hence there's a scheme point $x \in Z\sm$ that lies in
no other component of $\bb H_R$ than $Z$.  So $x \in Z\0\sm$.  Note
$\dim Z = \mb F + r$ by (1), and $x \in X$.  So $X_R$ is smooth of
dimension $\mb F+r$ at $x$.

Consider the local rings $\mc O_{X,x}$ and $\mc O_{X_{\rm red},x}$.
Note that $\mc O_{X,x}\ox_kR = \mc O_{X_R,x}$ and $\mc O_{X_{\rm
red},x}\ox_kR = \mc O_{(X_{\rm red})_R,x}$.  But $X_R$ is smooth at $x$;
hence, $\mc O_{X_R,x}$ is reduced, so equal to $\mc O_{(X_{\rm
red})_R,x}$.  Thus $\mc O_{X,x}\ox_kR = \mc O_{X_{\rm red},x}\ox_kR$.

Let $\go n$ be the nilradical of $\mc O_{X,x}$.  So the sequence $0 \to
\go n \to \mc O_{X,x} \to \mc O_{X_{\rm red},x} \to 0$ is exact.  So it
remains exact after tensoring over $k$ with $R$, as $X_{\rm red}$ is
$k$-flat since $k$ is a Dedekind domain and $X_L \neq \emptyset$.  So
$\go n\ox_kR = 0$.  So $\go n = 0$ by Nakayama's Lemma over $\mc
O_{X,x}$.  So $\mc O_{X,x} = \mc O_{X_{\rm red},x}$.  So $\mc O_{X,x}$
is $k$-flat.  But $X_R$ is $R$-smooth of dimension $\mb F+r$ at $x$.
Hence $X$ is $k$-smooth of fiber dimension $\mb F+r$ at $x$.  But $X$
has no embedded nilpotents by the definition of the scheme structure on
an irreducible component.  Thus $X$ is an $(\mb F+r+1)$-dimensional
integral component of\/ $\bb H$.

However, $L$ is the fraction field of $k$; hence, $X_L$ is an $(\mb
F+r)$-dimensional integral component of\/ $\bb H_L$.  Now, $Y$ is
integral of dimension $\mb F+r$ by (1) with $L$ for $R$.  But $Y \subset
X_L$.  So $Y = X_L$.  Thus (6) holds.
 \end{proof}

\begin{remark}\label{rePGor}
 Preserve the conditions of \eqref{prLift}.  Recall the first step in
the proof of (1): it shows $T_{\bb F}= T_{\mb h} = F^0T$ where $T_{\mb
h}$ is the tangent space at $C$ to $\bb F \Psi_{A}^{\mb h_{\bar s}}$; in
fact, it shows $T_{\mb h} = F^0T$ even when $C$ isn't I-compressed.
Notice that a small modification of the proof shows that, if $C$ is
homogeneous, then the tangent spaces at $C$ to $\bb H\Lambda_A^{\{\mb
h_m\}}$ and $\bb H \Psi_{A}^{\mb h_{\bar s}}$ are equal, and equal to
$T_0$; in fact, the latter two spaces are equal even when $C$ isn't
I-compressed.

When $C$ is homogeneous, but not necessarily $I$-compressed, then the
equation $T_{\mb h} = F^0T$ was already proved in \cite{PGor}; see the
left-hand expression in the second display on p.\,615, noting the
definition of $Z_H(S)$ on p.\,614 and Remark~1.11(i) on p.\,617.
(Moreover, the right-hand expression is equal to the obstruction space
for $\bb F \Psi_{A}^{\mb h_{\bar s}}$ at $C$; this obstruction theory
agrees with Jelisiejew's \cite{JJ-JLMS2019}*{Thm.\,4.2, p.\,258}.)  The
proof of \cite{PGor}*{Thm.\,1.10, p.\,615} suggests using either
Bennett's \cite{Bennett}*{Thm.\,(2.13), p.\,29} or the lengthy
elementary argument on pp.\,616--617 in \cite{PGor} (to treat the
tangent space, but not the obstruction space); of course, Bennett's
theorem yields more results, as noted above.
 \end{remark}

\section{Compressed Quotients of Gorenstein Artinian Rings}\label{GCQ}

\begin{setup}\label{se11}
 Keep the setup of \eqref{se10}; in particular, keep the Noetherian base
scheme $S$ and the $\mc O_S$-algebra $\mc A$.  Also, fix a Noetherian
ring $k$, a $k$-flat finitely generated graded algebra $A :=
\bigoplus_{p\ge0}A_p$ with $A/F^1A = k$.

Call a field $K$ an {\it $S$-field} if $K$ contains a residue field of
 \indt{Sfie@$S$-field}
$S$.  Similarly, call a field $K$ an {\it $k$-field} if $K$ is a
 \indt{kfie@$k$-field}
$k$-algebra.

Given a graded $A$-module $M = \bigoplus_p M_p$ and $g_i\in A_{d_i}$
with $d_i>0$ for $1\le i\le m$, note that the usual Koszul complex
$K_\bu(g_1,\dotsc,g_m;M)$ is naturally graded: its terms are sums of
right-shifts of M; and its differentials are maps of degree 0.  For
example, $K_\bu(g_1;M)$ is the complex $M(-d_1) \to M$ whose
differential is multiplication by $g_1$.  Furthermore,
\begin{equation*}\label{eqse11a}
 K_\bu(g_1,\dotsc,g_m;M) = K_\bu(g_1;A) \ox_A \dotsb
    \ox_A K_\bu(g_m;A) \ox_A M,
 \end{equation*}
 and the grading on the left-hand side is induced by that on the factors
on the right.

Fix an indeterminate $z$.  If $M_p$ is locally free of rank $\mb m(p)$
 \indn{HmbrMz@$\mb H(M,z)$}
for all $p$, set
                $$\ts\mb H(M,z) := \sum_p \mb m(p)z^p.$$

  Given Laurent series $\mb A(z) = \sum a_p z^p$ and $\mb B(z) = \sum
b_p z^p$ and $\mb C(z) = \sum c_p z^p$ with real coefficients and $c_q =
0$ for $q < 0$, let $\mb A(z) = \mb B(z) \mod z^n$ mean $a_p = b_p$ for
$p < n$.  Then $\mb A(z)\mb C(z) = \mb B(z)\mb C(z) \mod z^n$.  Also,
let $\mb A(z) \le \mb B(z) \mod z^n$ mean $a_p \le b_p$ for $p < n$.  If
$c_q \ge 0$ for $q\ge 0$, then $0\le\big(\mb B(z) - \mb A(z)\big)\mb
C(z) \mod z^n$ and so $\mb A(z)\mb C(z) \le \mb B(z)\mb C(z) \mod z^n$.
Moreover, let $\mb A(z) \le \mb B(z)$ mean $a_p \le b_p$ for all $p$.
If $0 \le \mb C(z)$, then $\mb A(z)\mb C(z) \le \mb B(z)\mb C(z)$.
 \end{setup}

\begin{lemma}\label{leCplx}
 Let $M \xto\mu N \xto\nu P$ be a sequence of filtered $A$-modules.
Assume the filtrations are exhaustive, and there's $n$ with $G_pM \to
G_pN \to G_pP$ exact for all $p < n$ and with\/ $\nu\mu M \subset
F^{n+1}P$.  Set $D := \Im\nu$ and $F^pD := D \cap F^pP$ for all $p$.

\(1) Then $F^qM/F^pM \to F^qN/F^pN \to F^qP/F^pP$ is exact for all $q <
p \le n$.

\(2) Then $G_p D = \Im G_p\nu$ for all $p\le n$.

\(3) Fix $p<n$. Assume $(M/F^{p+1}M)\ox_Ak \to (N/F^{p+1}N)\ox_Ak$
vanishes, and assume $((G_\bu N)\ox_Ak)_p \risom G_p(N\ox_Ak)$.  Then
$((G_\bu D)\ox_Ak)_p \risom G_p(D\ox_Ak)$.
 \end{lemma}

\begin{proof}
  {\bf In (1),} as $\nu\mu M \subset F^{n+1}P \subset F^pP$, the
sequence is a complex.  So it's exact as $G_rM \to G_rN \to G_rP$ is
exact for $q \le r < p$ and $L \mapsto G_\bu L$ is biexact by
\eqref{prBiex}.

{\bf For (2),} fix $p$.  Note $G_p D \subset G_p P$ as $F^{p+1} D := (
D) \cap F^{p+1}P$.  So we have to show $G_p N \to G_p D$ is surjective.
Thus given $z\in F^p D$, we have to find $y\in F^pN$ with $\nu y - z \in
F^{p+1}P$.  But $D := \Im\nu$; so $z = \nu y'$ with $y' \in F^qN$ for
some $q$.

As $z\in F^p D$, owing to (1) there's $x \in F^qM$ with $y := y'-\mu
x \in F^pN$.  Then $\nu y = \nu y' - \nu\mu x$.  But $\nu\mu x \in
F^{n+1}P \subset F^{p+1}P$.  Also $z = \nu y'$.  Thus $\nu y - z \in
F^{p+1}P$.

{\bf For (3),} form this induced commutative square with surjective
right column:
 \begin{equation*}\label{eqeqCplx}
 \begin{CD}
 ((G_\bu N)\ox_Ak)_p @>\nu_t>> ((G_\bu D)\ox_Ak)_p \\
          @VV\simeq V                          @VVV\\
     G_p(N\ox_Ak)    @>\nu_b>>      G_p(D\ox_Ak)
 \end{CD}
 \end{equation*}
 To prove (3), let's prove that $\nu_t$ is surjective and that $\nu_b$
is an isomorphism.

Given any graded $A$-module $L$, note that $(L\ox_Ak)_p =
L_p\big/((F^1A)L)_p$ and that $((F^1A)L)_p = \sum_{q>0}A_qL_{p-q}$.
Hence the hypotheses of (3) and its assertion don't involve $F^nM$ and
$F^nN$ and $F^nP$.  Thus replacing $M$ and $N$ and $P$ by $M/F^nM$ and
$N/F^nN$ and $P/F^nP$, we may assume $F^rM$ and $F^rN$ and $F^rP$ vanish
if $r\ge n$.

For all $r$, let's see $M/F^rM \to N/F^rN \to P/F^rP$ is now exact.
Indeed, given $y \in N$ with $\nu y \in F^rP$, say $y \in F^qN$.  Note
(1) holds now with $p$ replaced by any $r$.  Thus (1) yields $x \in
F^qM$ with $\mu x -y \in F^rN$, as desired.

For all $r$, the image of $N/F^rN$ in $P/F^rP$ is plainly $D/(D\cap
F^rP)$.  However, $F^rD := D \cap F^rP$.  Thus $M/F^rM \to N/F^rN \to
D/F^rD \to 0$ is exact.

For all $r$, form the following commutative diagram with exact rows and
columns:
 \begin{equation*}\label{eqeqCplxa}
 \begin{CD}
     @.       @.  M/F^{r+1}M @>>> M/F^rM @>>> 0\\
 @.      @.          @VVV           @VVV\\
 0 @>>> G_rN @>>> N/F^{r+1}N @>>> N/F^rN @>>> 0\\
 @.  @VV\gamma V  @VV\beta V    @VV\alpha V\\
 0 @>>> G_rD @>>> D/F^{r+1}D @>>> D/F^rD @>>> 0
 \end{CD}
 \end{equation*}
 As $\alpha$ and $\beta$ are surjective, a diagram chase shows $\gamma$
is too.  So $G_\bu N \to G_\bu D$ is surjective; so $(G_\bu N)\ox_Ak \to
(G_\bu D)\ox_Ak$ is too.  Thus $\nu_t$ is surjective, as desired.

For all $r$, note $(M/F^rM)\ox_Ak \to (N/F^rN)\ox_Ak \to (D/F^rD)\ox_Ak
\to 0$ is exact owing to the second paragraph above.  The first map
vanishes for $r := p+1$ by hypothesis.  So it vanishes for all $r \le
p+1$, as $(M/F^{p+1}M)\ox_Ak \to (M/F^rM)\ox_Ak$ is surjective.  Thus
$(N/F^rN)\ox_Ak \risom (D/F^rD)\ox_Ak$ for all $r \le p+1$.

But for any filtered $A$-module $L$, note $(L/F^rL)\ox_Ak =
(L\ox_Ak)\big/F^r(L\ox_Ak)$ as $F^r(L\ox_Ak)$ is the image of
$(F^rL)\ox_Ak$ in $L\ox_Ak$.  Hence, in the following commutative
diagram, the right two vertical maps are isomorphisms, as indicated:
 \begin{equation*}\label{eqeqCplxb}
 \begin{CD}
 0 @>>> G_p(N\ox_Ak) @>>> (N\ox_Ak)\big/F^{p+1}(N\ox_Ak)
                                @>>> (N\ox_Ak)\big/F^p(N\ox_Ak) @>>> 0\\
 @.        @V\nu_bVV              @VV\simeq V  @VV\simeq V\\
 0 @>>> G_p(D\ox_Ak) @>>> (D\ox_Ak)\big/F^{p+1}(D\ox_Ak)
                                @>>> (D\ox_Ak)\big/F^p(D\ox_Ak) @>>> 0          
 \end{CD}
 \end{equation*}
 The horizontal sequences are exact. So $\nu_b$ is an isomorphism.  Thus
(3) holds.
 \end{proof}

\begin{lemma}\label{leInj}
 Fix $m$, and for $1\le i\le m$, fix nonzero $g_i\in A_{d_i}$ with
$d_i>0$.  Fix a $k$-flat, finitely generated, graded $A$-module $M =
\bigoplus_p M_p$.  Set $M^{(0)} := M$.  For $1\le q\le m$, set $M^{(q)}
:= M\big/\sum_{i=1}^qg_iM$, set $\mb P^{(q)}(z) := \prod_{i=1}^q
(1-z^{d_i})$, and for all $p$, define $\mu^{(q)}_p\: M^{(q-1)}_{p-d_q}
\to M^{(q-1)}_p$ by $\mu^{(q)}_p(x) := g_q x$.  Fix $n$.

\(1) Assume $k$ is a field.  Then $\mb H(M,z) \le \mb H(M^{(m)},z) \big/
\mb P^{(m)}(z)$.  Further, $\mu^{(q)}_p$ is injective for $1\le q\le m$
and all $p < n$ iff $\mb H(M,z) = \mb H(M^{(q)},z) \big/ \mb P^{(q)}(z)
\mod z^n$ for $1\le q\le m$, iff $\mb H(M,z) = \mb H(M^{(m)},z) \big/
\mb P^{(m)}(z) \mod z^n$.

\(2) Then $\mu^{(q)}_p$ is injective and $M^{(q)}_p$ is $k$-flat for
$1\le q\le m$ and all $p < n$ iff $\mb H(M,z) = \mb
H((M\ox_kK)^{(m)},z) \big/ \mb P^{(m)}(z) \mod z^n$ for every $k$-field
$K$.

\(3) For all $r\ge 1$, the Koszul homology group $H_r(g_1,\dotsc,g_m;M)$
is $0$ in degree $p < n$, iff $H_1(g_1,\dotsc,g_m;M)$ is so, iff
$\mu^{(q)}_p$ is injective for $1\le q\le m$ and all $p < n$.
 \end{lemma}

\begin{proof} {\bf For (1),} cf.\ \cite{RStan78}*{Cor.\,3.2, p.\,62}.
Given $p$, form this right exact sequence:
 \begin{equation}\label{eqInj1}
  M^{(m-1)}_{p-d_m} \xto{\mu^{(m)}_p} M^{(m-1)}_p \to M^{(m)}_p \to 0.
 \end{equation}
 Set $\mb m^{(j)}(p) := \dim M^{(j)}_p$.  Then \eqref{eqInj1} yields
$\mb m^{(m-1)}_p - \mb m^{(m-1)}_{p-d_m} \le \mb m^{(m)}_p$, with
equality iff $\mu^{(m)}_p$ is injective.  Thus
 \begin{equation}\label{eqInj4}
 \mb
H(M^{(m-1)},z)(1-z^{d_m}) \le \mb H(M^{(m)},z).
 \end{equation}
 Also, $\mb H(M^{(m-1)},z) (1-z^{d_m}) = \mb H(M^{(m)},z) \mod z^n$ iff
$\mu^{(m)}_p$ is injective for $p < n$.

Induction on $m$ gives $\mb H(M,z) \le \mb H(M^{(m-1)},z) \big/ \mb
P^{(m-1)}(z)$.  Now, multiply \eqref{eqInj4} by $1\big/\mb P^{(m)}(z)$.
The inequality is preserved as $1\big/\mb P^{(m)}(z)$ has nonnegative
coefficients.  But $\mb P^{(m)}(z) = \mb P^{(m-1)}(z)(1-z^{d_m})$.  Thus
\begin{equation}\label{eqInj3+}
 \mb H(M,z) \le  \mb H(M^{(m-1)},z) \big/ \mb P^{(m-1)}(z)
\le \mb H(M^{(m)},z) \big/ \mb P^{(m)}(z)
 \end{equation}

Further, suppose $\mu^{(q)}_p$ is injective for $1\le q\le m$ and all $p
< n$.  Then by induction, $\mb H(M,z) = \mb H(M^{(q)},z) \big/ \mb
P^{(q)}(z) \mod z^n$ for $1\le q < m$.  But $\mu^{(m)}_p$ is injective
for $p < n$.  So  $\mb H(M^{(m-1)},z) (1-z^{d_m}) = \mb H(M^{(m)},z)
\mod z^n$.  Thus also
 $$\mb H(M,z) = \mb H(M^{(m-1)},z) \big/ \mb P^{(m-1)}(z)
              = \mb H(M^{(m)},z) \big/ \mb P^{(m)}(z) \mod z^n. $$

Conversely, suppose $\mb H(M,z) = \mb H(M^{(m)},z) \big/ \mb P^{(m)}(z)
\mod z^n$.  Then in \eqref{eqInj3+}, the two inequalities become
equalities mod $z^n$.  So by the first and by induction,
$\smash{\mu^{(q)}_p}$ is injective for $1\le q < m$ and all $p < n$.
But as $\mb P^{(m)}(z)$ is an ordinary polynomial, the second equality
implies $\mb H(M^{(m-1)},z) (1-z^{d_m}) = \mb H(M^{(m)},z) \mod z^n$.
So as shown above, $\mu^{(m)}_p$ is injective for $p < n$.  Thus (1)
holds.

{\bf For (2),} first assume $\mu^{(q)}_p$ is injective and $M^{(q)}_p$ is
$k$-flat for $1\le q\le m$ and all $p < n$.  For any $k$-algebra $K$,
then $\mu^{(q)}_p\ox_kK$ is injective and $M^{(q)}_p\ox_kK$ is $K$-flat.
But $\mb H(M,z) = \mb H(M\ox_kK,z)$ and $M^{(q)}_p\ox_kK =
(M\ox_kK)^{(q)}$.  Thus (1) yields $\mb H(M,z) = \mb
H((M\ox_kK)^{(q)},z) \big/ \mb P^{(m)}(z) \mod z^n$ for every $k$-field
$K$.

Conversely, assume $\mb H(M,z) = \mb H((M\ox_kK)^{(m)},z) \big/ \mb
P^{(m)}(z) \mod z^n$ for every $k$-field $K$.  Then (1) implies $\mb
H(M,z) = \mb H((M\ox_kK)^{(m-1)},z) \big/ \mb P^{(m-1)}(z) \mod z^n$ for
every $k$-field $K$.  So by induction on $m$, then $\mu^{(q)}_p$ is
injective and $M^{(q)}_p$ is $k$-flat for $1\le q < m$ and all $p < n$.
In particular, $M^{(m-1)}_p$ is flat.  Also, (1) implies
$\mu^{(m)}_p\ox_kK$ is injective.  So by the Local Criterion of
Flatness, $\mu^{(m)}_p$ is injective and $M^{(m)}_p$ is $k$-flat for all
$p < n$.  Thus (2) holds.

{\bf For (3),} cf.\ \cite{SerreAlLoc}*{Prop.\,3, p.\,IV-5}.  Form the
Koszul complexes $K_\bu := K_\bu(g_m;A)$ and $L_\bu :=
K_\bu(g_1,\dotsc,g_{m-1};M)$.  Then \cite{SerreAlLoc}*{Prop.\,1,
p.\,IV-2} yields, for all $r$, the following exact sequence of graded
modules and maps of degree 0: \begin{equation}\label{eqInj2}
 0 \to H_0(K_\bu\ox_A H_r(L_\bu)) \to  H_r(K_\bu\ox_A L_\bu)
    \to H_1(K_\bu\ox_A H_{r-1}(L_\bu)) \to 0.
 \end{equation}
 
First assume $\mu^{(q)}_p$ is injective for $1\le q\le m$ and $p < n$.
Induct on $m$.  So assume $H_r(L_\bu)_p = 0$ for $r \ge 1$ and $p < n$.
So $H_0(K_\bu\ox_A H_r(L_\bu))_p = 0$, and if $r\ge 2$, then
$H_1(K_\bu\ox_A H_{r-1}(L_\bu))_p = 0$.  But $H_0(L_\bu)) = M^{(m-1)}$.
Furthermore, $K_\bu\ox_A M^{(m-1)} = K_\bu(g_m,M^{(m-1)})$.  So
$H_1(K_\bu\ox_A H_0(L_\bu))_p = 0$ as $\mu^{(m)}_p$ is injective.  Thus
\eqref{eqInj2} yields $H_r(g_1,\dotsc,g_m;M)_p = 0$, as desired.

Finally, fix $p < n$ and assume $H_1(g_1,\dotsc,g_m;M)_p = 0$.  Then
\eqref{eqInj2} yields $\Ker\mu^{(m)}_p = 0$ and $H_1(L_\bu)_p \big/
(g_m\cdot H_1(L_\bu)_{p-d_m}) = 0$.  As $d_m > 0$, the latter yields
$H_1(L_\bu)_p = 0$.  So $\Ker\mu^{(q)}_p=0$ for $1\le q< m$ by induction
on $m$. Thus (3) holds.
 \end{proof}

\begin{proposition}\label{prOpens}
 Fix $d_1,\dotsc,d_m >0$, and set $\mb P(z) := \prod_{i=1}^m(1-z^{d_i})$
and set $\bb P := \prod_{i=1}^m \bb P(\mc A^*_{d_i})$.  Fix an $S$-flat,
finitely generated, graded $\mc A$-module $\mc M = \bigoplus_p \mc M_p$.
Given an $S$-field $K$, set $A_K := \Gamma(\mc A_{\Spec K})$ and $M_K :=
\Gamma(\mc M_{\Spec K})$.  Given a $K$-point of $\bb P$, let
$g_1,\dotsc,g_m \in (A_K)_{d_i}$ correspond to it.  Set $M_K^{(m)} :=
M_K\big/\sum_{i=1}^mg_iM_K$.  Assume $S$ is irreducible.  Fix $n\ge 1$.

\(1) There's a nonempty open subset $U$ of $\bb P$ such that, given any
$S$-field $K$, a $K$-point of $\bb P$ lies in $U$ iff $\mb
H(M_K^{(m)},z) \le \mb H(M_{K^\di}^{(m)},z)\mod z^n$ for every
$K^\di$-point of $\bb P$ for every $S$-field $K^\di$.

\(2) Assume $\mb H(M_K,z)\mb P(z) = \mb H(M_K^{(m)},z) \mod z^n$ for
some $K$-point of $\bb P$ for some $S$-field $K$.  Then given any
$S$-field $K^\di$ and any $K^\di$-point of $\bb P$, the latter lies in
the open subset $U$ of \(1) iff\/ $\mb H(M_{K^\di},z)\mb P(z) = \mb
H(M_{K^\di}^{(m)},z) \mod z^n$.
 \end{proposition}

\begin{proof}
 First off, let's define a global version on $\bb P$ of the module
$M^{(q)}$ of \eqref{leInj}.  For $1\le i\le m$, set $\bb P_i := \bb
P(\mc A^*_{d_i})$, and denote the pullback to $\bb P$ of the
tautological injection $\mc O_{\bb P_i}(-1)\into (\mc A_{d_i})_{\bb
P_i}$ by $\mc L_i \into (\mc A_{d_i})_{\bb P}$.  Combine the latter with
the pullback of the multiplication map $\mc A\ox \mc M \to \mc M$ to get
a map $\mc L_i \ox_{\mc O_P}\mc M(-d_i)_{\bb P} \to \mc M_{\bb P}$.
Finally, set $\mc M^{(q)} := \cok\bigl(\bigoplus_{i=1}^q \mc L_i \ox_{\mc
O_P}\mc M(-d_i)_{\bb P} \to \mc M_{\bb P}\bigr)$.

Note that $S$ is irreducible, and that each $\smash{\mc A_{d_i}^*}$ is
locally free.  Thus $\bb P$ is irreducible.

{\bf For (1),} for $0\le q< n$, apply the theory of flattening
stratification to $\mc M^{(m)}_q$.  Let $U_q$ be the stratum of $\bb P$
on which the restriction of $\mc M^{(m)}_q$ has minimum rank, $r_q$ say.
Then $U_q$ is nonempty and open.  Set $U : = \bigcap_{q=0}^{n-1} U_q$.
Then $U$ is open.  Further, as $\bb P$ is irreducible, $U$ is nonempty
and irreducible.  Set $\mb R(z) := \sum_{q=0}^{n-1} r_qz^q$.  Then given
an $S$-field $K$ and a $K$-point of $\bb P$, the construction of $U$
yields $\mb R(z) \le \mb H(M_K^{(m)},z) \mod z^n$, with equality iff the
$K$-point lies in $U$.  Thus (1) holds.

{\bf For (2),} let's find an affine open subset $V^{(m)} \subset \bb P$
containing the given $K$-point such that, in the notation of
\eqref{leInj} with $M^{(q)} := \Gamma(V^{(q)},\mc M^{(q)})$, the map
$\mu^{(q)}_p$ is injective and the module $M^{(q)}_p$ is flat for $1\le
q\le m$ and $p < n$.  To find $V^{(m)}$, note $\mb H(M_K,z) = \mb
H(M_K^{(m-1)},z) \big/ \mb P^{(m-1)}(z) \mod z^n$ by \eqref{leInj}(1).
So induction on $m$ yields the analogous set $V^{(m-1)}$.  In
particular, $M^{(m-1)}_p$ is flat.  But $\mu^{(m)}_p\ox_kK$ is injective
by \eqref{leInj}(1).  So the Local Criterion of Flatness yields an
affine open subset $V^{(m)} \subset V^{(m-1)}$ containing the given
$K$-point such that the restriction of $\mu^{(m)}_p$ to $V^{(m)}$ is
injective and the restriction of $M^{(m)}_p$ is flat for all $p < n$, as
desired.

Both $V^{(m)}$ and $U$ are nonempty open subsets of $\bb P$, and $\bb P$
is irreducible.  So $V^{(m)} \cap U \neq \emptyset$. So $V^{(m)} \cap U$
contains a $K^\di$-point for some field $K^\di$.  Then \eqref{leInj}(2)
yields $\mb H(M_{K^\di},z) = \mb H(M_{K^\di}^{(m)},z) \big/ \mb
P^{(m)}(z) \mod z^n$. But $\mb P(z)$ is an ordinary polynomial.  So $\mb
H(M_{K^\di},z)\mb P(z) = \mb H(M_{K^\di}^{(m)},z) \mod z^n$.  But mod
$z^n$, for every $K^\di$-point of $U$, the value of $ \mb
H(M_{K^\di}^{(m)},z)$ is the same, and that value is strictly less than
the value of $ \mb H(M_{K^\di}^{(m)},z)$ for every $K^\di$-point of $\bb
P - U$.  Thus (2) holds.
 \end{proof}

\begin{proposition}\label{prNonempty}
  Assume $k$ a field; $A$ Gorenstein, Artinian of socle degree a,
generated by $A_1$.  Set $r := \mb a(1)$ and $\mb r(p) := \binom{p+
r-1}{r-1}$ and $v_0 := \inf\{\, v \mid \mb r(v) > \mb a(v)\, \}$.  Set
$\mb P(z) := \prod_j(1-z^j)^{\mb t(a-j)}$.  Define $\mb w^*(p)$ by
$\sum_p\mb w^*(p)z^p := \mb H(A^\dg,z) \mb P(z)$.  Fix $n \le v_0 - a$
with $\mb w^*(n-1)=0$.  Assume $\mb t^*(p)=0$ for $p\ge n$ and $p \le
-a$.  Set $\mb u(i) := \sum_{j > i}\mb t(j)$.  Assume $\mb u(1-n) \le
r$.  Set $\mb h(q) := \mb a(q)-\mb w(q)$ for $q > -n$ and\/ $\mb h(q) :=
\mb a(q)$ for $q \le -n$.  Then $\mb w^*(p) \ge 0$ for $p <n$, and $\mb
H\Psi_A^{\mb h, \mb t} \neq \emptyset$.
 \end{proposition}

\begin{proof}
 Since $A$ is Artinian, $A^\dg = A^*$.  Since $A$ is Gorenstein, $A^*$
is an invertible $A$-module, so free of rank 1.  Choose an $f\in
A^*_{-a}$ with $A^* = Af$.  Choose variables $X_1,\dotsc,X_r$ and set $R
:= k[X_1,\dotsc,X_r]$.  Fix a $k$-basis of $A_1$.  Mapping $X_i$ to its
$i$th element, view $A$ as a homogeneous quotient of $R$.  Thus $A^* =
Rf \subset R^\dg$.

Form the map $\mu\: R(a) \to A^*$ by $\mu(x) := xf$. It factors as
$R(a)\onto A(a)\risom A^*$ where $R\onto A$ is the quotient map.  Fix $i
< v_0$.  Then $\dim_k R_i = \dim_k A_i$ by definition of $v_0$.  So
$R_i\risom A_i$.  Thus $\mu_i\: R(a)_{i-a} \risom (A^*)_{i-a}$.

For $1 \le p < a+n-1$, note $u(a-p-1)\le r$, and for $u(a-p) < j \le
u(a-p-1)$, set $g_j := X^p_j$.  Set $J:= \sum Rg_j$.  As the $g_j$ form
an $R$-regular sequence, \eqref{leInj}(1) yields $\mb H(R,z) = \mb
H(R/J,\,z) \big/ \prod_{j=1}^{a+n-2} (1-z^j)^{\mb t(a-j)}$.  Multiplying
both sides by $\mb P(z)$ gives $\mb H(R,z)\mb P(z) = \mb
H(R/J,\,z)(1-z^{a+n-1})^{\mb t(1-n)}$ as $\mb t^*(p)=0$ for $p \ge n$.
However, $\mb H(R,z) = \mb H(A^*(-a),z) \mod z^{v_0}$.  Thus
  $$\ts \mb H(R/J,\,z) (1-z^{a+n-1})^{\mb t^*(n-1)}
    = \sum_p\mb w^*(p-a)z^p  \mod  z^{v_0}.$$

Compare terms of degree $a+p$ with $p< n-1$.  Thus $\mb w^*(p) = \dim
(R/J)_{a+p} \ge 0$ as  $a+p < v_0$.
 
 Similarly, $\dim (R/J)_{a+n-1} - \mb t^*(n-1) = \mb w^*(n-1) = 0$.
Thus there are $g_j$ in $R_{a+n-1}$, for $\mb u(1-n)< j \le \mb u(-n)$,
whose residues form a basis of $(R/J)_{a+n-1}$.

Set $D := \sum_{j=1}^{\mb u(-n)} Ag_jf\subset A^*$ and $C := D^*$.  Then
$J(a)_i\risom D_i$ for $i< n-1$.  But $\dim J(a)_i = \mb h^*(i)$ for
$i<n-1$.  Thus $\dim C_p = \mb h(p)$ for $p > 1-n$.  Now, by
construction, $D_{n-1} = (A^*)_{n-1}$.  But $A^* = Rf$; so $R_1\cdot
(A^*)_i = (A^*)_{i+1}$ for all $i$.  So $D_i = (A^*)_i$ for all $i\ge
n-1$.  Thus $\dim C_p = \mb h(p)$ for $p \le 1-n$.  Plainly the $g_jf$
form a minimal generating set of $D$.  Thus $D$ has local generator type
$\mb t^*$. Thus $C \in \mb H\Psi_{A}^{\mb h, \mb t}$.
 \end{proof}

\begin{proposition}\label{prN28May22}
 Assume $k$ is a field, and $A$ is generated by $A_1$.  Fix a finitely
generated, graded $A$-module $B$, and $C \in \mb H\Psi_B^{\mb h, \mb
t}$ for some $\mb h,\ \mb t$.  Fix $n$ with $\mb t^*(p)= 0$ for $p\ge
n$.  Fix $\wt C \in \mb H\Psi_C^{\wt{\mb h}, \wt{\mb t}}$ for some
$\wt{\mb h},\ \wt{\mb t}$.  Assume $(C^*)_p = (\wt C^*)_p$ for $p <
n-1$, and $A_1(C^*)_{n-1} \subset A_1(\wt C^*)_{n-1}$.  Set $d := \dim
(C^*/\wt C^*)_{n-1}$.

\(1) Then $\wt{\mb h}(q) = \mb h(q)$ for $q \neq 1-n$, and\/ $\wt{\mb
h}(1-n) = \mb h(1-n) - d$.

\(2) Then $\wt{\mb t}(q) = \mb t(q)$ for $q\neq 1-n$, and\/ $\wt{\mb
t}(1-n) = \mb t(1-n) - d$.

\(3) Set $\mb P(z) := \prod_j(1-z^j)^{\mb t(a-j)}$ and\/ $\wt{\mb P}(z)
:= \prod_j(1-z^j)^{\wt{\mb t}(a-j)}$.  Define $\mb w^*(p)$ by $\sum_p\mb
w^*(p)z^p := \mb H(\mc B^\dg,z) \mb P(z)$ and $\wt{\mb w}^*(p)$ by\/
$\sum_p\wt{\mb w}^*(p)z^p := \mb H(\mc B^\dg,z) \wt{\mb P}(z)$.  Set
$b:=\dim B_a$. Then $\wt{\mb w}^*(p) = \mb w^*(p)$ for $p < n-1$, and\/
$\wt{\mb w}^*(n-1) = \mb w^*(n-1) + bd$.
 \end{proposition}

\begin{proof} {\bf For (1) and (2),} for $p < n-1$, notice $\wt{\mb
h}^*(p) = \mb h^*(p)$ as $(\wt C^*)_p = (C^*)_p$.

For the same reason, $((F^1A)\wt C^*)_p = ((F^1A)C^*)_p$ but for $p\le
n-1$.  However, $\wt{\mb t}^*(p) := \dim \bigl( ({\wt C}^*)_{p} \big/
((F^1A){\wt C}^*)_{p} \bigr)$ and $\mb t^*(p) := \dim \bigl( (C^*)_{p}
\big/ ((F^1A) C^*)_{p} \bigr)$ for all $p$.  So $\wt{\mb t}^*(p) = \mb
t^*(p)$ for $p < n-1$, and\/ $\wt{\mb t}^*(n-1) = \mb t^*(n-1) - d$ as
$d := \dim (C^*/\wt C^*)_{n-1}$.

Notice $\wt{\mb h}^*(n-1) = \mb h^*(n-1)-d$, also as $d := \dim (C^*/\wt
C^*)_{n-1}$.

For $p\ge n$, note $\mb t^*(p) = 0$; so $((F^1A) C^*)_{p} = (C^*)_p$.
But $A_1$ generates $A$; hence, $((F^1A) C^*)_{p} = A_1(C^*)_{p-1}$.
Thus $A_1(C^*)_{p-1} = (C^*)_p$.

Recall $A_1(C^*)_{n-1} \subset A_1(\wt C^*)_{n-1}$.  So $(C^*)_n \subset
A_1(\wt C^*)_{n-1} \subset (\wt C^*)_n \subset (C^*)_n$.  So
\begin{equation}\label{eqN28}
 A_1(\wt C^*)_{p-1} = (\wt C^*)_p = (C^*)_p.
 \end{equation}
 holds for $p = n$.  Suppose \eqref{eqN28} holds for some $p\ge n$.
Then $A_1(\wt C^*)_p = A_1(C^*)_p$.  But $A_1(C^*)_p = (C^*)_{p+1}$.  So
$A_1(\wt C^*)_p = (C^*)_{p+1}$.  But $A_1(\wt C^*)_p \subset (\wt
C^*)_{p+1}$ and $(\wt C^*)_{p+1} \subset (C^*)_{p+1}$.  Thus
\eqref{eqN28} holds for all $p\ge n$ by induction on $p$.  Thus $\wt{\mb
t}^*(p) = 0$ and $\wt{\mb h}^*(p) = \mb h^*(p)$ for all $p > n-1$.  Thus
(1) and (2) hold.

{\bf For (3),} note $\mb t(1-n) - \wt{\mb t}(1-n) = d$ by (2).  Hence
$\mb P(z) = \wt{\mb P}(z)(1-z^{a+n-1})^d$.  So $\sum\mb w^*(p)z^p =
\bigl(\sum\wt{\mb w}^*(p)z^p\bigr) (1-z^{a+n-1})^d$.  Compare terms of
degree $p$ for $p \le n-1$, and note $\mb w^*(-a) = \dim B^\dg_{-a} = b$.
Thus (3) holds.
 \end{proof}

\begin{theorem}\label{thAffBdl}
 Assume $\mc A$ is $\mc O_S$-Gorenstein and $\mc O_S$-Artinian of socle
degree $a$.  Set $\mb P(z) := \prod_j(1-z^j)^{\mb t(a-j)}$.  Define $\mb
w^*(p)$ by $\sum_p\mb w^*(p)z^p := \mb H(\mc A^\dg,z) \mb P(z)$.  Fix $n$
with $\mb w^*(p) \ge 0$ for $p <n$. Set $\mb h(q) := \mb a(q)-\mb w(q)$
for $q > -n$, and\/ $\mb h(q) := \mb a(q)$ for $q \le -n$. Assume $\mb
t^*(p)=0$ for $p\ge n$ and $p \le -a$.  Finally, assume $\bb
H\Psi_{\mc A}^{\mb h, \mb t} \neq \emptyset$.

\(1a) Set $\mb R := \sum_p \mb t(p)\bigl(\sum_{p'<p}\mb w(p')\bigr)$.
Then $\bb F\Psi_{\mc A}^{\mb h, \mb t}$ is an affine-space bundle
over\/ $\bb H\Psi_{\mc A}^{\mb h, \mb t}$ of fiber dimension $\mb R$.

\(1b) Set $\mb H := \sum_p \mb t(p)\mb w(p)$.  Then $\bb H\Psi_{\mc
A}^{\mb h, \mb t}$ is covered by nonempty open subschemes, with each
isomorphic to an open subscheme of the affine space over $S$ of fiber
dimension $\mb H$.

\(1c) Set $\mb F := \mb H + \mb R$.  Then $\bb F\Psi_{\mc A}^{\mb h,
\mb t}$ is covered by nonempty open subschemes, with each isomorphic to
an open of the affine space over $S$ of fiber dimension $\mb F$.

\(2) Assume $S$ irreducible.  Given $S^\di/S$ and $\mb h^\di$ and $\mc
C^\di \in \mb H\Psi_{\mc A_{S^\di}}^{\mb h^\di, \mb t}$, necessarily
$\mb h^\di(q) \le \mb h(q)$ for all $q$.

\(3) Given $p, q > 0$, set $r := a - (p + q)$. Assume $\mb a(r) > \mb
h(r)$. Assume $\mb t(a-p) \ge 2$ if $p = q$, or else $\mb t(a-p) \ge 1$
and\/ $\mb t(a-q) \ge 1$ if $p \neq q$. Then $\mb h(r) < \mb h^{\rm
I}_{\bar s}(r)$.
 \end{theorem}

\begin{proof}
 For general use below, set $m := \sum_p \mb t(p)$.

{\bf For (1a),} as for \eqref{thFHdim}, it suffices to apply
\eqref{prAffBdl}.  To do so, we must see that the universal sheaf, $\mc
C$ say, on $\bb H\Psi_{\mc A}^{\mb h, \mb t}$ is very effective.  So
given a Noetherian affine $T/\bb H\Psi_{\mc A}^{\mb h, \mb t}$ and a
potential lifting $\mc F$ of $\mc C_T$, we must show that the subsheaves
$\mc C_T^*$ and $G_\bu \mc F^*$ of $\mc A^\dg_T$ are equal and that the
surjection $\gamma\:(G_\bu \mc F^*)\ox_{\mc A_T}\mc O_T \onto G_\bu (\mc
F^*\ox_{\mc A_T}\mc O_T)$ is an isomorphism.  But these issues are
local.  Thus we may assume $T$ is local.

As $\mc A$ is $\mc O_S$-Artinian, $\mc A^\dg = \mc A^*$.  But $\mc A$ is
$\mc O_S$-Gorenstein, and $T$ is local.  So $\mc A^\dg_T$ is cyclic.
Now, $\mb t^*$ is the local generator type of $\mc C^*$.  Hence for
$1\le i\le m$, there's $g_i\in \Gamma((\mc A_{d_i})_T)$ with $d_i>0$,
with $\#\{\, i \mid d_i = a-q\,\} = \mb t(q)$ for all $q$, and with $\mc
C_T^* = \sum g_i\mc A^\dg_T$.  Note $\mb P(z) = \prod_i(1-z^{d_i})$.

Set $\mc M^{(m)} := \mc A^\dg_T/\mc C_T^*$.  Then $\mc M^{(m)}_{p}$ is
locally free of rank $\mb w^*(p)$ for $p < n$.  So given any $T$-field
$K$, then $\mb H(\mc M^{(m)}\ox K,z) = \mb H(\mc A^\dg,z) \mb P(z) \mod
z^n$.  Thus \eqref{leInj}(2),\,(3) yield $H_1(g_1,\dotsc,g_m; \mc
A^\dg_T)_p = 0$ for all $p<n$.

As $\mc F$ is a potential lifting of $\mc C_T$, there's $f_i \in
F^{d_i}\mc A_T$ with $g_i$ as its initial form and with $\mc F = \sum
f_iA^\dg$.  Form the Koszul complex $K_\bu := K_\bu(f_1,\dotsc,f_m; \mc
A^\dg_T)$.  Note $G_\bu K_\bu = K_\bu(g_1,\dotsc,g_m; \mc A^\dg_T)$.  So
\eqref{leCplx} applies to the complex $K_2\to K_1\to K_0$.

Hence \eqref{leCplx}(2) yields $G_p(\mc F^*) = (\mc C_T^*)_p$ for $p \le
n$.  But $G_p(\mc F^*) = (\mc C_T^*)_p$ for $p \ge n$ as $(\mc
C_T^*)_p \subset G_p(\mc F^*) \subset (\mc A_T^\dg)_p$ and $\mb h^*(p)
:= \mb a^*(p)$.  Thus $\mc F$ is a lifting of $\mc C_T$.

Fix $p\ge n$.  Then $(\mc C^*\ox\mc O_T)_p = 0$ as $\mb t^*(p) = 0$.
But $\mc C^*\ox\mc O_T =: \mc C^*_T = G_\bu\mc F^*$, and $\gamma_p$ is
surjective.  Thus $\gamma_p$ is an isomorphism.

Finally, fix $p< n$, and let's apply \eqref{leCplx}(3).  Note $K_2\ox\mc
O_T\to K_1\ox \mc O_T$ vanishes as $d_i > 0$ for all $i$.  But $K_1 =
\mc A^{\dg\oplus m}_T$.  Hence $G_\bu (K_1\ox \mc O_T) = K_1\ox \mc O_T$
and $G_\bu K_1 = K_1$.  Hence $(G_\bu K_1)\ox \mc O_T = G_\bu (K_1\ox
\mc O_T)$.  Thus \eqref{leCplx}(3) implies $\gamma_p$ is an isomorphism.

Thus $\mc C$ is very effective.  Thus (1a) holds.

{\bf For (1b),} form the functor $\mb U_{\mb t}$ whose $T$-points are
the homogeneous quotients $\mc E = \mc A_T^\dg / \mc D$ with $\mc E_p$
locally free of rank $\mb w^*(p)$ for $p < n$ and with $\mc D$ of local
generator type $\mb t^*$.  Let's use the hypotheses $\mb w^*(p) \ge 0$
for $p <n$ and $\mb t^*(p)=0$ for $p\ge n$ and $p \le -a$, but not $\bb
H\Psi_{\mc A}^{\mb h, \mb t} \neq \emptyset$, to represent $\mb
U_{\mb t}$ via recursion on $s - \bar s$ where $s := \sup \{\, q\mid \mb
t(q) \neq 0\,\}$ and where $\bar s := \inf \{\, q\mid \mb t(q) \neq
0\,\}$.  Set $\bar t := \mb t(\bar s)$.

If $\bar s = s$.  set $U' := S$ and $\wt{\mc E'} = \mc A^\dg$; so set
$\wt{\mc D}' := 0$, set $\mb w'^* := \mb a^*$, and set $\mb H' := 0$.

If $\bar s < s$, set $\mb t'(\bar s) := 0$ and $\mb t'(q) := \mb t(q)$
for $q \neq \bar s$, and define $w'^*$ and $\mb H'$ to be the objects
corresponding to $w^*$ and $\mb H$.  Trivially, $\mb t'^*(p)=0$ for
$p\ge n$ and $p \le -a$.  Further, $\sum_p\mb w'^*(p)z^p =
\bigl(\sum_p\mb w^*(p)z^p\bigr)(1-z^{a - \bar s})^{-\bar t}$; hence,
$\mb w'^*(p) \ge \mb w^*(p) \ge 0$ for $p < n$.  So assume that $\mb
U_{\mb t'}$ is representable by an $S$-scheme $U'$ and a universal
quotient $\wt{\mc E}' = \mc A_{U'}^\dg \big / \wt{\mc D}'$ with $\wt{\mc
E}'_p$ locally free of rank $\mb w'^*(p)$ for $p < n$, with $\wt{\mc
D}'$ of local generator type $\mb t'^*$, and with $U'$ covered by
nonempty open subschemes of the affine space over $S$ of fiber dimension
$\mb H'$.

For any $s$, note $\bigl(\sum_p\mb w'^*(p)z^p\bigr)(1-z^{a - \bar
s})^{\bar t} = \sum_p\mb w^*(p)z^p$ and $\mb w'^*(-a) = 1$; therefore,
comparing terms of degree $-\bar s$ yields $\mb w'^*(-\bar s) - \bar t =
\mb w^*(-\bar s) \ge 0$. Set $\bb G := \Grass_{\mb w(\bar s)}(\wt{\mc
E}'_{-\bar s})$, the Grassmann\-ian of rank-$\mb w(\bar s)$ locally free
quotients.  Note that $\bb G$ is covered by nonempty open subschemes of
the affine space over $U'$ of fiber dimension $(\mb w'^*(-\bar s)-\mb
w(\bar s))\mb w(\bar s)$, or $\bar t \mb w(\bar s)$.  Thus $\bb G$ is
covered by nonempty open subschemes of the affine space over $S$ of
fiber dimension $\mb H' + \bar t \mb w(\bar s)$ or $\mb H$.

Let's adapt the argument surrounding \eqref{eqLev1}.  Let $\mc
E^\di_{-\bar s}$ be the universal quotient of $(\wt{\mc E}'_{-\bar
s})_\bb G$.  Recall that $\wt{\mc E}' = \mc A_{U'}^\dg \big / \wt{\mc
D}'$.  So say $\mc E^\di_{-\bar s} = (\mc A^\dg_\bb G)_{-\bar s} \big/
\mc D^\di_{-\bar s}$.  Then set $\mc D^\di := \mc A_\bb G \mc
D^\di_{-\bar s} \subset \mc A^\dg_\bb G$ and $\mc E^\di := \mc A^\dg_\bb
G \big/ \mc D^\di$.  For $p < n$, let $U_p \subset \bb G$ be the
subscheme such that a map $T\to \bb G$ factors through $U_p$ iff $(\mc
E^\di_p)_T$ is locally free of rank $\mb w^*(p)$.  (Of course, $U_{-\bar
s} = \bb G$.)  Set $U := \bigcap_{p < n}U_p$, set $\wt{\mc E} := \mc
E^\di\mid U$, and set $\wt{\mc D} := \mc D^\di\mid U$.  Let's prove that
the pair $(U,\, \wt{\mc E})$ represents $\mb U_{\mb t}$.

First, note that, by construction, $\wt{\mc E} = \mc A_{U}^\dg \big/
\wt{\mc D}$ is a homogeneous quotient with $\wt{\mc E}_p$ locally free
of rank $\mb w^*(p)$ for $p < n$.  Furthermore, $\wt{\mc D}$ is of local
generator type $\mb t^*$ since $\wt{\mc D}'$ is of local generator type
$\mb t'^*$ and since $\wt{\mc D} / \wt{\mc D}'_U$ is generated by the
pullback of the universal subsheaf on $\bb G$, which is locally free of
rank $\bar t$.

Next, given any $S$-scheme $T$ and any homogeneous quotient $\mc E = \mc
A_T^\dg / \mc D$ with $\mc E_p$ locally free of rank $\mb w^*(p)$ for $p
< n$ and with $\mc D$ of local generator type $\mb t^*$, we have to show
that there's a unique $S$-map $T\to U$ with $\mc E = \wt{\mc E}_T$ as
quotients of $\mc A^\dg_T$.

If $\bar s = s$, recall $U' := S$ and $\wt{\mc E'} = \mc A^\dg$, and set
$\mc E' := \mc A^\dg_T$.  Then trivially there's a unique $S$-map $T\to
U'$, and $\mc E' = \wt{\mc E'}_T$.

If $\bar s < s$, set $\mc D' := \mc A_T\big( \bigoplus_{j>\bar s}\mc
D_{-j}\big)$ and $\mc E' : = \mc A_T^\dg / \mc D'$.  Let's see that $\mc
E'_p$ is locally free of rank $\mb w'^*(p)$ for $p < n$ and that $\mc
D'$ is of local generator type $\mb t'^*$.  As these issues are local,
we may temporarily assume that $T$ is local.  Set $m' := \sum_{q>\bar
s}\mb t(q)$ and $m := \sum_{q}\mb t(q)$.  Proceed somewhat as in the
proof of (1a).

As $\mc A$ is $\mc O_S$-Artinian, $\mc A^\dg = \mc A^*$.  But $\mc A$ is
$\mc O_S$-Gorenstein, and $T$ is local.  So $\mc A^\dg_T = \mc A_Tf$ for
some $f\in \Gamma((\mc A^\dg_{-a})_T)$.  Now, $\mb t^*$ is the local
generator type of $\mc D$.  Hence for $1\le i\le m$, there's $g_i\in
\Gamma((\mc A_{d_i})_T)$ with $d_i>0$, with $\mc D = \sum g_i\mc
A^\dg_T$, and with $\#\{\, i \mid d_i = a-q\,\} = \mb t(q)$ for all $q$.
Rearrange the $g_i$ so that $a-d_i > \bar s$ iff $i \le m'$.  Then
$g_1f,\dotsc,g_{m'}f$ form a minimal generating set of $\mc D'$.  Thus
$\mc D'$ is of local generator type $\mb t'^*$.  Now, set $M :=
\Gamma(\mc E)$ and $M' := \Gamma(\mc E')$.  Given any $T$-field $K$,
note $\mb H(M\ox K,z) = \sum_p \mb w^*(p)z^p \mod z^n$.  So
\eqref{leInj}(1) yields $\mb H(M'\ox K,z) = \sum_p \mb w'^*(p)z^p \mod
z^n$.  Hence $M'_p$ is flat over $\Gamma(\mc O_T)$ for $p < n$ by
\eqref{leInj}(2).  Thus $\mc E'_p$ is locally free of rank $\mb w'^*(p)$
for $p < n$, as desired.

By recursion, for our original $T$, there's a unique $S$-map $T\to U'$
with $\mc E' = \wt{\mc E_s}_T$ as quotients of $\mc A^\dg_T$.  But $\mc
E_{-\bar s}$ is a rank-$\mb w(\bar s)$ locally free quotient of $\mc
E'_{-\bar s}$.  Hence there's a unique $U'$-map $T\to \bb G$ with $\mc
E_{-\bar s} = (\wt{\mc E}_{-\bar s})_T$ as quotients of $(\wt{\mc
E}'_{-\bar s})_T$.  So $\mc D_{-\bar s} = (\wt{\mc D}_{-\bar s})_T$.  So
$\mc D = \wt{\mc D}_T$.  Thus $\mc E = \wt{\mc E}_T$ as quotients of
$\mc A^\dg_T$.  Thus the pair $(U,\, \wt{\mc E})$ represents $\mb U_{\mb
t}$, as desired.

Finally, set $R_p := \Supp \wt{\mc E}_p \subset U$. Set $\wh U := U -
\bigl(\bigcap_{p\ge n}R_p\bigr)$ and $\wh{\mc E} := \wt{\mc E} \mid \wh
U$.  Then a map $T\to U$ factors through $\wh U$ iff $\wt{\mc E}_p = 0$
for $p \ge n$ by \eqref{leLev}(1).  But $(U,\, \wt{\mc E})$ represents
$\mb U_{\mb t}$.  So $(\wh U,\, \wh{\mc E})$ represents the functor
whose $T$-points are the $\mc E = \mc A_T^\dg / \mc D$ with $\mc E_p$
locally free of rank $\mb w^*(p)$ for $p < n$, with $\mc E_p=0$ for $p
\ge n$. and with $\mc D$ of local generator type $\mb t^*$.
Equivalently, set $\wh{\mc D} := \wt{\mc D} \mid \wh U$; then $(\wh U,\,
\wh{\mc D})$ represents the functor whose $T$-points are the flat $\mc D
\subset \mc A^\dg_T$ with Hilbert function $\mb h^*$, with local
generator type $\mb t^*$, and with $\mc A^\dg_T/\mc D$ flat.  Or, by
\eqref{thGMD3}, Macaulay Duality, $(\wh U,\, \wh{\mc D}^*)$ represents
the functor whose $T$-points are the flat quotients $\mc C$ of $\mc A_T$
with Hilbert function $\mb h$ and $T$-socle type $\mb t$.  Thus $\wh U =
\bb H\Psi_{\mc A}^{\mb h, \mb t}$.

Note $\wh U$ is an open subscheme of $U$, so one of $\bb G$.  So $\wh U$
is covered by open subschemes of the affine space over $S$ of fiber
dimension $\mb H$.  The latter open subschemes may be taken to be
nonempty as $\bb H\Psi_{\mc A}^{\mb h, \mb t} \neq \emptyset$.  Thus
(1b) holds.

 {\bf Plainly, (1c)} follows immediately from (1a) and (1b).

{\bf For (2),} note $\mb P(z) = \prod(1-z^{d_i})$.  Now, the Hilbert
function of $\mc C_T$ is $\mb h$.  So $\dim((\mc A^\dg_{T}/\mc C_T^*)_p
= \mb a^*(p) - \mb h^*(p)$.  But $\mb h(q) := \mb a(q)-\mb w(q)$ for $q
> -n$.  Furthermore, $\sum_p\mb w^*(p)z^p := \mb H(\mc A^\dg,z) \mb
P(z)$.  Thus $\mb H((\mc A^\dg_{T}/\mc C_T^*),\, z) = \mb H(\mc
A_T^\dg,z) \mb P(z)\mod z^n$.

Let $K^\di$ be an $S^\di$-field, $T^\di := \Spec K^\di$.  But $\mb t^*$
is the local generator type of $\mc C^{\di*}$ and $\mc A_{T^\di}^\dg$ is
cyclic.  So for $1\le i\le m$, there's $g^\di_i\in \Gamma((\mc
A_{d_i})_{T^\di})$ with $\mc C_{T^\di}^{\di*} = \sum g^\di_i\mc
A^\dg_{T^\di}$.

Set $\bb P := \prod_{i=1}^m \bb P(\mc A^*_{d_i})$.  Consider the
$T$-point of $\bb P$ defined by the $g_i$.  By \eqref{prOpens}(2), it
lies in the open subset $U$ of $\bb P$ provided by \eqref{prOpens}(1).
Consider the $T^\di$-point of $\bb P$ defined by the $g^\di_i$.  Note
$\mb H(\mc A^\dg,z) \mb P(z) \le \mb H( (\mc A^\dg_{T^\di}/\mc
C_{T^\di}^{\di}*),\, z) \mod z^n$ by \eqref{prOpens}(1).  So $\mb w^*(p) \le
\mb a^*(p) - \mb h^{\di*}(p)$ for $p < n$.  So $\mb h^\di(q) \le \mb
h(q)$ for $q > -n$.  But $h^\di(q) \le \mb a(q)$ for all $q$, and $\mb
h(q) = \mb a(q)$ for $q \le -n$.  Thus (2) holds.

 {\bf For (3),} set $\mc N := \bigoplus_{i=1}^m \mc A_T(a-d_i)$, and use
the $g_i$ to define a surjection $\mc N \onto \mc C^*$.  Reorder the
$g_i$ so $d_1 = p$ and $d_2 = q$.  Then $(g_2,-g_2,0,\dotsc,0)\neq0$ in
$\Gamma(\mc N_{-r})$, but it maps to 0 in $\Gamma(\mc C^*_{-r})$.  Now,
$\dim \Gamma(\mc N_{-r}) = \mb g_{\bar s}(r)$ by \eqref{eqsbImax}.  So
$\mb g_{\bar s}(r) > \mb h(r)$.  But $\mb h^{\rm I}_{\bar s}(r) :=
\min\{\,\mb g_{\bar s}(r),\, \mb a(r)\,\}$ by \eqref{eqsbImaxh}, and
$\mb a(r) > \mb h(r)$.  Thus (3) holds.
 \end{proof}

\begin{example}\label{exMetal}
 Let's see what the theory in this section says about Migliore,
Mir\'{o}-Roig, and Nagel's Examples 2.14--2.16 on pp.\,343--345 in
\cite{MMN} of an algebra $A$ and a level $\mb t$ (that is, $ \bar s =
s$) with $\mb{H\Psi}_A^{\mb h^{\rm I}_{\bar s},\mb t} = \emptyset$.
Their A is an unspecified ``general'' complete intersection of forms of
the same degree;. They prove $A$ has a homogeneous quotient $C$ whose
Hilbert function $\mb h$ is maximal for $\mb t$, but $\mb h \neq \mb
h^{\rm I}_{\bar s}$.

By contrast, below, we construct specific such $A$ and $C$.  Now, take
$\mc M$ to be a polynomial ring over a field $k$ in \eqref{prOpens}.
Its (2) yields a scheme $U$ whose $K$-points, for any $k$-field $K$,
represent (with repetition) exactly the quotients $A^\di$ of $M_K$ with
the right Hilbert function; the sequence that defines an $A^\di$ is
regular, as $A^\di$ is Artinian.  The proof of \eqref{prOpens} provides
a flat $\mc O_U$-algebra $\mc A$ whose fibers are the $A^\di$.

Consider \eqref{thAffBdl}.  Its $\mb h$ turns out to be the Hilbert
function of our $C$; therefore, $C \in \mb H\Psi_{\mc A}^{\mb h, \mb
t}$.  So \eqref{thAffBdl}(1b) implies $\bb H\Psi_{\mc A}^{\mb h, \mb
t}$ is $U$-flat, where $U$ and $\mc A$ were introduced above.  So the
image of $\bb H\Psi_{\mc A}^{\mb h, \mb t}$ in $U$ is a nonempty open
subscheme, whose points represent (with repetition) exactly the desired
$A^\di$.  Moreover, \eqref{thAffBdl}(1b) describes the fiber over a
given $A^\di$; it's the scheme $\bb H \Psi_{A^\di}^{\mb h,\mb t}$
parameterizing all the desired quotients of $A^\di$.

There's more.  First, \eqref{thAffBdl}(1a) describes the retraction map
to $\bb H \Psi_A^{\mb h,\mb t}$ of the scheme $\bb F \Psi_A^{\mb h,\mb
t}$ of arbitrary filtered quotients of $A$ with Hilbert function $\mb h$
and socle type $\mb t$.  Second, \eqref{thAffBdl}(2) asserts $\mb h$ is
maximal among the Hilbert functions of homogeneous quotients of $A$ of
$k$-socle type $\mb t$.  Third, \eqref{thAffBdl}(3) yields $\mb h \neq
\mb h^{\rm I}_{\bar s}$.

For Migliore et al.'s 2.14--2.16, respectively take $r := 3,\,4,\,3$ and
$e := 2,\,2,\,3$ in the following setup.  Let $k$ be any field, $R :=
k[X_1,\dotsc,X_r]$ the polynomial ring with $r \ge 2$.  Fix $e \ge 2$.
Set $f := X_1^{-e}\dotsb X_r^{-e}$ and $I := (0:_Rf)$; set $A := R/I$.
Plainly, $X_1^{i_i}\dotsb X_r^{i_r} \in I$ iff $i_j \ge e+1$ for some
$j$; hence, $I = (X_1^{e+1},\dots,X_r^{e+1})$.  Thus $A$ is a complete
intersection; its socle degree $a$ is $er$, and its initial degree $v_0$
is $e+1$.  Set $n := v_0-a = e+1-a$.

Fix $t\ge1$.  Take $\mb t(p) := t$ for $p = a-1$ and $\mb t(p) := 0$ for
$p\neq a-1$.  Migliore et al.'s examples have $t=2$; so $t = r-1$ in
2.14 and 2.16.  Let's do the general case: $t = r-1$.  Use the notation
of \eqref{prNonempty} and its proof.  So $\mb P(z) = (1-z)^t$, and $J:=
\sum_{j=1}^t RX_j$, and $D = Jf\subset A^*$.  But $1-a < n-1$ as $1 <
e$; so $\mb t(1-n) = 0$.  Plainly $R/J = k[X_r]$.  So $R/(I+J) =
k[X_r]/(X_r^{e+1})$.  However, $R/I = A^*(-a)$; hence, $R/(I+J) =
(A^*/D)(-a)$.  Thus $(A^*/D)_p = 0$ for $p\ge e+1-a =: n$.  Now, $\mb
H(R/J,\,z) = \sum_p\mb w^*(p-a)z^p \mod z^{v_0}$ by an argument in the
proof of \eqref{prNonempty}.  But $J_p = (I+J)_p$ for $p < v_0$.  So
$\dim (A^*/D)_p = \mb w^*(p)$ for $p < n$.  Thus the Hilbert function of
$D$ is $\mb h^*$, so that of $C := D^*$ is $\mb h$, as desired.

Migliore et al.'s 2.15 is a little different, requiring more attention
be paid to $J$ and to $\car k$.  Of course, $J:= Rg_1+Rg_2$ where
$g_1,\,g_2$ are linear forms, as $t=2$.  However, if $g_j := X_j$, then
$X_3X_4^2,\,X_3^2X_4 \notin (I+J)$; so $D_n \neq A^*_n$, where $n = -5$.

Fix $b \in k$, take $g_1 := X_1+X_2+X_3$ and $g_2:= X_2+b
X_3+X_4$, and define a map $u\: R_2\oplus R_2 \to A_{-5}^\dg$ by $u(x,y)
:= (xg_1+yg_2)f$.  It's straightforward to write out a 16 by 20 matrix
representing $u$; use as basis of $R_2$ its monomials, and as basis of
$A_{-5}^\dg$ the products $X_1^{i_1}\dotsb X_4^{i_4}f$ with $0\le
i_1,\dotsc,i_4 < 3$ and $\sum i_j = 3$.  Order the two bases
lexicographically.  Then Macaulay~2 (with {\tt ZZ[b]} as ring)
finds the first nonvanishing maximal minor to be $9b(1-b)$.
Thus, if $\car k\neq 3$ and $b \neq 0,\,1$, then $u$ is surjective,
and so $C \in \mb H\Psi_{ A}^{\mb h, \mb t}$.

The case $k = \Z/2$ is special.  Of course, there's no $b\in k$ with $b
\neq 0,\,1$; however, there is such a $b$ in any proper extension field.
So the $k$-scheme $\bb H \Psi_A^{\mb h,\mb t}$ is nonempty.
(Moreover, it has no $k$-points; indeed, Macaulay~2 was used to check
each of the six essentially different choices of the pair $(g_1,g_2)$.)
Thus \eqref{thAffBdl} still applies.

Assume $\car k = 3$.  Say $g_i = \sum_{j=1}^4 a_{ij}X_j$ for $i=1,\ 2$.
Set $\Delta := a_{11}a_{22} - a_{12}a_{21}$.  Note $g_i^3 = \sum
a_{ij}^3X_j^3$; so $\Delta^3 X_1^3,\ \Delta^3 X_2^3 \in
\bigl(g_1,g_2,X_3^3,X_4^3\bigr)$.  Assume $\Delta \neq 0$.  Then,
therefore, $I+J = \bigl(g_1,g_2,X_3^3,X_4^3\bigr)$.  Hence $R/(I+J)$ is
of socle degree $4$.  Hence $\dim (A^*/D)_p \neq 0$ for $p = n = -5$ and
$p = n+1 = -4$.  Denote the Hilbert function of $D$ by $\mb h_\Delta^*$.
Thus $\mb h_\Delta^*(p) \neq \mb h^*(p)$ for $p = -5,\,-4$.  (Moreover,
it's not hard to see using \eqref{prOpens} that $\mb h_\Delta^*$ is
maximal for $A$ and $\mb t$.) Thus $\bb H \Psi_A^{\mb h,\mb t} =
\emptyset$, although as seen above, $\bb H \Psi_A^{\mb h,\mb t} \neq
\emptyset$ when $\car k \neq 3$.  (In Ex.\,2.12 on p.\,340, Migliore et
al.\ work a similar example in characteristic 2 with $r := 3$ and $e :=
3$ and $t:=1$.)
 \end{example}

\begin{remark}\label{reFrbg}
 If true, Fr\"oberg's conjecture provides another major case where
  \indt{Fr\"oberg's conjecture}
\eqref{thAffBdl} applies; namely, $A$ is an Artinian general complete
intersection over an infinite field $k$.  Fr\"oberg's conjecture is
suitably formulated in Migliore, Mir\'{o}-Roig, and Nagel's Conjecture
2.7 on p.\,339 in \cite{MMN} and in Nenashev's Conjecture 1 on p.\,273
in \cite{Nen}, with $A$ not necessarily Artinian.  The conjecture has
been proved in three instances: (1)~when $A$ has two or three
generators, (2)~when $\car k = 0$ and $A$ is a complete intersection, or
almost so, and (3)~wnen $k$ is an algebraically closed of characteristic
0 and $A$ is a polynomial ring modulo forms of the same degree.  The
proofs are generally attributed to Stanley, Fr\"oberg, Anick,
Hocster--Laksov, Aubry, and Nenashev; see \cite{Nen}*{p.\,275}.

Our treatment is geared to the setup of \eqref{thAffBdl}.  However, in
essence, it's the same as Migliore, Mir\'{o}-Roig, and Nagel's treatment
of Conjecture 2.9 and Remark 2.10 on p.\,340, although they don't assume
$A$ is Artinian.  The two treatments differ in one notable way: we use
Macaulay Duality, whereas they use linkage in Lemma 2.4 on p.\,338.
Nevertheless, this difference is superficial owing to \eqref{coLink}(2).

So say $A = R/I$ with $R := k[X_1,\dotsc,X_r]$ the polynomial ring,
$\deg X_i =1$, and $I$ is generated by $r$ general forms of degrees
$d_1,\dotsc,d_r$.  Note $\mb H(R,z) = 1/(1-z)^r$.  Set $\mb Q(z) :=
\prod_{i=1}^r (1-z^{d_i})$.  Thus \eqref{leInj}(1) yields $\mb H(A,z) =
\mb Q(z)\big/(1-z)^r$.

Set $a := \sum (d_i-1)$; so $a$ is the socle degree of $A$.  Choose an
$f\in A^\dg_{-a}\subset R^\dg_{-a}$ with $A^\dg = Af$.  Then $A(a)\risom
A^\dg$ by $g\mapsto gf$.  Thus $\mb H(A^\dg,z) = \mb H(A,z)z^{-a}$.

Set $\mb P(z) := \prod_j(1-z^j)^{\mb t(a-j)}$.  Define $\mb w^*(p)$ by
$\sum_p\mb w^*(p)z^p := \mb H(A^\dg,z) \mb P(z)$.  Set $n := \sup\{\,
m\mid \mb w^*(p) \ge 0 \text{ for }p < m\,\}$. Assume $\mb t^*(p)=0$ for
$p\ge n$ and $p \le -a$.

For all $i$, fix general forms $g_{ij}\in R_i$ for $j=1, \dotsc, \mb
t(a-i)$.  Set $J:= \sum_{i,j}Rg_{ij}$.  Then Fr\"oberg's conjecture
asserts $\mb H(R/(I+J),z) = \mb P(z)\mb Q(z)\big/(1-z)^r \mod z^{a+n}$
and $R_p = (I+J)_p$ for $p = a+n$.  Note $\mb P(z)\mb Q(z)\big/(1-z)^r =
\sum_p\mb w^*(p)z^{p+a}$.

Set $D := JA^\dg$.  Then $(R/(I+J))(a) \risom A^\dg/D$.  So Fr\"oberg's
conjecture yields $\dim (A^\dg/D)_p = \mb w^*(p)$ for $p < n$ and
$(A^\dg/D)_p = 0$ for $p= n$; so $(A^\dg/D)_p = 0$ for $p \ge n$, as
$R_1\cdot A^\dg_p = R_1\cdot R_{p+a}f = A^\dg_{1+p}$.  Moreover, the
conjecture implies that, if one of the $g_{ij}$ is omitted, then at
least one value of $\dim (A^\dg/D)_p$ changes, since $\mb t^*(p)=0$ for
$p\ge n$; hence, the products $g_{ij}f$ form a minimal generating set of
$D$, and so the (local) generator type of $D$ is $\mb t^*$.

Finally, set $C := D^*$.  Set $a(q) := \dim A_q$, and set $\mb h(q) :=
\mb a(q)-\mb w(q)$ for $q > -n$, and\/ $\mb h(q) := \mb a(q)$ for $q \le
-n$.  Then $\dim C_q = \mb h(q)$.  Therefore, $C\in \bb H\Psi_A^{\mb
h, \mb t}$.  Thus \eqref{thAffBdl} applies; so it describes the geometry
of $\bb H \Psi_A^{\mb h,\mb t}$ and  $\bb F \Psi_A^{\mb h,\mb t}
\onto \bb H \Psi_A^{\mb h,\mb t}$, and it asserts $\mb h$ is maximal,
but often $\mb h \neq \mb h^{\rm I}_{\bar s}$.
 \end{remark}

\newpage

\indexprologue[\vspace{-12pt}]{}
\printindex[terminology]

\indexprologue[\vspace{-12pt}]{}
\printindex[notation]

 \end{document}